\documentclass{amsart}
\usepackage[all]{xy}
\usepackage{multirow,mathrsfs,upgreek}
\usepackage{pgfplots}
\usepackage{amssymb}
\usepackage{hyperref}
\usepackage[margin=2cm]{geometry}

\author{Yuri~G.~Prokhorov}
\address{Steklov Mathematical Institute of Russian Academy of Sciences, ul.~Gubkina~8, Moscow, 119991, Russia}
\address{Faculty of Mechanics and Mathematics, Lomonosov Moscow State University, Moscow, 119991, Russia}
\address{Laboratory of Algebraic Geometry and Its Applications, HSE University, ul.~Usacheva~6, Moscow, 119048, Russia}
\email{prokhoro@mi-ras.ru}

\title{Fano Threefolds}


\makeatletter
\@addtoreset{equation}{section}
\@addtoreset{footnote}{section}
\@addtoreset{teo}{section}
\makeatother

\newtheorem{teo}{Theorem}[section]
\newtheorem{cor}[teo]{Corollary}
\newtheorem{prp}[teo]{Proposition}
\newtheorem{lem}[teo]{Lemma}
\newtheorem{claim}[teo]{Claim}

\theoremstyle{definition}
\newtheorem{rem}[teo]{Remark}

\newtheorem{notation}[teo]{Notation}
\newtheorem{examples}[teo]{Examples}
\newtheorem{construction}[teo]{Construction}
\newtheorem{dfn}[teo]{Definition}
\newtheorem{exa}[teo]{Example}
\newtheorem*{direct}{Bibliographic indications}

\newcounter{eitem}\numberwithin{eitem}{section}
\newenvironment{zadachi}{\subsection*{Exercises}\setcounter{eitem}{0}}{}
\def\eitem{\par\smallskip\refstepcounter{eitem}{\bf\arabic{eitem}.}\hskip2mm\ignorespaces }

\newcommand{\hint}[1]{\par{\textit{Hint}: #1\par}}

\newenvironment{proofsk}{\begin{proof}}{\end{proof}}

\def\K3{\textup{K3}}
\newcommand{\type}[1]{$\mathrm #1$}
\newcommand{\Type}[1]{\mathrm #1}

\newcommand{\heading}[1]{\multicolumn1{c|}{#1}}
\newcounter{NNN}\numberwithin{NNN}{table}\renewcommand{\theNNN}{\arabic{NNN}\textdegree}
\newcommand\rownumber{\refstepcounter{NNN}\theNNN}

\newcounter{NN}\numberwithin{NN}{table}\renewcommand{\theNN}{\arabic{NN}\textdegree}
\def\nr{\refstepcounter{NN}{\theNN}}

\newcounter{NNr}\renewcommand{\theNNr}{\arabic{NNr}$^\#$}
\def\nrr{\refstepcounter{NNr}\theNNr}

\newcommand{\xarr}[1]{\xrightarrow{\,#1\,}}
\newcommand{\xdashrightarrow}[1]{\overset{#1}{\dashrightarrow}}
\newcommand{\xmapsto}[1]{\mathrel{\rule[.2pt]{.5pt}{5pt}\kern-1pt{\xrightarrow{#1}}}}
\newcommand{\approxident}{\mathrel{\vcenter{\offinterlineskip\hbox{$\sim$}\vskip-.35ex\hbox{$\sim$}\vskip-.35ex\hbox{$\sim$}}}}
\newcommand{\mo}{}
\newcommand{\A}{\mathbb A}
\newcommand{\FF}{\mathbb F}
\newcommand{\CC}{\mathbb C}
\newcommand{\PP}{\mathbb P}
\newcommand{\RR}{\mathbb R}
\newcommand{\QQ}{\mathbb Q}
\newcommand{\ZZ}{\mathbb Z}
\newcommand{\OO}{\mathbb O}

\newcommand{\PPP}{\mathscr{P}}
\newcommand{\MMM}{\mathscr{M}}
\newcommand{\NNN}{\mathscr{N}}
\newcommand{\EEE}{\mathscr{E}}
\newcommand{\OOO}{\mathscr{O}}
\newcommand{\LLL}{\mathscr{L}}
\newcommand{\III}{\mathscr{I}}
\newcommand{\FFF}{\mathscr{F}}
\newcommand{\JJJ}{\mathscr{J}}
\newcommand{\TTT}{\mathcal T}%

\newcommand{\rZ}{\mathrm{Z}}
\newcommand{\rC}{\mathrm{C}} 
\newcommand{\rR}{\mathrm{R}} 
\newcommand{\Sym}{\mathrm{S}}
\newcommand{\R}{\mathrm{R}} 
\newcommand{\N}{\mathrm{N}}
\newcommand{\J}{\mathrm{J}}
\newcommand{\g}{\mathrm{g}}
\newcommand{\di}{\mathrm{d}}
\newcommand{\dd}{\mathrm{d}}
\newcommand{\hr}{\mathrm{h}}
\newcommand{\const}{\operatorname{const}}
\newcommand{\Exc}{\operatorname{Exc}}
\newcommand{\Gr}{\operatorname{Gr}}
\newcommand{\ord}{\operatorname{ord}}
\newcommand{\mult}{\operatorname{mult}}
\newcommand{\deff}{\operatorname{def}}
\newcommand{\SL}{\operatorname{SL}}
\newcommand{\PSL}{\operatorname{PSL}}
\newcommand{\Hom}{\operatorname{Hom}}
\newcommand{\Supp}{\operatorname{Supp}}
\newcommand{\LCS}{\operatorname{LCS}}
\newcommand{\codim}{\operatorname{codim}}
\newcommand{\Cl}{\operatorname{Cl}}
\newcommand{\GL}{\operatorname{GL}}
\newcommand{\pr}{\operatorname{pr}}
\newcommand{\Pic}{\operatorname{Pic}}
\newcommand{\Bs}{\operatorname{Bs}}
\newcommand{\Sing}{\operatorname{Sing}}
\newcommand{\Eff}{\operatorname{Eff}}
\newcommand{\divi}{\operatorname{div}}
\newcommand{\Aut}{\operatorname{Aut}}

\newcommand{\pt}{\operatorname{pt}}
\newcommand{\rk}{\operatorname{rk}}
\newcommand{\NE}{\overline{\operatorname{NE}}}

\newcommand{\lcm}{\operatorname{lcm}}
\renewcommand{\gcd}{\operatorname{gcd}}

\newcommand{\Lines}{\operatorname{F_1}}
\newcommand{\Conics}{\operatorname{F_2}}
\newcommand{\Univ}{\operatorname{\mathfrak{U}}}

\newcommand{\p}{\mathrm{p}_{\mathrm{a}}}
\newcommand{\chit}{\upchi_{\mathrm{top}}}
\newcommand{\comp}{\circ}

\newcommand{\qq}{\mathbin{\sim_{\scriptscriptstyle{\QQ}}}}
\newcommand{\mumu}{\boldsymbol\mu}

\begin{document}

\begin{titlepage}
 \begin{center}
 \vspace*{1cm}

 {\huge \textbf{Fano threefolds}}

 \vspace{0.5cm}

 \vspace{1.5cm}

 {\normalsize \textbf{Yuri~G.~Prokhorov}}

 \vfill

 \vspace{0.8cm}
 
Steklov Mathematical Institute of Russian Academy of Sciences
\\
Faculty of Mechanics and Mathematics, Moscow Lomonosov State University
\\
Laboratory of Algebraic Geometry and Its Applications, HSE University

 2022-2024
 
 \end{center}
\end{titlepage}

\begin{abstract} 
The goal of these lecture notes is to present
the modern point of view on the classification of Fano threefolds.
We tried to offer a self-consistent treatment of the topics covered.
\par\medskip\noindent
These  notes  have been published in two versions: a Russian edition in 
\textit{Lektsionnye Kursy NOTs} \textbf{31}.  Steklov Inst. Math., Moscow (ISBN 978-5-98419-085-5),
doi: \href{https://doi.org/10.4213/lkn31}{10.4213/book1907}, and an English translation in the
\textit{Proc. Steklov Inst. Math.}, \textbf{328}, Suppl. 1 (2025), doi: \href{https://doi.org/10.1134/S0081543825020014}{10.1134/S0081543825020014} 
\end{abstract}

\maketitle

\medskip
\setcounter{tocdepth}{2}
\tableofcontents
\medskip

\newpage\section*{Introduction}

Fano varieties are projective algebraic varieties with ample anticanonical class.
The classification of such varieties was initiated in the works of Gino Fano
\cite{Fano1931,Fano1942}. More precisely, he considered three-dimensional varieties embedded
to a projective space so that their curve sections by codimension two 
linear subspaces are canonical curves. One can show that this definition is almost 
equivalent to the modern one: such varieties are exactly Fano varieties with \textit{very ample} 
anticanonical class.

With the appearance of the minimal models theory it became clear that Fano varieties are very important, since they form
an essential class in the birational classification of varieties of negative Kodaira
dimension~\cite{Mori:3-folds,KMM,KM:book}.

This publication is devoted to the classification of three-dimensional Fano varieties.
The notes do not pretend to be new or original at all.
The aim of the author was to systematize and present consistently the classification
of Fano threefolds of Picard number~$1$. This classification
is currently dispersed across numerous sources, some of them contain gaps and inaccuracies.
Of course, the author does not claim that this edition does not have similar problems.
Historical experience shows that almost all sufficiently complex classification results contain gaps.
The fact that errors have been detected indicates that the result is in demand, 
and the error correction is a non-easy continuous work.

The main classification result that we are going to prove can be summarized in the following theorem.

\begin{teo}

\label{theorem:main}
Let $X$ be a nonsingular Fano threefold of Picard number $\uprho(X)=1$ over an algebraically closed
field of characteristic $0$
and let $\iota(X)$ be its Fano index, i.e. 
the maximal natural number that divides 
the canonical class in the group $\Pic(X)$ \textup(see Definition~\ref{def:index}).
Then the following assertions hold.
\begin{itemize}
\item
If $\iota(X)\ge 3$, then $X$ is isomorphic to either the projective space $\PP^3$
or a nondegenerate quadric $Q\subset\PP^4$.
\item
If $\iota(X)=2$, then there exists exactly $5$ deformation types of such
Fano threefolds.
They are called \textsl{del Pezzo threefolds} and listed in Table~\ref{table:del-Pezzo-3-folds},
page~\textup{\pageref{table:del-Pezzo-3-folds}}.

\item
If $\iota(X)=1$, then there exists exactly $10$ deformation types of such
Fano threefolds.
They are listed in Table~\ref{table-main}, page~\textup{\pageref{table-main}}.
\end{itemize}
\end{teo}

Note that in the Mori theory~\cite{Mori:3-folds,KM:book}
singular Fano varieties (with terminal, canonical, or log terminal singularities) also should be considered.
However, due to lack of time we completely ignore this request: \textit{all
our Fano varieties are nonsingular}.

This text compiled from the detailed lecture notes delivered at the Scientific and Educational Center of MAIN
in the fall of 2017, as well as, at the Center for Geometry and Physics of the Pohang University (South Korea) in the spring of
2018. I am very grateful to the listeners of my lectures for their attention, patience,
correct timely questions, and comments.

The final version of these notes was prepared during my stay at the Center for 
Geometry and Physics. I would like to thank this institute and personally Jihun
Park for the invitation and excellent working conditions.

I also would like to express my gratitude to A.~G.~Kuznetsov for thorough reading of a 
preliminary version of the manuscript and numerous valuable comments and suggestions.

\begin{direct}
The author has learned most of the material presented here from the book
\cite{Iskovskikh1988}, original
papers~\cite{Isk:Fano1e,Isk:Fano2e,Isk:anti-e}, and
a course of lectures delivered
by V.~A.~Iskovskikh at the Faculty of Mechanics and Mathematics of Moscow State 
University in 1988. 
A later source, written in modern language is the book~\cite{IP99}. 
Detailed references are provided in the relevant sections.
\end{direct}

\newpage\section{Notation, preliminary concepts and the simplest properties of Fano 
varieties}
\label{sec1}

Throughout is text we assume that the ground field~is the field of complex numbers $\CC$.
\footnote{For the case of positive characteristic we refer to the recent series of papers \cite{Tan1,Tan2,Tan3,Tan4}.}

\subsection{Notation}

\begin{itemize}
\item
$\p(X):= (-1)^{\dim(X)}\left(\upchi (\OOO_X)-1\right)$ is the arithmetic genus 
of a variety $X$.
\item
$\Pic(X)$ denotes, as usual, the Picard group, i.e. the group of 
invertible sheaves on $X$ up to isomorphism.

\item
$\Cl(X)$ is the Weil divisor class group of a normal variety $X$.
We assume that there exists an identification of $\Pic(X)$ with the group of 
classes of
Cartier divisors modulo
linear equivalence. Thus $\Pic(X)\subset\Cl(X)$.

\item
$\sim$ and $\approxident$ denote the linear and numerical equivalences,
respectively.

\item
$\qq$ denotes the $\QQ$-linear equivalence of Weil divisors or $\QQ$-divisors,
i.e. $D_1\qq D_2$ if there exists $n\in \ZZ\setminus \{0\}$ such that
$nD_1$ and $nD_2$ are integral divisors and $nD_1\sim nD_2$.

\item
$\mathrm{NS}(X):= \Pic(X)/\Pic^0(X)$ is the Neron--Severi group.
\item
$\uprho(X):= \rk\mathrm{NS}(X)$ is the Picard number.
\item
$\varkappa(X)$ is the Kodaira dimension of a variety $X$.
\item
$K_X$ denotes the canonical class of a variety $X$.
\item
$\TTT_X$ is the tangent bundle to a variety $X$.
\item
$\hr^{p,q}(X):= \dim H^q(X,\Omega_X^p)$ are Hodge numbers.
\item
$\NNN_{Z/X}$ is the normal bundle to a subvariety $Z$ in $X$.
\item
$\PP_X(\EEE)$ is the projectivization of a vector bundle $\EEE$ on a variety
$X$,
i.e. the relative projective spectrum $\operatorname{Proj} \mathscr{S}$
of the symmetric algebra $\mathscr{S}:= \oplus_{d\ge 
0}\operatorname{S}^d(\EEE)$
(we use Grothendieck's definition, see~\cite[Ch.~II,~\S7]{Hartshorn-1977-ag}).
\item
$\FF_e=\PP_{\PP^1}(\OOO_{\PP^1}\oplus \OOO_{\PP^1}(-e))$ is the rational
ruled surface~\cite[Ch.~V, \S2]{Hartshorn-1977-ag};
by $\Sigma$ and $\Upsilon$ we usually denote its minimal section and fiber,
respectively.
Sometimes $\Sigma$ and $\Upsilon$ also denote the corresponding classes of in 
the group~$\Pic(\FF_e)$.
\item
$\Phi_{|D|}\colon X \dashrightarrow \PP^N$ is the rational map given by the
linear system
$|D|$, where $N=\dim|D|$.
\item
$T_{P,X}:= (\mathfrak{m}_{P,X}/\mathfrak{m}_{P,X}^2)^\vee$ is the Zariski 
tangent
space 
to a variety $X$ at a point $P$. If $X$ is embedded to a projective
space $\PP^N$,
then $T_{P,X}$ will be identified with an affine subspace in the corresponding
affine chart $\A^N\subset\PP^N$
and by $\overline{T_{P,X}}\subset\PP^N$ we denote its closure.
\item
$\PP (w_0,\dots,w_n)$ is the weighted projective space with weights $w_i$
(see~\ref{wps}).
\item
$\lceil \alpha \rceil$ (respectively, $\lfloor \alpha \rfloor$) denotes
upper (respectively, lower)
integral part of a real number $\alpha$. For a $\QQ$-divisor $B=\sum b_iB_i$
we put
$$
\lceil B \rceil:= \sum \lceil b_i \rceil B_i,\qquad \lfloor B \rfloor:= \sum
\lfloor b_i \rfloor B_i.
$$
\end{itemize}

\textit{A contraction} is a surjective projective morphism of normal
varieties whose fibers are connected.

\subsection{Definition of Fano varieties}

\begin{dfn}
A nonsingular projective variety $X$ is called \textit{Fano variety},
if its anticanonical divisor $-K_X$ is ample.
\end{dfn}

If $X$ is a curve of genus $g$, then $\deg K_X=2g-2$.
Thus $X$ is a Fano variety if and only if $g=0$.
In other words, the only 
one-dimensional Fano variety is the projective line.
Two-dimensional Fano varieties are traditionally called \textit{del Pezzo 
surfaces}. The classification of del Pezzo surfaces is well-known
and can be found in many textbooks (see 
e.~g.~\cite{Manin:book:74,Dolgachev-ClassicalAlgGeom,Prokhorov-re-rat-surf}).
We expect that the reader has already become familiar with this theory.

\begin{exa}
Let $X$ be a nonsingular projective variety of Picard number $\uprho(X)=1$
(e.g., a hypersurface of dimension $\ge 3$ in a projective space).
In this case for the canonical class there are only three possibilities:
\begin{enumerate}
\item
$K_X$ is ample, then $X$ is a variety \textit{of general type},
\item
$K_X$ is numerically trivial, then $X$ is a variety \textit{of numerical 
Calabi--Yau type} 
(it follows from the Abundance Theorem in this case that $K_X$ is a torsion 
element in the group $\Pic(X)$),
\item
$-K_X$ is ample, then $X$ is a Fano variety.
\end{enumerate}
Unlike Fano varieties, varieties of general type form a much more extensive 
class, in fact they are not bounded in the sense of moduli spaces.
\end{exa}

\subsection{Simplest properties}

Now we present some simplest properties of Fano varieties that follow directly
from the definitions and standard topological and algebro-geometric facts.

\begin{teo}
\label{thm-1-prop}
Let $X$ be a Fano variety.
Then
\begin{enumerate}
\item
\label{1-properties-kodaira}
$H^q(X, \OOO_X(mK_X))=0$ for $m\le 0$, $q>0$ and for $m\ge 1$, $q<\dim(X)$;
\item
\label{1-properties-irregularity}
$H^q(X,\,\OOO_X)=H^0(X,\Omega_X^q)=0$ for $q>0$, in particular,
$\upchi(\OOO_X)=1$;

\item
\label{1-properties-pi1}
$\uppi_1^{\mathrm{alg}}(X)=\{1\}$;

\item
\label{1-properties-H1}
$H^1(X,\ZZ)=H_1(X,\ZZ)=0$;

\item
\label{1-properties-Pic}
$\Pic(X)\simeq H^2(X,\ZZ)$ is a finitely generated free abelian group;

\item
\label{1-properties-nef}
if $M$ is a nef divisor on $X$, then the linear system $|nM|$
has no fixed components and base points for some $n>0$;
\item
\label{1-properties-aut} the automorphism group $\Aut(X)$ of $X$ is a
linear algebraic
group.
\end{enumerate}
\end{teo}

It is known also that any Fano variety is simply connected, i.e. its 
topological
fundamental group $\uppi_1(X)$ is trivial.
However, this is already a non-trivial result (see 
\cite[Theorem~3.5]{Campana1991a} and
\cite{Campana1992}).

\smallskip

\begin{proof}
The properties~\ref{1-properties-kodaira}
and~\ref{1-properties-irregularity}
follows from the Kodaira Vanishing Theorem, Serre duality and 
Hodge theory.

To prove~\ref{1-properties-pi1} it is sufficient to show that
$X$ has no non-trivial finite \'etale covers.
Assume, the converse: there exists an \'etale cover $\pi\colon X'\to X$
of degree $n>1$. Then the divisor $-K_{X'}=-\pi^*K_X$ is ample.
Therefore, $X'$ is a Fano variety. On the other hand,
$$
\upchi(X', \OOO_{X'})=n\upchi(X, \OOO_X)=n>1.
$$
This contradicts
\ref{1-properties-irregularity}.

To prove the assertion
\ref{1-properties-H1} we note that as $H^1(X,\,\OOO_X)=0$, the Hodge theory 
implies that $H^1(X,\CC)=0$.
On the other hand, according to the universal coefficient formulas
and the assertion~\ref{1-properties-pi1} the groups $H^1(X,\ZZ)$ and 
$H_1(X,\ZZ)$ 
have no torsions.

Let us prove~\ref{1-properties-Pic}. From the exponential sheaf sequence
\begin{equation}
\label{eq:exp-seq}
0\longrightarrow\ZZ \xarr{2\uppi \mathrm{i}} \OOO_X \xarr{\exp}
\OOO_X^* \longrightarrow 0
\end{equation}
and the property~\ref{1-properties-irregularity} we obtain an isomorphism 
$\Pic(X)\simeq
H^2(X,\ZZ)$.
Since $X$ is a compact (in the Hausdorff topology) variety, 
the last group is finitely generated and is torsion free
according to~\ref{1-properties-H1}.

The assertion
\ref{1-properties-nef} is an immediate consequence of the Base Point Free 
Theorem~\ref{th:mmp:bpf}.

Finally, the assertion
\ref{1-properties-aut} follows from the fact that
the linear system $\mo|-nK_X|$ for some $n>0$ defines an embedding
$X\hookrightarrow \PP^N$
which must be $\Aut(X)$-equivariant. Therefore, all the
automorphisms of $X$ are induced by the linear transformations of $\PP^N$
and so the group $\Aut(X)$ consists of elements of
$\operatorname{PGL}_{N+1}(\CC)$
preserving $X$.
\end{proof}

Note that the properties~\ref{thm-1-prop} are completely not obvious in positive 
characteristic.
However there are some partial results in low dimensions (see~\cite[\S 
III.3]{Kollar-1996-RC} and~\cite{Shepherd-Barron1997}).\footnote{See also the series of papers \cite{Tan1,Tan2,Tan3,Tan4}}

\begin{cor}
\label{cor:Omega}
Let $X$ be a Fano variety and let $X\dashrightarrow Z$ be a dominant
rational map to a nonsingular projective variety $Z$.
Then for all $q>0$ we have
$$
H^0(Z, \Omega_Z^q)=H^q(Z,\,\OOO_Z)=0.
$$
In particular, if $Z$ is a curve, then $Z\simeq \PP^1$.
\end{cor}

\begin{proof}
Indeed, as the map $X\dashrightarrow Z$ is dominant, it
defines an embedding
$$
H^0(Z, \Omega_Z^q)\subset H^0(X, \Omega_X^q)
$$
(see e.~g.~\cite[Ch.~III,~\S6,~Theorem~2]{Shafarevich:basic}).
\end{proof}

\begin{cor}
\label{cor:delPezzo:}
Let $X$ be a del Pezzo surface. Then
\begin{enumerate}
\item
\label{delPezzo:rational}
the surface $X$ is rational,
\item
\label{delPezzo:Noether}
$$
K_X^2+\uprho(X)=10.
$$
In particular, $1\le K_X^2\le 9$.
\end{enumerate}
\end{cor}

\begin{proof}
The assertion~\ref{delPezzo:rational} follows from the Castelnuovo rationality criterion
and the vanishing
$$
H^1(X,\,\OOO_X)=H^0(X,\,\OOO_X(2K_X))=0.
$$

To prove 
\ref{delPezzo:Noether}
we note that
$$
\chit(X)=2+\rk H^2(X,\ZZ)=2+\uprho(X)
$$
by
Theorem~\ref{thm-1-prop}\ref{1-properties-H1}--\ref{1-properties-Pic} and 
the Poincar\'e duality. Now~\ref{delPezzo:Noether}
follows from Theorem~\ref{thm-1-prop}\ref{1-properties-irregularity}
and Noether's formula
$$
K_X^2+\chit(X)=12\upchi(X, \OOO_X).\qedhere
$$
\end{proof}

Fano varieties of dimension greater than~$2$ can be non-rational.
However they are \textit{rationally connected}~\cite{Campana1992,Kollar-1996-RC}.
This is a much deeper result than the properties~\ref{thm-1-prop}.
It is based on the deformation and splitting techniques for rational curves (bend 
and
break~\cite{Mori:3-folds,Kollar-1996-RC}).
This
technique is not discussed in the present course. However, we will give one important consequence.

\begin{teo}[{\cite[Corollary~VI.3.8]{Kollar-1996-RC}}]
\label{th:Omega}
Let $X$ be a Fano variety. Then for any $q>0$ the following equality holds
$$
H^0\left(X,(\Omega_X^1)^{\otimes q}\right)=0.
$$
\end{teo}

\begin{proofsk}
Let us use the rational connectedness of $X$.
So, there exists a family $\mathcal C$ of rational curves such that for any
two general points $P_1,P_2\in X$ there exists a curve $C\in\mathcal C$ passing
through $P_1$ and $P_2$.
Let $C\subset X$ be a general rational curve from this family $\mathcal
C$ and let
$$
\nu\colon \PP^1\longrightarrow C\subset X
$$
be the normalization.
It can be shown that the vector bundle $\nu^*\TTT_X$ is ample. Therefore, 
$\nu^*\Omega_X^1$
is a sum of line bundles of negative degree.
Thus
$$
H^0(\PP^1, \nu^*(\Omega_X^1)^{\otimes q})=0
$$
for any $q>0$, i.e. any global section of the vector bundle
$(\Omega_X^1)^{\otimes q}$ vanishes along $C$. Since our
family is dominant, this section vanishes everywhere.
\end{proofsk}

\subsection{Index of Fano varieties}

\begin{dfn}
\label{def:index}
The maximal integer $\iota=\iota(X)$ such that
\begin{equation}
\label{eq:def:index}
-K_X=\iota H
\end{equation}
for some
ample divisor $H$ is called the \textit{Fano index} of a Fano variety~$X$.
Here the self-intersection number
$H^{\dim(X)}$ is called the \textit{degree} of a Fano variety~$X$.
We will denote it by~$\dd(X)$.
\end{dfn}
Note that by Theorem~\ref{thm-1-prop}\ref{1-properties-Pic}
the class of the divisor $H$ in~\eqref{eq:def:index} is uniquely defined.
Sometimes it is called the \textit{fundamental divisor}.

\begin{exa}[(see e.~g. {\cite[Proposition~6.2.1]{Prokhorov-re-rat-surf}})]

\label{ex-del-Pezzo}
Let $X$ be a del Pezzo surface. Then $\iota(X)\le 3$.
If $\iota(X)=3$, then $X\simeq \PP^2$.
If $\iota(X)=2$, then $X\simeq \PP^1\times \PP^1$.
\end{exa}

\subsection{Divisors on Fano varieties}

\begin{prp}
\label{Fano:adjunction}
Let $X$ be a Fano variety of dimension $n\ge 3$ and index $\iota=\iota(X)$.
Write $-K_X=\iota H$ and let
$Y\in|mH|$ be a nonsingular element.
\begin{enumerate}
\item
\label{Fano:adjunction1}
If $m<\iota$, then $Y$ is a Fano variety of index $\iota-m$.
\item
\label{Fano:adjunction2}
If $m=\iota$, then $Y$ is a Calabi--Yau variety. In particular, if
$\dim(X)=3$, then $Y$ is a \K3 surface.
\end{enumerate}
\end{prp}

\begin{proof}
\Ref{Fano:adjunction1} 
By the adjunction formula the divisor $-K_Y=(\iota-m)H|_Y$ is ample and
$Y$ is a Fano variety whose index is divisible by $\iota-m$.
On the other hand, by the Lefschetz hyperplane theorem
$\Pic(X)$ is embedded to $\Pic(Y)$ as a primitive sublattice.
Therefore, $H|_Y$ is a primitive element of $\Pic(Y)$.

\ref{Fano:adjunction2}
Again by the adjunction formula $K_Y=0$. From the exact sequence
$$
0\longrightarrow\OOO_X(K_X) \longrightarrow\OOO_X \longrightarrow\OOO_Y\longrightarrow0
$$
we obtain $H^q(Y,\,\OOO_Y)=0$ for $1\le q\le \dim(Y)-1$.
Therefore, $Y$ is a Calabi--Yau variety.
\end{proof}

\begin{prp}
\label{prop:surface00}
Let $X$ be a Fano threefold and let $S\subset X$ be 
an irreducible component of the divisor $D\in|{-}K_X|$. Then
only one of the following possibilities occurs:
\begin{enumerate}
\item
$S$ is a birationally ruled surface,
\item
$S=D$ and $S$ is a \K3 surface, possibly, having Du Val singularities.
\end{enumerate}
\end{prp}

\begin{proof}
Write
$$
D=mS+R\sim -K_X,
$$
where $R$ is an effective divisor that does not contain $S$ as its component
and $m$ is an integral positive number. Then
\begin{equation}
\label{eq:lemma:surface00}
-(K_X+S)=(m-1)S+R\qq\frac{m-1}m(-K_X)+\frac1m R.
\end{equation}
By Lemma~\ref{lemma:surface0} below the Kodaira dimension $\varkappa(S)$
is non-positive. Assume that
$\varkappa(S)=0$. Then again by Lemma~\ref{lemma:surface0}
the singularities of the surface $S$ are at worst Du Val and $K_S\qq 0$.
From the right hand side part of~\eqref{eq:lemma:surface00} we obtain that $m=1$
(because $-K_X$ is ample) and $R=0$ (because the support of $D$ is connected).
Therefore, $K_S\sim 0$ and then $S$ is a \K3 surface
(see Proposition~\ref{Fano:adjunction}\ref{Fano:adjunction2}).
\end{proof}

\begin{lem}
\label{lemma:surface0}
Let $X$ be a nonsingular projective three-dimensional variety and let $S\subset 
X$
be an irreducible \textup(possibly non-normal\textup) surface.
Assume that
\begin{equation}
\label{eq:lemma:surface0}
K_X+S\qq -G,
\end{equation}
where $G$ is an effective $\QQ$-divisor.
Then the Kodaira dimension $\varkappa(S)$ is non-positive. Moreover, if
$\varkappa(S)=0$, then the singularities of the surface $S$ are at worst Du Val and
$K_S\qq 0$.
\end{lem}

\begin{proof}
Let $f\colon \widetilde{X}\to X$ be the blowup with nonsingular center $C$
lying in the singular 
part of $S$, let $\widetilde{S}\subset\widetilde{X}$ be the proper transform 
of $S$ and
let $E$ be the exceptional divisor. We can write
$$
K_{\widetilde{X}}=f^* K_X+aE, \qquad \widetilde{S}\sim f^* S-mE,
$$
where $a=2-\dim(C)\le 2$ and $m\ge 2$
is the multiplicity $S$ along $C$.
Then
$$
K_{\widetilde{X}}+\widetilde{S}\sim f^* (K_X+S)+(a-m)E \qq -(m-a)E-f^*G, \quad 
a-m\le
0.
$$
Here the $\QQ$-divisor $\widetilde{G}:= (m-a)E+f^*G$ is effective.
Moreover,
if $\widetilde{S}\cap\Supp(\widetilde{G})=\varnothing$, then $S\cap
\Supp(G)=\varnothing$ and $a=m=2$, i.\,e.\ $f$ is the blowup of a point of multiplicity~$2$ 
on
$S$.

According to Hironaka's theorem, applying blowups of described above kind to $(X,S)$
we can obtain embedded resolution of singularities 
$$
(\widehat{X}\supset \widehat{S})\longrightarrow (X\supset S)
$$
of the surface $S$. Here $K_{\widehat{X}}+\widehat{S}\qq -\widehat{G}$ with 
$\widehat{G}\ge 0$.
Let $\varkappa(\widehat{S})\ge 0$. By the adjunction formula
$$
K_{\widehat{S}}\qq -\widehat{G}|_{\widehat{S}}.
$$
Therefore, $K_{\widehat{S}}\qq 0$ and 
$\widehat{S}\cap\Supp(\widehat{G})=\varnothing$.
According to our inductive construction, $S\cap\Supp(G)=\varnothing$ and
the resolution $\widehat{X}\to X$ is a composition of blowups of points of multiplicity~$2$ on the proper transform of $S$.
This means that $K_S\qq0$,
the surface $S$ is normal, and its singularities are at worst Du Val.
\end{proof}

\subsection{Cone of effective curves}

The following important fact is an immediate consequence of the Cone Theorem~\ref{cone-th}.

\begin{teo}[{\cite{Mori:3-folds}}]
\label{theorem:Fano-Mori}
Let $X$ be a Fano variety. Then the Mori cone $\NE(X)$
polyhedral and is generated by a finite number of extremal rays $\rR_1,\dots,\rR_n$.
Each extremal ray $\rR_i$ has the form $\rR_i=\RR_+[\ell_i]$,
where $\ell_i$ is a rational curve on $X$ such that
$$
-K_X\cdot \ell_i\le \dim(X)+1.
$$
\end{teo}

\begin{cor}
\label{cor::Fano-Mori}
A divisor $D$ on Fano threefold $X$ is ample
\textup(respectively, is nef\textup) if and only if
$D\cdot \rR>0$
\textup(respectively, $D\cdot\rR\ge 0$\textup) for any extremal
ray $\rR$ on $X$.
\end{cor}

\subsection{Riemann--Roch Theorem on Fano varieties}

We need the following asymptotic form of the Riemann--Roch Theorem
(see e.~g.~\cite[Ch.~VI, Theorem~2.15.9]{Kollar-1996-RC}).

\begin{teo}
Let $H$ be a divisor on a nonsingular projective variety $X$ of dimension~$n$.
Then its Euler characteristic $\upchi(X,\,\OOO_X(tH))$ is a polynomial of the 
form
\begin{equation}
\label{equation-RR}
\upchi(X,\,\OOO_X(tH))=\frac 1{n!} H^n t^n+\frac1 {2(n-1)!} (-K_X)\cdot
H^{n-1}t^{n-1}+
\cdots
\end{equation}
\end{teo}

From this we obtain the following facts that will be repeatedly used
below.

\begin{teo}
\label{theorem-sections}
Let $X$ be a Fano variety of dimension $n$ and index $\iota=\iota(X)$.
Let $H$ be a divisor such that $-K_X=\iota H$.
Then the following assertions hold:
\begin{enumerate}
\item
\label{theorem-large-index1}
$\iota(X)\le n+1$.
\item
\label{theorem-large-degree}
If $\iota(X)=n+1$, then $H^n=1$.
\item
\label{theorem-large-degree-q}
If $\iota(X)=n$, then $H^n=2$.
\item
\label{theorem-large-RR}
For $\iota(X)\ge n-2$ the dimension of the linear system $|H|$ is computed by the 
following formulas:
\begin{equation}
\label{eq:sections}
\dim H^0(X,\,\OOO_X(H))=
\begin{cases}
n+1,&\text{if $\iota(X)=n+1$,}
\\
n+2,&\text{if $\iota(X)=n$,}
\\
n+d-1,&\text{if $\iota(X)=n-1$,}
\\
n+g-1,&\text{if $\iota(X)=n-2$,}
\end{cases}
\end{equation}
where
$$
d:= \dd(X)=H^n,\qquad g:= \frac12 H^n+1
$$
are positive integers.
In particular, $|{-}K_X|\neq \varnothing$ if $\iota(X)\ge n-2$.
\item
\label{theorem:2dim-RR}
Let $n=2$, $\iota:= \iota(X)$. Then
\begin{equation}
\label{eq:theorem:2dim-RR}
\dim H^0(X,\,\OOO_X(tH))=\frac1{2}dt(t+\iota)+1.
\end{equation}
\item
\label{theorem:3dim-RR}
Let $n=3$. Then
\begin{equation}
\label{eq:theorem:3dim-RRi=1}
\dim H^0(X,\,\OOO_X(tH))=\frac1{6}(g-1)t(t+1)(2t+1)+2t+1
\end{equation}
for $\iota(X)=1$, $t\ge 0$ and
\begin{equation}
\label{eq:theorem:3dim-RRi=2}
\dim H^0(X,\,\OOO_X(tH))=\frac1{12}t(t+\iota)(2t+\iota)d+\frac{2t}{\iota}+1
\end{equation}
for $\iota:= \iota(X)=2$, $t>-\iota$.
\end{enumerate}
\end{teo}

\begin{proof}
According to~\eqref{equation-RR},
$$
\upchi(t):= \upchi(X,\,\OOO_X(tH))
$$
is a
polynomial of degree $n$. On the other hand, by the Kodaira Vanishing Theorem
$$
\upchi(t)=\hr^0(X, \OOO_X(tH)),
$$
if the divisor
$-K_X+tH$ is ample, i.\,e.\ for $t>-\iota$.
In particular, $\upchi(0)=1$ and
the polynomial
$\upchi(t)$ vanishes at points
$$
t=-1, \dots,-(\iota-1).
$$
Since $\upchi(t)$ is of degree $n$, we have $\iota\le n+1$.
This proves~\ref{theorem-large-index1}.\footnote{Boundedness of the index 
\ref{theorem-large-index1} also follows from
Theorem~\ref{theorem:Fano-Mori}.}

For $\iota\ge n-2$ the polynomial $\upchi(t)$ can be computed explicitly. This
proves~\ref{theorem-large-degree},~\ref{theorem-large-degree-q}
and~\ref{theorem-large-RR}.
We give these computations for the case $\iota=n-2$. 
The analysis of the remaining cases is simpler
and left to the reader.
Since $\upchi(t)$ vanishes at points $t=-1, \dots,-(n-3)$, we have
\begin{equation}
\label{equation-RR-3-K}
\upchi(t)=\frac {d}{n!} (t+1)\cdots (t+n-3) \upchi_3(t),
\end{equation}
where $\upchi_3(t)$ is a polynomial of degree $3$ with leading term $t^3$.
It follows from the Serre duality that
$$
\hr^i(X,\,\OOO_X(tH))=\hr^{n-i}(X,\,\OOO_X(K_X-tH)),\qquad
\upchi(t)=(-1)^n\upchi(-\iota-t).
$$
From this we obtain that the set of roots of the polynomial $\upchi(t)$ 
is symmetric
with respect to the point
$-\iota/2=-(n-2)/2$. Therefore, the last factor in~\eqref{equation-RR-3-K}
has the form
$$
\upchi_3(t)=\left(t+\frac{n-2}{2}\right)\left(\left(t+\frac{n-2}{2}\right)^2+a 
\right)
$$
for some $a\in\QQ$. Thus
\begin{equation}
\label{equation-RR-3-Ka}
\upchi(t)=\frac {1}{n!}\, d (t+1)\cdots (t+n-3) 
\left(t+\frac{n-2}{2}\right)\left(\left(t+\frac{n-2}{2}\right)^2+a
\right).
\end{equation}
Since $\upchi(0)=1$, from this we obtain
$$
n!=d (n-3)! \cdot\frac{n-2}{2}\cdot \left(\left(\frac{n-2}{2}\right)^2+a 
\right),
\quad
2n(n-1)=d \left(\left(\frac{n-2}{2}\right)^2+a \right).
$$
Finally
$$
a=\frac 2 d n(n-1)- \left(\frac{n-2}{2}\right)^2.
$$
Substituting this to~\eqref{equation-RR-3-Ka} we obtain
$$
\upchi(t)=\frac {1}{n!}\,\bigl(dt\left(t+n-2\right)+2 n(n-1)
\bigr) (t+1)\cdots (t+n-3) \left(t+\frac{n-2}{2}\right).
$$
From this equality we immediately deduce the
required formulas~\eqref{eq:sections},~\eqref{eq:theorem:3dim-RRi=1} and
\eqref{eq:theorem:2dim-RR}.
\end{proof}

Note that in the general case (for Fano varieties of higher dimension and 
low index) the non-emptiness of the anticanonical linear system is not known.

\begin{cor}
\label{cor:rat-surf}
Let $X$ be a Fano variety and let
$$
\psi\colon X\dashrightarrow Z
$$
be a dominant rational map to a surface $Z$.
Then $Z$ is rational.
\end{cor}

The rationality of $Z$ follows from the rational connectedness of Fano varieties.
Indeed, the rational connectedness is preserved under dominant maps
and rationally connected surfaces are rational.
Since this assertion very important in the sequel, we provide here
an independent (but ``handmade'') proof, working only in dimension~3.

\begin{proof}
Let $\dim(X)=3$. Without loss of generality we may assume that the surface
$Z$ is nonsingular, projective, and does not contain $(-1)$-curves. According to
Corollary~\ref{cor:Omega}, we have
$$
H^1(Z,\,\OOO_Z)=H^0(Z,\Omega_Z^q)=0,\qquad \forall q>0.
$$
Assume that the surface $Z$ is not rational.
Then $Z$ cannot be ruled (because $H^1(Z,\,\OOO_Z)=0$).
According to the classification of algebraic surfaces, $\varkappa(Z)\ge 0$.

Consider a resolution of indeterminacy loci of our map
$$
\xymatrix@R=2em{
&\overline{X}\ar[dl]_{\sigma}\ar[dr]^{\bar\psi}&
\\
X\ar@{-->}[rr]^{\psi}&&Z
}
$$
(here $\sigma$ is a composition of blowups with smooth centers).
If some $\sigma$-exceptional divisor $E$ dominates 
$Z$, then $\varkappa(Z)=-\infty$ that contradicts the above.
Therefore, the map $\psi$
is defined near its general fiber, i.e. near $\sigma(\bar\psi^{-1}(z))$
for a general point $z\in Z$.

By Theorem~\ref{theorem-sections} the linear system $|{-}K_X|$ is of
positive dimension.
Take any element $D\in|{-}K_X|$.
Since the divisor $D\sim -K_X$ is ample, it
has non-zero intersection with a general fiber of the map $\psi$.
Thus the divisor $D$ has a component, say $S$, that dominates 
$Z$ (possibly, $S=D$). Since $\varkappa(Z)\ge 0$, we have $\varkappa(S)\ge 0$.
According to Proposition~\ref{prop:surface00},
$S$ is a \K3 surface, possibly, having only Du Val singularities.
Therefore, $\varkappa(Z)=\varkappa(S)=0$.
Again using the classification of algebraic surfaces, we may assume that $Z$
is an Enriques surface.

Then there exists double unbranched cover $Z'\to Z$, where $Z'$ is a \K3 surface.
Now we make the base change for $\bar\psi$. Then
$$
\overline{X}'=\overline{X}\times_Z Z'\longrightarrow\overline{X}
$$
is also double unbranched cover.
Since the fundamental group $\uppi_1^{\mathrm{alg}}(X)$ is a birational
invariant in the category of nonsingular complete varieties, the variety $\overline{X}$
is simply connected (see~\ref{thm-1-prop}\ref{1-properties-pi1}). Hence,
the cover $\overline{X}'\to \overline{X}$ splits: 
$\overline{X}'=\overline{X}'_1\sqcup
\overline{X}'_2$,
where $\overline{X}'_i\simeq \overline{X}$. But then as in the proof of
Corollary~\ref{cor:Omega}, we have
$$
H^0(X,\Omega^2_X)=
H^0(\overline{X}'_i,\Omega^2_{\overline{X}'_i})
\supset H^0(\overline{Z},\Omega_{Z'}^2)\neq 0.
$$
This contradicts
Theorem~\ref{thm-1-prop}\ref{1-properties-irregularity}.
\end{proof}

\subsection{Fano varieties of large index}

It turns out that the structure of Fano varieties of large index is very simple:

\begin{teo}
\label{thm:large-index}
Let $X$ be a Fano variety of dimension $n$ and let $\iota=\iota(X)$ be its
its Fano index.
Then the following assertions hold.
\begin{enumerate}
\item
If $\iota(X)=\dim(X)+1$, then $X\simeq \PP^n$.
\item
If $\iota(X)=\dim(X)$, then $X$ is isomorphic to a quadric in $\PP^{n+1}$.
\end{enumerate}
\end{teo}

In connection with this simple result, the following definition can be formulated.

\begin{dfn}
\label{def:coindex}
\textit{The coindex} of a Fano threefold $X$ is the (non-negative) number
$$
\mathrm{c}(X):= \dim(X)+1 -\iota(X).
$$
Thus $\mathrm{c}(X)=0$ (respectively, $\mathrm{c}(X)=1$) if and
only if $X\simeq \PP^n$ (respectively, is isomorphic to a quadric in 
$\PP^{n+1}$). 
Fano varieties of coindex~$2$ are called \textit{
del Pezzo varieties}~\cite{Fujita:book}
and Fano varieties of coindex $3$~are called \textit{Mukai varieties}~\cite{Mukai-1989}.
For an $n$-dimensional Mukai variety $X$
the number
$$
\g(X):= \frac12 H^n+1
$$
is called its \textit{genus } (see~\eqref{eq:sections}). Here, as usual,
$H=-\frac1\iota K_X$.
\end{dfn}

It turns out that complexity in the classification of Fano varieties increases with increasing the coindex.
At the present, there is the classification of Fano varieties of coindex
$\mathrm{c}(X)\le3$ only.

\begin{proof}[Proof of Theorem~\ref{thm:large-index}]
According to Theorem~\ref{theorem-sections}, we have
\begin{equation}
\label{eq:large-index:1}
\begin{alignedat}{2}
\iota&=n+1&&\quad\Longrightarrow\quad H^n=1,\quad \dim|H|=n,
\\[3pt]
\iota&=n &&\quad\Longrightarrow\quad H^n=2,\quad \dim|H|=n+1.
\end{alignedat}
\end{equation}
By Lemma~\ref{lemma-smooth} below the linear system $|H|$ contains an element 
$Y$ that is a nonsingular variety. By the adjunction formula
$-K_Y=(\iota-1)H|_Y$. Therefore,
$Y$ is a Fano variety of index $\iota(Y)=\iota(X)-1$
(for $\dim(X)>2$; for $\dim(X)=2$ we have $Y\simeq \PP^1$).
From the exact sequence
\begin{equation}
\label{equation:1exact}
0 \longrightarrow\OOO_X \longrightarrow\OOO_X(H)\longrightarrow
\OOO_Y(H)\longrightarrow 0
\end{equation}
and the vanishing of $H^1(X,\,\OOO_X)$ we obtain that the restriction map
$$
H^0(X,\,\OOO_X(H))\longrightarrow H^0(Y,\,\OOO_Y(H))
$$
is surjective. By the inductive hypothesis the linear system $|H|_Y|$ has no
base points. Hence, the same holds and for $|H|$. Thus the linear
system $|H|$ defines a morphism $\Phi\colon X \to \PP^{n+d-1}$ so that
$$
\deg \Phi\cdot \deg \Phi(X)=H^n.
$$
Since $\Phi(X)$ does not lie in a hyperplane, we have $\deg \Phi=1$, i.\,e.\
the morphism $\Phi$ is birational onto its image. Since the divisor $H$ is ample, the morphism
$\Phi$ is finite. Since a variety of degree $\le 2$ is normal, the map $\Phi$
is an isomorphism. Thus Theorem~\ref{thm:large-index} is proved.
\end{proof}

It remains to prove the existence of a nonsingular element $Y\in|H|$.

\begin{lem}
\label{lemma-smooth}
In the above notation the linear system $|H|$ contains an element
$Y$ that is a nonsingular variety.
\end{lem}
In the next lecture we discuss a general method for proving the 
assertions of this kind
(see Exercise~\ref{zadacha-smooth} in Section~\ref{sec2} below).
Here we give a simple ``handmade'' proof, which works
only in dimension $\le 3$.

\begin{proof}
Consider the case $\dim(X)=3$. The case $\dim(X)=2$ it is much simpler and
left to the reader for his own analysis.

Let $\iota(X)=4$. Then according to~\eqref{eq:large-index:1} we have 
$\dim|H|=3$ and $H^3=1$.
Since $H$ is ample, any element of
the linear system $|H|$ reduced and
is irreducible.
Take two distinct elements $H_1,\, H_2\in|H|$.
Let $C:= H_1\cap H_2$ and let
$\PPP\subset |H|$ be the pencil generated by $H_1$ and
$H_2$.
It is clear that $\Bs\PPP= C$. By Bertini's theorem
a general member $H\in\PPP$ is nonsingular outside $C$.

Since $H\cdot C=1$, the curve $C$ is a reduced
and irreducible.
By the genus formula
$$
2\p(C)-2=(K_X+H_1+H_2)\cdot C=-2.
$$
This implies that $C$ is a nonsingular rational curve and so
all surfaces $H\in\PPP$ are nonsingular along $C$.
Hence, a general member $H\in\PPP$ is nonsingular everywhere.

Let $\iota(X)=3$. Then according to~\eqref{eq:large-index:1} we have 
$\dim|H|=4$ and
$H^3=2$.
Thus any element $H\in|H|$ is either
reduced and irreducible, or has two irreducible components of 
multiplicity~1, or
has one irreducible component of multiplicity~$2$.
Since $\iota(X)=3$, the latter option is impossible.

Assume that a general member $H\in|H|$ is reducible: $H=F+M$,
where $M,\, F>0$. 
Then
$$
H^3=2=H^2\cdot (F+M).
$$
Therefore,
$$
H^2\cdot F=H^2\cdot M=1.
$$
In particular, $F$ and $M$ are irreducible surfaces.
Since $\iota(X)=3$, the divisors $M$ and $F$ are 
not linearly equivalent. Thus we may assume that 
$F$ is the fixed and $M$ is the movable part of $|H|$ (see Appendix~\ref{app:Bertini}). 
In particular,
$\dim|M|=\dim|H|>0$.
The support of the ample divisor $H=F+M$ is connected.
Thus $F\cap M$ is a curve and
$$
H\cdot F\cdot M>0.
$$
If the intersection of general elements $M_1,\, M_2\in|M|$ is empty, then the linear
system $|M|$ is base point free. 
In this case, by Bertini's theorem a general element $M\in|M|$ is nonsingular
and $K_M=3H|_M$ by the adjunction formula. Therefore, $M$ is a del Pezzo surface 
of index $3$.
By the inductive hypothesis $M\simeq \PP^2$ and $H|_M$ is the class of a line. On the other hand, from the restriction exact sequence we obtain that 
$\hr^0(M,\,\OOO_M(H))\ge 4$, a contradiction. 
Hence, the intersection of general elements $M_1,\, M_2\in|M|$ is a curve
and $H\cdot M^2>0$.
But then
$$
H^2\cdot M=1=H\cdot M^2+H\cdot F\cdot M\ge 2.
$$
The contradiction shows that a general member $H\in|H|$ is irreducible and reduced.

Now we claim that a general member $H\in|H|$ is also normal.
Assume the converse. Then by the Serre criterion $H$ is singular along some curve
$Z$. By Bertini's theorem $Z$ is contained in the base locus $\Bs|H|$,
hence
any other element $H_1\in|H|$ is also singular along $Z$. Thus $Z$
is contained with multiplicity $\ge 4$ in the intersection $H\cap H_1$.
But then
$$
4Z\cdot H \le H\cdot H_1\cdot H=2.
$$
The contradiction proves that $H$ is a normal surface.

Let $\mu\colon \widetilde{H}\to H$ be the minimal resolution (see 
e.~g.~\cite{P:book:sing-re}).
Consider the curve $C:= H\cap H_1$ as a Cartier divisor on $H$ and put
$M:= \mu^*C$.
By the adjunction formula $K_H=-2C$. Thus we can write
\begin{equation}
\label{equation:1H}
K_{\widetilde{H}}\qq \mu^* K_H-\Delta\qq -2M-\Delta,
\end{equation}
where $\Delta$ is an effective (possibly equal to zero) $\QQ$-divisor with support 
in the exceptional locus the morphism $\mu$~\cite{P:book:sing-re}.
This set does not contain any $(-1)$-curves.
Thus for any $(-1)$-curve $E$ on $\widetilde{H}$ we have
$$
M\cdot E>0, \qquad \Delta\cdot E\ge 0.
$$
Hence, $K_{\widetilde{H}}\cdot E<-1$,
which is impossible. Therefore, $\widetilde{H}$ does not contain any $(-1)$-curves. It follows
from~\eqref{equation:1H} that $-K_{\widetilde{H}}$ is a big divisor
(because it is represented as a sum of big and effective divisors).
Therefore, $\varkappa(\widetilde{H})=-\infty$ and so
$\widetilde{H}$ is a minimal ruled surface. The morphism $\mu$ must 
contract exactly one irreducible curve $\Sigma$.
Let $\Upsilon$ be a fiber of ruled surface. Since the divisor~$M$ is big,
it is not proportional to $\Upsilon$.
If $\Delta\neq 0$, then $\Delta^2<0$ and so the $\QQ$-divisor $\Delta$ also
is not proportional to $\Upsilon$.
Since $K_{\widetilde{H}}\cdot \Upsilon=-2$,
from~\eqref{equation:1H} we obtain that $\Delta=0$ and $M\cdot \Upsilon=1$.
Then $K_{\widetilde{H}}\cdot \Sigma=0$ and so
$$
2\p(\Sigma)-2=(K_{\widetilde{H}}+\Sigma)\cdot \Sigma=\Sigma^2<0.
$$
Therefore, $\p(\Sigma)=0$, $\Sigma^2=-2$, and $\widetilde{H}$ is a rational 
ruled
surface $\FF_2$.
In this case $H$ is a quadratic cone $Q\subset\PP^3$
and the divisor~$C$ is very ample on $H$. Then the assertion follows from
the exact sequence~\eqref{equation:1exact}, as above.
\end{proof}

Further we provide a few examples of Fano varieties.

\begin{examples}
\begin{enumerate}
\item
The projective space $\PP^n$ is the simplest example of a Fano variety.
\item
A nonsingular hypersurface $X_d\subset\PP^n$ is a Fano variety
if and only if $d<n+1$.
\item
\textit{Complete intersections.} A nonsingular complete intersection 
$X_{d_1\cdots d_r}\subset\PP^n$ of
hypersurfaces of degrees $d_1,\dots,d_r$ is a Fano variety
if and only if $\sum d_i<n+1$.
\item
\textit{Products.}
$X\times Y$ is a Fano variety if and only if
$X$ and $Y$ are Fano varieties.
\end{enumerate}
\end{examples}

\subsection{Blowups}
In this subsection we discuss the question of whether a blow up of some variety
is a Fano variety.
Let $X$ be a nonsingular projective variety, let $\LLL$ be a 
linear system Cartier divisors on $X$ without fixed components, and
let $Z:= \Bs\LLL$ (as a scheme).
Let $f\colon \widetilde{X}\to X$ be the blowup of $Z$.
Note that the variety $\widetilde{X}$ can be very singular and can have a very complicated
structure.
Thus we assume that the scheme $Z$ is reduced and nonsingular.
In this case the variety $\widetilde{X}$ is also nonsingular.
Let $\widetilde{\LLL}:= f^{-1}_*\LLL$ be the proper transform.
Then the linear system $\widetilde{\LLL}$ is base point free.
In particular, it is nef and $f$-ample.
Assume that we can decompose the anticanonical class $-K_{\widetilde{X}}$
as a convex linear combination
$$
-K_{\widetilde{X}}\equiv \alpha \widetilde{\LLL}+\beta f^*H, \qquad 
\alpha,\,\beta>0
$$
where $H$ is an ample divisor on $X$. Then the divisor $-K_{\widetilde{X}}$ is 
ample by Kleiman's Ampleness Criterion~\ref{Kleiman}, i.e. $X$ is a Fano variety.

\begin{exa}
In the above notation, let $X=\PP^n$ and let $Z=\PP^k\subset\PP^n$ be a 
linear subvariety.
Then the blowup of $\widetilde{X}$ in $Z$ is a Fano variety.
Indeed, let $\LLL$ be the linear system hyperplanes,
passing through $Z$. Then
$$
-K_{\widetilde{X}}\sim -f^*K_X-(n-k-1)E, \quad f^*\LLL=\widetilde{\LLL}+E,
$$
where $E$ is the exceptional divisor.
Therefore,
$$
-K_{\widetilde{X}}\sim (n+1)f^*\LLL-(n-k-1)E\sim
(n-k-1)\widetilde{\LLL}+(k+2)f^*\LLL.
$$
According to the above arguments, this divisor is ample.
\end{exa}

Below we will need the intersection numbers of divisors on
the blowup of a variety along a nonsingular subvariety.
These formulas will be repeatedly applied below.

\begin{lem}
\label{lemma-blowup-curve-intersection}
Let $V$ be a nonsingular three-dimensional variety and let
$\sigma \colon \widetilde{V} \to V$ be the blowup with center a smooth curve 
$Z\subset V$ of genus $\g(Z)$. Let $E:= \sigma^{-1}(Z)$
be the exceptional divisor.
Then
$$
E\simeq \PP\left(\NNN_{Z/V}^\vee\right)
$$
and for any divisors $D_1$, $D_2$, $D_3$
on $V$ the following relations hold:
\begin{equation}
\label{eq:blowup-curve-intersection}
\begin{gathered}
\begin{aligned}
\sigma^* D_1\cdot \sigma^* D_2\cdot\sigma^* D_3&=D_1\cdot D_2\cdot D_3,
\\[3pt]
\sigma^* D_1\cdot \sigma^* D_2\cdot E&=0,
\\[3pt]
\sigma^* D_1\cdot E^2&=-D_1\cdot Z,
\end{aligned}\\[3pt]
E^3=-\deg (\NNN_{Z/V})=2-2\g(Z)+K_V\cdot Z.
\end{gathered}
\end{equation}
\end{lem}

\begin{proof}
All the relations follows from the projection formula and the fact that
$\OOO_E(E)\simeq \OOO_{\PP(\NNN_{Z/V}^\vee)}(-1)$.
In last relation the degree $\deg (\NNN_{Z/V})$ can be found from the exact
sequence
$$
0 \longrightarrow\TTT_Z \longrightarrow\TTT_V|_Z\longrightarrow\NNN_{Z/V}\longrightarrow 0.\qedhere
$$
\end{proof}

\subsection{Grassmannians and homogeneous spaces}

\begin{prp}
\label{proposition:homogeneous}
Let $G=\Gr(r,n)$ be a Grassmann variety.
Then $\Pic(G)\simeq \ZZ$ and $-K_G\sim n H$,
where $H$ is the ample generator of $\Pic(G)$.
In particular, $G$ is a Fano variety of index $n$.
\end{prp}

\begin{proof}
The element $a\in G$ is given by some $r\times n$-matrix $A$ of maximal rank.
Moreover such matrices differing in elementary transformations of rows
define the same element.
Write $A=\begin{pmatrix}
A_1 \mid A_2
\end{pmatrix}$, where $A_1$ (respectively, $A_2$) is a matrix of dimension
$r\times r$ (respectively, $r\times (n-r)$). Consider the open subset
$U\subset G$ given by the condition $\det A_1\neq 0$.
Any element $a\in U$ \textit{uniquely} represented by a matrix of the form
$A=\begin{pmatrix}
E \mid X
\end{pmatrix}$.
Let $X=(x_{i,j})$.
Here the entries $x_{i,j}$ can be regarded as coordinates on~$U$.
Thus
$U\simeq \mathbb{A}^{r(n-r)}$. Let $H=G\setminus U$.
It is clear that $H$ is a closed subvariety of codimension~$1$.
This variety is irreducible because the determinant regarded as
a polynomial in its entries is irreducible.
From the exact cutting sequence we have
$$
H\cdot \ZZ \longrightarrow\Pic(G) \longrightarrow\Pic(U)=0.
$$
Therefore, $\Pic(G)=\ZZ\cdot H$.

Below we assume that $n=5$, $r=2$.
The general case differs from it only in its bulkiness.
Then
$$
A=
\begin{pmatrix}
1&0 & x_{1,3}& x_{1,4}& x_{1,5}
\\
0&1 & x_{2,3}& x_{2,4}& x_{2,5}
\end{pmatrix}
$$
Now we move on to another chart.
For this by elementary row operations
we transform $A$ to the matrix
$$
\begin{pmatrix}
\frac1{x_{1,3}}&0 &1& \frac{x_{1,4}}{x_{1,3}}&
\frac{x_{1,5}}{x_{1,3}}
\\
-\frac{x_{2,3}}{x_{1,3}}&1 &0&
\scriptstyle{x_{2,4}}-\frac{x_{1,4}x_{2,3}}{x_{1,3}}&
\scriptstyle{x_{2,5}}-\frac{x_{1,5}x_{2,3}}{x_{1,3}}
\end{pmatrix}
$$
The new coordinates will be
$$
x_{1,3}'=\frac 1{x_{1,3}},\qquad
x_{2,3}'=-\frac {x_{2,3}}{x_{1,3}},
$$
$$
x_{1,j}'=\frac {x_{1,j}}{x_{1,3}},\quad x_{2,j}'=x_{2,j}- \frac
{x_{1,j}x_{2,3}}{x_{1,3}},\quad j=4,5.
$$
Consider the differential form
$\omega=\di x_{1,3}' \wedge\cdots \wedge \di x_{2,5}'$.
Then
$$
\omega=\frac{u}{x_{1,3}^5} \,
\di x_{1,3} \wedge\cdots \wedge \di x_{2,5},
$$
where $u$ is an invertible element. We obtain $K_G=\divi(\omega)=-5H$, where the divisor
$H$ is given by the equation
$x_{1,3}=0$.
\end{proof}
Note that the divisor $H$ is a hyperplane section under the Pl\"ucker embedding.
Therefore, the linear system $|mH|$ is very ample for $m>0$.

There is much more general result than Proposition 
\ref{proposition:homogeneous}.

\begin{prp}[{\cite[Ch.~V, Theorem~1.4]{Kollar-1996-RC}}]
Let $G$ be a linear algebraic group and let $X$ be a homogeneous
space with respect to some action of $G$.
If the variety $X$ is projective, then it is a Fano variety.
\end{prp}

In other words, the quotient of a linear algebraic group by its parabolic
subgroup is always a~Fano variety.

\subsection{Fano varieties and finite morphisms}

Let $f\colon X\to Y$ be a finite surjective morphism of nonsingular varieties.
Take a point $P\in X$ and take local coordinates $x_1,\dots,x_n$ in a neighborhood
$P\in U\subset X$.
Let $Q:= f(P)$ and let $y_1,\dots,y_n$ be local coordinates in a neighborhood
$Q\in V\subset Y$.
The map $f$ is given by some functions $y_i=y_i(x_1,\dots,x_n)$.
Let 
$$
\omega=\phi\,\di y_1\wedge \dots \wedge \di y_n
$$
be a
meromorphic differential form of top degree.
Then $K_Y|_V=\divi(\phi)$ and
$$
f^*K_Y|_U=\divi\big(\phi(y_1(x_1,\dots,x_n),\dots,y_n(x_1,\dots,x_n))\big).
$$
On the other hand, $K_X|_U$ can be given by the divisor of the meromorphic form
$f^*\omega$ that is equal to
$$
\phi (y_1(x_1,\dots,x_n),\dots,y_n(x_1,\dots,x_n))\det (\partial y_i/\partial
x_j)\,\di x_1\wedge \dots \wedge \di x_n
$$
Comparing these relations, we obtain that $K_Y-f^*K_X$ is given by vanishing
the Jacobian determinant
$\det (\partial y_i/\partial x_j)$. It is a holomorphic function,
vanishing along
the divisor that is the branch divisor.
We obtain the Hurwitz formula
\begin{equation}
\label{equation-Hurwitz-1}
K_X=f^*K_Y+R, \qquad R\ge 0.
\end{equation}
Assume now that $f$ is a cyclic cover of degree $m$.
Taking the coordinates $x_1,\dots,x_{n-1}$ along a component of $R$ and take $x_n$ along
a transversal direction, we may assume that $y_1=x_1$, \dots, $y_{n-1}=x_{n-1}$,
$y_n=x_n^m$.
Then
\begin{equation}
\label{equation:Hurwitz}
R=\frac {m-1}mf^*B,\quad \text{where}\quad B=f(R).
\end{equation}
Thus
\begin{equation}
\label{equation-Hurwitz-2}
K_X=f^*\left(K_Y+\frac{m-1}m B\right).
\end{equation}
It is clear that in this case the divisor $B$
must be nonsingular and its class in $\Pic(Y)$ must be divisible by $m$.

In particular, if varieties $X$ and $Y$ are projective, then
from~\eqref{equation:Hurwitz}
we obtain that $X$ is a Fano variety if and only if the divisor
$-(K_Y+\frac{m-1}m B)$ is ample.

\begin{exa}
Let in our notation $Y=\PP^n$ and let $B$ be a hypersurface of degree
$md$,
given by the homogeneous equation $\phi(x_0,\dots,x_n)=0$.
The variety $X$ can be realized as a hypersurface in the weighted 
projective space $\PP(1^{n+1},d)$ (see below) given by the equation
$y^m=\phi(x_0,\dots,x_n)$, where
$x_i$ are quasihomogeneous coordinates of degree~$1$ and $y$ is coordinate of 
degree
$d$.
Here $X$ is a Fano variety if and only if
$(m-1)d\le n$.
\end{exa}

\subsection{Fano varieties in weighted projective spaces}

\label{wps}

A lot of examples of Fano varieties are given as complete intersections in weighted 
projective spaces.
Here we provide without proofs some basic properties weighted projective spaces.
For details we refer to the survey~\cite{Dolgachev-1982}.

Consider the polynomial ring $R=\CC [x_0,\dots,x_n]$ with the non-standard
grading $\deg x_i=w_i$, $w_i\in\ZZ_{>0}$.
The projective spectrum $\operatorname{Proj} (R)$ is called the \textit{weighted 
projective space} of dimension $n$ with weights $w_i$. We denote it
by $\PP (w_0,\dots,w_n)$. Sometimes we use also the shortened notation:
$\PP(d_0^{k_0},\dots,d_m^{k_m})$ means that every $d_i$ is repeated $k_i$ times.
It is clear that without loss of generality we may assume that the numbers $w_i$ are setwise coprime.

The variety $\PP:= \PP (w_0,\dots,w_n)$ is toric and it can be realized
as the quotient 
$$
\PP^n/\mumu_{w_0}\times \cdots \times \mumu_{w_n},
$$
where $\mumu_{w_i}$ is a cyclic group of order $w_i$ that acts by
multiplication of coordinate $x_i$ by $w_i$-th root of unity.
The variety $\PP$ can be realized also as quotient $\CC^{n+1}/\CC^*$ by 
one-dimensional torus $\CC^*$ with the action 
$$
(x_0,\dots,x_n)
\xmapsto{\lambda\in\CC^*}
(\lambda^{w_0} x_0,\dots,\lambda^{w_n}x_n).
$$
This implies that for any $r\in\ZZ_{>0}$ there exists an isomorphism
$$
\PP (rw_0,\dots,rw_{i-1},w_i, rw_{i+1}, \dots,rw_n)\simeq \PP (w_0,\dots,w_n).
$$
Thus we always may assume that
$$
\gcd(w_0,\dots,w_{i-1}, w_{i+1}, \dots,w_n)=1\quad \text{for any $i$}
$$
i.e. any collection $n$ weights from $\{w_0,\dots,w_n\}$ is setwise coprime.
This assumption is important and always \textit{is supposed to be fulfilled
automatically}.
In this case we say that the collection $\{w_0,\dots,w_n\}$ is \textit{well-formed}.

The variety $\PP=\PP (w_0,\dots,w_n)$ is covered by affine charts
\begin{equation}
\label{eq:wps:U}
U_k=\{x_k\neq 0\}\simeq \CC^n/\mumu_{w_k}(w_0,\dots,w_{k-1}, w_{k+1},
\dots,w_n).
\end{equation}
where $\CC^n/\mumu_r(a_1, \dots,a_n)$ denotes the quotient of $\CC^n$ by the group
$\mumu_r$ acting on $\CC^n$
diagonal matrix
$$
\begin{pmatrix}
\exp \big(\frac{2\uppi \mathrm{i} a_0}{r
}\big) & & \\
& \ddots & \\
& & \exp \big(\frac{2\uppi \mathrm{i} a_n}{r
}\big)
\end{pmatrix}
$$
In particular, it follows that the space $\PP$ has only cyclic
quotient singularities and they are isolated if and only if the numbers $w_i$
pairwise coprime (under our assumption of well-formedness).
Also it is clear that $\PP$ is nonsingular if and only if it is isomorphic to the
usual projective space.

Any subvariety $X\subset\PP$ of codimension~$1$ can be given by a
quasihomogeneous polynomial $f(x_0,\dots,x_n)$ (with respect to weights $w_i$).
Such a subvariety is called a \textit{hypersurface} in $\PP$.
By the \textit{degree} of a hypersurface in $\PP$ we understand the weighted degree of
the corresponding polynomial $f$.
It is easy to show (using the exact cutting sequence~\cite[Ch.~3,
\S6]{Hartshorn-1977-ag}) that the hypersurfaces $D_k:= \{x_k=0\}$ generate
the Weil divisor class group $\Cl(\PP)$.
Moreover the group $\Cl(\PP)$ is cyclic and the map
$$
\deg\colon \Cl(\PP)\to \ZZ
$$
is an isomorphism.

The Picard group $\Pic(X)$ naturally embedded to $\Cl(X)$ as a subgroup of index $\lcm
(w_0,\dots,w_n)$.
Indeed, let $I:= [\Cl(X):\Pic(X)]$. It follows from
the presentation
\eqref{eq:wps:U} in the chart $U_k=\{x_k\neq 0\}$ that for any Weil divisor $D$ its multiplicity
$w_kD$
is a Cartier divisor in $U_k$. Therefore, $I$ divides $\lcm
(w_0,\dots,w_n)$. On the other hand, let $C_k\subset\PP$ be the coordinate
line $\langle x_k,\, x_l\rangle$, $k\neq l$.
By the projection formula in chart $U_k$
we compute $ID_l\cdot C_k=I/w_k$, where $D_l:= \{x_l=0\}$. This number must be
integral.
Therefore, $\lcm (w_0,\dots,w_n)$ divides $I$.

Let $A$ be the positive generator of $\Cl(\PP)$. Then the self-intersection 
number
$A^n$ is well-defined as a rational number and
$$
A^n=\frac1 {w_0\dots w_n}.
$$
The canonical class of $\PP$ is also well-defined (as a Weil divisor).
Computations with differential forms show that
$$
-K_{\PP}=\left(\sum w_i\right)A,
$$
where $A$ is the positive generator of $\Cl(\PP)$.
As on any projective spectrum, on the space $\PP$ 
the coherent sheaves~$\OOO(m)$ are defined~\cite[Ch.~3, \S5]{Hartshorn-1977-ag}.
All such sheaves are reflexive and
$$
\OOO_{\PP}(m)\simeq \OOO_{\PP}(mA).
$$
Moreover, such a sheaf is invertible if and only if $m\equiv 0\mod w_i$
for all~$i$.

Assume now that the numbers $w_i$ are pairwise coprime and let $d$ be a 
positive number that is divisible by $w_0\cdots w_n$. Then the space $\PP$
has only isolated singularities and by Bertini's theorem a general
hypersurface
$X=X_d\subset\PP$ of degree $d$ is nonsingular. By the adjunction formula
$$
K_X=\left(d-\sum w_i\right)A|_X.
$$
Thus $X$ is a Fano variety if and only if $d<
\sum w_i$.
If $n\ge 4$, then in our situation the Lefschetz hyperplane theorem holds for the Weil divisor class groups~\cite{Dolgachev-1982}, i.e. the restriction map
$$
\Cl(\PP)\longrightarrow\Pic(X)
$$
is an isomorphism. Then we can compute the index of our Fano threefold:
$$
\iota(X)=-d+\sum w_i.
$$
In a similar way one can construct Fano varieties that are complete intersections of
larger codimension.

\subsection{Fano varieties and birational morphisms}

\begin{lem}
\label{lemma:bir}
Let $f\colon X\to Y$ be a birational morphism of nonsingular varieties.
Write
\begin{equation}
\label{eq:discr-Fano}
K_X=f^*K_Y+\sum a_i E_i
\end{equation}
where $E_i$ are \textup(all\textup) irreducible exceptional divisors.
Then $a_i>0$ for any $i$.
\end{lem}

\begin{proof}
Take a point $P\in X$ and local coordinates $x_1,\dots,x_n$ in a neighborhood
$P\in U\subset X$.
Let $Q:= f(P)$ and let $y_1,\dots,y_n$ be local coordinates in a neighborhood
$Q\in V\subset Y$.
The map $f$ is given by functions $y_i=y_i(x_1,\dots,x_n)$.
Let $\omega=\phi \,\di y_1\wedge \dots \wedge \di y_n$ be a meromorphic
differential form of top degree.
Then $K_Y|_V=(\phi)$ and
$$
f^*K_Y|_U=\divi\big(\phi(y_1(x_1,\dots,x_n),\dots,y_n(x_1,\dots,x_n))\big).
$$
On the other hand, $K_X|_U$ can be given by the divisor of a meromorphic form
$f^*\omega$ that is equal to
$$
\phi \big(y_1(x_1,\dots,x_n),\dots,y_n(x_1,\dots,x_n)\big)\det (\partial
y_i/\partial
x_j) \,\di x_1\wedge \dots \wedge \di x_n
$$
Comparing these relations, we obtain that $K_X-f^*K_Y$ is given by the vanishing of
the Jacobian determinant
$\det (\partial y_i/\partial x_j)$. It is a holomorphic function,
vanishing along all components of the 
exceptional divisor.
\end{proof}

\begin{prp}
\label{proposition:bir-pt}
Let $X$ be a Fano variety and let $f\colon X\to Y$ be a birational morphism
to a nonsingular variety~$Y$ such that $\dim f(\Exc(f))=0$. Then $Y$ is 
also a Fano variety.
\end{prp}

\begin{proof}
Take $n\gg 0$ such that the divisor $\mo|-nK_X|$ is very ample. Since
$f_*\mo|-nK_X|\subset \mo|-nK_Y|$, the divisor $-nK_Y$ is big and 
$\Bs\mo|-nK_Y|\subset
f(\Exc(f))$.
This implies that the divisor $-nK_Y$ is nef. By the Base Point Free 
Theorem~\ref{th:mmp:bpf} the linear system $\mo|-mK_Y|$ has no
base points for $m\gg 0$. Assume that $Y$ is not a 
Fano variety. Then there exists an irreducible curve $C\subset Y$ such that $-K_Y\cdot
C=0$. Let $C'\subset X$ be the proper transform of $C$. Then $C'\not \subset
\Exc(f)$ and so $K_X\cdot C'\ge 0$ (see~\eqref{eq:discr-Fano}). a contradiction.
\end{proof}

\begin{cor}
Let $X$ be a del Pezzo surface of degree $d$.
If $X\not \simeq \PP^1\times \PP^1$ and $X\not \simeq \PP^2$, then
$X$ is the blowup of a point on a del Pezzo surface of degree $d+1$.
\end{cor}

\begin{proof}
Indeed, if the surface $X$ contains $(-1)$-curve, then
$X$ is represented as a blowup of a point on a nonsingular surface $Y$ that,
according to Proposition~\ref{proposition:bir-pt}, is a del Pezzo surface. Assume that $X$ does not contain $(-1)$-curves and $X\not \simeq
\PP^1\times \PP^1$, $\PP^2$. According to Corollary
~\ref{cor:delPezzo:}\ref{delPezzo:rational}, we have $X\simeq \FF_n$, $n\ge 2$.
But then
$$
K_X\cdot \Sigma=2\p(\Sigma)-2-\Sigma^2=n-2\ge 0.
$$
where $\Sigma$ is the exceptional section. This contradicts the ampleness of $-K_X$.
\end{proof}

\begin{zadachi}
\eitem
Find all del Pezzo surfaces that are complete intersections of
hypersurfaces in Grassmannians.

\eitem
Under what conditions the blowup of $m\ge2$ points on $\mathbb{P}^n$ is a
Fano variety?

\eitem
\label{zad:curve-dP}
Let $X$ be a projective surface and let $C\subset X$
be an ample effective divisor such that $\p(C)\le 1$.
Prove that $X$ is a del Pezzo surface.

\eitem
Let $X$ be a projective rational surface such that $-K_X\cdot C>0$
for any curve $C$.
Prove that $X$ is a del Pezzo surface.

\eitem
\label{zad:DP9}
Let $X$ be a del Pezzo surface of degree $9$.
Without using the classification of surfaces prove that $X\simeq \PP^2$.
\hint{Use the Poincar\'e duality and show that $\iota(X)=3$.}

\eitem
\label{zad:DP8}
Let $X$ be a del Pezzo surface of degree $8$.
Without using the classification of surfaces, prove that $X\simeq \PP^1\times 
\PP^1$ or $\FF_1$.
\hint{As in the previous exercise, use the Poincar\'e duality.
If the intersection form on $\Pic(X)$ is even, then you can prove that $\iota(X)=2$,
and if it is odd, then you can prove that there exists $(-1)$-curve on $X$.}

\eitem
Let $f\colon X\to \PP^3$ be the blowup of a nonsingular curve $C$.
Find a sufficient condition of the fact that $X$ is a Fano variety.

\eitem
Find all del Pezzo surfaces that are
complete intersections of hypersurfaces in a product of projective
spaces $\mathbb{P}^{n_1}\times\cdots\times \mathbb{P}^{n_l}$.

\eitem
Find all Fano varieties that are
complete intersections of hypersurfaces in a product of projective
spaces $\mathbb{P}^{n_1}\times\cdots\times \mathbb{P}^{n_l}$.

\eitem
Let $X=X_3\subset\PP^4$ be a nonsingular surface of degree $3$ that does not lie on a 
hyperplane.
Prove that $X$ is a del Pezzo surface.

\eitem
Let $X$ be a nonsingular projective surface such that its anticanonical
divisor $-K_X$
is nef. Prove that one of the following holds:
\begin{enumerate}
\item
$K_X\approxident 0$;
\item
$X$ is rational;
\item
$X$ is birationally equivalent to a ruled surface over an
elliptic curve.
\end{enumerate}

\eitem
Let $X$ and $Y$ be a del Pezzo surface.
Assume that there exists a finite morphism $f\colon X\to Y$ of degree~$2$.
Which values can take the degree $X$ and $Y$?
Consider all the cases.

\eitem
\label{zad:base:del-Pezzo}
Let $X$ be a Fano threefold and let $f\colon X\to Y$
be a flat morphism to a (nonsingular) surface.
Prove that $Y$ is a del Pezzo surface.
\hint{Show that $-f_*(K_X^2)\approxident 4K_Y+\Delta$, where $\Delta\subset
Y$ is the discriminant divisor.}

\eitem
\label{zad:D-notnef1}
Let $X$ be a Fano threefold and let $D$ be a 
prime divisor on $X$. Assume that $D$ is not nef.
Prove that $D$ is isomorphic either to projective plane $\PP^2$, or to nonsingular
ruled surface, or to
a quadratic cone $Q\subset\PP^3$. \hint{Use the classification of extremal rays
\ref{class:ext-rays}.}

\eitem
\label{zad:D-notnef2}
Let $X\subset \PP^N$ be a nonsingular three-dimensional variety such that
its general hyperplane section is a surface
of Kodaira dimension 0. Prove that $X$ is a Fano variety.

\eitem
Let $X$ be a nonsingular three-dimensional variety and let $D$ be a divisor 
with simple normal
crossings on $X$. Assume that the divisor $-(K_X+D)$ is ample.
Compute the polynomial $\upchi (X,\,\OOO_X(-t(K_X+D)))$.
\hint{Use the Kawamata--Viehweg Vanishing Theorem 
\ref{vanishing:KV}.}

\eitem
Find all Fano varieties that are complete
intersections of hypersurfaces in Grassmannians.
\end{zadachi}

\newpage\section{The existence of a smooth divisor}
\label{sec2}

\subsection{Main result and outline of the proof} 

In this lecture we prove the existence of a nonsingular divisor in the anticanonical linear system
on Fano threefolds.

\begin{teo}[V.~V.~Shokurov \cite{Shokurov:eleph}]
\label{theorem-smooth-divisor}
Let $X$ be a Fano threefold of index $\iota=\iota(X)$ and let
$H:= -\frac1{\iota}K_X$.
Then the linear system $|H|$ contains a
smooth irreducible divisor.
\end{teo}

This is a fundamental fact, on which the classification of nonsingular (and Gorenstein) 
Fano threefolds is based.
For the first time it was is proved by V. V. Shokurov~\cite{Shokurov:eleph}. Below we
will give the proof using more modern techniques.
The theorem has generalizations
to the higher-dimensional case as well as to the case of singular varieties (see e.~g.
\cite{Wilson1987,Reid:Kaw,Prokhorov1995a,Ambro-1999}).
The proof given here is a highly simplified version of the one in~\cite{Ambro-1999}.

\begin{rem}
We reduce the question on~the existence of ``good'' divisors in linear systems to the question
on non-vanishing of global sections of restrictions of the corresponding
line bundles to certain subvarieties.
By Bertini's theorem a general member a linear system is nonsingular if this 
linear
system is base point free
(and fixed components). Thus below we assume that
the base locus of
the studied linear system is non-empty.
\end{rem}

Consider a model example.

\begin{exa}
Let $X$ be a del Pezzo surface. By the Riemann--Roch Theorem
$\dim|{-}K_X|>0$.
Assume for simplicity that the linear system $|{-}K_X|$
has \textit{exactly one} fixed component~$F$. Since
$$
\hr^2(X,\,\OOO_X(F))=\hr^0(X,\,\OOO_X(K_X-F))=0
$$
by the Serre duality, again by the Riemann--Roch Theorem we have
$$
\hr^0(X,\,\OOO_X(F))\ge \frac12 F\cdot (F-K_X)+1.
$$
Since $\dim|F|=0$, we have $F^2<0$. Thus $F$ is a $(-1)$-curve.
In particular, $F\simeq \PP^1$ and so $H^0(F,\OOO_F(-K_X))\neq 0$.
Since $F$ is a fixed component, we have
$$
H^0(X, \OOO_X(-K_X-F))\simeq H^0(X,\,\OOO_X(-K_X)).
$$
From the restriction exact sequence
$$
0 \longrightarrow\OOO_X(-K_X-F) \longrightarrow\OOO_X(-K_X)\longrightarrow
\OOO_F(-K_X)\longrightarrow 0
$$
we obtain
$$
H^1(X, \OOO_X(-K_X-F))\neq 0
$$
and then by Kodaira Vanishing Theorem
the divisor $-2K_X-F$ is not ample. This implies that the divisor $-(K_X+F)$ is 
not 
nef, i.\,e.\ there exists an irreducible curve $C$ such that
$-(K_X+F)\cdot C<0$. But then $C$ must be a fixed component of the
linear system $\mo|-(K_X+F)|$ and therefore it must be a fixed component of the linear system $|{-}K_X|$.
By our assumption $C=F$. But then $-(K_X+F)\cdot C=2$, a contradiction.
\end{exa}

The considered example shows that for an application of the inductive approach 
to the proof of the existence of nonsingular divisors, two
ingredients are needed:
separation a ``good'' component of the base locus and vanishing of higher
cohomology.

Further we will need some basic information about the singularities of pairs.
Below we provide the basic definitions and properties.
A more detailed explanation
see for instance in~\cite{KM:book,Kollar95:pairs,P:book:sing-re}.

\subsection{Singularities of pairs}
\label{subsection:discr} 

Let $X$ be a normal (not necessarily complete) variety and let $B=\sum
b_iB_i$ be any $\QQ$-divisor on $X$.
Assume that the divisor $K_X+B$ is $\QQ$-Cartier.
Consider any birational morphism $f\colon \widetilde{X}\to X$, where the variety
$\widetilde{X}$
is also normal. Similar to~\eqref{eq:discr-Fano} we can write
\begin{equation}
\label{eq:KKaE}
K_{\widetilde{X}}+\widetilde{B}=f^*(K_X+B)+\sum a_i E_i,
\end{equation}
where the $E_i$ are exceptional divisors, $\widetilde{B}$ is the proper
transform of $B$
on $\widetilde{X}$, and the numbers $a_i$ are rational. These numbers are called
\textit{discrepancies}. Each $a_i$ depends on
$X$, $B$, and the discrete valuation of the function field $\CC(X)$ defined by 
the divisor
$E_i$ (and does not depend on $f$).
Thus it is natural to use the notation $a_i=a(X,B, E_i)$.
It is possible also to define discrepancies of prime divisors with support on the
variety~$X$ itself:
$$
a(X,B, E)=
\begin{cases}
-b_i, & \text{if $E=B_i$,}
\\
0, & \text{if $E\not\subset\Supp(B)$.}
\end{cases}
$$

\begin{rem}
\label{rem:discr1}
In the notation of~\ref{subsection:discr}, let $F$ be a $\QQ$-Cartier divisor.
Put $B_t:= B+tF$.
Then it follows from~\eqref{eq:KKaE} that
$$
a(X,B_t, E)=a(X,B, E)- t m_F,
$$
where $m_F$ is the multiplicity $f^*F$ along $E$. Thus the function $a(X,B_t,E)$ is linear in $t$.
\end{rem}

\begin{rem}
Usually, in the minimal model theory it is being considered the pairs $(X,B)$ consisting of a
normal variety~$X$ and an \textit{effective} $\QQ$- or $\RR$-divisor
$B=\sum b_iB_i$. The effectivity of $B$ is important in many assertions but 
the definitions can be formulated in general form.
\end{rem}

Recall that a \textit{birational model} of a variety $X$ is another
variety $\widetilde{X}$
such that there exists birational map $f\colon \widetilde{X} \dashrightarrow X$.
We say that a prime divisor $E$ on $\widetilde{X}$ \textit{has non-empty center } 
on $X$, if the map $f$
is regular at the general point of $E$.
In this case the \textit{center} $E$ on $X$ is the closure of the image $f(E)$.

Thus the discrepancy is defined for \textit{any} divisor $E$ on
some birational model
$\widetilde{X}$ of the variety $X$ under the condition that $E$ has a (non-empty) 
center on
$X$.

\begin{dfn}
\label{def:sing}
\begin{enumerate}
\item
\label{def:sing:lc}
We say that a pair $(X,B)$ has \textit{log canonical \textup(lc\textup) singularities}, if
$a(X,B, E)\ge -1$ for any divisor $E$ with center on $X$.
\item
\label{def:sing:klt}
We say that a pair $(X,B)$ has \textit{Kawamata log terminal \textup(klt\textup) singularities},
if
$a(X,B, E)> -1$ for any divisor $E$ with center on $X$.
\item
\label{def:sing:plt}
We say that a pair $(X,B)$ has \textit{purely log terminal \textup(plt\textup) singularities}, if
$a(X,B, E)> -1$ for any exceptional divisor $E$ with center on $X$.
\end{enumerate}
\end{dfn}

It is clear that for a plt pair $(X,B=\sum b_iB_i)$
the inequalities hold $b_i\le 1$ for any $i$.
Thus a plt pair is lc.
In the case $\lfloor B\rfloor\le 0$ both concepts of log terminal~\ref{def:sing:klt}
and~\ref{def:sing:plt} are equivalent and then we say that the singularities of $X$ are
log terminal.

\begin{rem}
\label{rem:discr2}
It is easy to show that it is sufficient to check the conditions $a(X,B, E)\ge -1$ (respectively, $a(X,B, 
E)>-1$) in Definition~\ref{def:sing} for one
fixed surjective birational morphism $f\colon \widetilde{X}\to X$
such that $\widetilde{X}$ is nonsingular and the union of the exceptional locus
and the proper transform of $B$
is a divisor with simple normal crossings. Thus the condition of
log canonical and log terminal is equivalent to the fulfillment of some
\textit{finite} number of inequalities. If all these components $B_i$ are 
$\QQ$-Cartier divisors, then
these inequalities are \textit{linear} with respect to coefficients $b_i$.
\end{rem} 

\begin{exa}
Let the variety $X$ be nonsingular and the support of the divisor $B$ has simple 
normal crossings.
In this case the pair $(X,B=\sum b_iB_i)$ has lc (respectively,
klt) singularities if and only if
$b_i\le 1$ (respectively, $b_i< 1$) for all $i$. The pair $(X,B=\sum b_iB_i)$
has plt singularities
if and only if
$b_i\le 1$ for all $i$ and the components, for which the equality holds 
$b_i=1$
do not meet each other.
\end{exa}

\begin{dfn}
Let $(X,B)$ be a pair consisting of a normal variety $X$ and let
a $\QQ$-divisor $B$ on $X$
such that the divisor $K_X+B$ is $\QQ$-Cartier.
An irreducible subvariety $W\subset X$ is said to be a \textit{log canonical
center} (or a \textit{center of log canonical
singularities}) of the pair $(X,B)$ if there exists a birational morphism 
$f\colon 
\widetilde{X}\to X$
and a prime divisor $E\subset\widetilde{X}$ such that $a(X,B, E)\le -1$ and 
$f(E)=W$.
For example, if $E=B_i$ is a component of $B$ with coefficient $b_i\ge 1$,
then $E$ is a log canonical center.
The union of all log canonical centers is called the \textit{locus of
log canonical singularities}.
We denote it by $\LCS(X,B)$.
\end{dfn}

\begin{prp}[{\cite[Proposition~1.5]{Kawamata1997}}]
\label{prop:lc-centers}
Let $(X,B)$ be an lc pair, where the variety $X$ has only
log terminal singularities,
and let $B$ be an effective $\QQ$-divisor.
Let $W_1,\, W_2\subset X$ be log canonical centers for $(X,B)$.
If $W_1\cap W_2\neq \varnothing$, then any irreducible component
of the intersection $W_1\cap W_2$
is also a log canonical center for $(X,B)$.
\end{prp}

In particular, the locus of log canonical centers contains minimal elements by 
inclusion.
Such centers are called \textit{minimal}.
If the pair $(X,B)$ is plt, then by Definition~\ref{def:sing}\ref{def:sing:plt} the log canonical centers are 
exactly the components of $\lfloor B\rfloor$. In particular, all they are minimal.

\subsection{Adjunction of log divisors}

The classical adjunction formula expresses the canonical class of a divisor
in terms of the canonical class of the total space.
It turns out that one can write down analogous adjunction formulas
for subvarieties that are log canonical centers of any codimension.
Such formulas are effectively applied in inductive procedures in the minimal
model theory.

\begin{teo}[{\cite[\S3]{Shokurov:flips}},
{\cite[Ch.~16]{Utah}}]
\label{adjunction-divisor}
Let $X$ be a normal variety and let $B$ be an effective $\QQ$-divisor on $X$ such that the pair
$(X,B)$ is plt.
Then $\lfloor B\rfloor$ is a normal subvariety.
Let $Z$ be a component of $\lfloor B\rfloor$ \textup(thus $Z$ is a minimal center of log canonical
singularities of codimension~$1$ for $(X,B)$\textup).
Then there exists a canonically defined effective $\QQ$-divisor $B_Z$ on~$Z$
such that
\begin{equation}
\label{adjunction:div}
(K_X+B)|_Z\qq K_Z+B_Z
\end{equation}
and the pair $(Z,B_Z)$ is klt.
\end{teo}

\begin{examples}
\begin{enumerate}
\item
If the variety $X$ is nonsingular, then $B_Z=(B-Z)|_Z$ and~\eqref{adjunction:div} 
is 
the usual adjunction formula.
\item
Let $X=X_2\subset\PP^3$ be a quadratic cone with vertex $P$
and let $Z\subset X$ be a line. Then for $B=Z$ the divisor $B_Z$ has the form
$B_Z=\frac12 P$.
\end{enumerate}
\end{examples}

In the general case there exists a recipe of computations of the coefficient $B_Z$ in any 
prime
divisor $P\subset Z$
in terms of $B$ and the singularities of $X$ along $P$.

\begin{exa}
Let $X$ be a surface and let $Z$ be a curve on~$X$ such that the pair
$(X,B=Z)$ is
plt. In this case near any point $P\in Z\subset X$,
which is singular on $X$, the pair $(X,Z)$ is analytically isomorphic to 
the quotient 
$$
(\CC^2,\,\{x_1=0\})/\mumu_m,
$$
where the action of $\mumu_m$ on $\CC^2$ is diagonal and free outside 
the origin.
Then the coefficient of $P$ in the divisor $B_Z$ is equal to $1-1/m$
\cite[Proposition~3.9]{Shokurov:flips},~\cite[Proposition~16.6]{Utah}.
\end{exa}

For applications, it is important to have a generalization of the 
formula~\eqref{adjunction:div}
to subvarieties of higher codimension.

\begin{teo}[\cite{Kawamata:Adj-1}]
\label{adjunction-codimension2}
Let $X$ be a variety with log terminal
singularities and let $B$ be an effective $\QQ$-divisor on~$X$ such that the pair
$(X,B)$ is lc. Let $Z$ be a minimal center of log canonical
singularities for $(X,B)$.
Assume that $\codim Z=2$.
Then there exist effective $\QQ$-divisors $B_Z$ and
$M_Z$ on $Z$ such that
\begin{equation}
\label{eq:adj:BM}
(K_X+B)|_Z\qq K_Z+B_Z+M_Z
\end{equation}
and the pair $(Z,B_Z+M_Z)$ is klt.
\end{teo}

The divisor $B_Z$ is called \textit{divisorial} and $M_Z$ is called
\textit{moduli}
part of the adjunction.
They have different nature.
The divisor $B_Z$ is ``physically'' fixed, it is uniquely defined by the pair
$(X,B)$ and its coefficients $B_Z$ are computed locally.
In turn, $M_Z$ is defined up to $\QQ$-linear equivalence and
depends on
variation of $(X,B)$ along $Z$.
In particular, $M_Z$ does not depend on those components of $B$ that does not pass 
through
$Z$.

There is a generalization of the formula~\eqref{eq:adj:BM} for subvarieties
of any codimension~\cite{Kawamata:Adj-2},
however we do not need it.

\subsection{Multiplier ideal}

An algebraic analogue of the locus of log canonical singularities
is the multiplier ideal. 
Below we provide its definition and simplest properties.
Due to lack of time, we omit the details,
see~\cite{Lazarsfeld2004,Lazarsfeld2010}.

\begin{dfn}
\label{def:m-i}
Let $B$ be an effective $\QQ$-divisor on a nonsingular variety $X$.
Fix a log resolution $f\colon \widetilde{X}\to X$ of the pair $(X,B)$,
i.e. a surjective proper birational morphism such that $\widetilde{X}$ 
is nonsingular, and 
the union of the exceptional locus and the proper transform of $B$
is a divisor with simple normal crossings. Then the
\textit{multiplier ideal}
is the sheaf
\begin{equation}
\label{eq:m-i}
\III(X,B):= f_* \OOO_{\widetilde{X}} \left(K_{\widetilde{X}/X}-\lfloor
f^*B\rfloor\right),
\end{equation}
where $K_{\widetilde{X}/X}=K_{\widetilde{X}/X}-f^*K_X$ is the relative 
canonical
class.
\end{dfn}
There are more general variants of the definition of the sheaf $\III(X,B)$ (e.g., for
singular varieties $X$). However they are needed for our purposes.

It can be checked directly that the sheaf $\III(X,B)$ does not depend on choice of
log resolution. We may assume that the support of the divisor $K_{\widetilde{X}/X}$
is concentrated 
in the exceptional locus and since the variety $X$ is nonsingular, it
is effective. Thus $f_* \OOO_{\widetilde{X}} (K_{\widetilde{X}/X})=\OOO_X$ and it follows
from~\eqref{eq:m-i} that there is an inclusion
$$
\III(X,B)\hookrightarrow f_* \OOO_{\widetilde{X}} (K_{\widetilde{X}/X})=\OOO_X.
$$
Therefore, $\III(X,B)$ is indeed an ideal sheaf in $\OOO_X$.
If the pair $(X,B)$ is lc, then the ideal $\III(X,B)$ defines
a reduced subscheme in $X$ whose support coincides with the locus of log canonical
singularities $\LCS(X,B)$.

\begin{examples}
\label{txampl-m-i}
\begin{enumerate}
\item
If the support of $B$ has simple normal crossings, then
$\III(X,B)=\OOO_X(-\lfloor B\rfloor)$.
\item
If $B$ is an integral divisor, then
$\III(X,B)=\OOO_X(-B)$.
\end{enumerate}
\end{examples}

Let us reformulate Definition~\ref{def:m-i}. Write
$$
f^*B=\sum m_iE_i,\quad m_i\in\QQ_{\ge 0},\qquad K_{\widetilde{X}/X}=\sum
a_iE_i,\quad m_i\in\ZZ_{\ge 0},
$$
where $E_i$ are distinct prime divisors on $\widetilde{X}$.
Then the sections of the sheaf $\III(X,B)$ over an open set $U\subset X$ are exactly those
functions $\varphi \in\CC(X)$ that are
regular on $U$ and such that
$$
\ord_{E_i}(\varphi) \ge \lfloor m_i \rfloor -a_i.
$$

\subsection{Vanishing Theorems}

We will need several generalizations of the Kodaira Vanishing Theorem.
The first of these was proposed independently by Kawamata and Viehweg 
(see~\cite{KMM}).

\begin{teo}[Kawamata--Viehweg]
\label{vanishing:KV}
Let $X$ be a nonsingular projective variety and let
$B$ be a nef and big $\QQ$-divisor on $X$ such that
the support of its fractional part has only normal intersections. Then
$$
H^q(X,\ \OOO_X(K_X+\lceil B\rceil))=0,\quad \forall q>0.
$$
\end{teo}

\begin{teo}[Nadel's Vanishing Theorem
\cite{Lazarsfeld2004}]
\label{theorem-Nadel}
Let $X$ be a nonsingular projective variety and
$B$ be an effective $\QQ$-divisor on $X$.
Let $M$ be an integral divisor on $X$ such that the difference
$M-(K_X+B)$ is a nef and big divisor.
Then
$$
H^q\bigl(X,\ \III(X,B)\otimes \OOO_X(M)\bigr)=0,\qquad \forall q >0.
$$
\end{teo}

Assume that in the above notation the support of $B$ has simple normal crossings.
Then
$$
\III(X,B)\otimes \OOO_X(M)=\OOO_X\bigl(M-\lfloor B\rfloor\bigr)=
\OOO_X\bigl(K_X+\lceil
M-K_X- B\rceil\bigr),
$$
where the divisor $ M-K_X- B$ is nef and big.
Thus in this case Theorem~\ref{theorem-Nadel} is exactly the Kawamata--Viehweg Vanishing Theorem.

The following fact is extremely useful when working with the locus of log canonical 
singularities.
It was first formulated and proved by V.~V.~Shokurov
{\cite[Lemma~5.7]{Shokurov:flips}}. An elegant proof
working in any dimension has been suggested by J. Koll\'ar 
{\cite[Theorem~17.4]{Utah}}.

\begin{teo}[Shokurov's connectedness theorem]
Let $X$ be a nonsingular projective variety and
let $B$ be an effective $\QQ$-divisor on $X$ such that the
$\QQ$-divisor
$-(K_X+B)$ is nef and big.
Then the set $\LCS(X,B)$ is connected.
\end{teo}

\begin{proof}
Let $Z:= \LCS(X,B)$.
We have the exact sequence
$$
0 \longrightarrow\III(X,B) \longrightarrow
\OOO_X \longrightarrow\OOO_Z\longrightarrow 0.
$$
By Nadel's Vanishing Theorem $H^1(X,\III(X,B))=0$.
Therefore, $H^0(X,\,\OOO_Z)=\CC$.
\end{proof}

\subsection{Two-dimensional case} 

Now we apply the above results to the proof of the existence
of good divisors.
We start with the two-dimensional case.

\begin{teo}
\label{stheorem:proba}
Let $X$ be a del Pezzo surface.
There exists a nonsingular curve $H\in|{-}K_X|$.
\end{teo}

\begin{proof}
By the Riemann--Roch Theorem
$$
\dim|{-}K_X|=K_X^2>0.
$$
In particular, $|{-}K_X|\neq\varnothing$.
Assume that a general divisor $H\in|{-}K_X|$ is singular.
Put
$$
\lambda:= \max \{ t\in\QQ \mid \text{the pair $(X,t H)$ is lc}\}.
$$
It is clear that $0<\lambda\le 1$ and the pair $(X,\lambda H)$ is not 
klt. In particular, $\LCS(X,\lambda H)\neq \varnothing$.
Let $Z\subset\LCS(X,\lambda H)$ be a minimal center of log canonical
singularities.
By Bertini's theorem $\Bs|{-}K_X|\neq \varnothing $ and $Z\subset\Bs|{-}K_X|$.

Assume that $Z$ is a curve. According to the above, $Z$ is a
fixed component of the linear system $|{-}K_X|$ and $\lambda H=Z+D$, where
$D$ is an effective $\QQ$-divisor that does not contain $Z$ as its component.
Since $H$ is an ample divisor, its support is connected.
According to Proposition~\ref{prop:lc-centers},
the curve $Z$ does not meet other log canonical centers. Thus
$Z$ is a connected component of $\lfloor \lambda H\rfloor$ and $\lambda<1$.
Then the divisor $-(K_X+\lambda H)$ is ample and by 
Shokurov's connectedness theorem
$Z=\LCS(X,\lambda H)$. Note also that by Theorem~\ref{adjunction-divisor}
the curve $Z$ is nonsingular and $-K_Z=(K_X+\lambda H)|_Z-D|_Z$
is an ample divisor. Therefore, $Z$ is rational.

Now we define an auxiliary $\QQ$-divisor $B$ on $X$ as follows.
If $Z=\LCS(X,\lambda H)$, we put $B:= \lambda H$.
Let $Z\varsubsetneq \LCS(X,\lambda H)$. Then $Z$ is a point.
Take a general hyperplane section $F\subset X$ passing through $Z$ and
for $0\le \varepsilon\le \lambda$, $t\ge 0$ we consider the following 
$\QQ$-divisor
$$
B_{\varepsilon, t}:= (\lambda-\varepsilon) H+t F.
$$
It is clear that the pair $(X,B_{\varepsilon, 0})$ is klt whenever
$\varepsilon>0$. Thus the number
$$
\delta=\delta(\varepsilon):= \max \left\{ t\in\QQ\mid \text{the pair 
$\big(X,\,(\lambda-\varepsilon) H+tF\big)$ is lc}\right\}
$$
is positive. Moreover, the function $\delta(\varepsilon)$ is piecewise linear,
monotonic and continuous,
and $\delta(0)=0$ (see Remarks~\ref{rem:discr1} and~\ref{rem:discr2}). Take 
$0<\varepsilon\ll 1$.
Then $0<\delta(\varepsilon)\ll 1$ and~$Z=\LCS(X,B_{\varepsilon,
\delta(\varepsilon)})$.
Put $B:= B_{\varepsilon, \delta(\varepsilon)}$.

We have the exact sequence
$$
0 \longrightarrow \III(X,B)\otimes \OOO_X(-K_X) \longrightarrow
\OOO_X(-K_X) \longrightarrow\OOO_Z(-K_X)\longrightarrow 0.
$$
Since the divisor $-2K_X-B$ ``very little" differs from
$$
-2K_X-\lambda H=(2-\lambda)H,
$$
$-2K_X-B$ is ample and by Theorem~\ref{theorem-Nadel} we have
$$
H^1\bigl(X,\III(X,B)\otimes \OOO_X(-K_X)\bigr)=0.
$$
Thus the restriction map
$$
H^0\bigl(X,\,\OOO_X(-K_X)\bigr) \longrightarrow H^0\bigl(Z, \OOO_Z(-K_X)\bigr)
$$
is surjective.
Recall that $Z$ is either a point or a nonsingular rational curve.
Thus
$$
H^0(Z, \OOO_Z(-K_X))\neq 0.
$$
Therefore, $Z\not\subset\Bs|{-}K_X|$, a contradiction.
\end{proof}

\subsection{Proof of the main theorem} 

Let us proceed to the proof of Theorem~\ref{theorem-smooth-divisor}.
According to~\eqref{eq:sections} the linear system $|H|$ is non-empty
and, moreover, $\dim|H|\ge 2$.

\begin{prp}
\label{proposition-H-H}
Let $H_1,\dots,H_m\in|H|$ be distinct divisors and let $D:= \sum \lambda_i
H_i$,
where $\lambda_i\ge 0$ and
\begin{equation}
\label{equation:lambda}
\sum \lambda_i<\iota+1.
\end{equation}
Assume that the pair $(X,D)$ is
lc.
Then any log canonical center of $(X,D)$ is not contained in $\Bs|H|$.
\end{prp}

\begin{proof}
We may assume that $\LCS(X,D)\neq \varnothing$, i.e.
the pair $(X,D)$ is not klt.
Let $Z'\subset\LCS(X,D)$ be a center of log canonical singularities.
Assume that $Z'\subset\Bs|H|$.
We may assume that $Z'$ is a minimal center.
According to~\eqref{equation:lambda}, the divisor $H-(K_X+D)$ is ample.

Now as in the proof of~\ref{stheorem:proba} we ``slightly perturb''
the divisor $D$ to localize the log canonical center.

\begin{lem}
In the above assumptions there exists an effective $\QQ$-divisor $B$ on~$X$ such
that the divisor $H-(K_X+B)$ is ample and 
the locus of log canonical singularities of the pair $(X,B)$ consists exactly of 
a single
center $Z$ that is contained in $\Bs|H|$.
\end{lem}

\begin{proof}
If $Z'=\LCS(X,D)$, we put $B:= D$ and $Z=Z'$.
Let $Z'\varsubsetneq \LCS(X,D)$. For $0\le \varepsilon\le 1$, $t\ge 0$
we consider the following $\QQ$-divisor
$$
D'_{\varepsilon, t}:= (1-\varepsilon)D+t F,
$$
where $F$
is a general hyperplane section ($F\in|nH|$, $n\gg 0$) passing
through $Z'$. Put
$$
\delta(\varepsilon):=\max\{ t\in\QQ \mid \text{the pair $(X,D'_{\varepsilon, t})$ is lc}\}.
$$
The function $\delta(\varepsilon)$ is piecewise linear, monotonous and continuous,
and $\delta(0)=0$ (see Remarks~\ref{rem:discr1} and~\ref{rem:discr2}). Take 
$0<\varepsilon\ll 1$.
Then $0<\delta(\varepsilon)\ll 1$ and
$$
Z'=\LCS\big(X,D'_{\varepsilon, \delta(\varepsilon)}\big)\varsubsetneq
\LCS(X,D).
$$
If $Z'$ is a
minimal log canonical center for $\big(X,D'_{\varepsilon,
\delta(\varepsilon)}\big)$, we put $B:= D'_{\varepsilon,
\delta(\varepsilon)}$ and $Z=Z'$.
If there exists a minimal center $Z''\varsubsetneq Z'$, then
we will continue the process replacing $Z'$ with $Z''$. The process terminates because
the dimension of the center of log canonical singularities is decreasing. As a result, 
we obtain $B$ and $Z$ as claimed in the assertion of the lemma.
\end{proof}

Going back to the proof of Proposition~\ref{proposition-H-H}, consider the exact sequence
$$
0 \longrightarrow\III(X,B)\otimes \OOO_X(H) \longrightarrow
\OOO_X(H) \longrightarrow\OOO_Z(H)\longrightarrow 0.
$$
By our assumption $Z\subset\Bs|H|$. Therefore, all the sections of
$H^0(X,\,\OOO_X(H))$
vanish on $Z$. On the other hand, since the divisor $H-K_X-B$ is ample, by
Theorem~\ref{theorem-Nadel} we have
$$
H^1(X,\III(X,B)\otimes \OOO_X(H))=0.
$$
Thus the restriction map
$$
H^0(X,\,\OOO_X(H)) \longrightarrow H^0(Z, \OOO_Z(H))
$$
is surjective. It remains to show that
$$
H^0(Z, \OOO_Z(H))\neq 0.
$$
This is obvious, if $Z$ is a point.
If $Z$ is a curve, then by Theorem~\ref{adjunction-codimension2}
it is nonsingular and
$$
2\g(Z)-2=\deg K_Z\le (K_X+B)\cdot Z<H\cdot Z.
$$
Thus
$$
\deg \OOO_Z(H)\ge \max \{1,\, 2\g(Z)-1\}.
$$
Then by the Riemann--Roch Theorem $H^0(Z, \OOO_Z(H))\neq 0$.

It remains to consider the case
$\dim(Z)=2$, i.\,e.\ where $Z$ is an irreducible component of
$\lfloor B\rfloor$. In this case, by our construction, the surface $Z$
is also a minimal log canonical center for $(X,D)$.
This means that the pair
$(X,D)$ is plt along $Z$, moreover,
$Z$ is an irreducible component of
$\lfloor D\rfloor$ and also it is a fixed component of $|H|$.
According to our assumption~\eqref{equation:lambda} the divisor $-(K_X+D)$ is 
ample.
By Theorem~\ref{adjunction-divisor} the surface $Z$ is normal and the following equality 
$$
(K_X+D)|_Z=K_Z+D_Z
$$
holds, where $D_Z\ge 0$, the pair $(Z, D_Z)$ is klt, and the divisor
$-(K_Z+D_Z)$ is ample.
By Lemma~\ref{lemma:log-del-Pezzo} below any ample invertible sheaf on $Z$
has non-zero global sections. The contradiction proves the proposition.
\end{proof}

\begin{dfn}
Let $S$ be a normal projective surface and
let $\Delta$ an effective $\QQ$-Weil divisor on $S$ such that the pair
$(S,\Delta)$ is
klt and the divisor $-(K_S+\Delta)$ is nef and big.
Then the pair $(S,\Delta)$ is called a \textit{weak log del Pezzo surface}.
\end{dfn}

\begin{lem}
\label{lemma:log-del-Pezzo}
Let $(S,\Delta=\sum \delta_i\Delta_i)$ be a weak log del Pezzo surface and
let $A$ be a nef and big Cartier divisor
on $S$.
Then the surface~$S$ is rational and
$$
H^0(S,\,\OOO_S(A))\neq 0.
$$
\end{lem}

\begin{proof}
First we assume that the surface $S$ is nonsingular.
Let us prove the rationality of $S$. Assume the converse. It is clear that
$|nK_S|=\varnothing$ for all $n>0$.
Thus the surface $S$ is birationally ruled and there exists a contraction 
$\tau\colon S\to \Gamma$ to a curve of positive genus so that the general fiber of 
$\tau$ is a nonsingular
rational curve. If for some component $\Delta_i$ we have
$\Delta_i^2\ge 0$, then we can replace $\Delta$ with $\Delta':= 
\Delta-\delta_i\Delta_i$ with preserving of all conditions. Thus we may assume,
that $\Delta_i^2<0$ for all $i$ (or $\Delta=0$). Then
$$
(K_S+\Delta_i)\cdot \Delta_i=(K_S+\Delta)\cdot
\Delta_i+(1-\delta_i)\Delta_i^2-(\Delta-\delta_i\Delta_i)\cdot \Delta_i<0.
$$
Therefore, $\p(\Delta_i)=0$, i.e. $\Delta_i$ is a smooth rational curve. If
$\tau(\Delta_i)=\Gamma$, then $\g(\Gamma)\le \p(\Delta_i)=0$.
This contradicts our assumption. Hence, all the components of $\Delta$ lie 
in the fibers of $\tau$.
For any curve $\Gamma$ lying in the fibers of $\tau$ we have
$$
2\p(\Gamma)-2=(K_X+\Gamma)\cdot \Gamma<0,\qquad \p(\Gamma)<0.
$$
This implies that the support of $\Delta$ has simple normal crossings. 
Since $\lceil-\Delta\rceil=0$, by the Kawamata--Viehweg Vanishing Theorem 
applied to
$-(K_S+\Delta)$ we obtain $H^1(S,\,\OOO_S)=0$ and then $\g(\Gamma)=0$. This 
again
contradicts our assumption. Thus the surface $S$ is rational.

Furthermore, by the Serre duality
$$
\hr^2(S,\,\OOO_S(A))=\hr^0(S,\,\OOO_S(K_S-A))=0.
$$
Since the divisors $-(K_S+\Delta)$ and $A$ are nef, we have
$-(K_S+\Delta)\cdot A\ge 0$. By the Riemann--Roch Theorem
$$
\hr^0(S,\,\OOO_S(A))\ge
\frac12(A^2-A\cdot K_S)+\upchi(\OOO_S)\ge \frac12(A^2+A\cdot \Delta)>0.
$$

Let now the surface $S$ be singular.
Consider its minimal resolution of singularities 
$f\colon \widetilde{S}\to S$. We can write
$$
K_{\widetilde{S}}+\widetilde{\Delta}=f^*(K_S+\Delta),
$$
where $\widetilde{\Delta}$ is an effective $\QQ$-divisor on $\widetilde{S}$ 
(see~\cite{P:book:sing-re})
and the pair $(\widetilde{S}, \widetilde{\Delta})$ is klt.
Thus $(\widetilde{S}, \widetilde{\Delta})$ is again a weak log del Pezzo surface.
According to the proven above, the surface $\widetilde{S}$ is rational, hence $S$
is rational as well.

Since $f_*\OOO_{\widetilde{S}}=\OOO_S$, we have $\OOO_S(A)=f_*f^*\OOO_S(A)$ (the projection formula).
The singularities of the surface $S$ are rational~\cite{P:book:sing-re}, i.\,e.\
$R^qf_*\OOO_{\widetilde{S}}=0$ for $q>0$. Then from the Leray spectral sequence
we obtain
$$
H^0(S,\,\OOO_S(A))=H^0(\widetilde{S},f^*\OOO_S(A))\neq 0.\qedhere
$$
\end{proof}

From Proposition~\ref{proposition-H-H} one can deduce the following

\begin{cor}
Let $H_1,\dots,H_m\in|H|$ be general elements. Put $D:= \sum \lambda_i
H_i$,
where $0\le \lambda_i\le 1$, $\sum \lambda_i\le \iota+1$.
Then the pair $(X,D)$ is
lc.
\end{cor}

\begin{proof}
Assume the converse.
Decreasing the coefficients $\lambda_i$, we can achieve the situation where the pair
$(X,D'=\sum \lambda_i' H_i)$ is maximally lc and $\lambda_i'<1$
for all $i$.
By Bertini's theorem the surfaces $H_i$ are nonsingular outside
$\Bs|H|$ and meet each other 
transversely.
Moreover the divisor $\sum H_i$ has only simple normal crossings on
the open subset
$X\setminus \Bs|H|$. Hence, the pair $(X,D')$ is klt outside $\Bs
|H|$.
According to Proposition~\ref{proposition-H-H}
any minimal log canonical center of $(X,D)$ is not contained in $\Bs|H|$, a 
contradiction.
\end{proof}

\begin{cor}
\label{corollary-LC}
Let $H_1,\dots H_{\iota+1}\in|H|$ be general divisors.
Then the pair $\left(X,\sum_{i=1}^{\iota+1}H_i\right)$
lc.
\end{cor}

\begin{cor}
The linear system $|H|$ has no fixed components.
\end{cor}

\begin{proof}
Indeed, otherwise the divisor $\sum_{i=1}^{\iota+1} H_i$ has a component
of multiplicity $\ge 2$, which contradicts the lc property of the pair $\left(X,\sum
H_i\right)$ as in the corollary above.
\end{proof}

\begin{cor}
\label{corollary-generically-smooth}
Let $P\in\Bs|H|$ be a point.
Then a general divisor $H\in|H|$ is nonsingular at~$P$.
\end{cor}

\begin{proof}
If all elements of $H$ are singular in $P$, then by blowing up $P$ we obtain an
exceptional
divisor $E$ with discrepancy
$$
a\Bigl(E,\sum_{i=1}^{\iota+1}H_i,X\Bigr)=2-
\sum_{i=1}^{\iota+1}\mult_P(H_i)\le 2-2(i+1)=-2i\le -2.
$$
This contradicts Corollary~\ref{corollary-LC}.
\end{proof}

\begin{cor}
\label{corollary:Bs:index2}
If $\iota(X)\ge 3$, then $\Bs|H|=\varnothing$.
If $\iota(X)=2$, then $\dim \Bs|H|\le 0$ and a general divisor $H\in|H|$
is nonsingular.
\end{cor}

\begin{proof}
Similarly to the proof of corollary above.
\end{proof}

Now we are in position to complete the proof of Theorem~\ref{theorem-smooth-divisor}.

\begin{proof}[Proof of Theorem~\ref{theorem-smooth-divisor}]
Assume that a general divisor $H$ is singular.
Then by Bertini's theorem $\Bs|H|\neq \varnothing$
and $\Sing(H)\subset\Bs|H|$.
According to Corollary~\ref{corollary:Bs:index2}, we may
assume that $\iota(X)=1$.
Moreover, the pair $(X,H+H')$ is
lc for general divisors $H, H'\in|H|$ (Corollary~\ref{corollary-LC}).
Since the surface~$H$ is nonsingular outside $H\cap H'$, the pair $(X,H)$ is
plt. By Theorem~\ref{adjunction-divisor} the surface $H$
is normal (and irreducible because it is connected) and has only
log terminal singularities.
On the other hand, by the adjunction formula $K_H=0$.
In particular, $K_H$ is a Cartier divisor. Log terminal singularities of
surfaces, whose the canonical class is Cartier, are so-called
\textit{Du Val singularities} \cite[\S9]{P:book:sing-re}.
As in the proof of
Proposition~\ref{Fano:adjunction}\ref{Fano:adjunction2} we have
$H^1(H,\,\OOO_H)=0$.
Thus $H$ is a singular \K3 surface (see Definition~\ref{def:DuVal:K3}).

Then we have $\dim \Bs|H|\le 1$.
Let $C$ be an irreducible component of the set $\Bs|H|$ that is contained in the singular locus of 
a general member $H\in|H|$.
According to Corollary~\ref{corollary-generically-smooth}, a general divisor
is nonsingular at a general point of $C$. Hence, $C$ is a curve and a general 
member $H\in|H|$
has only isolated singularities (i.\,e.\ elements of $H$
have a \textit{moving} singular point on $C$).
Since $H^1(X,\,\OOO_X)=0$, the restriction map
$$
H^0(X,\,\OOO_X(H')) \longrightarrow H^0(H,\,\OOO_H(H'))
$$
is surjective.
Thus the complete linear system $|H'_H|$ has the curve~$C$ as its fixed
component.
But according to Corollary~\ref{DuVal:K3} the surface
$H$ must be is nonsingular along $C$. The contradiction completes the proof of
Theorem~\ref{theorem-smooth-divisor}.
\end{proof}

\begin{zadachi}
\eitem
Let $C_1$ and $C_2$ be nonsingular curves 
given by the equations $y=0$ and $y=x^2$ on the plane $\CC^2$.
For which values $\alpha_1$ and $\alpha_2$
the pair $(\CC^2,\,\alpha_1C_1+\alpha_2 C_2)$ is lc?

\eitem
Let $C$ be a nonsingular curve 
given by the equation $y^2=x^3$ on the plane $\CC^2$.
For which values $\alpha$
the pair $(\CC^2,\,\alpha C)$ is lc?

\eitem
Prove assertions from Example~\ref{txampl-m-i}.

\eitem
Let $X$ be a del Pezzo surface and let $A$ be an ample divisor on $X$.
Show that the linear system $|A|$ contains a nonsingular curve.

\eitem
Let $X$ be a \K3 surface and let $A$ be an ample divisor on $X$.
Show that for a general divisor $D\in|A|$ the pair $(X,D)$ is lc.

\eitem
Let $X$ be a nonsingular three-dimensional variety and let $D$ be a divisor on $X$
with simple normal crossings. Assume that the divisor $-(K_X+D)$ is ample.
Prove that the linear system $\mo|-(K_X+D)|$ contains a nonsingular divisor.

\eitem
Generalize theorem on~the existence of a nonsingular curve in the linear system
$|{-}K_X|$ to the case of a \textit{weak} del Pezzo surface $X$ (i.e. a 
surface 
with nef and big anticanonical class).

\eitem
\label{zadacha-smooth}
Let $X$ be a Fano variety of dimension $n$ and index $\iota(X)\ge n$.
Show that the linear system $\mo|-\frac 1\iota K_X|$ contains a nonsingular
divisor.

\eitem
Using Theorem~\ref{theorem-smooth-divisor}, find an upper bound of the
Picard number of Fano threefolds (the bound will be not
optimal, cf. Theorem~\ref{theorem51}).

\eitem
Let $S$ be a normal projective surface and
let $\Delta$ an effective $\QQ$-Weil divisor on $S$ such that the pair
$(S,\Delta)$ is plt
and the divisor $-(K_S+\Delta)$ is ample.
Prove that the surface $S$ is rational (generalization of
Lemma~\ref{lemma:log-del-Pezzo}). Is it possible to weaken the conditions of this 
assertion replacing the plt property with lc one?
\end{zadachi}

\newpage\section{Fano threefolds of index~$2$}
\label{sect:index2}

In this section we discuss the classification of Fano threefolds of 
index~$2$.
Recall that such varieties are also called 
\textit{del Pezzo threefolds} (see Definition~\ref{def:coindex}).
The main classification result is summarized in Table 
\ref{table:del-Pezzo-3-folds}, page~\pageref{table:del-Pezzo-3-folds}.


Recall that the degree of a del Pezzo threefold is the number
$$
\dd(X):= \frac18(-K_X)^3.
$$

\begin{rem}
\label{rem:Lef}
According to Theorem~\ref{theorem-smooth-divisor}, on a del Pezzo threefold $X$
a general member $H\in \mo|-\frac12K_X|$ is nonsingular.
By the Lefschetz hyperplane theorem
the restriction $\Pic(X)\to \Pic(H)$ is an embedding and its image
is a primitive sublattice in $\Pic(H)$.
On the other hand, by the adjunction formula
$K_H=-H|_H$. Thus $-K_H$ is a primitive ample element
of the lattice $\Pic(H)$
and so
$H$ is a del Pezzo surface of degree $\dd(X)$
with $\iota(H)=1$. Thus
$$
1\le \dd(X)\le 8
$$ by Noether's formula (see
Corollary~\ref{cor:delPezzo:}) and $H\not \simeq \PP^1\times \PP^1$ in the case $\dd(X)=8$.
\end{rem}

\subsection{Inductive step} 
Below we will systematically use the induction on the dimension.
For this, the following lemma is very useful.

\begin{lem}
\label{lemma-linear-systems}
Let $V$ be an algebraic variety such that $H^1(V,\,\OOO_V)=0$. Let 
$D\subset V$ be a \textup(reduced\textup) subvariety of codimension~$1$
which is a Cartier divisor.
Then
$$
\Bs|\OOO_D(D)|=\Bs|\OOO_V(D)|.
$$
In particular, if the linear system $|\OOO_D(D)|$ on $D$ has no base
points, then the same holds for
the linear system $|D|$ on $V$.
In this case $|D|$ defines a morphism $\Phi_{|D|}\colon V\to \PP^{\dim|D|}$
so that the diagram
$$
\xymatrix@C=6em{
V\ar[r]^{\Phi_{|D|}} & \PP^{\dim|D|}
\\
D\ar[r]^{\Phi_{|\OOO_D(D)|}}\ar@{^{(}->}[u] & \PP^{\dim|D|-1}\ar@{^{(}->}[u]
}
$$
is commutative, where $\Phi_{|\OOO_D(D)|}$ is a morphism given by the linear system
$|\OOO_D(D)|$.
Here $\Phi_{|D|}(D)$ is a hyperplane section of the variety
$\Phi_{|D|}(V)$.
\end{lem}

\begin{proof}
It is clear that $\Bs|D| \subset\Supp (D)$ and $\Bs|D|=\Bs|D|_D$, where $|D|_D$ 
is 
the restriction of the linear system $|D|$ to the subvariety $D$.
The ideal sheaf of the subscheme~$D$ in $X$ is invertible and coincides with $\OOO_V(-D)$. 
Thus
there exists the following
exact sequence (see~\cite[Ch.~II, \S\S3,~6]{Hartshorn-1977-ag})
$$
0\longrightarrow\OOO_V(-D) \longrightarrow\OOO_V \longrightarrow\OOO_D \longrightarrow 0.
$$
The sheaf $\OOO_V(D)$ is also invertible, hence the tensor multiplication by it
preserves 
the exactness:
$$
0\longrightarrow\OOO_V \longrightarrow\OOO_V(D) \longrightarrow\OOO_D(D) \longrightarrow 0.
$$
By our condition $H^1(V,\,\OOO_V)=0$
the restriction map
$$
H^0(V,\,\OOO_V(D)) \longrightarrow H^0(D,\OOO_D(D))
$$
is surjective. Therefore, $|D|_D=|\OOO_D(D)|$.
\end{proof}

\subsection{Anticanonical map of del Pezzo surfaces} 

We apply the result~\ref{theorem-smooth-divisor} on~the existence of a smooth
divisor to a description of projective embeddings of del Pezzo
surfaces and del Pezzo threefolds.
Recall that the degree $d=K_X^2$ of a del Pezzo surface $X$ takes the values
$1\le d \le 9$ (see
Corollary~\ref{cor:delPezzo:}). The del Pezzo surfaces of degree~$8$ 
and~$9$
are described by Exercises~\ref{zad:DP9} and~\ref{zad:DP8} Section~\ref{sec1}.

\begin{prp}
\label{del-pezzo-projective}
Let $X$ be a del Pezzo surface of degree $d$. Then the following
assertions hold.
\begin{enumerate}
\item
\label{del-pezzo-projective-d=1}
For $d=1$ the linear system $|{-}K_X|$ is a pencil with a unique base 
point,
the linear system $|{-}2K_X|$ defines a morphism
$$
\Phi_{|{-}2K_X|}\colon X\longrightarrow\PP^3,
$$
that is a double cover of its image
$Q\subset\PP^3$; this image is a quadratic cone, the branch divisor is 
nonsingular
and is cut out by a cubic that does not pass through the vertex. Conversely, if
a nonsingular
surface $X$ is represented in the form of
a double cover of a quadratic cone $Q\subset\PP^3$
branched over a curve $B=B_6\subset Q$
that is cut out by a cubic,
then $X$ is a del Pezzo surface of degree~$1$.
\end{enumerate}
\begin{enumerate}
\item
\label{del-pezzo-projective-d=2}
For $d=2$ the linear system $|{-}K_X|$ defines a morphism
$$
\Phi_{|{-}K_X|}\colon X\longrightarrow\PP^2
$$ 
that is a double cover branched over a nonsingular
curve $B_4\subset\PP^2$ of degree $4$.
Conversely, if a nonsingular surface $X$ is represented in the form of
a double cover of $\PP^2$ branched over a curve $B=B_4\subset\PP^2$
of degree $4$,
then $X$ is a del Pezzo surface of degree~$2$.

\item
\label{del-pezzo-projective-d=3}
For $d=3$ the linear system $|{-}K_X|$ defines a morphism
$$
\Phi_{|{-}K_X|}\colon X\longrightarrow\PP^3
$$
that is an embedding and its image $X_3\subset\PP^3$
is a cubic surface. Conversely, any nonsingular cubic surface
$X_3\subset\PP^3$ is a del Pezzo surface of degree $3$.

\item
\label{del-pezzo-projective-d=4}
For $d=4$ the linear system $|{-}K_X|$ defines a morphism
$$
\Phi_{|{-}K_X|}\colon X\longrightarrow\PP^4
$$
that is an embedding and its image $X_4\subset\PP^4$
is an intersection of two quadrics.
Conversely, any nonsingular intersection of two quadrics
$X_4\subset\PP^4$ is a del Pezzo surface of degree $4$.
\end{enumerate}
\end{prp}

\begin{rem}
In the case~\ref{del-pezzo-projective-d=1} the curve $B=B_6\subset Q$ is the
\textit{divisorial part} of the ramification.
Here the curve $B$ is nonsingular and the morphism 
$\Phi_{|{-}2K_X|}$
is branched also over the 
vertex of the cone $o\in Q$, i.e. 
the inverse image of~$o$ 
is 
a single (nonsingular) point on $X$.
The quadratic cone in $\PP^3$ is isomorphic to weighted projective plane
$\PP(1^2,2)$.

Also it should be noticed that in this case the map $\Phi_{|{-}K_X|}\colon 
X\dashrightarrow \PP^1$
is decomposed as follows
$$
\Phi_{|{-}K_X|}\colon X\xarr{\Phi_{|{-}2K_X|}} Q \dashrightarrow \PP^1,
$$
where $Q\dashrightarrow \PP^1$ is the projection from the vertex of the cone.
\end{rem}

\begin{proof}
We claim that for $d\ge 2$ the linear system $|{-}K_X|$ is base point free and defines a morphism
$$
\Phi_{|{-}K_X|}\colon X\longrightarrow\PP^d.
$$
Indeed, this follows from Lemma~\ref{lemma-linear-systems} applied
to a nonsingular element $D\in|{-}K_X|$
and the corresponding fact on elliptic curves: 
a linear system of degree $\ge 2$ on an elliptic curve is base point free 
(see~\cite[Ch.~IV,
Corollary~3.2]{Hartshorn-1977-ag}). Similarly, for $d=1$ a general member $D\in
|{-}2K_X|$ is a nonsingular curve of genus~$2$ and $-2K_X|_D=2K_D$ by 
the adjunction formula.
Again by Lemma~\ref{lemma-linear-systems} the linear system $|{-}2K_X|$ is base point free and defines a morphism
$$
\Phi_{|{-}2K_X|}\colon X\longrightarrow\PP^3.
$$
Also, for $d=1$
the linear system $|{-}K_X|$
is an elliptic pencil with a unique base point. The latter follows
from the fact that 
a linear system
of degree~$1$ on an elliptic curve is zero-dimensional, i.\,e.\ consists of 
unique element.

To prove~\ref{del-pezzo-projective-d=4}\footnote{Later we will give
another, more
conceptual proof of this fact.} denote by $Y\subset\PP^4$
the image of the morphism
$\Phi_{|{-}K_X|}$ and note that
$$
(\deg Y)\cdot (\deg \Phi_{|{-}K_X|})=(-K_X)^2=4.
$$
Since $Y$ does not lie in a hyperplane, we have $\deg Y>2$.
The only possibility is:
$$
\deg Y=4, \qquad \deg \Phi_{|{-}K_X|}=1.
$$
Thus $\Phi_{|{-}K_X|}$ is a birational morphism onto its image.
Since the divisor $-K_X$ is ample, this morphism is also finite.
This means that the map
$$
\Phi=\Phi_{|{-}K_X|}\colon X\longrightarrow Y
$$
either is an isomorphism or
coincides with the normalization of $Y$. For a general curve $D\in|{-}K_X|$ the 
restriction 
$\Phi_D\colon D\to \PP^3$ is an isomorphism onto its image (see~\cite[Ch.~IV,
Corollary~3.2]{Hartshorn-1977-ag}).
By Lemma~\ref{lemma-linear-systems} the image $\Phi(D)$ is a hyperplane section of
$Y$.
Hence, a general hyperplane section of $Y$ is nonsingular and so
$\dim \Sing(Y)\le 0$.

Let $\JJJ_Y$ be the ideal sheaf $Y$ in $\PP^4$. We have the exact sequence
$$
0 \longrightarrow\JJJ_Y(2) \longrightarrow\OOO_{\PP^4}(2)
\longrightarrow\OOO_Y(2) \longrightarrow 0.
$$
Since $\hr^0(\PP^4, \OOO_{\PP^4}(2))=15$ and
$$
\hr^0(Y, \OOO_Y(2))\le \hr^0(X, \Phi^*\OOO_Y(2))=\hr^0(X,
\OOO_X(-2K_X))=13
$$
(see~\eqref{eq:theorem:2dim-RR}), from the above exact sequence
we obtain $\hr^0(Y,\JJJ_Y(2))\ge 2$.
Two linearly independent global sections of the sheaf $\JJJ_Y(2)$ define two
distinct quadrics
$Q_1$, $Q_2\subset\PP^4$ passing through $Y$.
Thus
$Y\subset Q_1\cap Q_2$. It is clear that these quadrics $Q_i$ are irreducible.
Thus $\dim (Q_1\cap Q_2)=2$ and
from degree matching we obtain $Y=Q_1\cap Q_2$.
Now we know that $Y\subset\PP^4$ is a complete intersection.
Since singularities $Y$ are isolated, by the Serre criterion
(see~\cite[Ch.~III,~Proposition~2]{Mumford:red},
\cite[\S11.2]{Eisenbud1995}) the variety $Y$ is normal.
Therefore, $\Phi\colon X\to Y$ is an isomorphism.

The proof~\ref{del-pezzo-projective-d=3} is similar and left to the reader.

In the case~\ref{del-pezzo-projective-d=2} the morphism
$$
\Phi=\Phi_{|{-}K_X|}\colon X\to \PP^2
$$
is finite and has degree~$2$.
Let $B\subset\PP^2$ be the branch curve, let $l$ be the class of a line on
$\PP^2$
and let $n:= \deg B$.
By the Hurwitz formula~\eqref{equation-Hurwitz-2} we have
$$K_X=\Phi^*\left(K_{\PP^2}+\frac 12 B\right)=\Phi^*\left(\frac{n-6}2l\right).
$$
On the other hand, the morphism $\Phi$ is given by the linear system $|{-}K_X|$.
Therefore, $-K_X=\Phi^* l$. Since $\Pic(\PP^2)=\ZZ\cdot [l]$
and $\Phi^*l$ is not a torsion element in $\Pic(X)$ (because this is an
ample divisor), from the above formula
we immediately obtain $n=4$.
The last statement in~\ref{del-pezzo-projective-d=2} also follows 
from~\eqref{equation-Hurwitz-2}.

In the case~\ref{del-pezzo-projective-d=1} we have a morphism that is finite
is onto its 
image 
$$
\Phi=\Phi_{|{-}2K_X|}\colon X\longrightarrow Q\subset\PP^3,
$$
where
$$
(\deg Q)\cdot \left(\deg \Phi_{|{-}K_X|}\right)=(-2K_X)^2=4.
$$
The restriction of $\Phi$ to a general elliptic curve $D\in|{-}K_X|$ is finite of degree~$2$ onto
its image.
Thus $\deg Q=2$ and $ \deg \Phi=2$ and so $Q\subset\PP^3$ is an
(irreducible) quadric and the image of $D$ is a line on $Q$.
Moreover, since $D^2=(-K_X)^2=1$, all elliptic curves $D\in
|{-}K_X|$ pass through the unique base point $D_1\cap D_2$,
where $D_1,\, D_2\in|{-}K_X|$, $D_1\neq D_2$.
Thus on $Q$ there exists a pencil of lines $\Phi(D)$ passing through one 
point.
This means that $Q$ is a cone.

Let $l$ be the class of a line on $Q$. Then
$\Cl(Q)=\ZZ\cdot l$ and $\Pic(Q)=\ZZ\cdot (2l)$
(see~\cite[Ch.~II, \S6]{Hartshorn-1977-ag})
and $-K_Q=4l$. Let $B\subset Q$ be the branch curve. Then $B\sim n l$ for some $n$.
By the Hurwitz formula
\begin{equation}
\label{eq:Hurwitz1}
K_X=\Phi^*\left(K_Q+\frac 12 B\right)=\Phi^*\left(\frac{n-8}2l\right).
\end{equation}
As in the case~\ref{del-pezzo-projective-d=2} we obtain that $n=6$,
i.e. $B\sim6l$. From the exact sequence
$$
0\longrightarrow\OOO_{\PP^3}(1)\longrightarrow\OOO_{\PP^3}(3)\longrightarrow\OOO_Q(6l)\longrightarrow 0
$$
and the vanishing of $H^1(\PP^3,\,\OOO_{\PP^3}(1))$ we obtain that $B$ is cut out
by a cubic.
\end{proof}

\subsection{Half-anticanonical map of del Pezzo threefolds}
\begin{prp}
\label{DP:projective}
Let $X$ be a del Pezzo threefold of degree $d=\dd(X)$
and let $-K_X=2H$. Then for $d=1$,
the linear system $|H|$
has a unique base point and for $d\ge 2$ the linear system $|H|$ is base point free and defines a morphism $\Phi_{|H|}\colon X\to \PP^{d+1}$.
Moreover, the following assertions hold.
\begin{enumerate}
\item
\label{DP:projective-d=1}
For $d=1$ the linear system $|{-}K_X|$ is base point free and defines a morphism
$$
\Phi_{|{-}K_X|}\colon X\longrightarrow\PP^6,
$$
which is a double cover of its image
$Y\subset\PP^6$; this image is a cone over the Veronese surface $S_4\subset
\PP^5$ and the branch divisor is nonsingular and is cut out on $Y$ by a cubic.
Conversely, if a nonsingular three-dimensional variety~$X$ is represented in the form of
a double cover of the cone $Y=Y_4\subset\PP^6$ over the Veronese surface
$S_4\subset\PP^5$ branched over a surface $B\subset Y$ that is cut out on $Y$ 
by a cubic,
then $X$ is a del Pezzo threefold of degree~$1$.

\item
\label{DP:projective-d=2}
For $d=2$ the map
$$
\Phi_{|H|}\colon X\longrightarrow\PP^3
$$
is a double cover branched over a nonsingular
surface $B_4\subset\PP^3$ of degree $4$. Conversely, if a nonsingular 
three-dimensional
variety
$X$ is represented in the form of
a double cover of $\PP^3$ branched over a surface $B=B_4\subset\PP^3$
of degree $4$,
then $X$ is a del Pezzo threefold of degree~$2$.

\item
\label{DP:projective-d=3}
For $d=3$ the map $\Phi_{|H|}$ is an isomorphism onto its image and the image of
$X_3\subset\PP^4$
is a cubic hypersurface.
Conversely, any nonsingular cubic hypersurface
$X_3\subset\PP^4$ is a del Pezzo threefold of degree $3$.

\item
\label{DP:projective-d=4}
For $d=4$ the map $\Phi_{|H|}$ is an isomorphism onto its image and the image of
$X_4\subset\PP^5$
is an intersection of two quadrics. Conversely, any nonsingular intersection of 
two 
quadrics
$X_4\subset\PP^5$ is a del Pezzo threefold of degree $4$.
\end{enumerate}
\end{prp}

\begin{rem}
In the case where $d=1$ the branch surface $B\subset Y$ does not pass 
through the vertex of the
cone $o\in Y$.
However the cover
$$
\Phi_{|{-}K_X|}\colon X\longrightarrow Y\subset\PP^6
$$
is branched also over $o$. The singular point $o\in Y$ is local analytically isomorphic to
the quotient $\CC^3/\mumu_2(1,1,1)$.
Thus the cover $X\to Y$ coincides with the universal cover in a small
punctured neighborhood of the point~$o$.
Therefore, $\Phi_{|{-}K_X|}^{-1}(o)$ is a nonsingular point of the variety~$X$.
The map $\Phi_{|H|}\colon X\dashrightarrow \PP^2$
is decomposed as follows
$$
\Phi_{|H|}\colon X\xarr{\Phi_{|{-}K_X|}} Y \dashrightarrow \PP^2,
$$
where $Y\dashrightarrow \PP^2$ is the projection from the vertex of the cone 
$Y$.

Note also that a cone over the Veronese surface is a toric variety.
It is isomorphic to the weighted projective space $\PP(1^3,2)$
\cite{Dolgachev-1982}.
\end{rem}

\begin{proof}
Proofs of the
most of the assertions are
completely similar to the proofs of the corresponding
facts for del Pezzo surfaces and are left to the reader.
Minor changes are needed only
in the case~\ref{DP:projective-d=1}.

Let $\dd(X)=1$. Then $\dim|H|=2$ and $\dim|{-}K_X|=6$ according to
\eqref{eq:theorem:3dim-RRi=2}.
Consider the morphism
$$
\Phi=\Phi_{|{-}K_X|}\colon X\longrightarrow\PP^6.
$$
Denote by $Y\subset\PP^6$ its image.
Let $H\in|H|$ be a general member (nonsingular del Pezzo surface of degree
$1$).
Since
$$
\OOO_H(-K_X)=\OOO_H(2H)=\OOO_H(-2K_H)
$$
and
$$
\quad H^1(X,\,\OOO_X(-K_X-H))=H^1(X,\,\OOO_X(H))=0,
$$
from the exact sequence 
$$
0\longrightarrow\OOO_X(-K_X-H) \longrightarrow\OOO_X(-K_X) \longrightarrow\OOO_H(-K_X) \longrightarrow 0
$$
we obtain that the map
$$
H^0(X,\,\OOO_X(-K_X)) \longrightarrow H^0(H,\,\OOO_H(-2K_H))
$$
is surjective. Hence, the morphism-restriction $\Phi_H\colon H\to \Phi(H)$
is given by the complete linear system $|{-}2K_H|$ and so it is finite of degree~2
(see Proposition~\ref{del-pezzo-projective}\ref{del-pezzo-projective-d=1}).
In particular, the morphism $\Phi\colon X\to Y$ is not birational and $\deg 
\Phi=2$.

Moreover, all the surfaces $\Phi(H)$ pass through one point,~which is 
the image of the unique base point of the linear system $|H|$ and a general 
such a
surface
is a quadratic cone with vertex $o$
(see~\ref{del-pezzo-projective}\ref{del-pezzo-projective-d=1}). Hence, $Y$
is also a cone over some surface $S\subset\PP^5$.
Further,
$$
(\deg Y)\cdot (\deg \Phi)=(-K_X)^3=8.
$$
Therefore, $\deg Y=4$ and $\deg S=4$. Since $S$ is a general hyperplane section 
of
the variety $Y$ and $\Phi(H)$ is a quadratic cone, the intersection $S\cap
\Phi(H)=\PP^5\cap\Phi(H)$
is a plane conic. Hence, $S=S_4\subset\PP^5$ is the Veronese surface
(see Proposition~\ref{varieties-minimal-degree}),
i.\,e.\
the image of $\PP^2$ under the map given by the complete linear system of conics. For the branch divisor
$B\subset Y$ by the Hurwitz formula we have $B\sim 3S$
(cf.~\eqref{eq:Hurwitz1} and~\eqref{equation-Hurwitz-2}).
Then, according to Proposition~\ref{cor:VMD:R}, the surface~$B$ is cut out on $Y$
by a cubic.
\end{proof}

\subsection{The graded algebra} 

Further we consider projective embeddings of del Pezzo threefolds in detail.
For this it is very convenient to consider divisorial algebras.

Let $V$ be a projective algebraic variety and let $D$ a divisor on
$V$.
Consider the graded algebra
$$
\R(V,D):= \bigoplus_{d\ge 0}H^0(V,\,\OOO_V(dD)).
$$
Recall that
$$
H^0(V,\,\OOO_V(dD))=\{\varphi \in\CC(V)\mid \divi(\varphi)+D \ge 0\}\cup \{0\}
$$
is a subspace in $\CC(V)$. Thus it is convenient to regard $\R(V,D)$
as a subalgebra of the algebra $\CC(V)[t]$,
where every space $H^0(V,\,\OOO_V(dD))$ is embedded to the component $\CC(V)t^d$.

For any graded algebra
$$
R=\bigoplus_{d\ge 0} R_d
$$
by $R^{[n]}$ we denote its \textit{Veronese} (or \textit{truncated}) 
\textit{subalgebra}
$$
R^{[n]}=\bigoplus_{k\ge 0} R_{nk}.
$$
Thus $\R(V,D^{[n]})=\R(V,nD)$.

\begin{prp}
\label{proposition:R:very-ample}
Let $A$ be an ample divisor on $V$.
Assume that algebra $\R(V,\,A)$ is generated by its component of degree~$1$.
Then the divisor $A$ is very ample.
\end{prp}

\begin{proof}
It is clear that the sheaf $\OOO_V(A)$ is generated by its global sections.
So, it is sufficient to prove that the global sections of $H^0(V,\,\OOO_V(A))$ separate
points and tangent vectors. Let $P_1,\, P_2\in V$ be distinct points. Let 
$W\subset H^0(V,\,\OOO_V(A))$ the subspace of codimension~$1$ consisting of
all sections vanishing at the point~$P_1$. Take a basis 
$s_1,\dots,s_m\in W$
and complete it to a basis
$$
s_0, s_1,\dots,s_m\in H^0(V,\,\OOO_V(A)).
$$
Assume that all elements of $W$ vanish also at $P_2$. For
some $d\gg 0$ the divisor $dA$ is very ample, hence the elements of
$H^0(V,\,\OOO_V(dA))$ separate points, i.\,e.\ there exists an element $f\in
H^0(V,\,\OOO_V(dA))$ such that $f(P_1)=0$ and $f(P_2)\neq 0$. By our condition $f$
is represented in the form of a linear combination of sections
$$
s_0^{d_0}s_1^{d_1}\cdots s_m^{d_m},\qquad \sum d_i=d.
$$
Since $f(P_1)=0$, $s_0(P_1)\neq 0$, and $s_i(P_1)=0$ for $i=1,\dots m$, the term $s_0^d$ is missing
in this linear combinations.
Since $f(P_2)\neq 0$, we have $s_i(P_2)\neq 0$ for some $1\le i\le m$.
This means that $s_i$ separates points $P_1$ and $P_2$.
Similarly one can show that sections of $H^0(V,\,\OOO_V(A))$
separate tangent vectors.
\end{proof}

\begin{prp}
Let $A$ be an ample divisor on $V$.
Then the algebra $\R(V,\,A)$ is finitely generated and $V\simeq \operatorname{Proj}
\R(V,\,A)$.
\end{prp}

\begin{proof}
First, assume first that the divisor $A$ is very ample.
Let 
$V\subset\PP^N$ be the corresponding embedding and let $\mathcal{J}_V\subset
\OOO_{\PP^N}$ be the 
ideal sheaf. From the exact sequence
$$
0\longrightarrow\mathcal{J}_V(d) \longrightarrow\OOO_{\PP^N}(d)
\longrightarrow\OOO_V(dA) \longrightarrow0
$$
and Serre Vanishing Theorem
\cite[Ch.~III, Theorem~5.2]{Hartshorn-1977-ag}
we obtain the surjectivity of the maps
$$
H^0(\OOO_{\PP^N}(d))=\Sym^dH^0(\OOO_{\PP^N}(1))=\Sym^dH^0(\OOO_V(A)) \longrightarrow
H^0(\OOO_V(dA))
$$
where $d\ge d_1$ for some $d_1$.
This means that $\R(V,\,A)$ is generated by its components of degree $\le d_1$.

For arbitrary ample divisor $A$ take $d_0$ so that $d_0A$
is very ample and consider the Veronese subalgebra $\R(V,\,d_0A)$.
It is easy to show that algebra $\R(V,\,A)$ is finitely generated if and only
if
so is $\R(V,\, d_0A)$ (see e.~g.~\cite[Ch.~III, Sect.~1,
Subsect.~3]{Bourbaki:comm-alg-e}). This proves the finite generation of $\R(V,\,A)$.
Therefore, $\operatorname{Proj} \R(V,\,A)$ is a scheme of finite type
and the natural map $X\to \operatorname{Proj} \R(V,\,A)$
is an isomorphism.
\end{proof}

Consider the simplest one-dimensional example that will be very essential below:

\begin{prp}
\label{e-curves:}
Let $C$ be an elliptic curve and let $A$ be a divisor on
$C$ of degree $d>0$.

\begin{enumerate}
\item
\label{e-curves:deg=1}
If $d=1$, then the graded algebra $\R(C,A)$ has the form
\begin{equation}
\label{eq:elliptic}
\R(C,A)=\CC[x, y, z]/(f),
\end{equation}
where $\deg x=1$, $\deg y=2$, $\deg z=3$, and $f\in\CC[x, y, z]$ is 
a quasihomogeneous polynomial of weighted degree
$6$.

\item
\label{e-curves:deg=2}
If $d=2$, then the graded algebra $\R(C,A)$ is generated by its
components of degree~$1$ and~$2$.
Moreover, in this case
$$
\R(C,A)=\CC[x_1, x_2, y_1]/(h),
$$
where $\deg x_i=1$, $\deg y_1=2$, and $h\in\CC[x_1, x_2, y_1]$ is a
quasihomogeneous
polynomial of weighted degree
$4$.

\item
\label{e-curves:deg=3}
If $d\ge 3$, then the graded algebra $\R(C,A)$ is generated by its
component of degree~$1$.

\item
\label{e-curves:deg=4}
If $d\ge 4$, then the ideal of relations in $\R(C,A)$ between elements of degree~$1$
is generated by relations of degree~$2$.
\end{enumerate}
\end{prp}
First we prove the lemma.

\begin{lem}
\label{lemma:ell-cu}
In the notation of Proposition~\ref{e-curves:}, there exists a point $P\in C$ such 
that $A\sim dP$.
\end{lem}

\begin{proof}
Consider the map
$$
\phi\colon C \longrightarrow\Pic^0(C), \qquad P \longmapsto [dP-A].
$$
The map is not constant.
Hence, it is surjective and so $\phi(P)=0$ for some point $P\in C$.
This means that $A\sim dP$.
\end{proof}

\begin{proof}[Proof of Proposition~\ref{e-curves:}]
\Ref{e-curves:deg=1} 
By the Riemann--Roch Theorem we have $\hr^0(C,\OOO_C(nA))=n$ for $n>0$.
Let $x$ be any non-zero element of $H^0(C,\OOO_C(A))$.
Take an element $y\in H^0(C,\OOO_C(2A))$ that is not proportional to $x^2$,
and an element $z\in H^0(C,\OOO_C(3A))$ that is not a linear combination of 
$x^3$
and $xy$. Let $P\in C$ be the point given by the section $x$.
For $m>1$ we have the exact sequence
$$
0\longrightarrow H^0(C,\OOO_C((m-1)A))
\xarr{\cdot x} H^0(C,\OOO_C(mA))\longrightarrow H^0(P,\,\OOO_P)
\longrightarrow 0.
$$
This implies, in particular, that $y$ and $z$ do not vanish in $P$.
Also, every space $H^0(C,\OOO_C(mA))$ is generated by
$xH^0(C,\OOO_C((m-1)A))$ and one extra element that does not
vanish at $P$.
This element can be chosen in the form $y^kz^l$, where $2k+3l=m$.
From this we obtain that $x,y,z$ generate the algebra $\R(C,A)$. Also it
follows that the elements
$x^4,\, x^2y,\, y^2,\, xz$ form a basis of the space $H^0(C,\OOO_C(4A))$,
the elements
$$
x^5,\ x^3y,\ xy^2,\ x^2z,\ yz
$$
form a basis of $H^0(C,\OOO_C(5A))$,
and between the elements
$$
x^6,\ x^4y,\ x^2y^2,\ x^3z,\ xyz,\ z^2,\ y^3
$$
in the algebra $\R(C,A)$ there exists one relation $f(x,y,z)=0$ of degree~$6$.
We may assume that this relation has the form
\begin{equation}
\label{eq:new:real:ell}
f(x,y,z)=z^2+y^3+xg(x,y,z)=0.
\end{equation}
The kernel of the homomorphism
$$
\CC[x,y,z] \longrightarrow\R(C,A)
$$
is a homogeneous prime ideal $I$ of height~$1$. Since $\CC[x,y,z]$ is 
a unique factorization domain,
such an ideal must be principal.
Thus $I=(f)$.

\ref{e-curves:deg=2}
Consider now the case $d=2$. By Lemma~\ref{lemma:ell-cu} we have $A\sim 2P$
for some point $P\in C$.
Therefore, there exists an embedding $\R(C,A)\subset\R(C,P)$
so that $\R(C,A)$
coincides with Veronese subalgebra:
$$
\R(C,A)=\R(C,P)^{[2]}=(\CC[x, y, z]/(f))^{[2]}.
$$
Put $x_1=x^2$, $x_2=y$, $y_1=xz$.
Using our description of the algebra $\R(C,P)$ given in 
\ref{e-curves:deg=1},
it can be shown that the elements $x_1$, $x_2$, $y_1$ generate $\R(C,A)$ and between
them there exists
a relation of degree $4$.

\ref{e-curves:deg=3}
As above by Lemma~\ref{lemma:ell-cu} we have $A\sim dP$ for some point
$P\in C$.
Therefore, there exists an embedding $\R(C,A)\subset\R(C,P)$ so that
$\R(C,A)$
coincides with the Veronese subalgebra:
$$
\R(C,A)=\R(C,P)^{[d]}=(\CC[x, y, z]/(f))^{[d]}
$$
(see~\eqref{eq:elliptic}). Taking the relation~\eqref{eq:new:real:ell} into account 
we obtain that the
homogeneous component $\R(C,A)_n=(\CC[x, y, z]/(f))_{dn}$ is generated by the
monomials
of the form
\begin{equation}
\label{equation-xy-relations}
x^ky^lz^m,\qquad k+2l+3m=dn,\quad m\in\{0,\, 1\}.
\end{equation}
We show that these elements can be expressed in terms of
elements of the component $\R(C,A)_1=(\CC[x, y, z]/(f))_d$.
Assume, the converse, i.e. some monomial $M\in\R(C,A)_n$
of the form~\eqref{equation-xy-relations}
cannot be expressed in terms of
monomials from the component $\R(C,A)_1=(\CC[x, y, z]/(f))_d$.
Take such $n$ to be minimal under this condition.
Consider the case of even $d$ (the odd case is similar and left to the reader).
If $k\ge d$, then $M=x^dM'$, where $M'$ again has
the form~\eqref{equation-xy-relations} with $d'=d-1$. This contradicts our 
assumption. Therefore, $k\le d-1$.
If $l\ge d/2$, then similarly $M=y^{d/2}M'$, $M'\in\R(C,A)_{n-1}$. Hence,
$l\le d/2-1$ and so
$k+2l\le 2d-3$. Taking~\eqref{equation-xy-relations} into account we obtain
$$
k+2l=2d-3,\quad n=2,\quad k=d-1,\quad l=d/2-1\ge 1.
$$
In this case we can write
$$
M=x^{d-2}y M',\quad M'\in\R(C,A)_1.
$$
Again we have a contradiction.
This proves the assertion.

The proof of the assertion~\ref{e-curves:deg=4} is left to the reader.
\end{proof}

\begin{cor}
\begin{enumerate}
\item
Any elliptic curve can be embedded to $\PP(1,2,3)$
and given there by an equation of degree $6$.
\item
Any elliptic curve can be embedded to $\PP(1^2,2)$
and given there by an equation of degree $4$.
\item
Any elliptic curve can be embedded to $\PP^2$
and given there by an equation of degree $3$.
\item
Any elliptic curve can be embedded to $\PP^3$
and given there by two equations of degree~$2$.
\end{enumerate}
\end{cor}

\begin{teo}[{(Hyperplane Section Principle, cf. 
\cite[Theorem~3.6]{Mori-1975})}]
\label{th-hyp-sect}
Let $X$ be an irreducible projective
variety of dimension $\ge 2$, let $\LLL$ be an invertible sheaf on~$X$, and let $Y\in
|\LLL|$ be an effective divisor.
Put $\LLL_Y:= \LLL\otimes \OOO_Y$.
Assume that
\begin{equation}
\label{eq:H1}
H^1(X,\,\LLL^{\otimes m})=0,\qquad \forall m\ge 0.
\end{equation}
Let $x_0\in H^0(X,\,\LLL)$ be a section defining the divisor $Y$.
Then the following assertions hold.
\begin{enumerate}
\item
\label{th-hyp-sect-0}
There is the exact sequence
$$
0 \longrightarrow\R(X,\,\LLL)\xarr{\cdot x_0} \R(X,\,\LLL)\xarr{\varphi} \R(Y,\,\LLL_Y)
\longrightarrow 0,
$$
where $\varphi$ is the restriction map.
\item
\label{th-hyp-sect-1}
Let $\overline{x_1},\dots,\overline{x_k}$ be 
homogeneous elements generating $\R(Y,\,\LLL_Y)$
and let $x_1,\dots,x_k \in\R(X,\,\LLL)$ be homogeneous elements such
that $\overline{x_i}=\varphi(x_i)$.
Then $\R(X,\, \LLL)$ is generated by elements $x_0, x_1,\dots,x_k$.
\item
\label{th-hyp-sect-2}
If the graded ring $\R(Y,\,\LLL_Y)$ is generated by
its components of degree $\le r$, then 
the graded ring $\R(X,\, \LLL)$ is also generated by its components of degree $\le r$.

\item
\label{th-hyp-sect-3}
In the assumptions of~\ref{th-hyp-sect-1} let
$\overline{f_1},\dots,\overline{f_n}$ be homogeneous elements generating
the ideal of relations between $\overline{x_1},\dots,\overline{x_k}$. Then 
there exist
homogeneous relations $f_1,\dots,f_n$ between $x_0, x_1,\dots,x_n$
in~$\R(X,\LLL)$
such that $\overline{f_i}=\varphi(f_i)$ and $f_1,\dots,f_n$ generate 
all relations
between $x_0, x_1,\dots,x_n$.
\end{enumerate}
\end{teo}

\begin{proof}
The assertions~\ref{th-hyp-sect-1} and~\ref{th-hyp-sect-2} are proved by 
induction on $m$.
According to~\eqref{eq:H1}, there is the following
exact sequence of
sheaves
\begin{equation}
\label{eq:seq:LLL}
0 \longrightarrow H^0(X,\LLL^{\otimes (m-1)})
\xarr{\cdot x_0} H^0(X, \LLL^{\otimes m})
\longrightarrow H^0(Y, \LLL_Y^{\otimes m})
\longrightarrow 0.
\end{equation}
This implies~\ref{th-hyp-sect-0}.
Let $s \in H^0(X,\LLL^{\otimes m})$ be an arbitrary element and
$\bar{s} \in H^0(Y,\LLL_Y^{\otimes m})$
be its image. Then by our condition 
$\bar{s}=p(\overline{x_1},\dots,\overline{x_k})$, where
$p$ is some quasihomogeneous polynomial of degree $m$. The
section $s- p(x_1,\dots,x_k)$ vanishes on $Y$ and so it has the form
$x_0 v$ for some $v\in H^0(X,\LLL^{\otimes (m-1)})$.
By the inductive hypothesis $v=q(x_0,\dots,x_k)$, where $q$ is quasihomogeneous
polynomial of degree $m-1$. Therefore,
$$
s=p(x_1,\dots,x_k)+x_0q(x_0,\dots,x_k),
$$
i.e. $x_0, x_1,\dots,x_k$ generate $\oplus_{i\le m}\R(X,\LLL)_m$.
This proves~\ref{th-hyp-sect-1} and~\ref{th-hyp-sect-2}. The assertion
\ref{th-hyp-sect-3} is proved similar to~\ref{th-hyp-sect-1} 
by using the sequence~\eqref{eq:seq:LLL}.
\end{proof}

Now from Proposition~\ref{e-curves:}, Theorem~\ref{theorem-smooth-divisor}, and
Theorem~\ref{th-hyp-sect} by induction we obtain the following two assertions.

\begin{prp}
\label{projective-models:del-Pezzo}
Let $Y$ be a del Pezzo surface of degree $d$.
\begin{enumerate}
\item
\label{projective-models-degree=1-dim=2}
If $d=1$, then
$$
\R(Y,-K_Y)=\CC[x_0, x_1, y, z]/(f),
$$
where $\deg x_i=1$, $\deg y=2$, $\deg z=3$, and $f$ is a quasihomogeneous 
polynomial of weighted degree
$6$. In particular, $Y$ is isomorphic to a hypersurface
of degree $6$ in $\PP(1^2,2,3)$.

\item
If $d=2$, then
$$
\R(Y,-K_Y)=\CC[x_0, x_1, x_2, y]/(h),
$$
where $\deg x_i=1$, $\deg y=2$, and $h$ is a quasihomogeneous polynomial of 
weighted degree
$4$. In particular, $Y$ is isomorphic to a hypersurface
of degree $4$ in $\PP(1^3,2)$.

\item
\label{projective-models-degree-ge3-dim=2}
If $d\ge 3$, then the graded algebra $\R(Y,\,-K_Y)$ is generated by its
component of degree~$1$.
In particular, the divisor $-K_Y$ is very ample and defines an embedding 
$Y=Y_d\subset\PP^d$.

\item
If $d\ge 4$, then the anticanonical image $Y=Y_d\subset\PP^d$
is an intersection of quadrics.
\end{enumerate}
\end{prp}

\begin{prp}
\label{projective-models-3-folds}
Let $X$ be a del Pezzo threefold of degree~$d$ and let $-K_X=2H$.
\begin{enumerate}
\item
\label{projective-models-degree=1-dim=3}
If $d=1$, then
$$
\R(X,H)=\CC[x_0, x_1, x_2, y, z]/(f),
$$
where $\deg x_i=1$, $\deg y=2$, $\deg z=3$, and $f$ is a quasihomogeneous 
polynomial of weighted degree
$6$.
In particular, $X$ is isomorphic to a hypersurface
of degree $6$ in $\PP(1^3,2,3)$.

\item
\label{projective-models-degree=2}
If $d=2$, then
$$
\R(X,H)=\CC[x_0, x_1, x_2, x_3, y]/(h),
$$
where $\deg x_i=1$, $\deg y=2$, and $h$ is a quasihomogeneous polynomial of 
weighted 
degree
$4$. In particular, $X$ is isomorphic to a hypersurface
of degree $4$ in $\PP(1^4,2)$.

\item
\label{projective-models-degree-ge3-dim=3}
If $d\ge 3$, then the graded algebra $\R(X,H)$ is generated by its
component of degree~$1$.
In particular, the divisor~$H$ is very ample and defines an embedding $X=X_d\subset
\PP^{d+1}$.

\item
If $d\ge 4$, then the half-anticanonical image $X=X_d\subset\PP^{d+1}$ is a
an intersection of quadrics.
\end{enumerate}
\end{prp}

\begin{rem}
In the case~\ref{projective-models-degree=1-dim=3} the projection
$X \to \PP(1^3,2)$ is a regular morphism (see~\eqref{eq:new:real:ell}).
It is the double cover described in Proposition 
\ref{DP:projective}\ref{DP:projective-d=1}.
Similarly, in the case~\ref{projective-models-degree=2} the projection
$X \to \PP(1^4)=\PP^3$ is a regular morphism of degree $2$ described in Proposition~\ref{DP:projective}
\ref{DP:projective-d=2}.
\end{rem}

The following assertion is a consequence of the Lefschetz hyperplane theorem
(see~\cite[Theorem~3.2.4(i)]{Dolgachev-1982}).

\begin{cor}
\label{corollary-i=2-rho}
Let $X$ be a del Pezzo threefold of degree $\dd(X)\le 4$.
Then $\Pic(X)\simeq \ZZ$.
\end{cor}

Further we show that for a del Pezzo threefold
the condition $\Pic(X)\simeq\ZZ$ is equivalent to our condition $\dd(X)\le 5$.

\subsection{Del Pezzo threefolds with higher Picard rank} 

Now we study del Pezzo threefolds with $\uprho(X)>1$. In this case we
can apply the techniques of extremal rays (see Appendix~\ref{section:mori}).
First we provide two general facts.

\begin{lem}
\label{lemma-intersection-exceptional}
Let $X$ be a Fano threefold, let $\rR_1,\dots,\rR_n$ be all the 
extremal rays on $X$ and let $f_1,\dots,f_n$ be the corresponding
extremal
contractions.
Assume that all
the contractions $f_i$
on $X$ are birational and let $E_1,\dots,E_n$ be their their exceptional
divisors
\textup(it is possible that $E_i=E_j$ for $i\neq j$\textup).
Then $E_i\cap E_j\neq \varnothing$ for some $i$ and $j$ such that $E_i\neq
E_j$.
In particular, the image $f_i(E_i)$ is not a point
for at least
one~$i$.
\end{lem}

\begin{proof}
For any $i$, let $C_i$ be any curve generating the ray $\rR_i$. Then
$E_j\cdot C_j<0$. Furthermore, let
$Z$ is a curve that is the intersection in $X$ of
two general hyperplane sections.
Then $E_j\cdot Z>0$ for any $j$ and by the Cone Theorem
$$
Z\approxident \sum \alpha_i C_i,\quad \alpha_i\ge 0.
$$
Assume that $E_i\cap E_j=\varnothing$ whenever $E_i\neq E_j$.
Then
$$
E_j\cdot C_i\quad
\begin{cases}
\quad =0, & \text{if $E_i\neq E_j$,}
\\
\quad <0, & \text{if $E_i= E_j$.}
\end{cases}
$$
This implies that
$E_j\cdot Z\le 0$, a contradiction.

The last assertion follows from the fact that fibers of different extremal 
contractions
cannot intersect each other by a set of positive dimension.
\end{proof}

\begin{lem}
\label{lemma-del-Pezzo}
Let $X$ be a Fano threefold. Assume that on~$X$
there exists an extremal contraction of type~\type{D}.
Then the second extremal contraction is of type~\type{C} or~\type{B_1}.
\end{lem}

\begin{proof}
Aa above it follows from the fact that fibers of extremal contractions cannot 
intersect each other by a set of positive dimension.
\end{proof}

From the classification of extremal rays on three-dimensional
varieties~\ref{class:ext-rays} we obtain the following.

\begin{cor}
\label{cor:DP:contr}
Let $X$ be a del Pezzo threefold with $\uprho(X)>1$
and let $f\colon X\to Y$ be a extremal Mori contraction.
Then $f$ can be only one of the following types:
\begin{enumerate}
\item
[\type{B_2}:]
the morphism $f$ is the blowup of a point,
$Y$ is a del Pezzo threefold or $Y\simeq \PP^3$
and $(-K_Y)^3=(-K_X)^3+8$;
\item
[\type{C_2}:]
the morphism $f$ is a $\PP^1$-bundle and
$Y$ is a del Pezzo surface of degree $10-\uprho(Y)=11-\uprho(X)$;
\item
[\type{D_2}:]
the morphism $f$ is a quadric bundle and $Y \simeq \PP^1$.
\end{enumerate}
\end{cor}

\begin{proof}
Since the length of the corresponding extremal ray in our situation is 
divisible by~$2$, from Theorem~\ref{class:ext-rays}
we obtain that the contraction can be only types \type{B_2}, \type{C_2} or
\type{D_2}.
In the case \type{B_2} the canonical divisor $K_Y=f_*K_X$ is divisible by~$2$
and
$Y$ is a Fano variety of even index by Proposition~\ref{proposition:bir-pt}.
In the case \type{C_2}
the restriction morphism $f_H\colon H\to Y$ to a general (nonsingular) 
divisor 
$H\in|{-}K_Y|$ is birational.
Therefore, $Y$ is a del Pezzo surface again by 
Proposition~\ref{proposition:bir-pt}.
\end{proof}

\begin{prp}
\label{classification-rho-ge2-iota=2}
Let $X$ be a del Pezzo threefold with $\uprho(X)>1$.
Then one of the following assertions holds:
\begin{enumerate}
\item
\label{classification-rho-ge2-iota=2-d=7}
$\dd(X)=7$ and $X=X_7\subset\PP^8$ is 
the blowup of a point on $\PP^3$;
\item
\label{classification-rho-ge2-iota=2-d=6b}
$\dd(X)=6$ and
$X\simeq \PP^1\times \PP^1\times \PP^1$;
\item
\label{classification-rho-ge2-iota=2-d=6a}
$\dd(X)=6$ and
$X$ is a divisor of bidegree $(1,1)$ in $\PP^2\times \PP^2$.
\end{enumerate}
\end{prp}

\begin{proof}
Let, as usual, $H:= -\frac12K_X$ and $d:= \dd(X)$.

First consider the case $\uprho(X)=2$.
Let $f\colon X\to Y$ and $f'\colon X\to Y'$ be different extremal contractions.
According to Corollary~\ref{cor:DP:contr}, up to permutation for types of
contractions $f$ and $f'$
there are only the following possibilities:
\begin{quote}
\type{C_2}--\type{D_2},\quad \type{C_2}--\type{C_2},\quad
\type{C_2}--\type{B_2},\quad
\type{D_2}--\type{D_2},\quad \type{B_2}--\type{B_2},\quad 
\type{B_2}--\type{D_2}.
\end{quote}
By Lemmas~\ref{lemma-del-Pezzo} and
\ref{lemma-intersection-exceptional} the cases \type{D_2}--\type{D_2},
\type{B_2}--\type{B_2}, \type{B_2}--\type{D_2}
are impossible. Thus
we may assume that $f$ is of type \type{C_2}, i.e.
$\dim(Y)=2$ and $f$ is a $\PP^1$-bundle.
Since $\uprho(Y)=1$, we have $Y\simeq \PP^2$.

Let $l\subset\PP^2$ be a line and let $M:= f^{-1}(l)$.
From the exact sequence~\eqref{equation-exact-sequence} we obtain that
$$
\Pic(X)=\ZZ\cdot H\oplus \ZZ\cdot M.
$$
It is clear that $M^3=0$ and
$$
K_X\cdot M^2=K_X\cdot f^{-1}(\pt)=-2.
$$
Hence, $H\cdot M^2=1$. The surface $M$ is rational and geometrically ruled.
Thus $K_M^2=8$.
By the adjunction formula
$$
8=K_M^2=(K_X+M)^2\cdot M=(M-2H)^2\cdot M=-4+4 H^2\cdot M.
$$
Therefore, $H^2\cdot M=3$.

Assume that $\dim(Y')=1$. Then $Y'\simeq \PP^1$ and $f'\colon X\to \PP^1$ is 
fibered into two-dimensional
quadrics.
Let $F$ be a general fiber of $f'$. Again it follows from the exact 
sequence~\eqref{equation-exact-sequence} that
$$
\Pic(X)=\ZZ\cdot H\oplus \ZZ\cdot F=\ZZ\cdot H\oplus \ZZ\cdot M.
$$
Hence, we can write $F\sim aH\pm M$, where $a\in\ZZ_{>0}$.
Since $F$ is not ample, we have $F\sim aH- M$.
It is clear that $F^2\approxident 0$. Thus
$$
0=F^2\cdot M=(aH-M)^2\cdot M=3a^2-2a.
$$
The last equation has no integral positive solutions, a contradiction.

Assume that $\dim(Y')=2$. Then $Y'\simeq \PP^2$ and $f'\colon X\to \PP^2$ is 
also a $\PP^1$-bundle.
Let $l'\subset Y'=\PP^2$ be a line and let $M':= f'^{-1}(l')$.
As above we have
$$
\Pic(X)=\ZZ\cdot H\oplus \ZZ\cdot M',\qquad M'\sim aH-M,
$$
$$
3=M'\cdot H^2=(aH-M)\cdot H^2=ad-3.
$$
Since we already established that $d\ge 5$ (see Corollary~\ref{corollary-i=2-rho}),
the only possibility is
$$
d=6,\qquad a=1,\qquad H\sim M+M'.
$$
Furthermore, the morphism
$$
\hat f=f\times f'\colon X\longrightarrow Y\times Y'=\PP^2\times \PP^2
$$
is finite onto its image $V:= \hat f(X)\subset\PP^2\times \PP^2$.
From our relations we obtain also that $M^2\cdot M'=M\cdot M'^2=1$.
This shows that
the morphism $\hat f$ is also birational onto its image and $V$ is a divisor of 
bidegree
$(1,1)$. Such a variety $V$ must be normal (see Exercise~\ref{zad:DP:1-1:6} at 
the end of this section).
Therefore, $\hat f\colon X\to V$ is an isomorphism. We obtain the case 
\ref{classification-rho-ge2-iota=2-d=6a}.

Assume that $\dim(Y')=3$. Then $f'\colon X\to Y'$ is the blowup of a point and 
$Y'$ is a
Fano variety
of index~$2$ or $4$ with $\uprho(Y')=1$.
Let $E\subset X$ be the exceptional divisor.
If $\iota(Y')=4$, then $Y'\simeq \PP^3$
and we obtain the case~\ref{classification-rho-ge2-iota=2-d=7}. Let 
$\iota(Y')=2$.
In this case the divisor $f'_* H$ generates the group $\Pic(Y')$.
As above we can write $E\sim aH-M$,
$$
1=H^2\cdot E=H^2\cdot(aH-M)=ad-3.
$$
Therefore, $d\le 4$. This contradicts Corollary~\ref{corollary-i=2-rho}.

Further we consider the case $\uprho(X)\ge 3$. Then
$X$ has no extremal contractions of type \type{D_2}. By Lemma
\ref{lemma-intersection-exceptional} there exists
a contraction $f\colon X\to Y$ of type \type{C_2}, i.e. $X$ has a structure of
a $\PP^1$-bundle over
the surface $Y$. If there exists also a birational contraction $f'\colon X\to Y'$,
then its exceptional divisor is isomorphic to $\PP^2$ and dominates $Y$.
This implies that $Y\simeq \PP^2$.
This contradicts our assumption $\uprho(X)\ge 3$.
Hence, all extremal contractions on $X$ are $\PP^1$-bundles.

Assume that the surface $Y$ contains a $(-1)$-curve $C$.
Let $F:= f^{-1}(C)$ and let $\Sigma$ be the minimal section of
ruled surface $F\simeq \FF_n$.
We have
$$
(K_F+\Sigma)\cdot \Sigma=-2=K_X\cdot \Sigma+F\cdot \Sigma+(\Sigma)_F^2.
$$
By the projection formula $F\cdot \Sigma=f^*C\cdot \Sigma=C^2=-1$. Therefore,
$$
(\Sigma)_F^2=-1+2H\cdot \Sigma\ge 1,
$$
a contradiction.
Therefore, we may assume that $Y$ is a
surface that does not contain any $(-1)$-curves. Since $Y$ is a del Pezzo surface, we 
have
$Y\simeq \PP^1\times \PP^1$ and so $\uprho(X)=3$.

Let 
$$
\phi_1\colon X\xarr{f} Y=\PP^1\times \PP^1\longrightarrow \PP^1=:T_1
$$
be the composition of
the morphism $f$ with projection to the first factor.
It is clear that this is a smooth morphism.
By the adjunction formula any fiber $F_1$ of this morphism is a del Pezzo 
surface
with $\iota(F_1)=2$.
Thus, $F_1\simeq \PP^1\times \PP^1$, i.e. $\phi_1$ is a quadric bundle.
Furthermore, $\uprho(X/T_1)=2$.
Therefore, the relative Mori cone $\NE(X/T_1)$
has two extremal rays and so in addition to $f$ there exists another
extremal
contraction $f'\colon X\to Y'=\PP^1\times \PP^1$ over $T_1$:
$$
\xymatrix@R=1em{
&X\ar[dl]_f\ar[dr]^{f'}&
\\
\PP^1\times \PP^1\hspace{0em}\ar[dr]&&\hspace{0em}\PP^1\times \PP^1\ar[dl]
\\
&\PP^1&
}
$$
Hence, there exists a morphism
$$
\phi\colon X\longrightarrow Y\times_{\PP^1} Y'=\PP^1\times \PP^1\times \PP^1.
$$
Let $\phi_i\colon X\to \PP^1$, $i=1,\, 2,\, 3$ be compositions of the morphism
$\phi$ with projections to factors.
For every pair of subscript indices $1\le i\neq j\le 3$,
the morphism
$$
\phi_i\times \phi_j\colon X\longrightarrow\PP^1\times \PP^1
$$
is an extremal Mori contraction (because $-K_X$ is ample and
$\uprho(X/\PP^1\times \PP^1)=1$).
Let $F_1$, $F_2$, and $F_3$ be the fibers of $\phi_1$, $\phi_2$, and $\phi_3$, 
respectively
As above, it is not difficult to see that $F_i\simeq \PP^1\times \PP^1$
and $\OOO_{F_i}(H)\simeq \OOO_{\PP^1\times \PP^1}(1,1)$.
Since $\hr^0(X,\, \OOO_X(H))=\dd(X)+2\ge 7$, from the standard exact 
sequence
$$
0\longrightarrow \OOO_X(H-F_i) \longrightarrow \OOO_X(H) \longrightarrow \OOO_{F_i}(H) \longrightarrow 0
$$
we obtain that $|H-F_i|\neq \varnothing$. Since there are no birational 
contractions on the variety 
$X$, every effective
divisor on $X$ is nef (see Corollary~\ref{cor::Fano-Mori}).
Therefore, $H\cdot C\ge F_i\cdot C$
for any curve $C$.

By the Hurwitz formula
$$
-2H=K_X=\phi^* K_{\PP^1\times \PP^1\times \PP^1}+G=-2F_1-2F_2-2F_3+G,
$$
where $G$ is the branch divisor.
This implies that there exists an (integral) divisor $D$ such that
$$
2D\sim G,\qquad F_1+F_2+F_3\sim H+D.
$$
For a fiber $\ell$ the contractions $\phi_1\times \phi_2$ we have $\ell 
\approxident 
F_1\cdot F_2$, \,
$F_1\cdot \ell=F_2\cdot\ell=0$, \, $F_3\cdot \ell \le H\cdot \ell=1$, and
$$
1\ge F_1\cdot F_2\cdot F_3= F_3\cdot \ell= (F_1+F_2+F_3)\cdot \ell
=H\cdot \ell+D\cdot \ell=1+D\cdot \ell.
$$
Therefore, $D\cdot \ell=0$.
Similarly we obtain that $D$ trivially intersect the fibers of
all contractions $\phi_i\times \phi_j$, $i\neq j$.
It follows that $D\approxident 0$ and $G\approxident 0$.
But this means that $G=0$
(because the divisor $G$ is effective).
Thus the morphism $\phi$ is \'etale.
Since the variety $\PP^1\times \PP^1\times \PP^1$ is simply connected, it must 
be an isomorphism.
We obtain the case~\ref{classification-rho-ge2-iota=2-d=6b}.
Proposition~\ref{classification-rho-ge2-iota=2} is proved.
\end{proof}

\subsection{Quintic del Pezzo threefolds} 

Now we can complete the classification of del Pezzo threefolds.
It remains to consider the case $\uprho(X)=1$.
According to Proposition~\ref{projective-models-3-folds}, we can restrict ourselves to the
cases $\dd(X)\ge 5$. The main result is Theorem~\ref{th:d5},
but first we
provide some necessary facts on~families of lines on del Pezzo
threefolds.

\begin{prp}
\label{proposition-index=2-lines}
Let $X=X_d\subset\PP^{d+1}$ be a del Pezzo threefold
of degree $d=\dd(X)\ge4$
with $\uprho(X)=1$. The following assertions hold.
\begin{enumerate}
\item
\label{proposition-index=2-lines-existence}
There exists a line on $X$.

\item
\label{proposition-index=2-lines-hs}
Let $l\subset X$ be a line. There exists a nonsingular hyperplane section passing through $l$.

\item
\label{proposition-index=2-lines-nb}
$\NNN_{l/X}\simeq \OOO_{\PP^1}\oplus \OOO_{\PP^1}$ or $\OOO_{\PP^1}(-1)\oplus
\OOO_{\PP^1}(1)$.

\item
\label{proposition-index=2-lines-sh}
Let $\Lines(X)$ be the Hilbert scheme parametrizing lines on $X$.
Then $\Lines(X)$ is reduced, nonsingular, and each its component is 
two-dimensional.

\item
\label{proposition-index=2-lines-pt}
There are a finite \textup(non-zero\textup) number
of lines passing through any point $P\in X$.

\item
\label{proposition-index=2-lines-N}
The line $l\subset X$ corresponding to a sufficiently general point of a component of
$\Lines(X)$ has normal bundle of the form $\OOO_{\PP^1}\oplus \OOO_{\PP^1}$.

\item
\label{proposition-index=2-lines-dl}
Each line meets at least one other line.
\end{enumerate}
\end{prp}

\begin{proof}

\ref{proposition-index=2-lines-existence}
A line exists on a nonsingular hyperplane section $H\subset X$
because $K_H^2\le 8$ and $H\not \simeq \PP^1\times \PP^1$ (see
Remark~\ref{rem:Lef}).

\ref{proposition-index=2-lines-hs} follows from 
the parameter count: the hyperplane sections passing through $l$ form a subspace 
of codimension~$2$ in $(\PP^{d+1})^*$ and the hyperplane sections that are 
singular at some point of $l$ 
form a subvariety of codimension $\ge 3$.

\ref{proposition-index=2-lines-nb}
We have a decomposition $\NNN_{l/X}=\OOO_{\PP^1}(a)+\OOO_{\PP^1}(b)$. Further,
$$
0\longrightarrow\NNN_{l/H}\longrightarrow\NNN_{l/X}
\longrightarrow\NNN_{H/X}|_l\longrightarrow 0
$$
where $\NNN_{l/H}=\OOO_{\PP^1}(-1)$ and $\NNN_{H/X}|_l=\OOO_{\PP^1}(1)$.
Therefore, $a+b=0$. Since $\hr^0(\NNN_{l/X})\le \hr^0(\NNN_{H/X}|_l)=2$, we 
have 
$a,\, b\le 1$.

The assertion \ref{proposition-index=2-lines-sh}
follows from the deformation theory (see~\cite[Theorems~2.8 and 
2.15]{Kollar-1996-RC}).
Since $H^1(X,\NNN_{l/X})=0$, 
the scheme $\Lines(X)$ is nonsingular at the point $[l]$. Since
$$
\dim T_{[l], \Lines(X)}=\hr^0(l,\NNN_{l/X})=2,
$$
this scheme is two-dimensional.

\ref{proposition-index=2-lines-pt}
If there is a one-dimensional family of lines passing through a point $P\in X$,
then these lines cover a surface $F\subset X$, which is contained in the
intersection $X\cap\overline{T_{P,X}}$, where $\PP^3=\overline{T_{P,X}}\subset
\PP^{d+1}$
is the embedded projective tangent space. But then a hyperplane
section
$H\subset X$ passing through $\overline{T_{P,X}}$ is reducible.
This contradicts our condition $\Pic(X)=\ZZ\cdot H$.

Thus there are at most a finite number
of lines passing through each point $P\in X$.
Consider the universal family of lines and the diagram
\begin{equation}
\label{eq:dia:un:sem}
\vcenter{
\xymatrix{
&\Univ(X)\ar[dl]_p\ar[dr]^{q}&
\\
\Lines(X)&&X
} }
\end{equation}
Since $\dim \Univ(X)=3=\dim(X)$, the morphism $q$ is finite.
Hence, there is at least one line passing through each point $P\in X$.

\ref{proposition-index=2-lines-N}
Again consider the diagram~\eqref{eq:dia:un:sem}.
Since $\operatorname{Hom} (\OOO_{\PP^1},\,\OOO_{\PP^1}(1))=0$, 
the differential 
$$
\mathrm{d} q \colon \NNN_{l/\Univ(X)}=\OOO_{\PP^1}\oplus \OOO_{\PP^1}
\longrightarrow\NNN_{l/X}
$$
degenerates along $l$ if and only if $\NNN_{l/X}\simeq \OOO_{\PP^1}
(1)\oplus
\OOO_{\PP^1}(-1)$.
In other words, $\NNN_{l/X}\simeq \OOO_{\PP^1} (1)\oplus
\OOO_{\PP^1}(-1)$ if and only if the line $l\subset X$ is contained in 
the branch divisor of the morphism~$q$.

Finally, we prove~\ref{proposition-index=2-lines-dl}.
Let $l$ be an arbitrary line on $X$.
Take a curve $\Gamma\subset \Lines(X)$ that does not pass through the point 
$[l]\in
\Lines(X)$. Then $q(p^{-1}(\Gamma))$ is a surface that is covered
by a one-dimensional family of lines that does not contain $l$.
Since $\Pic(X)\simeq \ZZ$, the subvariety $q(p^{-1}(\Gamma))$ is an
ample divisor and so $l\cap q(p^{-1}(\Gamma))\neq \varnothing$.
Therefore, $l$ meets at least one line from the family~$\Gamma$.
\end{proof}

\begin{teo}
\label{th:d5}
Let $X$ be a del Pezzo threefold with $\dd(X)\ge 5$ and
$\uprho(X)=1$.
Then $\dd(X)=5$, the variety $X=X_5\subset\PP^6$ is unique up to
an isomorphism and $X=\Gr(2,5)\cap\PP^6\subset\PP^9$ \textup(with Pl\"ucker 
embedding\textup).
\end{teo}

\begin{proof}
As usual, we write $-K_X=2H$.
Let $l\subset X$ be a line and let $\psi\colon X \dashrightarrow Y\subset
\PP^{d-1}$
be the projection from $l$, where $Y=\psi(X)$ is the image of $\psi$.
The map $\psi$ is given by the linear system $|H-l|$ of
hyperplanes passing through $l$.
In this case $l$ coincides (as a scheme) with the base locus $\Bs|H-l|$ of this
linear system.
Let $\sigma\colon \widetilde{X}\to X$ be the blowup of $l$,
let $H^*:= \sigma^*H$, and let $E$ is the exceptional divisor.
Then the proper transform of $|H-l|$ on $\widetilde{X}$ is the complete linear
system
$|H^*-E|$ of dimension
\begin{equation}
\label{equation-index=2-dim1}
\dim|H^*-E|=d-1.
\end{equation}
According to the construction, this linear system is base point free.
We can write
$$
-K_{\widetilde{X}}=\sigma^*(-K_X)-E=H^*+H^*-E,
$$
where the divisors $H^*$ and $H^*-E$ are nef and not proportional.
Since $\uprho(\widetilde{X})=2$, their sum is an ample divisor, i.e.
$\widetilde{X}$ is a Fano variety.

Since $\deg\NNN_{l/X}=0$ (see Proposition~\ref{proposition-index=2-lines}
\ref{proposition-index=2-lines-nb}), 
from Lemma~\ref{lemma-blowup-curve-intersection} we obtain
\begin{equation}
\label{eq:iota=2}
H^{* 3}=d, \quad H^{* 2}\cdot E=0, \quad H^{* }\cdot E^2=-1, \quad E^3=0.
\end{equation}
From~\eqref{equation-index=2-dim1} and the exact sequence
$$
0 \longrightarrow\OOO_{\widetilde{X}}(H^*-2E) \longrightarrow\OOO_{\widetilde{X}} (H^*-E)
\longrightarrow\OOO_E(H^*-E) \longrightarrow 0
$$
we obtain
\begin{equation}
\label{equation-index=2-dim}
\dim|H^*-2E| \ge d-5.
\end{equation}
In particular, $|H^*-2E|\neq \varnothing$.
Recall that $\uprho(\widetilde{X})=2$.
Therefore, the cone $\NE(\widetilde{X})$ has exactly two extremal rays
(and they are both negative with respect to $K_{\widetilde{X}}$).
Let $\rR_{\varphi}\subset \NE(\widetilde{X})$ be the extremal ray other than 
the ray
generated by the exceptional curves of $\sigma$ and let
$\varphi\colon\widetilde{X}\to Y$ be the
contraction of $\rR_{\varphi}$.
Then $\Pic(Y)\simeq \ZZ$. Let $M$ be the ample generator of this group.
Let $J$ be a line on $X$ meeting $l$ and let $\widetilde{J}$ be its
proper transform on $\widetilde{X}$. Then $(H^*-E)\cdot \widetilde{J}=0$.
Since the divisor $H^*-E$ is nef, it is a supporting divisor of the ray 
$\rR_{\varphi}=\RR_+[\widetilde{J}]$,
i.e. the equality $(H^*-E)\cdot C=0$ holds the some curve $C$ on
$\widetilde{X}$ if and only if $[C]\in\rR_{\varphi}$.
Thus $H^*-E=\varphi^*M$.
It follows also that the fibers of $\varphi$ are the proper transforms of lines
meeting
$l$.

Note that any element of the linear system
$$
\sigma_*\bigl(|H^*-2E|\bigr)\subset |H|
$$
is irreducible. However the linear system $|H^*-2E|$ can
have the divisor $E$ as its fixed component.
Take $m$ maximal such that $|H^*-mE|\neq \varnothing$ and let $D\in
|H^*-mE|$
be an arbitrary divisor.
(Actually, from our computations below we obtain that $m=2$.)
Then $D$ is an irreducible divisor. According to~\eqref{equation-index=2-dim},
we have
$m\ge 2$. Thus $D\cdot \widetilde{J}<0$. This means that $\varphi$ is 
a birational contraction and
$D$ is its exceptional divisor.

But then
$$
\dim|H^*-mE|=\dim|H^*-2E|=0.
$$
Therefore, $d=5$ (see~\eqref{equation-index=2-dim}). Taking
\eqref{eq:iota=2} into account we can write
$$
0=(\varphi^*M)^2\cdot D=(H^*-E)^2\cdot D=(H^*-E)^2\cdot (H^*-mE)=4-2m.
$$
Hence, $m=2$, i.e. $D\sim H^*-2E$.
Further,
$$
\varphi^*M\cdot D^2=(H^*-E)\cdot D^2=(H^*-E)\cdot
(H^*-2E)^2=-3.
$$
Therefore, $\varphi$ contracts $D$ to a curve $Z$. Moreover, $\varphi$ is 
a contraction of type \type{B_1}.
This means that $Y$ is a nonsingular variety, $Z$ is a nonsingular curve,
and $\varphi$ is the blowup of $Y$ along~$Z$.
Moreover,
$$
M\cdot Z = -\varphi^*M\cdot D^2=3.
$$
Since
$$
-K_Y=\varphi_*(-K_X)= \varphi_*(2H^*-E)= \varphi_*(3\varphi^*M - D)=3M,
$$
$Y$ is a Fano threefold of index $3$, i.e.
$Y=Q\subset\PP^4$ is a nonsingular quadric.
Here $Z\subset Q\subset\PP^4$ is a curve of degree~$3$.
Such a curve cannot also
lie in a plane.
Therefore, $Z$ is a rational twisted cubic curve in $\PP^3$.

It is not difficult also to compute that $(H^*-E)^2\cdot E=2$, i.\,e.\ 
$\varphi(E)$ is 
a hyperplane section.
If $\NNN_{l/X}\simeq \OOO_{\PP^1}\oplus \OOO_{\PP^1}$, 
then
$E\simeq \PP^1\times \PP^1$
and $\varphi(E)$ is nonsingular hyperplane
section, is isomorphic to $\PP^1\times \PP^1$.
In this case $Z$ is a divisor of bidegree $(1,2)$ on 
$\varphi(E)\simeq\PP^1\times
\PP^1$.
Moreover, if $\NNN_{l/X}\simeq\OOO_{\PP^1}(-1)\oplus \OOO_{\PP^1}(1)$,
then $E\simeq \FF_2$ and the morphism $\varphi$ contracts the negative 
section of this
ruled surface.
Therefore, $\varphi(E)$ is a quadratic cone in this case.
In both cases the divisor $\varphi(E)$ is uniquely defined by the curve $Z$: 
it is cut out on $Q$ by the linear span $\langle Z\rangle=\PP^3$.

Thus we obtain the following diagram
\begin{equation}
\label{diagram-i=2:projection}
\vcenter{
\xymatrix@R=2em{
& \widetilde{X}\ar[dl]_{\sigma}\ar[dr]^{\varphi} &
\\
X\ar@{-->}[rr]^{\psi}&& Q
} }
\end{equation}
where $\sigma$ is the blowup of a line $l\subset X$, $\varphi$
is the blowup of a rational twisted cubic curve $Z$ on a nonsingular quadric
$Q$, and $\psi$ is the projection from $l$.
The construction~\eqref{diagram-i=2:projection} is completely determined the 
pair
$(X,l)$. On the other hand,
starting with a (nonsingular) quadric $Q\subset\PP^4$ and a twisted cubic curve 
$Z\subset
Q$ one can uniquely restore the diagram, and hence the pair $(X,l)$.

We may assume that $\NNN_{l/X}\simeq \OOO_{\PP^1}\oplus \OOO_{\PP^1}$ (see
Proposition
\ref{proposition-index=2-lines}\ref{proposition-index=2-lines-N}).
The automorphism group $\Aut(Q)$ of the quadric $Q$ transitively acts on
the set of nonsingular hyperplane sections so that the stabilizer of any 
nonsingular hyperplane
section transitively acts on the set of nonsingular divisors of bidegree 
$(1,2)$.
This means that our variety $X$ is unique up to isomorphism.
On the other hand, a nonsingular intersection $\Gr(2,5)\cap\PP^6\subset\PP^9$ 
is a
Fano threefold
of index~$2$ and degree $5$. Therefore, $X\simeq \Gr(2,5)\cap\PP^6$.
Theorem~\ref{th:d5} is proved.
\end{proof}

\begin{cor}
\label{cor:V5-rat}
The del Pezzo threefold $X_5\subset\PP^6$ of degree~$5$ is rational.
\end{cor}
This result can be generalized on forms of the variety $X_5\subset\PP^6$
over non-closed fields~\cite{KP-rF}.

\begin{rem}
The construction, similar to one used in the proof of Theorem~\ref{th:d5},
can be applied to a del Pezzo threefold $X_4\subset\PP^5$ of degree~$4$, which 
is a complete
intersection of two quadrics (see Exercise~\ref{ex:sl:V4} in Section 
\ref{section:sl} below).
\end{rem}

\begin{rem}[(see~{\cite{Mukai-Umemura-1983}}, 
{\cite{KPS:Hilb}})]
\label{remark:MU}
The the variety $X_5\subset\PP^6$ is quasihomogeneous with respect to
an action of the group $\PSL_2(\CC)$.
To explain this, we denote by $M_d$ the
space of
binary forms of degree $d$ in variables $x_1$, $x_2$. The elements
$M_d$ are homogeneous polynomials of degree $d$ in $x_1$,~$x_2$.
The group $\SL_2(\CC)$ naturally acts on the space $M_d$ and
induces an action of
$\PSL_2(\CC)$ on the projective space $\PP(M_d)$.

Consider the polynomial
$$
f(x_1,x_2)=x_1x_2(x_1^4-x_2^4)\in M_6
$$
and let $[f]\in\PP(M_6)$ be the corresponding point.
The polynomial $f$ is semi-invariant with respect to the binary octahedral 
group 
$\mathrm{Oct}\subset\SL_2(\CC)$
\cite[Ch.~4]{Springer1977}.
We claim that the closure
$\overline{\PSL_2(\CC)\cdot [f]}\subset\PP(M_6)$
of the orbit $\PSL_2(\CC)\cdot [f]\simeq \SL_2(\CC)/ \mathrm{Oct}$ is a 
Fano threefolds $X_5\subset\PP^6$ (see~\cite{Mukai-Umemura-1983}).
Indeed, there is a decomposition of $\SL_2(\CC)$-representations:
$$
\wedge^2 M_4=M_6\oplus M_2.
$$
Here the element $f$ is represented in the form $f=h_1\wedge h_2\in 
M_6\subset\wedge^2 M_4$,
where $h_1,\, h_2\in M_4$ are two distinct semi-invariants of the tetrahedral
subgroup
$\mathrm{Tet}\subset\mathrm{Oct}$.
The group $\PSL_2(\CC)$ naturally acts on
$$
\Gr(2,5)=\Gr(2,M_4)\subset\PP(\wedge^2 M_4)
$$
so that the variety $\Gr(2,M_4)\cap\PP(M_6)$ is invariant and contains the point
$[f]=[h_1\wedge h_2]$. Therefore, it contains also
the orbit $\PSL_2(\CC)\cdot [f]$, which must be a dense open subset.
It can be shown that $\Gr(2,M_4)\cap\PP(M_6)$ is nonsingular, therefore,
this is a del Pezzo threefold of degree~$5$ (see Exercise~\ref{zad:Gr25} at the 
end 
of 
this section).
\end{rem}

\subsection{Summary of the results}
The Fano threefolds of index~$2$ are described by the following
Table~\ref{table:del-Pezzo-3-folds} (recall that $\hr^{1,2}(X)=\dim 
H^{1,2}(X)$ is middle Hodge number of a threefold~$X$).
\begin{table}[ht]\small
\caption{Del Pezzo threefolds}
\label{table:del-Pezzo-3-folds}
\begin{center}\def\arraystretch{1.3}
\begin{tabular}{|c|c|c|c|p{110mm}|}
\hline
& $\dd(X)$ & $\uprho(X)$ & $\hr^{1,2}(X)$ & \heading{$X$}
\\\hline
\setcounter{NNN}{0}\rownumber
\label{table:Fano2:1} & $1$ &$1$ & 21
&a hypersurface of degree $6$ in $\PP(1^3,2,3)$
\\\hline
\rownumber
\label{table:Fano2:2} & $2$ &$1$ & 10 &a hypersurface of degree $4$
in 
$\PP(1^4,2)$
\\\hline
\rownumber
\label{table:Fano2:3} & $3$ &$1$ & 5 &$X_3\subset\PP^4$, a
hypersurface
of degree $3$
\\\hline
\rownumber
\label{table:Fano2:4}
& $4$ &$1$ & 2 &$X_4\subset\PP^5$, an intersection of two quadrics
\\\hline
\rownumber & $5$ &$1$ & 0 &$X_5\subset\PP^6$, a section of Grassmannian
$\Gr(2,5)$
embedded by Pl\"ucker to $\PP(\wedge^2 \CC^5)=\PP^9$ by a
subspace of codimension $3$
\\\hline
\rownumber & $6$ &$3$ & 0&$X_6\subset\PP^7$, $X\simeq \PP^1\times \PP^1\times
\PP^1$
\\\hline
\rownumber & $6$ &$2$ & 0 &$X_6\subset\PP^7$, a divisor of bidegree $(1,1)$ on
$\PP^2\times \PP^2$
\\\hline
\rownumber & $7$ &$2$ & 0&$X_7\subset\PP^8$, the blowup of a point on $\PP^3$
\\\hline
\end{tabular}
\end{center}
\end{table}
\begin{rem}
Del Pezzo varieties of arbitrary dimension (i.e. $n$-dimensional
Fano varieties of index $n-1$) are classified in the works of T.~Fujita. They 
have 
the same description as three-dimensional ones (see Exercises~\ref{zad:DP1-4} 
and~\ref{zad:DP6} 
below).
\end{rem}

\begin{rem}[(Added in translation)]
A complete classification of higher-dimensional del Pezzo varieties with 
terminal singularities 
was obtained in the recent work \cite{KP:dP}.
\end{rem}

\begin{zadachi}
\eitem
\label{exercise-elliptic-curve}
Prove that, under a suitable choice of generators $x,y,z$ of the algebra 
$\R(C,A)$,
the relation $f$ in~\eqref{eq:elliptic} can be written in the form
$f=z^2+y^3+a yx^4+bx^6$, $a,\, b\in\CC$. What conditions are imposed here
on the coefficients
$a,\, b$? Deduce from this an explicit form of the relation $f$ in 
\ref{projective-models:del-Pezzo}\ref{projective-models-degree=1-dim=2} and
\ref{projective-models-3-folds}\ref{projective-models-degree=1-dim=3}.

\eitem
Compute the canonical algebra $\R(C,K_C)$ of a curve $C$ of genus~$2$.
Deduce from this that $C$ is embedded to the weighted projective plane
$\PP(1^2,3)$.

\eitem
Compute the Hodge numbers $\hr^{1,2}(X)$ for a del Pezzo threefold of degree 
$d\le
4$.

\eitem
\label{zad:DP:1-1:6}
Let $V$ be a hyperplane section of $\PP^2\times\PP^2\subset \PP^8$
(with Pl\"ucker embedding). Prove that there exists exactly three classes of
(up to isomorphism) such varieties~$V$: nonsingular,
having one ordinary double point, and reducible one having two components.
In particular, the variety $V$ is normal if it is irreducible.

\eitem
\label{zadacha:Fano-neFano}
Let $Y$ be a three-dimensional nonsingular variety and let $f\colon X\to Y$
be the blowup of a nonsingular curve $C\subset Y$. Assume that $X$ is a Fano 
threefold and~$Y$ is not Fano.
Prove that the curve $C$ is rational and the divisor $-K_Y$
is nef.
\hint{Apply the Cone Theorem to $X$ and find a curve that negatively intersects
the exceptional divisor $E=f^{-1}(C)$.}

\eitem
Which assertions of Proposition~\ref{proposition-index=2-lines} become wrong for 
del Pezzo threefolds of degree $\le 3$?

\eitem
Prove that there are the following two analogs of the construction from 
Remark~\ref{remark:MU} (another analog, for the icosahedron group see in 
Example~\ref{remark:MU:22}).
\begin{enumerate}
\item
The binary tetrahedral group
$\mathrm{Tet}\subset\SL_2(\CC)$ has two semi-invariants
$h_i=x_1^4\pm 2\sqrt{3}x_1^2x_2^2+x_2^4\in M_4$.
The closure of the orbit $\PSL_2(\CC)\cdot [h_i]$ is a
nonsingular quadric.
\item
The binary dihedral group
$\mathfrak{D}_3\subset\SL_2(\CC)$ has two semi-invariants
$h_i=x_1^3\pm x_2^3\in M_3$.
The closure of the orbit $\PSL_2(\CC)\cdot [h_i]$ coincides with 
$\PP(M_3)=\PP^3$.
\end{enumerate}
In both cases describe the $\PSL_2(\CC)$-orbits of dimension~$\le2$.

\eitem
\label{zad:Gr25}
Prove that the variety $\Gr(2,M_4)\cap\PP(M_6)$ in Remark~\ref{remark:MU} is 
nonsingular
(and has expected dimension).
Describe the $\PSL_2(\CC)$-orbits of dimension $\le 2$ on this variety.

\eitem
Let $X=X_5\subset\PP^6$ be a del Pezzo threefold of degree $5$.
Prove that the Hilbert scheme $\Lines(X)$ of lines on $X$ is irreducible and 
isomorphic to the
projective plane.
Prove that there are there lines passing through a general point of $X$.
\hint{Use the construction~\eqref{diagram-i=2:projection} and show
that the lines on $X$ meeting given one are parametrized by a rational curve and
two such general curves meet each other at a single point.}

\eitem
Describe the Hilbert scheme $\Lines(X)$ of lines on del Pezzo threefolds of 
degree~$\ge 6$.

\eitem
\label{zad:DP1-4}
Prove that the higher-dimensional del Pezzo varieties of degree $d\le 4$ (i.e.
Fano varieties of dimension $n$ and index $n-1$)
have the same description as~\ref{table:Fano2:1}--\ref{table:Fano2:4}
Table~\ref{table:del-Pezzo-3-folds}.

\eitem
\label{zad:DP6}
Prove that a Fano variety of dimension $n\ge 4$, index
$n-1$, and degree $d\ge 6$ is isomorphic to $\PP^2\times\PP^2$.

\eitem
Assume that a Fano threefold $X$ is a
double cover of another Fano threefold $Y$.
Describe all possibilities for such covers.
\end{zadachi}

\newpage\section{Base points in the anticanonical system}
\label{sec4}


In this lecture we discuss the question of when the anticanonical linear
system
$|{-}K_X|$ on a Fano threefold has base points.

\subsection{Main results}
Since we are interested in
varieties of Picard number~$1$, the main result of this lecture is the 
following.
\begin{teo}
\label{theorem-Bs-rho=1}
Let $X$ be a Fano threefold with $\uprho(X)=1$.
Then the linear system $|{-}K_X|$ is base point free.
\end{teo}

However Fano threefolds having base points in the linear system
$|{-}K_X|$ exist. This illustrates the following example.

\begin{exa}
\label{example-Bs}
Let $Y$ be a del Pezzo threefold of degree~$1$ and let $M:= -\frac12 K_Y$.
Consider the rational map
$\pi\colon Y \dashrightarrow \PP^2$ given by the
linear system $|M|$. The general fiber of this map is a nonsingular curve of 
genus
$1$.
Let $B\subset Y$ be any nonsingular fiber and let $\varphi \colon X\to Y$ be 
the
blowup with center $B$. Denote $\overline{M}:= \varphi^*M$
and let $D:= \varphi^{-1}(B)$ be the exceptional divisor.
It is clear that the curve $B$ is a scheme-theoretical intersection of two 
elements 
$M_1,M_2\in|M|$. Thus the linear system $|\overline{M}-D|$ is a pencil without 
base
points.
We have $K_X=-2\overline{M}+D$. By Kleiman's Ampleness Criterion~\ref{Kleiman} 
the divisor
$$
-K_X\sim (\overline{M}-D)+\overline{M}
$$
is ample,
i.e. $X$ is a Fano threefold.
Using the relations~\eqref{eq:blowup-curve-intersection} we obtain
$$
(-K_X)^3=(2\overline{M}-D)^3=8-6+2=4.
$$
Hence, $X$ is a Fano threefold of index~$1$ and genus\footnote{Recall that the 
genus of a Fano threefold $X$
is the integral positive number $\g(X)=\frac12(-K_X)^3+1$
(see~\ref{def:coindex})} $3$ with 
$\uprho(X)=2$.

There exists exactly two extremal contractions on $X$: the birational morphism
$\varphi \colon X\to Y$
and the morphism $\psi\colon X\to \PP^1$ given by the linear system 
$|\overline{M}-D|$.
It is easy to see that $\psi\colon X\to \PP^1$ is a del Pezzo fibration
of degree~1.
Let $\Omega\in|\overline{M}-D|$ be a general fiber of this fibration. Since
$-K_X-\Omega\sim \overline{M}$, by the Kawamata--Viehweg Vanishing Theorem
we have
$$
H^1(X,\,\OOO_X(-K_X-\Omega))=0.
$$
Therefore, by Lemma~\ref{lemma-linear-systems}
$$
\Omega\cap\Bs|{-}K_X|=\Bs|{-}K_\Omega|.
$$
Since the linear system $|{-}K_\Omega|$ on $\Omega$ has a unique base
point $P=\Bs |{-}K_\Omega|$, the set
$\Bs|{-}K_X|$ is a curve $Z=\varphi^{-1}(P)$, which is a section of the 
fibration
$\psi\colon X\to \PP^1$.
\end{exa}

It turns out that even in the case of arbitrary Picard number there are only a 
few Fano threefolds with $\Bs|{-}K_X|\neq \varnothing$.
It is possible to get a complete classification:

\begin{teo}
\label{theorem-Bs}
Let $X$ be a Fano threefold such that the linear system
$|{-}K_X|$ has base points.
Then $X$ is isomorphic either to
the product $F\times \PP^1$, where $F$ is a del Pezzo surface of degree~$1$, or 
the variety from Example~\ref{example-Bs}.
\end{teo}

The whole lecture is devoted to the proof of Theorems 
\ref{theorem-Bs-rho=1} and~\ref{theorem-Bs}.
The first step in the proof of these two theorems is the following fact that 
is an immediate consequence of
the corresponding facts on \K3 surfaces (see Appendix~\ref{ch:K3}).

\subsection{Base locus}
\begin{prp}
\label{proposition-Bs}
Let $X$ be a Fano threefold of genus $g=\g(X)$ and let $S\in
|{-}K_X|$ be a nonsingular element.
Assume that the linear system $|{-}K_X|$ has base points.
Then
\begin{equation}
\label{base-point-index=1}
\OOO_S(-K_X)=\OOO_S(Z+gC),\qquad g\ge 3,
\end{equation}
where $Z$ is a smooth rational $(-2)$-curve on $S$ and 
$C$ is a fiber of an elliptic pencil without base points on $S$ and
$Z\cdot C=1$. In particular,
\begin{equation}
\label{eq:Bs1}
-K_X\cdot C=1,\quad \iota(X)=1\quad \text{and}\quad -K_X\cdot Z=g-2.
\end{equation}
Moreover, $\Bs|{-}K_X|=Z$ \textup(as a scheme\textup).
The image of the map $\Phi_{|{-}K_X|}\colon X \dashrightarrow \PP^{g+1}$ is a
surface
of minimal degree \textup(see Appendix~\ref{sect:VMD})
$$
W=W_g\subset\PP^{g+1}.
$$
\end{prp}

\begin{proof}
Recall that $S$ is a \K3 surface (see Proposition~\ref{Fano:adjunction}).
Let $A:= -K_X|_S$.
Then the sheaf $\OOO_S(A)=\OOO_S(-K_X)$
is ample on $S$ and $A^2=2g-2$.
Since $H^1(X,\,\OOO_X)=0$, by Lemma~\ref{lemma-linear-systems} we have
$$
\Bs|A|=S\cap\Bs|{-}K_X|=\Bs|{-}K_X|.
$$
Hence, by Theorem~\ref{theorem-AB-Saint-Donat} we have
$|A|=Z+m|C|$, where $Z$ and $C$ are such as in~\eqref{base-point-index=1},
and $m\ge 3$ because
$$
-2+m=(Z+mC)\cdot Z=A\cdot Z=-K_X\cdot Z>0.
$$
Therefore, $|{-}K_X|$ has a unique base curve $Z$ and has no
other base points because the linear system $|mC| $ on $S$ has no
base points as well.
Furthermore,
$$
2g-2=(-K_X)^3=A^2=(Z+mC)^2=-2+2m.
$$
Hence, $g=m$.
The restriction $\Phi_S$ of the map $\Phi$ to the surface $S$ is a morphism,
given by the linear system $|gC|$.
This linear system is composed of elliptic pencil $|C|$, hence
$\Phi_S$ is decomposed as a composition
$$
\Phi_S\colon S \xarr{\Phi_{|C|}} \PP^1 \xarr{\Phi_{|\OOO_{\PP^1}(g)|}} 
\Lambda\subset \PP^g,
$$
where $\Lambda \subset\PP^g$ is a rational normal curve of degree $g$.
But this curve $\Lambda=\Phi_S(Z)$ is nothing but a section of $W=\Phi(X)$ by a
hyperplane in $\PP^{g+1}$ corresponding to the divisor $S$.
Therefore, $\dim(W)=2$ and $\deg W=g$.
The proposition is proved.
\end{proof}

Let $\sigma \colon \widetilde{X}\to X$ be the blowup of 
the curve $Z$, let $E:= \sigma^{-1}(Z)$ be an exceptional divisor, and let
$\widetilde{S}\subset\widetilde{X}$ is the proper transform of the surface $S$. 
Then
$E$ with projection to $Z$
is a rational ruled
surface, which is isomorphic to $\FF_e$ for some $e\ge 0$. We denote by
$\Sigma$ the exceptional section of
$E\simeq \FF_e$ and by $\Upsilon$ we denote its fiber.
Put $H^*:= \sigma^*(-K_X)$.
We have
$-K_{\widetilde{X}}\sim H^*-E\sim \widetilde{S}$.
From~\eqref{eq:blowup-curve-intersection} we obtain
\begin{equation}
\label{equation-base-points-intersections}
H^{*3}=2g-2,\quad H^{*2}\cdot E=0,\quad H^*\cdot E^2=2-g,\quad E^3=4-g.
\end{equation}

Consider the map
$$
\delta:= \Phi_{|{-}K_X|} \comp \sigma \colon \widetilde{X} \dashrightarrow
\PP^{g+1}.
$$
It is given by the linear system
$$
|\widetilde{S}|=|{-}K_{\widetilde{X}}|=|H^*-E|,
$$
which is base point free because the curve $Z$ is the
scheme base locus of the linear system $|{-}K_X|$. Therefore, the divisor
$-K_{\widetilde{X}}=H^*-E$ is nef and
$\delta$ is a morphism.
Its general fiber is a geometrically irreducible
smooth elliptic curve.
Indeed, among the fibers of the morphism $\delta$ there are
proper transforms of curves from elliptic pencil $| C|$ on $S$.
Thus we have the following commutative diagram
$$
\xymatrix@R=2em{
&\widetilde{X}\ar[dl]_{\sigma}\ar []!<0pt,0pt>;[dr]!<-40pt,5pt>^{\delta}&
\\
X\ar@{-->}[rr]^(.4){\Phi}&&W=W_g\subset\PP^{g+1}
}
$$

\begin{lem}
\label{lemma:-Bs-hyp}
The equality
\begin{equation}
\label{equation-Bs-hyp}
-K_{\widetilde{X}}|_E \sim \Sigma+\frac{g+e}2\Upsilon,\qquad g\ge e
\end{equation}
holds.
In particular, $g\equiv e \mod 2$. Furthermore, the restriction of $\delta$ to 
$E$ is a birational morphism, which
maps the fibers of the ruled surface $E$ to lines
on $W=W_g\subset\PP^{g+1}$. Therefore, $W=W_g\subset\PP^{g+1}$ is
either a nonsingular ruled surface \textup(in the case where $g>e$\textup) or a 
cone
over a rational normal curve \textup(if $g=e$\textup).
\end{lem}

\begin{proof}
Since $E$ is an exceptional divisor of the blowup, we have
$$
E=\PP_Z\bigl(\NNN_{Z/X}^\vee\bigr),
\qquad
\OOO_E(E)\simeq \OOO_{\PP_Z\left(\NNN_{Z/X}^\vee\right)}(-1).
$$
Therefore,
$$
-K_{\widetilde{X}}|_E=(H^*-E) |_E \sim \Sigma+a\Upsilon
$$
for some $a$.
Using~\eqref{equation-base-points-intersections} we compute:
$$
-e+2a=(\Sigma+a\Upsilon)^2=(H^*-E)^2\cdot E=g.
$$
From this we obtain~\eqref{equation-Bs-hyp} and also
$\deg \delta=1$ because $\deg W=g$.
Since $(H^*-E) |_E \sim \Sigma+a\Upsilon$, the
fibers of the ruled surface $E=\FF_e$ are mapped to lines by the map
given by the linear system $|H^*-E|$.
Thus $W$ is either a nonsingular ruled surface or a cone
over a rational normal curve.
\end{proof}

\subsection{Case $\uprho(X)=1$}
\begin{proof}[Proof of Theorem~\ref{theorem-Bs-rho=1}]
In this case $\uprho(\widetilde{X})=2$. Therefore, $\uprho(W)=1$
and so $W=W_g\subset\PP^{g+1}$ is a cone over a rational normal curve
of degree $g\ge 3$.
Let $l\subset W$ be a line-generator of the cone
and let $w_0\in W$ be its vertex.
Then $gl$ is a hyperplane section of the surface $W$ and so
$-K_{\widetilde{X}}\sim \delta^*(gl)$.
Since $\Pic(\widetilde{X})=\ZZ\cdot (-K_{\widetilde{X}})\oplus \ZZ\cdot E$
in our case, any divisor $\widetilde{D}$ on $\widetilde{X}$, that does not meet 
a
general fiber of the morphism $\delta$, is proportional to $-K_{\widetilde{X}}$,
i.e. it is the pull-back of a divisor on $W$.
In particular, the image of $\widetilde{D}$ cannot be a point.
Therefore, $\dim \delta^{-1}(w_0)=1$.
By Bertini's theorem, for general choice of $l$, the closure
$\widetilde{F}:= \overline{\delta^{-1}(l\setminus \{w_0\})}$ is an irreducible
surface.
On the open subset $\widetilde{U}:= \widetilde{X}\setminus 
\delta^{-1}(w_0)$
there is the linear equivalence $\delta^*(gl)\sim g\widetilde{F}$.
We can extend it to the whole variety~$\widetilde{X}$:
$$
-K_{\widetilde{X}}\sim \delta^*(gl)\sim g\widetilde{F}.
$$
But then $-K_X=-\sigma_*K_{\widetilde{X}}$ is divisible by $g$.
Therefore, $\iota(X)\ge g>1$. This contradicts~\eqref{eq:Bs1} and this
proves Theorem~\ref{theorem-Bs-rho=1}.
\end{proof}

\subsection{General case}
Now we proceed to the proof of Theorem~\ref{theorem-Bs}.

\begin{lem}
\label{lemma:bs:ample}
The divisor $E$ is relatively ample over $W$.
\end{lem}

\begin{proof}
Assume the converse. Then by the relative version of Kleiman's Ampleness 
Criterion we have $E\cdot z\le 0$ for some
non-zero
element $z\in \NE(\widetilde{X}/W)$.
Therefore,
$$
\sigma^*(-K_X)\cdot z=(E-K_{\widetilde{X}})\cdot z=E\cdot z \le 0.
$$
Since $\sigma$ is an extremal contraction, we have $\sigma^*(-K_X)\cdot z=0$ and
$z\approxident\alpha [\Upsilon]$, where $\Upsilon$ is a fiber of $E\simeq 
\FF_e$
and $\alpha>0$. But since $-K_{\widetilde{X}}\cdot \Upsilon>0$, 
the element~$z$ cannot lie in $\NE(\widetilde{X}/W)$, a contradiction.
\end{proof}

\begin{lem}
\label{lemma:bs:1fiber}
Any fiber of the morphism $\delta$ over a point $w\in W\setminus \Sing(W)$
is a reduced
irreducible curve of arithmetic genus~$1$.
\end{lem}

\begin{proof}
Let $\widetilde{X}_w$ be the scheme fiber of the morphism $\delta$ over a 
point
$w\in W\setminus \Sing(W)$.

If $N\subset \widetilde{X}_w$ is a two-dimensional component, then the 
intersection of $E\cap
N$ is non-empty by Lemma~\ref{lemma:bs:ample} and this set must be
a curve contracted on~$E$, which, in turn, must be the exceptional
section $\Sigma$ of the
ruled surface $E=\FF_e\to Z$. Therefore, $W$ is a cone with vertex~$w$. 
This contradicts our assumption $w\in W\setminus \Sing(W)$.

Hence, all the fibers over smooth points $w\in W$ are one-dimensional.
Then the fiber $\widetilde{X}_w$
is a locally complete intersection in $\widetilde{X}$.
In particular, $\widetilde{X}_w$ has no embedded components.
If $\widetilde{X}_w\subset \widetilde{S}$, then $E\cdot \widetilde{X}_w=(Z\cdot 
\sigma(\widetilde{X}_w))_S=1$.
Since the variety $\widetilde{X}$ is nonsingular, 
the morphism~$\delta$ is flat over $W\setminus \Sing(W)$
(see e.~g.~\cite[Theorem 23.1]{Matsumura1986}).
Therefore,
the equality $E\cdot \widetilde{X}_w=1$ holds for any point $w\in W\setminus 
\Sing(W)$.
Since the divisor $E$ is relatively ample, the
fiber $\widetilde{X}_w$ is irreducible and reduced.
Since $K_{\widetilde{X}}\cdot \widetilde{X}_w=0$, we have 
$\p(\widetilde{X}_w)=1$.
\end{proof}

\begin{lem}
There is a decomposition
\begin{equation}
\label{equation-Bs-Ka}
-K_{\widetilde{X}}=\widetilde{D}+\frac {g+e}2\widetilde{F},\qquad g\ge e,\quad 
g\ge 3,
\end{equation}
where $\widetilde{D}$ is a prime divisor with support in 
$\delta^{-1}(\delta(\Sigma))$ and $\widetilde{F}$ is a fiber over a line
$\delta(\Upsilon)$, i.e.
$$
\widetilde{F}:= \overline{\delta^{-1}(\delta(\Upsilon)\setminus 
\delta(\Sigma))},
$$
moreover, $\dim|\widetilde{F}|>0$ and $\widetilde{F}$ is an irreducible surface.
\end{lem}

\begin{proof}
Put $l:= \delta(\Upsilon)$ and consider the following divisor on $W$:
$$
\Theta:= \begin{cases}
\ \delta(\Sigma), &\text{if $\delta_E\colon E\to W$ is an isomorphism},
\\
\ 0, &\text{if $W$ is a cone}.
\end{cases}
$$
By the construction, the divisor $-K_{\widetilde{X}}$ is the pull-back of the
class of
hyperplane section $L$ of the surface
$W=W_g\subset\PP^{g+1}$, moreover, $L\sim \Theta+\frac {g+e}2 l$.
On the open subset $U:= \widetilde{X}\setminus \widetilde{D}$
there is the linear equivalence $\delta^*L\sim \frac {g+e}2\widetilde{F}$.
We can extend its to the whole variety $\widetilde{X}$:
$$
-K_{\widetilde{X}}\sim \delta^*L=\widetilde{D}+\frac {g+e}2\widetilde{F},
$$
where $\widetilde{D}$ is an effective divisor with support in 
$\delta^{-1}(\delta(\Sigma))$.
Restrict this equality to~$E$.
Since the divisor $E$ is $\delta$-ample, any component of the right hand side
meets~$E$.
By the construction, $\Supp(\widetilde{D}\cap E)=\Sigma$ and 
$\Supp(\widetilde{F}\cap
E)=\Upsilon$.
Comparing with~\eqref{equation-Bs-hyp}, we obtain $\widetilde{D}\cap E=\Sigma$ 
and
$\widetilde{F}\cap E=\Upsilon$ (the intersections are in the scheme sense).
Thus the divisor~$\widetilde{D}$ is irreducible and cut out on $E$ the 
exceptional 
section
$\Sigma$. By the construction the divisor $\widetilde{F}$ varies in a family.
Therefore, $\dim|\widetilde{F}|>0$.
\end{proof}

\begin{cor}
If $W$ is a cone, then fiber of $\delta$ over the vertex $w_0\in W$ is 
two-dimensional.
\end{cor}

\begin{cor}
There is a relation
\begin{equation}
\label{equation-Bs-Ka-X}
-K_X=D+\frac {g+e}2F,\qquad g\ge e,\quad g\ge 3,
\end{equation}
where $F:= \sigma(\widetilde{F})$, $\dim|F|>0$ and $D:= 
\sigma(\widetilde{D})$ is a prime divisor such that $D\cap S=Z$.
Moreover, for a general element $S\in|{-}K_X|$, the intersection 
$F\cap S=C$ is an element of an elliptic pencil and so $F\cdot C=0$.
\end{cor}

\begin{lem}
If the divisor $D$ is nef, then $X\simeq Y\times \PP^1$,
where $Y$ is a del Pezzo surface of degree~$1$.
\end{lem}

\begin{proof}
By the Cone Theorem 
there exists a ($K_X$-negative) extremal ray $\rR$ such that $F\cdot
\rR>0$. Let $\varphi\colon X\to Y$ be its contraction.
For the corresponding minimal extremal curve $\ell$ we have
(see Theorem~\ref{class:ext-rays})
$$
3\ge -K_X\cdot \ell=\upmu(\rR)=D\cdot \ell+\frac {g+e}2F\cdot \ell\ge
\frac {g+e}2\ge 2.
$$
where $\upmu(\rR)$ is the length of the extremal ray $\rR$.
Since $g\ge 3$ and $D\cdot \ell\ge 0$, we have $F\cdot \ell=1$ and $g+e\le 6$.

If $g+e> 4$, then $g+e=6$ and $D\cdot \ell=0$. Moreover, in this case
$\upmu(\rR)=3$. According to the classification of extremal
rays~\ref{class:ext-rays}, the
ray~$\rR$ is of type \type{D_3} and the divisor $D$ is a fiber of the 
contraction
$\varphi$.
In particular, $\dim |D|>0$ and $\Bs |D|=\varnothing$.
But then the curve $Z=D\cap S$ must be movable on $S$, a contradiction.

Therefore, $g+e=4$ and the ray $\rR$ can be only of types \type{D_2}, 
\type{D_3},
\type{C_2} or \type{B_2} (by Theorem~\ref{class:ext-rays}
because $\upmu(\rR)\ge 2$).
Take two general elements $F_1,\, F_2\in|F|$.
Then intersections
$F_i\cap S$ are distinct elements of the elliptic pencil $|C|$ and so $F_1\cap 
F_2\cap S=\varnothing$.
Since $S\in|{-}K_X|$ is an ample divisor, we have $F_1\cap F_2=\varnothing$.
If $\varphi$ has two-dimensional fiber~$N$, then
$F_1\cap F_2\cap N\neq \varnothing$ (because the divisors $F_1$ and $F_2$
are ample with respect to $\varphi$). The contradiction
shows that $\rR$ is of type \type{C_2},
i.e. $\varphi$ is a $\PP^1$-bundle over a nonsingular surface $Y$.
Here $F_1$ and $F_2$ are disjoint sections. Since $F_1\sim F_2$, we have 
$X\simeq Y\times \PP^1$. It is clear that $Y$ is a del Pezzo surface
and $K_Y^2=1$ because $\Bs|{-}K_Y|\neq \varnothing$.
\end{proof}

Further we assume that the divisor $D$ is not nef.
Then there exists a ($K_X$-negative) extremal ray $\rR$ such that $D\cdot\rR<0$.
Let $\varphi\colon X\to Y$ be its contraction.
Then $\varphi$ is birational and its exceptional divisor coincides with $D$.
By Kleiman's Ampleness Criterion (see Exercise~\ref{zad:nef} at the end of this 
section)
the divisor
$$
-2K_X-D=-K_X+\frac {g+e}2F
$$
is ample.
Therefore, by the Kodaira Vanishing Theorem
$$
H^1(X, \OOO_X(-K_X-D))=0.
$$
Then, by Lemma~\ref{lemma-linear-systems}, the
complete ample linear system $|\OOO_D(-K_X)|$
on $D$
has a curve~$Z$ as its fixed component.

Again, according to the classification of extremal rays
(Theorem~\ref{class:ext-rays}), the surface~$D$ is isomorphic to either $\PP^2$ 
or
a quadratic cone, or it is a nonsingular minimal
ruled surface. An ample linear system on such a surface
can have base points only in latter case under the additional assumption that 
the surface is not rational~\cite[Ch.~V, Theorem~2.17]{Hartshorn-1977-ag}.
Thus $D$ is a ruled
surface over a non-rational curve.
In particular, $\rR$ is of type \type{B_1}, i.e. $Y$ is nonsingular and 
$\varphi$ is 
the blowup of a nonsingular non-rational curve $B\subset Y$.
Then $Y$ is a Fano threefold (see Exercise~\ref{zadacha:Fano-neFano} in Section 
\ref{sect:index2}).

For a fiber $\ell$ of the ruled surface $D/B$ we have
$-K_X\cdot \ell=1$ and $D\cdot \ell=-1$.
From~\eqref{equation-Bs-Ka-X} we obtain
$F\cdot \ell=1$ and $g+e=4$.
Hence, $g\le 4$. Since $g\ge 3$, we have $g>e$ and so $E\simeq W$ (see
Lemma~\ref{lemma:-Bs-hyp}).
Then
$\delta(\widetilde{D})=\delta(\Sigma)$ is a curve and $\widetilde{D}\to 
\delta(\widetilde{D})$
is an elliptic fibration (see Lemma~\ref{lemma:bs:1fiber}).
Therefore, $B:= \varphi(D)$ is a nonsingular elliptic curve.
We also have $-K_X\cdot Z=g-2$. Since $Z$ is a nonsingular rational
curve lying on $D$, it is a fiber of the ruled surface $D$.
Thus $D\cdot Z=-1$ and $-K_X\cdot Z=1=g-2$ (see~\eqref{eq:Bs1}), i.e. $g=3$ and
$e=1$.
Taking relations~\eqref{eq:blowup-curve-intersection} into account 
we can write
$$
4=2g-2=(-K_X)^3=(-K_Y)^3+2\g(B)-2+2 K_Y\cdot B=(-K_Y)^3+2 K_Y\cdot B.
$$
From~\eqref{equation-Bs-Ka-X} it follows that
$-K_Y=\varphi_*(-K_X)=2M$, where $M:= \phi_*F$. Thus the index
$\iota(Y)$ of the Fano threefold $Y$ is even and
$$
2M^3=1+M\cdot B.
$$
This implies that $M$ is not divisible by~$2$ and so $\iota(Y)=2$.

Consider two general elements $F_1$, $F_2\in|F|$ and let $M_i:= \varphi(F_i)$.
Then $B\subset M_1\cap M_2$. On the other hand, $M_1\cap M_2$ is a curve
of arithmetic genus~$1$. It follows that $B=M_1\cap M_2$
and then $M\cdot B=M^3=1$. We obtain the case from Example
\ref{example-Bs}. Theorem~\ref{theorem-Bs}
is completely proved.

In conclusion we note that singular Fano threefolds (with canonical Gorenstein
singularities) are classified in~\cite{Jahnke-Radloff-2006}.

\begin{zadachi}
\eitem
Let $X=Y\times \PP^1$,
where $Y$ is a del Pezzo surface of degree~$1$.
Prove that the linear system $|{-}K_X|$ has base points.

\eitem
\label{zad:nef}
Let $X$ be a Fano threefold and let $|F|$ be a linear system
on $X$
of positive dimension without fixed components. Prove that the divisor $F$
is nef.
\hint{Use the classification of extremal rays~\ref{class:ext-rays}.}

\eitem
\label{zad:g=2}
Let $X$ be a three-dimensional nonsingular projective variety such
that the anticanonical class $-K_X$ is nef and big (a weak
Fano threefold). Let $(-K_X)^3=2$ and assume that the linear system
$|{-}K_X|$
contains a nonsingular surface.
Prove that $|{-}K_X|$ defines a morphism
$$
\Phi_{|{-}K_X|}\colon X \longrightarrow\PP^3
$$
which is 
generically finite of degree~$2$.
\end{zadachi}

\newpage\section{Hyperelliptic Fano threefolds}
\label{sec5}

\subsection{Anticanonical model}
\begin{dfn}
A Fano threefold $X$ is said to be \textit{hyperelliptic}, if
the linear system
$|{-}K_X|$ is base point free (and defines a morphism), but is not very
ample.
\end{dfn}

For example, according to
Proposition~\ref{DP:projective}\ref{DP:projective-d=1}
del Pezzo threefolds of degree~$1$ are hyperelliptic.
The term ``hyperelliptic Fano threefolds'' clarifies the following Lemma 
\ref{hyperelliptic} 
below.

First recall that by Theorem~\ref{theorem-smooth-divisor}
and Proposition~\ref{Fano:adjunction} a general member
$S\in |{-}K_X|$ on any 
Fano threefold $X$ is a nonsingular \K3 surface.
Moreover, if the linear system $|{-}K_X|$ is base point free, then
the intersection $C=S_1\cap S_2$ of general elements 
$S_1,\, S_2\in |{-}K_X|$ is a nonsingular curve 
and by the adjunction formula $K_C=-K_X|_C$.

\begin{lem}
\label{hyperelliptic}
Let $X$ be a Fano threefold $X$ such that the linear system
$|{-}K_X|$ is base point free. Then the following
conditions are equivalent:
\begin{enumerate}
\item
\label{hyperelliptic-3}
$X$ is hyperelliptic,
\item
\label{hyperelliptic-2}
there exists an element $S\in|{-}K_X|$, which is a nonsingular hyperelliptic
\K3 surface \textup(with respect to the polarization
$\OOO_S(-K_X)$, see Theorem~\ref{theorem-A-Saint-Donat}),
\item
\label{hyperelliptic-1}
there exist elements $H_1, H_2\in|{-}K_X|$ whose intersection is a
nonsingular hyperelliptic curve.
\end{enumerate}
Moreover, the equivalence would hold if in~\ref{hyperelliptic-2} and 
\ref{hyperelliptic-1} ``there exists'' replace with
``any nonsingular''.
\end{lem}

\begin{proof}
The implications~\ref{hyperelliptic-1} $\Longrightarrow$~\ref{hyperelliptic-2}
$\Longrightarrow$~\ref{hyperelliptic-3} obviously follow from definitions and 
surjectivity of the corresponding restriction maps.
Let us prove~\ref{hyperelliptic-3} $\Longrightarrow$~\ref{hyperelliptic-2}.
If the \K3 surface $S$ is not hyperelliptic, then according to
Theorem~\ref{theorem-A-Saint-Donat}\ref{th:K3:surjective}
the graded
algebra $\R(S, -K_X|_S)$
is generated by its component of degree~$1$.
Since
$$
H^1(X,\,\OOO_X(-mK_X))=0\qquad \forall m\ge 0,
$$
by Theorem~\ref{th-hyp-sect} the same holds and for
the algebra $\R(X, -K_X)$.
Therefore, the divisor $-K_X$ is very ample. The proof~\ref{hyperelliptic-2}
$\Longrightarrow$~\ref{hyperelliptic-1} is
similar and uses Noether's Theorem~\ref{th:M-Noether} on canonical
curves.
\end{proof}

\begin{cor}
\label{corollary:hyperelliptic}
Let $X$ be a hyperelliptic Fano threefold $X$ of genus $g$.
Then the linear system $|{-}K_X|$ defines a double cover
$\Phi=\Phi_{|{-}K_X|}\colon X \to Y \subset\PP^{g+1}$ onto its image 
$Y=\Phi(X)$.
Here the image $Y$ is a variety of minimal degree \textup(see 
Appendix~\ref{sect:VMD}). If
$\iota(X)=1$, then the variety~$Y$ is nonsingular.
\end{cor}

\begin{proof}
The restriction map
$$
H^0(X,\,\OOO_X(-K_X)) \longrightarrow H^0(S,\,\OOO_S(-K_X))
$$
is surjective. Then, according to Theorem~\ref{theorem-A-Saint-Donat} and
Lemma~\ref{hyperelliptic}, the morphism
$$
\Phi_{|{-}K_X|}\colon X \longrightarrow Y \subset\PP^{g+1}
$$
is not birational onto its image. We have
$$
2g-2=(-K_X)^3=(\deg \Phi)\cdot (\deg Y),
$$
where $\deg \Phi\ge 2$ and $\deg Y\ge g+1-3+1$ according to
Proposition~\ref{proposition:minimal-degree}.
Therefore, $\deg \Phi=2$ and $\deg Y=g-1$.

Let us prove the nonsingularity of $Y$. Assume the converse, that is, the 
variety $Y$ is 
singular.
In particular, $X\not\simeq\PP^3$ and so $g=\g(X)\ge 3$.
Recall that any variety of minimal degree is normal
(see Proposition~\ref{varieties-minimal-degree}).

Consider the case $\dim \Sing(Y)=1$. Then $\Sing(Y)=L$ is a line and $Y$ is 
a cone with vertex $L$ over a rational normal curve $C=C_{g-1}\subset
\PP^{g-1}$ (see Proposition~\ref{varieties-minimal-degree}). It is not difficult
to compute that in this case the group $\Cl(Y)$ is generated by the class of 
plane
$\Pi\subset Y$ (see Exercise~\ref{zad:cone1} at the end of this section). Since
$\Phi$ is a dominant finite morphism, there exists a well defined
homomorphism
$$
\Phi^*\colon \ZZ\simeq \Cl(Y)\longrightarrow \Cl(X)=\Pic(X),
$$
which is an embedding because its image cannot be
a torsion group. Note that $-K_X=\Phi^*
H_Y$, where $H_Y$ is the class of hyperplane section of $Y$. The degree count
gives us
$H_Y\sim (g-1)\Pi$. Therefore, $-K_X\sim (g-1)\Phi^*\Pi$, i.e. $\iota(X)\ge 
g-1\ge
2$. This contradicts our assumption.

Assume that $\dim \Sing(Y)=0$. Then $\Sing(Y)=\{o\}$ is a point and $Y$ is 
a cone with vertex~$o$ either over a nonsingular rational ruled surface
$S=S_{g-1}\subset\PP^g$ or over the Veronese surface $S=S_4\subset\PP^5$
(again by Proposition~\ref{varieties-minimal-degree}). In the latter case, as 
above, we obtain $\iota(X)>1$. Consider the case where $S=S_{g-1}\subset
\PP^g$ is a nonsingular rational ruled surface. Then $Y$
contains a family of planes meeting each other only at the vertex~$o$. Their
inverse images on $X$ meet each other in a finite set $\Phi^{-1}(o)$. But 
this is impossible on a
nonsingular variety.
\end{proof}

\subsection{Case $\uprho(X)=1$}
\begin{teo}
\label{thm:hyper}
Let $X$ be a hyperelliptic Fano threefold with $\uprho(X)=1$.
Then $X$ is one of the following:
\begin{enumerate}
\item
\label{thm:hyper:i2}
a del Pezzo threefold of degree~$1$;
\item
\label{thm:hyper:g2}
$\iota(X)=1$, $\g(X)=2$, and $X=X_6\subset\PP(1^4,3)$ is a hypersurface
of degree $6$;
\item
\label{thm:hyper:g3}
$\iota(X)=1$, $\g(X)=3$, and $X=X_{2\cdot 4}\subset\PP(1^5,2)$ is 
the intersection of a quadratic cone and a hypersurface of degree~$4$. In this 
case
$X$ can be given in $\PP(1^5,2)$ by the following equations:
$$
\phi_2(x_0,\dots,x_4)=\phi_4(x_0,\dots,x_4, y)=0,\quad \deg x_i=1,\quad \deg
y=2.
$$
\end{enumerate}
\end{teo}

\begin{proof}
According to Proposition~\ref{del-pezzo-projective} for $\iota(X)>1$ 
we obtain the case~\ref{thm:hyper:i2}. Thus from now on we assume that
$\iota(X)=1$.
Let $Y\subset\PP^{g+1}$ be the anticanonical image of $X$, where $g:= \g(X)$.
According to Corollary~\ref{corollary:hyperelliptic}, \ $Y=Y_{g-1}\subset
\PP^{g+1}$ is a nonsingular variety of minimal degree.
If $\uprho(Y)>1$, then it follows from 
Proposition~\ref{varieties-minimal-degree}
that
the variety $Y$ has a structure of $\PP^2$-bundle over $\PP^1$.
Thus in this case $X$ has surjective morphism to a curve.
This contradicts our assumption $\uprho(X)=1$.

Therefore, $\uprho(Y)=1$.
Again it follows from Proposition~\ref{varieties-minimal-degree} that the 
variety
$Y$ is isomorphic to either the
projective space $\PP^3$ or a nonsingular quadric $Y_2\subset\PP^4$.

In the case where $Y\simeq \PP^3$ by the Hurwitz formula
\eqref{equation-Hurwitz-2} the branch divisor $B\subset\PP^3$
is a surface of degree 6 and then $Y=Y_2\subset\PP^4$ and
the branch divisor $B$ is cut out on $Y_2\subset\PP^4$
by a hypersurface of degree~$4$. Furthermore, as in the case of del Pezzo 
threefolds one can construct
an embedding~$X$ to a weighted projective space.
\end{proof}

\subsection{General case}

In the case where $\uprho(X)>1$ the hyperelliptic Fano threefolds also can be
classified by using Corollary~\ref{corollary:hyperelliptic}
and Proposition~\ref{prop:scrolls}. We reproduce this classification
below. Our approach is slightly different from one of~\cite[Ch.~2, Theorem 
2.2]{Isk:anti-e}
and
is based on the method of extremal rays~\cite{Mori-Mukai-1981-82}.

\begin{teo}[{\cite[Ch.~2, Theorem 2.2]{Isk:anti-e}}]
\label{th:hyp:any}
Let $X$ be a hyperelliptic Fano threefold of genus $g:= \g(X)$ with 
$\uprho(X)>1$.
Then $X$ can be presented in the form of a double cover of a nonsingular
rational scroll $Y=\PP_{\PP^1}(\EEE)$,
where
\begin{equation}
\label{eq:gip:E}
\EEE=\OOO_{\PP^1} (d_1)\oplus \OOO_{\PP^1} (d_2)\oplus \OOO_{\PP^1} (d_3),
\quad d_i>0
\end{equation}
with the embedding
$$
\Phi_{|\OOO_{\PP(\EEE)}(1)|}\colon Y \hookrightarrow \PP^{g+1},
$$
given by the tautological linear system $|M|=|\OOO_{\PP(\EEE)}(1)|$,
where
\begin{equation}
\label{eq:gip:d}
g=d_1+d_2+d_3+1.
\end{equation}
The branch divisor $B\subset Y$ of the cover $\Phi\colon X\to Y$ is contained 
in 
the linear
system
\begin{equation}
\label{eq:gip:branch}
\left|4M+2(2-d_1-d_2-d_3)F\right|,
\end{equation}
where $F$ is a fiber of the projection $\pi\colon Y=\PP_{\PP^1}(\EEE)\to \PP^1$.

There are only the cases in Table~\ref{table-gip}.
All these cases do occur.
\end{teo}

\begin{table}[ht]\small
\caption{Hyperelliptic Fano threefolds with 
$\uprho(X)>1$}
\label{table-gip}
\begin{center}\def\arraystretch{1.3}
\begin{tabular}{|c|c|c|p{40mm}|p{40mm}|p{40mm}|}
\hline
$g$& $d_i$&$\uprho$ &\heading{$Y$} & \heading{$B$} & \heading{$X$}
\\\hline
$4$& $(1,1,1)$&$2$ &$\PP^1\times \PP^2$
with the Segre embedding to $\PP^5$
& a divisor of bidegree $(2,4)$& see Exercise~\ref{zad:gip} at the end of this 
section
\\\hline
$5$& $(2,1,1)$&$2$& the blowup of $\PP^3$ along a line $l\subset \PP^3$& the 
proper
transform of a nonsingular quartic
$\overline{B}\subset\PP^3$ meeting the line $l$ transversally&
the blowup of a del Pezzo threefold $V$ of degree~$2$
along a smooth curve,~which is an intersection of two divisors 
$H_1,H_2\in\mo|-\frac12 K_V|$
\\\hline
$7$& $(2,2,2)$&$9$& $\PP^1\times \PP^2$, the embedding to $\PP^8$
is given by a divisor of bidegree $(2,1)$ &a divisor of bidegree $(0,4)$& 
$F\times
\PP^1$, $F$ is a del Pezzo surface of degree~$2$
\\\hline
\end{tabular}
\end{center}
\end{table}

\begin{proofsk}
Note that $\iota(X)=1$ according to Proposition~\ref{del-pezzo-projective}.
Let 
$$
\Phi\colon X \longrightarrow Y=Y_{g-1} \subset\PP^{g+1}
$$
be the anticanonical map, where $Y=\Phi(X)$.
By our assumption and Corollary~\ref{corollary:hyperelliptic} the 
morphism $\Phi$ is a double cover,
and $Y=Y_{g-1}\subset
\PP^{g+1}$ is a nonsingular variety of minimal degree.
Let $B\subset Y$ be the branch divisor of the morphism $\Phi$.

It is obvious that for $g\le 3$ the variety $Y$ is isomorphic to $\PP^3$ or a 
quadric in $\PP^4$.
In these cases
the variety $X$ is a complete intersection in a weighted projective space
(as in the cases~\ref{thm:hyper}\ref{thm:hyper:g2} and~\ref{thm:hyper:g3}). 
But then $\uprho(X)=1$. This contradicts our assumptions.
Therefore, $g\ge 4$.
Then according to Proposition~\ref{varieties-minimal-degree} the variety $Y$
has the form $\PP_{\PP^1}(\EEE)$, where $\EEE=\pi_*\OOO_Y(1)$ is a locally 
free sheaf (vector bundle) of rank $3$. Furthermore, 
the class of a hyperplane section
coincides with the class of the tautological line bundle $\OOO_{\PP(\EEE)}(1)$.
Let us decompose $\EEE$ in a sum of line bundles: $\EEE=\oplus 
\OOO_{\PP^1}(d_i)$.
Since vector bundle $\EEE$ is ample, we have $d_i>0$ for any $i$.
Write $\EEE$ in the form~\eqref{eq:gip:E}.
Then
$$
g+2=\hr^0\bigl(\PP_{\PP^1}(\EEE),\,\OOO_{\PP_{\PP^1}(\EEE)}(1)\bigr)=
\hr^0(\PP^1,\,\EEE)=\sum d_i+3.
$$
From this we obtain~\eqref{eq:gip:d}. Furthermore, by the Hurwitz formula
$$
K_X=\Phi^* \left(K_Y+\frac12 B\right).
$$
Let $M$ be the tautological divisor on $Y=\PP_{\PP^1}(\EEE)$.
Since $M$ is also a hyperplane section of the variety $Y=Y_{g-1} 
\subset\PP^{g+1}$, 
we have $-K_X=\Phi^* M$.
The formula for the canonical divisor
~\ref{prop:scrolls}\ref{scroll:canonical-divisor} in our situation we can
write in the form
$$
K_Y=-3M+\left(-2+\sum d_i\right)F.
$$
Thus
$$
-\Phi^* M=K_X=\Phi^* \left(-3M+\left(-2+\sum d_i\right)F+\frac12 B\right).
$$
Therefore,
$$
B\sim 4M+2 \left(2-\sum d_i\right)F.
$$
This proves~\eqref{eq:gip:branch}.

Furthermore, the composition
$$
\lambda\colon X\xarr{\Phi} Y \xarr{\pi} \PP^1
$$
is a contraction 
(not necessarily extremal). By the adjunction formula a
general fiber $S$ of $\lambda$ is a del Pezzo surface.
The restriction of the anticanonical map $\Phi$ to $S$ is a finite
morphism $S\to \Phi(S)=\PP^2$ of degree~$2$ onto a fiber of
$\PP_{\PP^1}(\EEE)$ over $\PP^1$. It is clear that this morphism branched 
over the quartic
$\Phi(S)\cap B$.
Therefore, $S$ is a del Pezzo surface
of degree~$2$.
It follows immediately
from~\eqref{eq:gip:d} and~\eqref{eq:gip:branch} that in the cases
$g=4$ or $g=5$ for $X$ only the
corresponding cases in Table~\ref{table-gip} occur.
Further we assume that $g>5$.

For any $k$ there is a natural isomorphism
$$
H^0\bigl(Y,\,\OOO_Y(M+kF)\bigr)\simeq 
H^0(\PP^1,\,\pi_*\OOO_Y(M+kF))=H^0(\PP^1,\,\EEE(k)).
$$
Since $\deg\EEE=\sum d_i=g-1\ge 5$, by the Riemann--Roch Theorem
$$
\dim H^0\bigl(Y,\,\OOO_Y(M-2F)\bigr)>1.
$$
Therefore, $\dim |M-2F|>0$ and so
there is a decomposition
$$
-K_X\sim 2S+D,
$$
where $D=\Phi^*(M-2F)$ is an effective divisor such that
$\dim|D|>0$.

By the Cone Theorem on $X$ there exists
an extremal ray $\rR$ such that $S\cdot\rR>0$.
Let $\varphi$
be its contraction, let
$\upmu(\rR)$
be its length, and let $\ell$ be the corresponding minimal rational
curve (see~\eqref{eq:def:length}).
Since $S\cdot \rR>0$, none of the curves on $X$ is simultaneously 
contracted by both morphisms $\varphi$ and $\lambda$, hence none of the fibers 
of
$\varphi$ is two-dimensional. Therefore, $\rR$ is of type \type{C} or 
\type{B_1}.

Assume that $\rR$ is of type \type{B_1} (i.e. $\varphi$ is the
blowup of a smooth curve on a smooth variety).
Let $E$ be the corresponding exceptional divisor. Since $E\cdot \ell=-1$ and
$$
1=-K_X\cdot \ell=2 S\cdot \ell+D\cdot \ell\ge 2+D\cdot \ell,
$$ 
we have $D\cdot \ell<0$ and $E$ is a fixed component of the linear system $|D|$,
i.e. we can write
$$
|D|=rE+|L|,
$$
where $r>0$ and $|L|$ is a linear system that does not have $E$ as its fixed
component.
Then, as above, we have
$$
1=-K_X\cdot \ell=-r+2 S\cdot \ell+L\cdot \ell.
$$
If $|L|$ is composed of fibers of $\lambda$, then $L\cdot \ell>0$ and $r>1$. In 
this
case $-K_S=-K_X|_S=rE|_S$ is divisible by $r>1$ in the group $\Pic(S)$. This is 
impossible
on a del Pezzo surface of degree~$2$. Thus the divisor $L$ meets a general
fiber $S$.
Since $|L|$ is a movable linear system, $E\cap S\neq \varnothing$, and
$(rE+L)|_S=-K_S$, we have
$$
2=K_S^2= r(-K_S)\cdot (E\cap S)+(-K_S)\cdot (L\cap S)\ge r+1\ge 3.
$$
The contradiction shows that $\rR$ is of type \type{C}.
Then
$$
\upmu(\rR)=-K_X\cdot \ell \ge 2 S\cdot \ell\ge 2.
$$
Therefore, $\rR$ is of type \type{C_2} (i.e. $\varphi$ is a
$\PP^1$-bundle) and $S\cdot \ell=1$. Thus $S$ and $S'$ are disjoint
sections. Since they are linearly equivalent, the variety
$X$ actually is a direct product of $\PP^1$ and a del Pezzo surface
of degree~$2$ (the case $g=7$). The theorem is proved.
\end{proofsk}

Later we will need the following fact, which is related to varieties more 
general than Fano threefolds.

\begin{prp}
\label{proposition:nef-big}
Let $X$ be a nonsingular three-dimensional projective variety such that
the anticanonical linear system $|{-}K_X|$ is big and base point free.
Let 
$$
\Phi=\Phi_{|{-}K_X|}\colon X \longrightarrow X_1\subset\PP^N
$$
be the anticanonical morphism, where $X_1$ is its image.
Then the variety $X_1$ is normal and one of the following holds:
\begin{enumerate}
\item
\label{proposition:nef-big:bir}
the morphism
$\Phi\colon X\to X_1$ is birational, has connected fibers, and the variety 
$X_1$ has
at worst canonical Gorenstein singularities;
\item
\label{proposition:nef-big:2}
the morphism $\Phi\colon X\to X_1$ is finite at a general point, has 
degree~$2$, and 
$X_1\subset
\PP^N$ is a variety of minimal degree.
\end{enumerate}
\end{prp}

\begin{proof}
As in the proof of Theorem~\ref{theorem-sections}, using 
Riemann--Roch formula and the Kawamata--Viehweg Vanishing Theorem, we write
$$
N=\dim|{-}K_X|=g+1,\quad\text{where}\quad (-K_X)^3=2g-2.
$$
It is clear that
$$
(\deg \Phi)\cdot (\deg X_1)=(-K_X)^3=2g-2.
$$
From this, according to Proposition~\ref{proposition:minimal-degree}, we have 
$\deg
\Phi\le 2$. If $\deg \Phi=2$, then $X_1\subset\PP^{g+1}$ is a variety
of minimal degree $g-1$ (in particular, it is normal). We obtain the case 
\ref{proposition:nef-big:2}.

Assume that $\deg \Phi=1$. Then the morphism $\Phi$ is birational and $\deg
X_1=2g-2$.
Consider a general member $S\in|{-}K_X|$. By Bertini's theorem $S$ is a 
nonsingular
surface. It is not difficult to check that $S$ is a \K3 surface. The divisor
$A=-K_X|_S$ is nef and big. Moreover, the surface $S$
with quasi-polarization given by the divisor $-K_X|_S$ is not hyperelliptic. 
Thus, according to
Theorem~\ref{theorem-A-Saint-Donat}\ref{th:K3:surjective}, the algebra
$\R(S,\, -K_X|_S)$ is generated by its component of degree~$1$. Again the 
Kawamata--Viehweg Vanishing Theorem the restriction map
$$
H^0(X,\,\OOO_X(-nK_X)) \longrightarrow H^0(S,\,\OOO_X(-nK_X))
$$
is surjective for $n\ge 1$. Hence, by Theorem~\ref{th-hyp-sect}, the algebra
$\R(X,-K_X)$ is also generated by its component of degree~$1$.

Consider the Stein factorization 
$$
\Phi\colon X \xarr{\psi_0} X_0 \xarr{\psi_1} X_1.
$$
Let $H_1$ be a hyperplane section of $X_1$ and let $H_0:= \psi_1^* H_1$. Then
$H_1$ is a Cartier divisor and $\psi_0^*H_0=-K_X$. There is an isomorphism
graded algebras
$$
\R(X,\, -K_X)\simeq \R(X_0,\, H_0).
$$
and the algebra $\R(X_0,\, H_0)$ is also generated by its component of 
degree~$1$.
Therefore, the divisor $H_0$ is very ample (see
Proposition~\ref{proposition:R:very-ample}) and $\psi_1$ is an isomorphism.
Thus the variety $X_1$ is normal and $-K_X=\Phi^*H_1$.
Apply $\Phi_*$ to this equality: $K_{X_1}=-H_1$. In particular, this shows that 
$K_{X_1}$ 
is a Cartier divisor and $K_X=\Phi^*K_{X_1}$. This means that the singularities 
of
$X_1$ 
are 
canonical Gorenstein. We obtain the case~\ref{proposition:nef-big:bir}.
\end{proof}

In conclusion we note that hyperelliptic Fano threefolds with 
canonical Gorenstein singularities were classified in the work
\cite{Przhiyalkovskij-Cheltsov-Shramov-2005en}.

\begin{zadachi}
\eitem
\label{zad:g=3}
Show that the hyperelliptic varieties
of type~\ref{thm:hyper}\ref{thm:hyper:g3}
and three-dimensional quartics in $\PP^4$ are contained in one irreducible 
family.
Moreover, the varieties~\ref{thm:hyper}\ref{thm:hyper:g3}
are degenerations of quartics.

\eitem
\label{zad:cone1}
Let $X=X_d\subset\PP^{d+2}$ be a three-dimensional cone over a rational normal
curve
$C_d\subset\PP^d$ (with the vertex a line). Prove that the 
Weil divisor class group
$\Cl(X)$ is generated by the class of plane $\Pi\subset X$.
How is embedded in it to the Picard group?

\eitem
\label{zad:cone2}
Let $S=S_d\subset\PP^{d+1}$ be a nonsingular surface
of minimal degree that is isomorphic to $\FF_e$ and let
$X=X_d\subset\PP^{d+2}$ be a cone over 
$S$. What are generators of the Weil divisor class group
$\Cl(X)$?
How is embedded in it the Picard group?

\eitem
Describe all extremal rays for varieties from Table~\ref{table-gip}.

\eitem
\label{zad:gip}
Prove that the variety of genus 4 from the table in Theorem~\ref{th:hyp:any}
admits an embedding to the weighted projective space $\PP(1^{12},2)$ and
is realized there as the intersection of the cone over $\PP^1\times \PP^2$
embedded to 
$\PP^{11}$ by the linear system of bidegree $(1,2)$ and a quadric.
\end{zadachi}

\newpage\section{Trigonal Fano threefolds}
\label{sec6}


\subsection{Anticanonical model}
Recall that the classical Noether-Enriques-Petri Theorem asserts that a 
smooth canonical curve $C\subset \PP^{g-1}$ of genus $g\ge 4$ is an intersection 
of quadrics except for the cases where $C$ is either a trigonal curve (i.e. it 
possesses a one-dimensional 
linear series~$\mathfrak 
g^1_3$) or a curve of genus $6$ that is isomorphic to a plane quintic (see 
Theorem 
\ref{th:NEP:}).
Similar to the hyperelliptic property for canonical curves, the
trigonal one has analogs for \K3 surfaces
(Theorem~\ref{th:K3:}) and
for Fano threefolds:

\begin{dfn}
A Fano threefold $X$ is called \textit{trigonal}, if the linear
system $|{-}K_X|$ is very ample,
but the anticanonical image $X=X_{2g-2}\subset\PP^{g+1}$ is not a
an intersection of quadrics.
\end{dfn}

\begin{lem}
\label{trigonal}
Let $X$ be a Fano threefold $X$ such that the divisor
$-K_X$ is very ample. Then the following
conditions are equivalent:
\begin{enumerate}
\item
\label{trigonal-3}
$X$ is trigonal,
\item
\label{trigonal-2}
there exists an element $S\in|{-}K_X|$ that is a trigonal \K3 surface
\textup(with respect to the polarization
$\OOO_S(-K_X)$\textup),
\item
\label{trigonal-1}
there exist elements $H_1,\, H_2\in|{-}K_X|$ whose intersection is a
nonsingular trigonal canonical curve.
\end{enumerate}
Moreover, the equivalence remains valid if in condition \ref{trigonal-2} we replace ``there exists an element \dots that is'' with ``any nonsingular element \dots is'' and in condition 
\ref{trigonal-1} we replace ``there exist elements \dots whose intersection is” with “for any distinct elements \dots their intersection, if nonsingular, is.''
\end{lem}

\begin{proof}
Implications~\ref{trigonal-1} $\Longrightarrow$~\ref{trigonal-2}
$\Longrightarrow$~\ref{trigonal-3} obviously
follow from definitions.
Let us prove~\ref{trigonal-3} $\Longrightarrow$~\ref{trigonal-2}.
If the \K3 surface $S$ is not trigonal, then according to
Theorem~\ref{th:K3:} the
homogeneous ideal $I_S\subset \mathrm{S}^* H^0(S, \OOO_S(-K_X))$ is generated by elements of 
degree
$2$. By Theorem~\ref{th-hyp-sect} the same holds and for the ideal
$I_X\subset \mathrm{S}^* H^0(X, \OOO_X(-K_X))$. The proof~\ref{trigonal-2}
$\Longrightarrow$~\ref{trigonal-1} is
similar and uses the Noether-Enriques-Petri Theorem on canonical
curves~\ref{th:NEP:}.
\end{proof}

\begin{cor}
\label{corollary:trigonal}
Let $X=X_{2g-2}\subset\PP^{g+1}$ be a trigonal 
Fano threefold. The intersection of all quadrics in $\PP^{g+1}$ passing through 
$X$ is 
a four-dimensional
variety
$W=W_{g-2}\subset\PP^{g+1}$ of minimal degree \textup(see
Appendix~\ref{sect:VMD}).
If $g\ge 5$, then this variety is nonsingular.
\end{cor}

\begin{proof}
Fix a point $P\in W$ and consider a sufficiently general linear
subspace in $\PP^g\subset\PP^{g+1}$ such that $S:= X\cap\PP^g$ is a 
nonsingular
surface (of type \K3).
By Theorem~\ref{th-hyp-sect}\ref{th-hyp-sect-3}
any quadric $Q\subset\PP^g$ passing through~$S$ can be extended to
a quadric
$Q'\subset\PP^{g+1}$ passing through $X$.
Hence, $V:= W\cap\PP^g$ is an intersection of quadrics in $\PP^g$ containing $S$ 
and by virtue of
Theorem~\ref{th:K3:}\ref{th:K3:2}
the variety~$V$ is irreducible and is a variety of minimal degree
$V=V_{g-2}\subset\PP^g$.
Therefore, $W$ is also irreducible and is a variety of minimal
degree $W=W_{g-2}\subset\PP^{g+1}$.
If $g\ge 5$, then according to Theorem~\ref{th:K3:}\ref{th:K3:3},
the variety~$V$ is nonsingular at the point $P$ and the same holds for $W$.
Therefore, in this case $W$ is nonsingular everywhere and the corollary
is proved completely.
\end{proof}
\subsection{Case $\uprho(X)=1$}
\begin{teo}
\label{thm:trig:1}
Let $X=X_{2g-2}\subset\PP^{g+1}$ be a 
Fano threefold with $\iota(X)=1$ and $\uprho(X)=1$ such that the linear system 
$|{-}K_X|$ is
very ample. Assume that $X$ is trigonal.
Then $X$ is one of the following:
\begin{enumerate}
\item
$\g(X)=3$, $X=X_4\subset\PP^4$ is a quartic;
\item
$\g(X)=4$, $X=X_{2\cdot 3}\subset\PP^5$ is 
the intersection of a quadric and cubic.
\end{enumerate}
\end{teo}

\begin{proof}
Let $W\subset\PP^{g+1}$ is the intersection of all quadrics passing through 
$X$.
According to Corollary~\ref{corollary:trigonal}, \ $W=W_{g-2}\subset\PP^{g+1}$
is a variety of minimal degree, where $g=\g(X)$. Moreover,
if $g\ge 5$, then it is nonsingular. 
It follows from Proposition~\ref{varieties-minimal-degree} that,
under the assumption $g\ge 5$, the variety $W$ has a structure of $\PP^3$-bundle 
over $\PP^1$.
Hence, in this case there is a surjective morphism from $X$ to a curve.
This contradicts our assumption $\uprho(X)=1$.

Therefore, $g=3$ or $4$. In the case $g=3$ the variety $X=X_4\subset\PP^4$
is a hypersurface. In the case $g=4$ the variety $X=X_6\subset\PP^5$ is of 
degree $6$ and is contained in a quadric $W=W_2\subset\PP^4$.
By Theorem~\ref{th:K3:}\ref{th:K3:23} there exists an irreducible
cubic passing through $X$.
\end{proof}

\subsection{General case}
In the case $\uprho(X)>1$, trigonal Fano threefolds (as in
hyperelliptic case) also can be
classified by using Corollary~\ref{corollary:trigonal}
and Proposition~\ref{prop:scrolls}. We reproduce this classification
below. The proof is similar to the proof of Theorem~\ref{th:hyp:any} and
is based on the method of extremal rays~\cite{Mori-Mukai-1981-82}.

\begin{teo}[{\cite[Ch.~2, Theorem 3.4]{Isk:anti-e}}]
\label{th:trig:any}
Let $X=X_{2g-2}\subset\PP^{g+1}$ be a Fano threefold
of genus $g=\g(X)$ with $\uprho(X)>1$ such that the linear system $|{-}K_X|$
very ample. Assume that $X$ is trigonal.
Let $W\subset\PP^{g+1}$ be the intersection of all quadrics passing through 
$X$.
Then $W=W_{g-2}\subset\PP^{g+1}$ is a nonsingular rational scroll 
$W=\PP_{\PP^1}(\EEE)$,
where
\begin{equation}
\label{eq:trig:E}
\EEE=
\OOO_{\PP^1} (d_1)\oplus\cdots\oplus
\OOO_{\PP^1} (d_4), \quad d_i>0,
\end{equation}
with the embedding
$$
\Phi=\Phi_{|\OOO_{\PP(\EEE)}(1)|}\colon
W\hookrightarrow \PP^{g+1}
$$
given by the tautological linear system $|M|=|\OOO_{\PP(\EEE)}(1)|$ and 
\begin{equation}
\label{eq:trig:p}
g=\sum d_i+2.
\end{equation}
Here $X\subset W$ is a divisor contained in the linear system
\begin{equation}
\label{eq:trig:lin-syst}
\left|3M\otimes \left(2-\sum d_i\right)F\right|,
\end{equation}
where $F$ is a fiber of the projection $\pi\colon W=\PP_{\PP^1}(\EEE)\to \PP^1$.

There are only the cases in Table~\ref{table-trig}.
All these cases do occur.
\end{teo}

\begin{table}[ht]\small
\caption{Trigonal Fano threefolds with 
$\uprho(X)>1$}
\label{table-trig}
\begin{center}\def\arraystretch{1.3}
\begin{tabular}{|c|c|c|p{60mm}|p{60mm}|}
\hline
$g$& $d_i$&$\uprho$ &\heading{$W$} & \heading{$X$}
\\\hline
$6$& $(1,1,1,1)$ & $2$ &$\PP^1\times \PP^3\subset\PP^7$,
with the Segre embedding
& a divisor of bidegree $(1,3)$
\\\hline
$7$& $(2,1,1,1)$& $2$ & the blowup of $\PP^4$ along a plane&
the blowup of $V_3\subset\PP^4$ along a plane cubic
\\\hline
$8$
& $(2,2,1,1)$& $3$& small resolution of a quadric $Q\subset\PP^5$ of corank~$2$
(blowup of 
$\PP^3\subset Q$)
& the proper transform of a divisor in $Q$ of type $(3,2)$
\\\hline
$10$& $(2,2,2,2)$&$8$& $\PP^1\times \PP^3\subset\PP^{11}$, the embedding
is given by a divisor of bidegree $(2,1)$ &$X\simeq F\times \PP^1$, where $F$ is 
a
del Pezzo surface of degree $3$
\\\hline
\end{tabular}
\end{center}
\end{table}

\begin{rem}
The Fano threefold of genus $8$ from Table~\ref{table-trig} was erroneously 
omitted by Iskovskikh~\cite{Isk:anti-e} and it appears for the first time
in the paper by Mori and 
Mukai~\cite[Table~3,~\textnumero~2]{Mori-Mukai-1981-82}.
This was pointed out to the author by I.~Cheltsov (see
\cite[Theorem~1.6,~$\mathrm T_{11}$]{Przhiyalkovskij-Cheltsov-Shramov-2005en},
\cite[Lemma~8.2]{Cheltsov-Shramov:2008-e}).
\end{rem}

\begin{proofsk}
Let 
$$
X=X_{2g-2}\subset\PP^{g+1}
$$
be a trigonal Fano threefold of genus $g$ with $\uprho(X)>1$
and let
$$
W=W_{g-2} \subset\PP^{g+1}
$$
be the four-dimensional variety being the intersection of quadrics passing 
through
$X_{2g-2}$
(see Corollary~\ref{corollary:trigonal}).

If $g\le 4$, then
$W$ is isomorphic to $\PP^4$ or a quadric in $\PP^5$. But then
$X$ is a complete intersection as in Theorem~\ref{thm:trig:1} and, therefore, 
$\uprho(X)=1$.
This contradicts our assumptions.
Hence, we may assume that $g\ge 5$. Then again by Corollary 
\ref{corollary:trigonal} the variety $W_{g-2}$ is nonsingular and has the form
$W=\PP_{\PP^1}(\EEE)$.

Let $M$ be the tautological divisor on $W=\PP_{\PP^1}(\EEE)$ and let $F$ be 
a fiber of the projection $\PP_{\PP^1}(\EEE)\to \PP^1$.
Recall that according to Proposition~\ref{varieties-minimal-degree}
we may assume that $M$ is a hyperplane section of $W$, i.e. $-K_X=M|_X$.
Since the vector bundle $\EEE$ is ample, we have $d_i\ge 1$ for all~$i$.
Besides, $\deg W=\sum d_i=g-2$. From this one obtains~\eqref{eq:trig:p}.
The formula~\ref{prop:scrolls}\ref{scroll:canonical-divisor} for the canonical 
divisor in our situation we can write in the form
$$
K_W=-4M+\left(-2+\sum d_i\right)F.
$$
By the adjunction formula
$$
-M|_X=K_X=(K_W+X)|_X=\big(-4M+\left(-2+\sum d_i\right)F+X\big)|_X.
$$
Hence, there is the following linear equivalence divisors on $W$:
$$
X\sim 3M+\left(2-\sum d_i\right)F.
$$
This proves~\eqref{eq:trig:lin-syst}.

Furthermore, we write $\EEE$ in the form~\eqref{eq:trig:E}.
Then
$$
g+2=\hr^0\bigl(\PP_{\PP^1}(\EEE),\, \OOO_{\PP_{\PP^1}(\EEE)}(1)\bigr)=
\hr^0(\PP^1,\, \EEE)=\sum d_i+4.
$$
It immediately follows that for $g=6$ and $g=7$ for $X$ there are only the
corresponding cases in Table~\ref{table-trig}.
Further we assume that $g>7$.

The projection $\pi\colon W=\PP_{\PP^1}(\EEE)\to \PP^1$ induces
a contraction $\lambda\colon X\to\PP^1$.
Here a general fiber $S$ of the morphism $\lambda$ is a del Pezzo
surface, which is anticanonically embedded to a fiber of $\PP_{\PP^1}(\EEE)$ 
over 
$\PP^1$. Hence, $S$ is a cubic surface in $\PP^3$.

For any $k$ there is a natural isomorphism
\begin{equation}
\label{eq:SpL}
H^0\bigl(W,\,\OOO_W(M+kF)\bigr)\simeq H^0(\PP^1,\,\EEE(k)).
\end{equation}
Since $\deg\EEE=\sum d_i=g-2\ge 6$, by the Riemann--Roch Theorem
$$
\dim H^0\bigl(W,\,\OOO_W(M-2F)\bigr)>1.
$$
Therefore, $\dim \mo|-K_X-2S|>0$, i.e.
there is a decomposition
$$
-K_X\sim 2S+D,
$$
where $D$ is an effective divisor such that
$\dim|D|>0$.

By the Cone Theorem there exists an extremal ray
$\rR$ on the variety $X$ such that $S\cdot \rR>0$.
Let $\varphi$ be its contraction, let $\upmu(\rR)$
be its length, and let $\ell$ is the corresponding minimal rational
curve (see~\eqref{eq:def:length}). None of the curves on $X$ is simultaneously 
contracted by both morphisms $\varphi$ and $\lambda$, hence none of the fibers 
of
$\varphi$ is two-dimensional. Therefore, $\rR$ is of type \type{C} or 
\type{B_1}. If $\rR$ is of type \type{C}, then
$$
\upmu(\rR)=-K_X\cdot \ell \ge 2 S\cdot \ell\ge 2.
$$
Hence, $\rR$ is of type \type{C_2} (i.e. $\varphi$ is a
$\PP^1$-bundle). Since $S$ and~$S'$ are disjoint sections that are
linearly equivalent, the variety
$X$ actually is a direct product of $\PP^1$ and a del Pezzo surface
of degree~$3$ (the case $g=10$ in Table~\ref{table-trig}).

Further we assume that $\rR$ is of type \type{B_1} (i.e. $\varphi$ is the
blowup of a smooth curve on a smooth variety).
Let $E$ be the corresponding exceptional divisor. Since $E\cdot \ell=-1$ and
$$
1=-K_X\cdot \ell=2 S\cdot \ell+D\cdot \ell\ge 2+D\cdot \ell,
$$ we have $D\cdot \ell<0$ and $E$ is a fixed component of the linear system 
$|D|$,
i.e. we can write
$$
|D|=rE+|L|,
$$
where $r>0$ and $|L|$ is a linear system that does not have $E$ as its fixed
component.
Then, as above, we have
\begin{equation}
\label{eq:gip:m}
1=-K_X\cdot \ell=-r+2 S\cdot \ell+L\cdot \ell.
\end{equation}
If $|L|$ is composed of fibers of $\lambda$, then $L\cdot \ell>0$ and $r>1$. In 
this case $-K_S=-K_X|_S=rE|_S$ is divisible by $r>1$ in $\Pic(S)$ but this is 
impossible on a cubic surface. Hence, the divisor~$L$ meets the general fiber 
$S$.
Since $|L|$ is a movable linear system, $E\cap S\neq \varnothing$ and
the divisor
$$
(rE+L)|_S=-K_S
$$
is a hyperplane section of the cubic surface $S=S_3\subset\PP^3$,
the restriction $|L|\big|_S$ must be a pencil of conics, $r=1$, and $\ell':= 
E\cap 
S$ is 
a line on $S$. It follows from~\eqref{eq:gip:m} that
$$
S\cdot \ell=1,\qquad L\cdot \ell=0.
$$
Since $E\cdot \ell'=(\ell')^2_S=-1$, the curve $\ell'$ is also extremal
and~$E$ admits a contraction in another
direction.
In particular, $E\simeq \PP^1\times \PP^1$ and $\uprho(X)>2$.
Since $(L-S)\cdot \ell<0$, we have $H^0(X,\,\OOO_X(L-S))=0$. Then from
the exact sequence
$$
0 \longrightarrow \OOO_X(L-S)\longrightarrow\OOO_X(L)\longrightarrow\OOO_S(L) \longrightarrow 0
$$
we obtain that $\dim|D|=\dim|L|=1$. Therefore, $g=8$.
Then there are only two possibilities:
$$
(d_1,\dots,d_4)=(3,1,1,1)\quad \text{or}\quad (2,2,1,1).
$$
In the former case from~\eqref{eq:SpL} we obtain that the linear system
$|M-3F|$ is non-empty and the linear system
$|M-F|$ is base point free. Take any elements
$G\in|M-F|$ and $G'\in|M-3F|$. Since $X$ cannot be a component of
$G'$, we have $X\cdot G'\cdot G^2\ge 0$. On the other hand, it can be 
immediately computed:
$$
X\cdot G'\cdot G^2= (3M-4F)\cdot (M-3F)\cdot (M-F)^2= 3M^4-19M^3\cdot F=-1.
$$
The contradiction shows that $(d_1,\dots,d_4)\neq (3,1,1,1)$.
We obtain the case $g=8$ in Table~\ref{table-trig}.
The theorem is proved.
\end{proofsk}

\subsection{Summary of the results}
Let us summarize the results obtained in three previous lectures.

\begin{teo}[V.A. Iskovskikh~\cite{Isk:anti-e}]

\label{theorem-1.1}
Let $X$ be a Fano threefold of genus $g$ with $\iota(X)=1$,
$\uprho(X)=1$. Then
\begin{enumerate}
\item
the linear system $|{-}K_X|$ is base point free;
\item
the linear system $|{-}K_X|$ is not very ample only in the following cases:
\begin{enumerate}
\item
$g=2$ and $\Phi_{|{-}K_X|}\colon X\to\PP^3$ is a double cover branched over 
a surface of degree $6$;
\item
$g=3$ and $\Phi_{|{-}K_X|}\colon X\to Q\subset\PP^4$ is double
cover of a quadric branched over a surface of degree $8$.
\end{enumerate}
\item
Let the linear system $|{-}K_X|$ define an embedding $\Phi_{|{-}K_X|}\colon
X\hookrightarrow\PP^{g+1}$. Then $g\ge 3$ and the image $X=\Phi_{|{-}K_X|}(X)$ 
is of
degree $2g-2$. Moreover,
\begin{enumerate}
\item
if $g=3$, then the variety $X=X_4\subset\PP^4$ is a quartic,
\item
if $g=4$, then the variety $X=X_6\subset\PP^4$ is a complete intersection of a
quadric and
cubic.
\end{enumerate}
\item

\label{theorem-1.1:int:q}
For $g\ge 5$ the image $X_{2g-2}=\Phi_{|{-}K_X|}(X)$ is a
an intersection of quadrics \textup(and $X=X_8\subset\PP^6$ is a
complete intersection of quadrics for $g=5$\textup).
\end{enumerate}
\end{teo}

Singular trigonal Fano threefolds (with canonical Gorenstein singularities) are 
classified in~\cite{Przhiyalkovskij-Cheltsov-Shramov-2005en}.

\begin{zadachi}
\eitem
\label{problem:Fano:g-ge5}
Let $X$ be a Fano threefold of genus $\g(X)\le 5$
(not necessarily with $\uprho(X)=1$). Assume that the
anticanonical divisor $-K_X$ is very ample.
Without using Theorem~\ref{th:trig:any} prove that the variety $X$ is
one of the following: a quartic in $\PP^4$, an intersection of a quadric and 
cubic in $\PP^5$
or
an intersection of three quadrics in $\PP^6$.

\eitem
Without using Theorem~\ref{thm:trig:1} and~\ref{th:trig:any},
explore on trigonality Fano threefolds of index $>1$.

\eitem
\label{problem-ruled-surface}
Prove that a nonsingular ruled surface $F_d\subset\PP^{d+1}$ is 
an intersection of $\frac12 d(d-1)$ linearly independent quadrics.

\eitem
\label{problem:canonical-curve}
Prove that there are exactly $\frac12 (g-2)(g-3)$ linearly independent quadrics 
passing through a canonical curve $C_{2g-2}\subset\PP^{g-1}$.
\hint{Use Max Noether's Theorem~\ref{th:M-Noether}.}

\eitem
Prove that a nonsingular divisor of bidegree $(1,3)$ on $\PP^1\times \PP^3$
(the case $g=6$ from Table~\ref{table-trig}) is the blowup of $\PP^3$ along
a nonsingular intersection of two cubic surfaces.

\eitem
Describe extremal rays on Fano threefolds from Table~\ref{table-trig}.
Prove that for $g=8$ no extremal birational contractions produces a Fano 
threefold.

\eitem
Prove that the variety described in the case $g=8$ of the table in
Theorem~\ref{th:trig:any} is indeed a trigonal Fano threefolds.
\end{zadachi}

\newpage\section{Elementary transformations (Sarkisov links)}
\label{section:sl}

\subsection{Set-up}
\label{construction:sl} 

Let $X$ be a Fano threefold such that $\uprho(X)=1$.
Let $\sigma\colon \widetilde{X}\to X$ be the blowup with a nonsingular center
$C$, where $C$ is either a point or an irreducible nonsingular curve of degree 
$d$ and genus $\g(C)$.
Assume that the divisor $-K_{\widetilde{X}}$ is nef and big. Then
by the Base Point Free Theorem~\ref{th:mmp:bpf} for some 
$n>0$
the linear system $\mo|-nK_{\widetilde{X}}|$ defines a morphism
$$
\Phi_{|-nK_{\widetilde{X}}|}\colon \widetilde{X}\longrightarrow X_{\bullet}\subset\PP^N,
$$
which must be generically finite onto its image (because
$-K_{\widetilde{X}}$ is big).
Consider the Stein factorization 
$$
\Phi_{|-nK_{\widetilde{X}}|} \colon \widetilde{X} \xarr{\theta} X_0 
\xarr{\gamma}
X_{\bullet}\subset\PP^N.
$$
According to Corollary~\ref{cor:mmp:bpf} there exists an ample Cartier divisor
$A$ on
$X_0$ such that
\begin{equation}
\label{eq:SL1}
-K_{\widetilde{X}}=\theta^*A.
\end{equation}
On the complement to the subset $\theta(\Exc(\theta))$ of codimension $\ge 2$ 
the divisors
$-K_{X_0}$ and $A$ coincide. Hence, they coincide everywhere:
$$
-K_{\widetilde{X}}=\theta^*(-K_{X_0}).
$$
From this we obtain (cf. Proposition~\ref{proposition:nef-big})

\begin{cor}
\label{cor:SL:singX0c}
The variety $X_0$ has at worst canonical Gorenstein singularities.
\end{cor}
Further we will assume that
\begin{enumerate}
\item[$(\star)$]
the morphism $\Phi_{|-nK_{\widetilde{X}}|}$ does not contract any divisors.
\end{enumerate}
Thus the exceptional locus $\Exc(\theta)$ of the morphism $\theta$ is either 
empty or a union of a finite number of curves.
Then Corollary~\ref{cor:SL:singX0c} can be strengthened:

\begin{cor}
\label{cor:SL:singX0}
The variety $X_0$ has at worst terminal Gorenstein singularities.
If $\Exc(\theta)=\varnothing$, then $X_0$ is nonsingular.
\end{cor}

\subsection{Constructing the link}
\label{construction:sl-cases} 

Thus one of the following cases occurs:
\begin{enumerate}
\item
\label{case:nef-big}
the divisor $-K_{\widetilde{X}}$ is not ample and then $\theta$ is a small
$K_{\widetilde{X}}$-trivial contraction,
\item
\label{case:ample}
the divisor $-K_{\widetilde{X}}$ is ample (i.e. $\widetilde{X}$ is a Fano 
threefold) and
then $\theta$ is an isomorphism.
\end{enumerate}

In the case~\ref{case:nef-big},
according to Theorem~\ref{flop:theorem}, there exists a flop
$$
\vcenter{
\xymatrix@R=2em{
&\widetilde{X}\ar[dr]_{\theta}\ar@{-->}[rr]^\chi && 
\overline{X}\ar[dl]^{\bar{\theta}}
\\
&&X_0&&
}}
$$
where the variety $\overline{X}$ is nonsingular.
Recall also that the variety $X_0$ has only isolated
hypersurface cDV singularities and, according to~\eqref{eq:SL1},
\begin{equation}
\label{eq:SL2}
-K_{\overline{X}}=\bar{\theta}^*A.
\end{equation}

Furthermore, since $\uprho(\overline{X})=2$, the Mori cone
$\NE(\overline{X})$ is an angle on the plane and has exactly two extremal rays
(see Fig.~\ref{figure:Mori-cone}).
\begin{figure}[!t]\small
\begin{center}
\begin{tikzpicture}
\draw[white, fill=gray!20,fill opacity=0.5] (4,-1)--(0,0)--(4,3) plot
[smooth,tension=0.5,white] coordinates{(4,-1) (4.9,0) (4.9,1.9) (4,3)};
\draw[black, thick] (4,-1)--(0,0)--(4,3);
\fill (0,0) circle(1.5pt);
\node[above,yshift=1] () at (2,-1.2) {$\rR_\varphi$};
\node[above,yshift=1] () at (2,1.5) {$\rR_{\bar\theta}$};
\node () at (4.0,0.9) {the case~\ref{case:nef-big}};
\draw[black,dashed, line width=1.9pt] (-0.7,-0.5)--(4,3);
\node[right] () at (-0.9,-0.8) {$K_X$};
\draw [-stealth,black,line width=1.9pt](0,0)--(-0.6,1);
\node[right] () at (-0.5,1.2) {$\mathbf{+}$};
\end{tikzpicture}
\hspace{2.7em}
\begin{tikzpicture}
\draw[white, fill=gray!20,fill opacity=0.5] (4,-1)--(0,0)--(4,3) plot
[smooth,tension=0.5,white] coordinates{(4,-1) (4.9,0) (4.9,1.9) (4,3)};
\draw[black, thick] (4,-1)--(0,0)--(4,3);
\fill (0,0) circle(1.5pt);
\node[above,yshift=1] () at (2,-1.2) {$\rR_\varphi$};
\node[above,yshift=1] () at (2,1.5) {$\rR_{\sigma}$};
\node () at (4.0,0.2) {the case~\ref{case:ample}};
\draw[black,dashed, line width=1.6pt] (0,-1.2)--(0,2.9);
\node[right] () at (0,2.9) {$K_{\overline{X}}$};
\draw [-stealth,black,line width=1.9pt](0,0)--(1,0);
\node[right] () at (1,0) {$\mathbf{-}$};
\end{tikzpicture}
\end{center}
\caption{The Mori cone $\NE(\overline{X})$}
\label{figure:Mori-cone}
\end{figure}
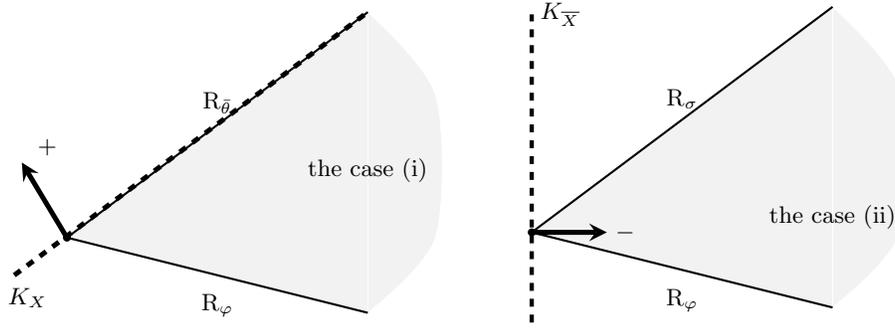
One of these rays $\rR_{\bar{\theta}}$ is generated by the curves in the fibers
of 
the morphism
$\bar{\theta}$ and so $K_{\overline{X}}\cdot \rR_{\bar{\theta}}=0$.
Another ray $\rR_{\varphi}$ must have negative intersection with the
canonical class
(otherwise we obtain a contradiction with~\eqref{eq:SL2}).
According to general theory (see Theorem~\ref{th:mmp:contraction}), there exists
a contraction $\varphi\colon \overline{X}\to Y$ of the
ray $\rR_{\varphi}$. We obtain the diagram
\begin{equation}
\label{diagram}
\vcenter{
\xymatrix@R=1.8em{
&\widetilde{X}\ar[dr]_{\theta}\ar@/_4.0pt/[ddl]_{\sigma}\ar@{-->}[rr]^\chi &&
\overline{X}\ar[dl]^{\bar{\theta}}\ar@/^4.0pt/[ddr]^{\varphi}
\\
&&X_0&&
\\
X\ar@{-->}[rrrr]^{\Psi}&&&&Y
}}
\end{equation}

In the case~\ref{case:ample} we put $\widetilde{X}=\overline{X}$.
Then again the cone $\NE(\overline{X})$ has exactly two extremal rays
$\rR_{\sigma}$ and
$\rR_{\varphi}$, where $\rR_{\sigma}$ is generated by the curves in the fibers 
of
$\sigma$.
In this case both rays are $K_{\widetilde{X}}$-negative and there exists a
contraction 
$\varphi\colon \overline{X}\to Y$ of the
ray $\rR_{\varphi}$. Then we will consider a (degenerate) 
diagram~\eqref{diagram}:
\begin{equation}
\label{diagram-v}
\vcenter{
\xymatrix@R=2em{
& X=X_0=\overline{X}
\ar@/_4.0pt/[dl]_<(+.4){\sigma}
\ar@/^4.0pt/[dr]^<(+.4){\varphi}
&
\\
X\ar@{-->}[rr]^{\Psi}&&Y
}}
\end{equation}
We have already considered such a construction in the proof of
Theorem~\ref{th:d5} (see~\eqref{diagram-i=2:projection}).

The diagram~\eqref{diagram} is a particular case of a \textit{Sarkisov link}.
It is uniquely restored from its center~$C\subset X$.

\begin{lem}
\label{lemma:i-form}
Let $\widetilde{D}_1$, $\widetilde{D}_2$ be any divisors on $\widetilde{X}$ 
and
let $\overline{D}_1$, $\overline{D}_2$ be their proper transforms on 
$\overline{X}$.
Then
$$
(-K_{\widetilde{X}})\cdot \widetilde{D}_1\cdot \widetilde{D}_2
=(-K_{\overline{X}})\cdot \overline{D}_1\cdot \overline{D}_2, \quad
(-K_{\widetilde{X}})^2\cdot \widetilde{D}_i=(-K_{\overline{X}})^2\cdot 
\overline{D}_i.
$$
In particular, $(-K_{\widetilde{X}})^3=(-K_{\overline{X}})^3$.
\end{lem}

\begin{proof}
From~\eqref{eq:SL1},~\eqref{eq:SL2} and the projection formula we obtain
$$
(-K_{\widetilde{X}})^2\cdot \widetilde{D}_i=\theta^*A^2 \cdot \widetilde{D}_i=
A^2\cdot \theta_*D_i=A^2\cdot \bar{\theta}_*\overline{D}_i
=\bar{\theta}^*A^2\cdot \overline{D}_i=
(-K_{\overline{X}})^2\cdot \overline{D}_i.
$$
This implies the second relation. The first one is proved similarly.
\end{proof}

\begin{notation}
Below by $E$ we denote the exceptional divisor $\sigma^{-1}(C)$ of the blowup
$\sigma$ and by $\overline{E}$ we denote its proper transform on $\overline{X}$.
\end{notation}

\begin{rem}
\label{rem:defect}
The abelian groups
$$
\Pic(\widetilde{X})=\ZZ\cdot K_{\widetilde{X}}\oplus \ZZ\cdot E,
\qquad
\Pic(\overline{X})=\ZZ\cdot K_{\overline{X}}\oplus \ZZ\cdot \overline{E}
$$
are identified by means of the map $\chi_*$:
$$
\chi_*\colon \Pic(\widetilde{X})\xarr{\simeq } \Pic(\overline{X}),\quad
\chi_* K_{\widetilde{X}}=K_{\overline{X}},\quad \chi_*E=\overline{E}.
$$
Thus on the lattice
$$
\ZZ^2=\Pic(\widetilde{X})=\Pic(\overline{X})
$$
there are two trilinear intersection forms (induced from
$\Pic(\widetilde{X})$
and $\Pic(\overline{X})$). Lemma~\ref{lemma:i-form} shows that the values 
of these
forms coincide, if one of the factors is the canonical class. However in general
case, these forms differ. It is clear that they coincide exactly when
$E^3=\overline{E}^3$. The number
$$
\deff(\Psi):= E^3-\overline{E}^3
$$
we call \textit{the defect} of the link~\eqref{diagram}. Thus the defect
is equal to~0, if 
$\chi$ is an isomorphism.
According to the proposition below, the converse is also true.
\end{rem}

\begin{prp}
\label{prop:defect}
$\deff(\Psi)\ge 0$.
If $\deff(\Psi)=0$, then $\chi$ is an isomorphism.
\end{prp}

Actually, it can be shown that the defect
$\deff(\Psi)$ is equal to the number of flopping curves, if this number is 
counted
``in the right way'' (i.e. with multiplicities, 
see~\cite[definition~5.3]{Reid:MM}).

\begin{proof}
Assume that $\chi$ is not an isomorphism.
Since $E$ is a non-zero effective divisor, the divisor $-E$ cannot be nef. 
Since 
$E\cdot 
\rR_{\sigma}=0$, we have $E\cdot \rR_{\theta}>0$,
i.e. $E$ is strictly positive on the components of the exceptional locus
$\Exc(\theta)$.
Thus for some $a,\, b\gg 0$ the divisor $D=aE+b(-K_{\widetilde{X}})$ is very 
ample on
$\widetilde{X}$.

According to Corollary~\ref{cor:flop:DC} the divisor $\overline{E}$ 
is strictly negative on the components of the exceptional locus 
$\Exc(\bar{\theta})$.
Let $\Gamma_1, \dots,\Gamma_N\subset\Exc(\bar{\theta})$ be its irreducible 
components. Then, according to the above,
\begin{equation}
\label{eq:SL:KGamma}
K_{\overline{X}}\cdot \Gamma_i=0,\qquad \overline{E}\cdot \Gamma_i<0,\quad 
\forall i.
\end{equation}
Thus for the proper transform $\overline{D}\subset \overline{X}$ of
the divisor $D$ we have $\overline{D}\cdot \Gamma_i<0$.
Take general elements $D_1,\, D_2\in|D|$ and let 
$\overline{D}_1,\,\overline{D}_2\in
|\overline{D}|$ be their proper transforms. Then $L:= D_1\cap D_2$ is a
nonsingular
curve that does not meet the exceptional locus $\Exc(\theta)$ and the 
intersection
$\overline{D}_1\cap\overline{D}_2$ can be written in the form
$$
\overline{D}_1\cap\overline{D}_2=\bar{L}+\sum \gamma_i\Gamma_i,
$$
where $\bar{L}\subset\overline{X}$ is the proper transform of $L$ and
$\gamma_i\ge 0$.
Since $\overline{D}_j\cdot \Gamma_i<0$, 
we have
$\Gamma_i\subset\Bs|\overline{D}|$
and so $\gamma_i>0$.
Since $\overline{D}\cdot \bar{L}=D\cdot L=D^3$, we have
$$
\overline{D}^3=\overline{D}\cdot\left (\bar{L}+\sum \gamma_i\Gamma_i\right)=
D^3+\overline{D}\cdot \sum \gamma_i\Gamma_i.
$$
From this by Lemma~\ref{lemma:i-form} we have
$$
a^3\deff(\Psi)=a^3(E^3-\overline{E}^3)=D^3-\overline{D}^3
=-a\overline{E}\cdot \sum \gamma_i\Gamma_i\ge a\sum \gamma_i\ge aN.
$$
This proves our assertion.
\end{proof}

\subsection{The left hand side of the link} 

The contraction $\varphi\colon \overline{X}\to Y$ belongs to one of types given
in the table from Theorem~\ref{class:ext-rays}.
Moreover, in our case there are certain restrictions on the variety $Y$:
\begin{itemize}
\item
if $\varphi$ is of type \type{D}, then $Y\simeq\PP^1$
according to Corollary~\ref{cor:Omega};
\item
if $\varphi$ is of type \type{C}, then $Y\simeq\PP^2$
according to Corollary~\ref{cor:rat-surf} (because $\uprho(Y)=1$);
\item
if $\varphi$ is of type \type{B_1} or \type{B_2}, then $Y$ is a
nonsingular Fano threefold (because $\uprho(Y)=1$ and $\mo|-nK_X|\neq 
\varnothing$ for some $n>0$);
\item
if $\varphi$ is of type \type{B_3}, \type{B_4} or \type{B_5}, then, as 
above, $Y$ is a singular Fano threefold.
\end{itemize}
In all cases $\Pic(Y)\simeq \ZZ$.

\begin{lem}
\label{lemma:barE-varphi}
The morphism $\varphi$ cannot contract the divisor $\overline{E}$.
\end{lem}

\begin{proof}
Indeed, if $\varphi$ is a contraction of fiber type, then $\overline{E}$
cannot be composed of fibers because
$\dim|r\overline{E}|=\dim|rE|=0$ for any $r>0$ and any effective divisor on
$Y\simeq \PP^1$ or $\PP^2$ is movable.
Assume that the contraction $\varphi$ is birational and $\overline{E}$ is its
exceptional divisor.
If $\chi$ is not an isomorphism, then the divisor $E$ must be
$\theta$-ample
(otherwise the divisor $-E$ is nef because it is positive on the fibers
$\sigma$). Therefore, the divisor
$\overline{E}$ is $\bar{\theta}$-negative (see
Corollary~\ref{cor:flop:DC}) and so it is positive on the fibers of
$\varphi$. This contradicts its exceptionality.
Further, if $\chi$ is an isomorphism, then according to our choice of the 
extremal
ray $\rR_{\varphi}$, the surface $E=\overline{E}$ is ruled and has two
structures of contraction to a curve.
This is possible only if $E\simeq \PP^1\times \PP^1$. In this case
the divisor $-K_{\widetilde{X}}+E$ is trivial on both extremal rays of the cone
$\NE(\widetilde{X})$.
Since $\uprho(\widetilde{X})=2$, this gives us 
$-K_{\widetilde{X}}+E\approxident
0$, a contradiction.
\end{proof}

We need some relations on intersection numbers of divisors
in the lattice $\Pic(\overline{X})$.
Let us introduce an additional notation.
Let $M$ be the ample generator of the group $\Pic(Y)\simeq \ZZ$ and let
$\overline{M}:= \varphi^*M$.
If $\varphi$ is a birational contraction, then by $\overline{F}$ we denote
its exceptional divisor.
For uniformity, we also put $\overline{F}=\overline{M}$, if the
contraction $\varphi$
is of fiber type (\type{D} or \type{C}).

\begin{lem}
\label{lemma:sl-comput}
Depending on the type of the contraction $\varphi\colon \overline{X}\to Y$
there are the following relations in $\Pic(\overline{X})$.

\textup{Type \type{D}:}
\begin{equation}
\label{equations-D}
\overline{M}^3=\overline{M}^2\cdot (-K_{\overline{X}})=0,\qquad
\overline{M}\cdot (-K_{\overline{X}})^2=K_{\overline{X}_\eta}^2,
\end{equation}
where $\overline{X}_\eta$ is the general fiber of $\varphi$.
\textup(Recall that $1\le K_{\overline{X}_\eta}^2\le 6$ in the case \type{D_1},
$K_{\overline{X}_\eta}^2=8$ in the case \type{D_2} and 
$K_{\overline{X}_\eta}^2=9$
in the case \type{D_3}\textup).

\smallskip
\textup{Type \type{C}:}
\begin{equation}
\label{equations-C}
\overline{M}^3=0,\qquad \overline{M}^2\cdot (-K_{\overline{X}})=2,\qquad
\overline{M}\cdot (-K_{\overline{X}})^2=12-\deg\Delta,
\end{equation}
where $\Delta\subset\PP^2$ is the discriminant curve of the conic bundle
$\varphi\colon \overline{X}\to Y=\PP^2$ \textup(see e.~g. 
\cite[\S3]{P:rat-cb:e}\textup).
In the case where \type{C_2} we put $\deg \Delta=0$.

\smallskip
\textup{Types \type{B_2}--\type{B_5}:}
\begin{equation}
\label{equations-E25}
\overline{F}^2\cdot (-K_{\overline{X}})=-2,\qquad
\overline{F}\cdot (-K_{\overline{X}})^2=k,\qquad
\overline{F}^3=4/k,
\end{equation}
where
$$
k=
\begin{cases}
4& \text{type \type{B_2}}
\\
2& \text{types \type{B_3} and \type{B_4}}
\\
1& \text{type \type{B_5}}
\end{cases}
$$

\textup{Type \type{B_1}:}
\begin{equation}
\label{equations-E1}
\begin{gathered}
(-K_{\overline{X}}+\overline{F})^2\cdot (-K_{\overline{X}})=(-K_Y)^3,\qquad
(-K_{\overline{X}}+\overline{F})\cdot \overline{F}\cdot 
(-K_{\overline{X}})=-K_Y\cdot Z,
\\[3pt]
\overline{F}^2\cdot (-K_{\overline{X}})=2\g(Z)-2,
\end{gathered}
\end{equation}
where $Z:= \varphi (\overline{F})$ is the center of the blowup
$\varphi$ \textup(nonsingular curve\textup).
\end{lem}

\begin{proof}
Let us prove, for example,~\eqref{equations-C}.
The relation $\overline{M}^3=0$ follows from the projection formula.
To prove the equality $\overline{M}^2\cdot (-K_{\overline{X}})=2$ we have to 
note that
a general geometric fiber $\overline{X}_y=\varphi^{-1}(y)$, $y\in Y$ is a
nonsingular
conic and
$$
K_{\overline{X}}\cdot {\overline{X}_y}=\deg K_{\overline{X}_y}=-2
$$
by the adjunction formula. Finally, for a general line $M\subset Y=\PP^2$
the surface $\overline{M}=\varphi^{-1}(M)$ is nonsingular and the projection
$\varphi|_{\overline{M}}:\overline{M}\to M$ is a conics bundle
whose number of degenerate fibers equals $M\cdot
\Delta$.
Thus the surface $\overline{M}$ is rational and $\uprho(\overline{M})=2+M\cdot 
\Delta$.
By Noether's formula $K_{\overline{M}}^2=8-M\cdot \Delta$
and by the adjunction formula 
$K_{\overline{M}}=(K_{\overline{X}}+\overline{M})_{\overline{M}}$.
Therefore,
$$
\overline{M}\cdot (-K_{\overline{X}})^2=\overline{M}\cdot 
(K_{\overline{X}}+\overline{M})^2- 2
K_{\overline{X}}\cdot \overline{M}^2
=12-\deg\Delta. \qedhere
$$
\end{proof}

\begin{lem}
\label{lemma-effective-divisors}
The cone of effective divisors $\Eff(\overline{X})$ on $\overline{X}$ 
is generated by the classes of divisors~$\overline{E}$ and $\overline{F}$.
\end{lem}

\begin{proof}
Note that the classes of divisors $\overline{E}$ and $\overline{F}$ are 
linearly independent in the lattice
$\Pic(\overline{X})$ by Lemma~\ref{lemma:barE-varphi}.
Let $\overline{D}$ be a prime divisor on $\overline{X}$ that is distinct from
$\overline{E}$ and
$\overline{F}$.
For some $a,\, b\in \QQ$ we have
\begin{equation}
\label{eq:DFE}
\overline{D}\qq a \overline{E}+b\overline{F}.
\end{equation}
It is sufficient to prove that $a,\, b \ge 0$.
Applying $\chi^{-1}_*$ to the relation~\eqref{eq:DFE} we obtain
$$
\widetilde{D}\qq a E+b\widetilde{F},
$$
where $\widetilde{D}$ and $\widetilde{F}$ are proper transforms on 
$\widetilde{X}$
of divisors
$\overline{E}$ and $\overline{F}$,
respectively. Therefore, $ \sigma_*\widetilde{D}\sim b\sigma_* \widetilde{F}$
and so $b>0$.
If the morphism $\varphi$ is birational, then similarly 
$\varphi_*\overline{D}\sim a
\varphi_*\overline{E}$ and $a>0$. And if $\varphi$ is not birational,
then intersecting both parts of~\eqref{eq:DFE} with a curve~$\bar{J}$ lying in a 
fiber of
$\varphi$
we obtain $0<\overline{D}\cdot \bar{J}=a \overline{E}\cdot \bar{J}$. From this 
again we obtain
$a>0$.
\end{proof}

\begin{lem}
\label{lemma-F}
Let $\iota(X)=1$. Put $\iota:= \iota(Y)$ and $\upmu:= \upmu(\rR_{\varphi})$
\textup(the length of the extremal ray see~\eqref{eq:def:length}).
Then for some positive integer $a$ we have
$$
\overline{M} \sim a(-K_{\overline{X}})-\upmu \overline{E}.
$$
Furthermore, let the contraction $\varphi$ be birational. Then
$$
\overline{F}\sim
\begin{cases}
(\iota a-1)(-K_{\overline{X}})-\iota \overline{E}, &
\text{if $\rR_{\varphi}$ is of type \type{B_1}, \type{B_3}, \type{B_4},}
\\
\frac 12 (\iota a-1)(-K_{\overline{X}})-\iota \overline{E}, &
\text{if $\rR_{\varphi}$ is of type \type{B_2},}
\\
2 (\iota a-1)(-K_{\overline{X}})-2\iota \overline{E},&
\text{if $\rR_{\varphi}$ is of type \type{B_5}.}
\end{cases}
$$
\end{lem}

\begin{proof}
Since $\iota(X)=1$, we have $\Pic(X)=\ZZ\cdot K_X$ and
$$
\Pic(\overline{X})=\ZZ\cdot K_{\overline{X}}\oplus \ZZ\cdot \overline{E}.
$$
Thus we can write $\overline{M}\sim a(-K_{\overline{X}})-b \overline{E}$ for 
some
$a,b\in\ZZ$.
Since $K_{\overline{X}}\cdot \rR_\varphi<0$, $\overline{E}\cdot \rR_\varphi<0$, 
and
$\overline{M}\cdot \rR_\varphi=0$, we have $ab>0$. If the
inequalities $a<0$ and $b<0$ hold, then the divisor
$$
\overline{E} \approxident
\frac ab(-K_{\overline{X}})-\frac 1b \overline{M}
$$
is ample, which is impossible. Therefore, $a>0$ and $b>0$.

To prove the first assertion it remains to note that the exact
sequence~\eqref{equation-exact-sequence} and the definition of the length
$\upmu(\rR_{\varphi})$ imply that the divisors $\overline{M}$ and 
$-K_{\overline{X}}$ generate in the lattice
$\Pic(\overline{X})$
a sublattice of index $\upmu$.

To prove the remaining assertions we write
$$
K_{\overline{X}}=\varphi^* K_Y+\alpha \overline{F}=-\iota \overline{M}+\alpha 
\overline{F},
$$
where $\alpha$ is the discrepancy. For it we have
$\alpha=1$ in the cases \type{B_1}, \type{B_3} and \type{B_4},
$\alpha=2$ in the case \type{B_2}, and
$\alpha=1/2$ in the case \type{B_5} (see Remark~\ref{rem:Mdiscr}).
Then
$$
\overline{F}\sim \frac {\iota}{\alpha} \overline{M}+\frac 
1{\alpha}K_{\overline{X}} \sim
\frac {a\,\iota-1}{\alpha}(-K_{\overline{X}})-\frac {\upmu\,
\iota}{\alpha}\overline{E}.\qedhere
$$
\end{proof}

\begin{exa}
\label{example:sl-cubic}
Let $X=X_3\subset\PP^4$ be a nonsingular cubic hypersurface
and let $C$ be a point on $X$ such that there are only a finite number
of lines passing through it (i.e. $C$ is not an is a generalized Eckardt point
\cite[Lemma~2.2.6, Remark~2.2.7]{KPS:Hilb}).
Let $\psi\colon X\dashrightarrow \PP^3$ be the projection from $C$.
It is clear that this map is generically finite and has degree~$2$.
Let $\sigma\colon \widetilde{X}\to X$ be the blowup of $C$.
Then the composition $\psi\comp \sigma$ is a regular morphism that is again 
generically
finite.
Therefore, its Stein factorization has the form
$$
\psi\comp \sigma\colon \widetilde{X} \xarr{\theta} X_0 \xarr{\gamma} \PP^3,
$$
where $\gamma$ is a finite morphism of degree~$2$ and the morphism $\theta$ 
contracts the proper
transforms $\widetilde{L}_1,\dots,\widetilde{L}_k\subset\widetilde{X}$ of lines
$L_1,\dots,L_k\subset X$ passing through $C$
and is an isomorphism outside them.
The Galois involution of the cover $X_0/\PP^3$ induces a map $\chi \colon
\widetilde{X}\dashrightarrow \widetilde{X}$
which is regular outside $\widetilde{L}_1,\dots,\widetilde{L}_k$.
We obtain the diagram~\eqref{diagram}, where $\overline{X}\simeq \widetilde{X}$ 
and $Y\simeq X$.
Note that the maps $\chi$ and $X\dashrightarrow Y=X$ are not
identical.
In particular, $\varphi$ contracts the proper transforms of tangent
hyperplane section of~$X$ at the point~$C$.
\end{exa}

\begin{rem}
\label{remark:sl:involution}
The situation with a \textit{symmetric} link \textup(such as in Example
\ref{example:sl-cubic}) occurs quite often.
Indeed, assume that in the conditions of~\ref{construction:sl} we have
$n=1$ and the variety $X_\bullet$ is normal.
According to Proposition~\ref{proposition:nef-big}, the morphism $\gamma\colon 
X_0\to
X_\bullet$ is either an isomorphism or finite of degree~$2$. Assume that
$\gamma\colon X_0\to X_\bullet$ is double cover. This obviously holds, for 
example,
if $(-K_{\widetilde{X}})^3=2$ (see Exercise~\ref{zad:g=2} in Section 
\ref{sec4}).
Then on $X_0$ there exists the Galois involution
$\tau\colon X_0\to X_0$ which can be extended to a birational involution 
$\chi\colon
\widetilde{X}\to \widetilde{X}$. As in the example above we obtain
the diagram~\eqref{diagram}, where $\overline{X}\simeq \widetilde{X}$ and 
$Y\simeq X$. Note
that the action of $\tau^*$ on the group $\Cl(X_0)\simeq \Pic(\widetilde{X})$ 
non-trivial.
Therefore,
$$
\rk \Cl(X_\bullet)=\rk \Cl(X_0)^{\langle\tau^* \rangle}=1.
$$
\end{rem}

\begin{zadachi}
\eitem
\label{ex:sl:V4}
Let $X=X_{2\cdot 2}\subset\PP^5$ be a nonsingular intersection of two 
quadrics,
let $l\subset X$ be a line, and let $\sigma \colon\widetilde{X}\to X$ be its
blowup.
Prove that the divisor $-K_{\overline{X}}$ is ample and there exists (the 
simplest) 
Sarkisov link~\eqref{diagram-v}, where $\varphi\colon \overline{X}\to Y=\PP^3$ 
is 
the blowup of a
smooth curve $Z\subset\PP^3$ of degree~$5$ and genus~$2$.
In particular, the variety $X$ is rational.
\hint{Similar to the proof of Theorem~\ref{th:d5}.}

\eitem
Use result of the previous exercise and prove that the lines on 
$X$
meeting given general line $l\subset X=X_{2\cdot 2}\subset\PP^5$ and 
distinct from
$l$, are parametrized by a curve $C_l$ of genus~$2$. Show that the isomorphism 
class of
this curve does not depend on choice of $l$.
Prove that the Hilbert scheme $\Lines(X)$ of lines on $X$ is an abelian
surface, the Jacobian of $C_l$.
\hint{The surface $\Lines(X)$ is birational to the surface of bisecant lines 
to the curve
$Z\subset\PP^3$.}

\eitem
Let $X=X_3\subset\PP^4$ be a nonsingular cubic,
let $l\subset X$ be a line, and let $\sigma \colon\widetilde{X}\to X$ be its
blowup.
Prove that there exists a Sarkisov link
with center $l$ and describe the result.

\eitem
Let $X=X_d\subset\PP^{d+1}$ be a del Pezzo threefold of degree $d=4$ or
$5$ and
let $C\subset X$ be a nondegenerate conic.
Prove that there exists a Sarkisov link
with center $C$ and describe the result.

\eitem
Let $X=X_{2\cdot 3}\subset\PP^5$ be a nonsingular intersection of a quadric 
and cubic.
Let $l\subset X$ be a line. Assume that the line $l$ meets at most
a finite number of other lines in $X$ (in fact, this holds for any line
on $X=X_{2\cdot 3}\subset\PP^5$). Prove that there exists a (symmetric)
Sarkisov link~\eqref{diagram} with center $l$, where $Y$ is again an 
intersection of 
a quadric and a cubic, and $\varphi$ is the blowup of a line.

\eitem
Let $X=X_{2\cdot 2\cdot 2}\subset\PP^6$ be an intersection of three quadrics and
let $C\subset X$ be a sufficiently general conic
(in particular, its meets at most
a finite number of lines in $X$).
Prove that there exists a Sarkisov link
\eqref{diagram} with center $C$, where
$Y$ is again an intersection of three quadrics and 
$\varphi$ is the blowup of a conic.
\end{zadachi}

\newpage\section{Fano threefolds of Picard number~$1$.~I}
\label{sec:FanoI}


In this and subsequent lectures we discuss the classification of 
Fano threefolds
of index~$1$ and Picard number~$1$. 
We will follow the following plan.
The proof of the classification is given in 
Sections~\ref{sec:FanoI}--\ref{sec:FanoIII}.
The main classification Theorem~\ref{thm:index1:main} is based on
the existence of a line (Theorem~\ref{thm:line-exist})
which, in turn, is deduced from the existence of conics
(Corollary~\ref{cor:ex-line}),
and the last fact is proved in Corollary~\ref{cor:exist:line+conic}.
Sections~\ref{sect:Fano-IV} and~\ref{sect:Fano-V} dedicated mostly
to the proof of
the existence of Fano threefolds by using of inversions of
birational links.

\begin{teo}[V.A. Iskovskikh
\cite{Isk:Fano2e,Isk:anti-e}]
\label{thm:index1:main}
Let $X$ be a Fano threefold of index $\iota(X)=1$ and genus $g=\g(X)$
with $\Pic(X)=\ZZ$. Then $2\le g\le 12$
and $g\neq 11$. All these possibilities do occur.
For each $g\in\{2,\dots,10, 12\}$ all such Fano threefolds are
deformation equivalent.
The detailed classification is presented in Table~\ref{table-main}, 
page~\textup{\pageref{table-main}}.
\end{teo}

Ideas of the proof of this theorem boils down to the work of G. Fano
\cite{Fano1942}.
The first rigorous proof was given by Iskovskikh in 
\cite{Isk:Fano2e,Isk:anti-e}.
A modern presentation can be found in the works
\cite{Iskovskikh1990,Cutkosky1989,IP99}.
Here we will provide a complete proof. As it was said above, it relies on
the existence of a line,
i.e. a nonsingular rational curve $l\subset X$ such that $-K_X\cdot l=1$:

\begin{teo}[V.~V.~Shokurov~{Shokurov1980a}]
\label{thm:line-exist}
Let $X$ be a Fano threefold with $\uprho(X)=1$, $\iota(X)=1$ and
$\g(X)\ge 7$.
Then there exists
a line on $X$.
\end{teo}

Note that for $\g(X)\le 6$ a line exists as well, but this fact is not needed 
for our classification.

The proof of this theorem will be obtained as a consequence of detailed study of
some birational transformations and Sarkisov links in subsequent lectures (see 
Corollary~\ref{cor:exist:line+conic}).
Note that the Cone Theorem~\ref{cone-th} gives the existence of an
irreducible rational curve of degree $\le 4$, however it is not possible to 
improve 
this result by using only methods of deformation theory
\cite{Mori:3-folds}. Note also that in the original version of this theorem
\cite{Shokurov1980a} it was asserted that
a line exists under essentially weaker assumptions:
$X$ is a Fano threefold such that the anticanonical divisor
is very ample and cannot be decomposed in a sum of two ample divisors.
However the assertion~\ref{thm:line-exist} is sufficient for our purposes.

\subsection{Double projection from a line: main result}
The existence of a line allows us to apply the following construction.
Let $X=X_{2g-2}\subset\PP^{g+1}$ be an anticanonically embedded
Fano threefold and let $l\subset X$ be a line.
By $|H-2l|$ we will denote the linear system of hyperplane sections
singular along $l$. The rational map
defined by this linear system is called the \textit{double projection from a 
line}.

\begin{teo}[V. A. Iskovskikh]

\label{th:double-projection}
Let $X=X_{2g-2}\subset\PP^{g+1}$, $g\ge 7$ be an anticanonically embedded
Fano threefold with $\Pic(X)=\ZZ\cdot H$, where $H=-K_X$.
Assume that there exists a line on $X$.
Then the
double projection from $l$ maps $X$ to a subvariety $Y\subset\PP^{g-6}$:
\begin{equation}
\label{eq:double-projection}
\Psi=\Phi_{|H-2l|} \colon X \dashrightarrow Y\subset\PP^{g-6}.
\end{equation}
Here the map $\Psi$ and the blowup $\sigma\colon \widetilde{X}\to 
X$ of a line $l$
suit to the Sarkisov link~\eqref{diagram}.

Furthermore, let $E:= \sigma^{-1}(l)$ be the exceptional divisor, let
$\overline{E}\subset\overline{X}$ be the proper transforms of $E$,
let $M$ be the ample generator of $\Pic(Y)\simeq \ZZ$, and let
$\overline{M}=\varphi^*M$. Then $\overline{M}\sim 
-K_{\overline{X}}-\overline{E}$ and
for the extremal contraction $\varphi$
there is one of the following possibilities:

\begin{enumerate}
\item
\label{g=7}
$g=7$, $Y\simeq\PP^1$
and $\varphi\colon\overline{X}\to \PP^1$ is a del Pezzo fibration 
of degree $5$;
\item
\label{g=8}
$g=8$, $Y\simeq\PP^2$
and $\varphi\colon\overline{X}\to \PP^2$ is a standard conic bundle whose 
discriminant curve $\Delta\subset\PP^2$ has degree $5$;
\item
\label{g=9-12}
$g \in\{9,\, 10,\, 12\}$ and in these cases
the morphism $\varphi\colon\overline{X}\to Y$ is birational and
contracts the divisor
$$
\overline{F}\sim \delta(-K_{\overline{X}})-(\delta+1)\overline{E},
$$
where $\delta=3$,~$2$,\, $1$ for $g=9$,\, $10$,\, $12$, respectively. 
Moreover, $\varphi$ is the
blowup of a nonsingular irreducible curve $Z\subset Y$. Furthermore:
\begin{enumerate}
\item
\label{g=9}
if $g=9$, then $Y\simeq \PP^3$, $Z\subset\PP^3$ is a non-hyperelliptic curve
of genus $3$ and degree $7$ lying
on a cubic surface $F_3=\varphi (\overline{E})$;
\item
\label{g=10}
if $g=10$, then $Y=Q\subset\PP^4$ is a nonsingular quadric, $Z\subset Y$ is a
curve of genus
$2$ and degree $7$ lying on a surface $F_4=\varphi (\overline{E})$ of 
degree $4$;
\item
\label{g=12}
if $g=12$, then $Y=Y_5\subset\PP^6$ is a del Pezzo threefold
of degree $5$, $Z\subset Y$ is a normal rational curve of degree $5$
lying on hyperplane section $F_5=\varphi (\overline{E})$.
\end{enumerate}
The birational map
$\Psi^{-1}\colon Y \dashrightarrow X$ that is inverse to the double projection,
is given by the linear system $|(2\delta+1)M-2Z|$.
\end{enumerate}
In particular, $\g(X)\le 12$ and $\g(X)\neq 11$.
\end{teo}

\begin{rem}
\label{remark:jacobian}
In the conditions of~\ref{g=9-12} of Theorem~\ref{th:double-projection} the 
pair $(Y\supset Z)$ is uniquely determines $(X\supset l)$. In particular, the 
variety
$X$ can be uniquely reconstructed from $(Y\supset Z)$. On the other hand, 
different embeddings of
the same curve $Z$ in $Y$ can define non-isomorphic Fano threefolds $X$,
and for different choices of the line $l$ on $X$ we can obtain curves $Z\subset 
Y$ that are not projectively equivalent. However it should be noticed that the 
curve $Z$ up to an \emph{abstract isomorphism} is uniquely restored from the 
variety $X$. 
This follows from the Torelli Theorem for curves and the fact that the Jacobian 
of 
the curve $Z$, as a principally polarized abelian variety, is isomorphic to the
intermediate Jacobian of $X$ (see~\cite{Clemens-Griffiths}):
$$
\J(Z)\simeq \J(Z)\oplus \J(Y)\simeq\J(\overline{X})\simeq 
\J(\widetilde{X})\simeq
\J(X)\oplus \J(l) \simeq \J(X).
$$
\end{rem}

\begin{cor}
\label{Euler-numbers}
Let $X$ be a 
Fano threefold with $\iota(X)=1$ and $\uprho(X)=1$.
Assume that there exists a line on $X$.
Then for $\g(X)\ge 8$ the topological Euler number $\chit(X)$ takes the 
following values:
$$
\begin{array}{rrrrr}
\g(X): & 8 & 9 & 10 & 12 \\[5pt]
\chit(X): & -6 & -2 & 0 & 4
\end{array}
$$
\end{cor}

\begin{proof}
Consider, for example, the case $g\ge 9$.
The construction of flops (see the proof of Theorem~\ref{flop:theorem}) shows
that $\chit(\widetilde{X})=\chit(\overline{X})$.
Using the invariants of the curve $Z$ given 
in~\ref{th:double-projection}\ref{g=9-12}
we obtain
$$
\chit(Y)+\chit(Z)=\chit(\overline{X})=\chit(\widetilde{X})=\chit(X)+\chit(l).\
qedhere
$$
\end{proof}

\begin{cor}
\label{corollary:rationality}
Let $X$ be a 
Fano threefold with $\iota(X)=1$ and $\uprho(X)=1$.
Assume that there exists a line on $X$.
Then for $\g(X)=7$ and $\g(X)\ge 9$ the variety~$X$ is 
rational\footnote{Later, in 
Theorem~\ref{th:main:point}\ref{th:main:point:g=8} it will be
is proved that for $g=8$ the variety $X$ is birationally equivalent to
a nonsingular three-dimensional cubic in $\PP^4$, hence it is non-rational.}.
\end{cor}

Note however that the rationality of Fano threefolds of genus $\g(X)=7$ and 
$\g(X)\ge 9$ over non-closed fields is already nontrivial question 
(see\cite{KP-rF}).

\smallskip
\begin{proof}
For $g=9$, $10$ and $12$ the assertion obviously follows
from~\ref{th:double-projection}\ref{g=9-12} and Corollary~\ref{cor:V5-rat}.
For $g=7$ the rationality $\overline{X}$ follows from the fact that 
the fiber~$\overline{X}_\eta$ over the
general point $\eta\in\PP^1$ is a nonsingular del Pezzo surface of degree
5 over the field of rational functions $\CC(\PP^1)$ and such surface is 
$\CC(\PP^1)$-rational~\cite{Manin:book:74},~\cite{Prokhorov-re-rat-surf}.
\end{proof}

\begin{cor}
\label{corollary:line-flop}
In the conditions of Theorem~\ref{th:double-projection} the map $\chi$ from
the diagram~\eqref{diagram} is not an isomorphism and $\widetilde{X}$ is not a
Fano threefold.
\end{cor}

\begin{cor}
\label{cor:line-line}
Let $X$ be a Fano threefold with $\iota(X)=1$, $\uprho(X)=1$ and $\g(X)\ge 7$,
and let $l\subset X$ be a line with normal bundle
$$
\NNN_{l/X}\simeq \OOO_{\PP^1}\oplus \OOO_{\PP^1}(-1)
$$
\textup(see~\eqref{eq:line-N}). Then there exists a line on $X$
meeting $l$.
\end{cor}

Proofs Corollaries~\ref{corollary:line-flop} and~\ref{cor:line-line}
will be given at the end of this section.

To prove Theorem~\ref{th:double-projection} we need some information
on the Hilbert scheme of lines on Fano threefolds.

\subsection{Hilbert scheme of lines}

Assume that on the variety $X\subset\PP^N$ there exists at least one
line.
According to the theory of Hilbert schemes, the family of all lines on $X$
is parametrized by some projective
scheme $\Lines=\Lines (X)$; this~is the scheme parametrizing closed subschemes 
in
$X$ with Hilbert polynomial
$\upchi(t)=t+1$. Since $X\subset\PP^N$, the scheme $\Lines (X)$ is realized as 
a 
closed subscheme of the Grassmannian
$\Gr(2,N+1)$ of lines in $\PP^N$.

\begin{prp}
\label{prop:index1-lines}
Let $X=X_{2g-2}\subset\PP^{g+1}$ be an anticanonically embedded 
Fano threefold of index $\iota=1$
with $\uprho(X)=1$, $g\ge 5$ and let
$l\subset X$ be a line on $X$. The following assertions hold.
\begin{enumerate}
\item
\label{prop:index1-lines-hs}
There is a nonsingular hyperplane section
passing through $l$.

\item
\label{prop:index1-lines-nb}
For the normal bundle $\NNN_{l/X}$ there are only
the following two possibilities:
\begin{equation}
\label{eq:line-N}
\NNN_{l/X}\simeq
\begin{cases}
\OOO_{\PP^1} \oplus\OOO_{\PP^1}(-1) & \text{type $(0,-1)$,}
\\
\OOO_{\PP^1}(1) \oplus\OOO_{\PP^1}(-2) & \text{type $(1,-2)$.}
\end{cases}
\end{equation}

\item
\label{prop:index1-lines-sh}
Any component of the scheme $\Lines$ is one-dimensional.
If $[l]\in\Lines $ is the point corresponding to the line $l\subset X$, then
$[l]\in\Lines$ is nonsingular if and only if $\NNN_{l/X}$ is of type
$(0,-1)$.

\item
\label{prop:index1-lines-pt}
There are at most a finite number of lines passing through any point $P\in X$.
\end{enumerate}
\end{prp}

\begin{proof}
Recall that for $g\ge 5$ the anticanonical divisor is very ample and the image 
of
the corresponding embedding is an intersection of quadrics
(see Theorem~\ref{theorem-1.1}).

The assertion~\ref{prop:index1-lines-hs}
is proved similar to
Proposition~\ref{proposition-index=2-lines}\ref{proposition-index=2-lines-hs}
by the parameter count: the hyperplane sections passing through $l$ form a 
subspace of codimension~$2$ in $(\PP^{g+1})^*$ and the hyperplane sections that 
are singular at some point of $l$
form a subvariety of codimension $\ge 3$.

The assertion~\ref{prop:index1-lines-nb} is
also similar to
Proposition~\ref{proposition-index=2-lines}~\ref{proposition-index=2-lines-nb}:
we have a decomposition $\NNN_{l/X}=\OOO_{\PP^1}(a)+\OOO_{\PP^1}(b)$.
Further,
$\NNN_{l/H}$ is represented in the form of
extension
$$
0\longrightarrow\NNN_{l/H}\longrightarrow\NNN_{l/X}
\longrightarrow\NNN_{H/X}|_l\longrightarrow 0,
$$
where $\NNN_{H/X}|_l=\OOO_{\PP^1}(1)$.
Since $H$ is a \K3 surface and $l\subset H$ is a
nonsingular curve of genus zero, we have $\NNN_{l/H}=\OOO_{\PP^1}(-2)$.
Therefore, $a+b=-1$. Since $\hr^0(\NNN_{l/X})\le \hr^0(\NNN_{H/X}|_l)=2$, we 
have 
$a,\, b\le 1$.

Let us prove~\ref{prop:index1-lines-sh}.
According to the deformation theory (see~\cite[Theorems~2.8
and~2.15]{Kollar-1996-RC}), at
any point $[l]\in\Lines$ we have
$$
\dim_{[l]}\Lines \ge \hr^0 (\NNN_{l/X})-\hr^1 (\NNN_{l/X})=1.
$$
Moreover, if
$\NNN_{l/X}$ is of type $(0,-1)$, then $\hr^0 (\NNN_{l/X})=1$ and
$\hr^1(\NNN_{l/X})=0$. In this case
the scheme $\Lines$ is nonsingular and has
dimension~$1$ at the corresponding point $[l]$.
If $\NNN_{l/X}$ is of type $(1,-2)$, then $\hr^0(\NNN_{l/X})=2$ and
$\hr^1(\NNN_{l/X})=1$.
In this case
$\dim T_{[l],\Lines}=2 $ and there are two possibilities:
\begin{enumerate}
\item
\label{Hilbert:lines-1}
$\dim_{[l]} (\Lines)=1$, the scheme $\Lines$ is singular (or is not reduced) at 
the point $[l]$;
\item
\label{Hilbert:lines-2}
$\dim_{[l]} (\Lines)=2$ and the scheme $\Lines$ is nonsingular at the point 
$[l]$.
\end{enumerate}
Let $\mathrm{F}_1'\subset\Lines$ be an irreducible component.
Consider the universal family of lines $\Univ'\to \mathrm{F}_1'$ and the 
diagram
$$
\xymatrix@R=1.8em{
&\Univ'\ar[dl]_{p}\ar[dr]^{q}&
\\
\mathrm{F}_1'&&X
}
$$
where $p$ is the corresponding
$\PP^1$-bundle and $q_l\colon p^{-1}([l]) \to l\subset X$ is the tautological
map for any point $[l] \in\Lines$.
Assume that the condition~\ref{Hilbert:lines-2} holds.
Since $\dim(\mathrm{F}_1')=2$, we have
$\dim(\Univ')=3$ and $\dim q(\Univ')=2$ or $3$. If $\dim q(\Univ')=3$,
then $q(\Univ')=X$ and the map $q\colon \Univ'\to X$
is generically finite. But this is impossible. Indeed, for any
point $[l]\in\Lines$ we have the natural morphism of normal bundles
$$
\mathrm{d} q \colon \NNN_{l/\Univ'}\simeq\OOO_{\PP^1}\oplus \OOO_{\PP^1}
\longrightarrow\NNN_{l/X}\simeq \OOO_{\PP^1}(1)\oplus \OOO_{\PP^1}(-2).
$$
Since $\Hom (\OOO_{\PP^1}, \OOO_{\PP^1}(-2))=0$, the morphism
$\mathrm{d}q$ has at every point $[l]\in S_0$ 
at least one-dimensional kernel.
This means that a general line $l$ is contained in the branch locus of the 
morphism
$q\colon \Univ'\to X$,
which is impossible
(here the assumption
on the zero characteristic of the ground field is used).

Therefore, $q(\Univ')$ is a surface.
By the construction it contains a two-dimensional family of lines.
But among all irreducible projective surfaces this property is satisfied only 
for the
plane $\PP^2$ (see Exercise~\ref{zad:P2-2} at the end of this section).
By our condition the group $\Pic(X)$ is generated by the class of 
a hyperplane section. Thus, in particular, $X$ does not contain planes and,
therefore, the case~\ref{Hilbert:lines-2} in our situation does not occur.

Finally note that~\ref{prop:index1-lines-pt} is
completely similar to the proof of the 
assertion~\ref{proposition-index=2-lines-pt} in
Proposition~\ref{proposition-index=2-lines}.
\end{proof}

\subsection{Flopping contraction}

\begin{notation}
\label{notation:X2g-2}
Let $X=X_{2g-2}\subset\PP^{g+1}$ be an anticanonically embedded 
Fano threefold of index~$1$ and genus $g\ge 5$.
Put $H:= -K_X$. Thus $\Pic(X)=\ZZ\cdot H$.
Let $l\subset X$ be a line and let $\sigma\colon \widetilde{X}\to X$
be the blowup of $l$. Let $E:= \sigma^{-1}(l)$ be the exceptional divisor and 
let
$H^*:= \sigma^*H$, where $H\sim -K_X$ is the class of hyperplane section of 
$X$.
Note that $E\simeq \FF_1$ or $\FF_3$ (see~\eqref{eq:line-N}).
\end{notation}

First we study the map given by the linear system
$$
|{-}K_{\widetilde{X}}|=|H^*-E|.
$$

\begin{lem}
In the notation~\ref{notation:X2g-2} we have
\begin{equation}
\label{eq:intersection-numbers}
E^3=1,\qquad
-K_{\widetilde{X}}^3=2g-6,\qquad
(-K_{\widetilde{X}})^2\cdot E=3,\qquad
-K_{\widetilde{X}}\cdot E^2=-2.
\end{equation}
\end{lem}

\begin{proof}
The proof follows from the fact that $-K_{\widetilde{X}}=H^*-E$, 
$(H^*)^3=H^3=2g-2$
and the formulas~\eqref{eq:blowup-curve-intersection}.
\end{proof}

\begin{notation}

\label{notation:X2g-2a}
Let 
$$
\psi=\Phi_{|H-l|}\colon X \dashrightarrow X_{\bullet}\subset\PP^{g-1}
$$
be the projection from a line $l$, where $X_{\bullet}=\psi(X)$. Consider 
the induced morphism
$\phi=\Phi_{|{-}K_{\widetilde{X}}|} \colon \widetilde{X} \to 
X_{\bullet}\subset\PP^{g-1}$:
$$
\xymatrix@R=3em@C=5em{
&\widetilde{X}\ar[dl]_{\sigma}\ar[dr]^(.35){\phi}&
\\
X\ar@{-->}[rr]^(.4){\psi} && X_{\bullet}\subset\PP^{g-1}
}
$$
\end{notation}

\begin{lem}
\label{lemma:l-s:d-p}
In the notation~\ref{notation:X2g-2} and~\ref{notation:X2g-2a} 
the following assertions hold.
\begin{enumerate}
\item
\label{double-projection:E}
$$
-K_{\widetilde{X}}|_E=
\begin{cases}
\Sigma+2\Upsilon, & \text{if $\NNN_{l/X}$ is of type $(0,-1)$},
\\
\Sigma+3\Upsilon,& \text{if $\NNN_{l/X}$ is of type $(1,-2)$},
\end{cases}
$$
where $\Sigma$, $\Upsilon$ are classes of the exceptional
section and fiber of the rational ruled surface
$E=\FF_1$ or $\FF_3$, respectively.

\item
\label{double-projection:bir}
The morphism $\phi \colon \widetilde{X} \to X_{\bullet}\subset\PP^{g-1}$ is 
birational and
any its nontrivial fiber is either the proper transform of a line on $X$
meeting $l$ or, in the case where $\NNN_{l/X}$ is of type $(1,-2)$,
the exceptional section $\Sigma$ of the surface $E\simeq \FF_3$.
The image $X_{\bullet}=\phi ( \widetilde{X})$ is a normal variety and it is a 
Fano threefold with canonical
Gorenstein singularities.

\item
\label{dimension-double-projection}
$\hr^0 (\OOO_E (-K_{\widetilde{X}}))\le 5$, \qquad
$\hr^0 (\OOO_{\widetilde{X}} (-K_{\widetilde{X}}-E))\ge g-5$.
\end{enumerate}
\end{lem}

\begin{proof}
The assertion~\ref{double-projection:E}
can be easily deduced from the fact that
$$
(-K_{\overline{X}})^2\cdot \overline{E}=3\quad \text{and}\quad
-K_{\overline{X}})\cdot\Upsilon=-E\cdot\Upsilon=1.
$$

To prove~\ref{double-projection:bir}
recall that $X$ is an intersection of quadrics.
The projection $X \dashrightarrow \PP^{g-1}$ is induced by the linear projection
$\PP^{g+1} \dashrightarrow \PP^{g-1}$ from $l$ and the fibers of the last 
projection are two-dimensional planes $\Pi\subset\PP^{g+1}$ passing through $l$.
Every quadric $Q$ passing through $l$ either contains $\Pi$ or cut out on
$\Pi$, additionally to $l$, the residual line~$l_Q$.
Thus any fiber of $\psi$ is an intersection of lines $l_Q\subset\Pi$ taken by 
all
quadrics $Q\supset X$, $Q\not\supset \Pi$.
Therefore, any fiber of the projection $\psi$ is either a point or a line 
meeting $l$.
Hence, nontrivial fibers of the contraction $\phi$ outside $E$ are proper 
transforms of
such lines. In the case where $\NNN_{l/X}$ is of type $(0,-1)$, it follows
from~\ref{double-projection:E} that $E$ does not contain fibers of 
$\phi$ and in the case $(1,-2)$ a unique fiber of $\phi$ lying in $E$, is the 
exceptional section
$\Sigma$.
The last assertion in~\ref{double-projection:bir} is obtained from 
Proposition~\ref{proposition:nef-big}.

The first inequality in~\ref{dimension-double-projection}
follows from~\ref{double-projection:E}, the second~one follows from the exact 
sequence 
$$
0\longrightarrow H^0(\OOO_{\widetilde{X}}(-K_{\widetilde{X}}-E))\longrightarrow
H^0(\OOO_{\widetilde{X}}(-K_{\widetilde{X}}))\longrightarrow 
H^0(\OOO_E(-K_{\widetilde{X}})),
$$
because $\hr^0(\widetilde{X}, \OOO_{\widetilde{X}}(-K_{\widetilde{X}}))=g$.
\end{proof}


From now on we will assume that $g\ge 7$.
Then
$$
\dim|{-}K_{\widetilde{X}}-E|\ge 1.
$$

\begin{rem}
\label{line:linear-system}
Any element of the linear system
$$
\sigma_*\bigl(|{-}K_{\widetilde{X}}-E|\bigr)=\sigma_*\bigl(|H^*-2E|\bigr)\subset 
|H|
$$
is irreducible. Thus the only possibility for a fixed component of the linear 
system
$|{-}K_{\widetilde{X}}-E|$ is the exceptional divisor $E$.
Write
\begin{equation}
\label{eq:FanoI:-KEm}
|{-}K_{\widetilde{X}}-E|=(m-1)E+|{-}K_{\widetilde{X}}-mE|,\qquad m\ge 1,
\end{equation}
where $|{-}K_{\widetilde{X}}-mE|$
has no fixed components and
$$
\dim|{-}K_{\widetilde{X}}-mE|=\dim|{-}K_{\widetilde{X}}-E|> 0.
$$
Then a general element $D\in|{-}K_{\widetilde{X}}-mE|$ is irreducible.
Actually, it follows from our computations below that $m=1$.
\end{rem}

\begin{cor}
\label{cor:F1:E}
The morphism $\phi$ does not contract any divisors.
\end{cor}

\begin{proof}
Indeed, if $C$ is a curve in a fiber of the morphism $\phi$, then
$(H^*-E)\cdot C=K_{\widetilde{X}}\cdot C=0$. Since $H^*\cdot C>0$, we have 
$E\cdot
C>0$ and
$(-K_{\widetilde{X}}-mE)\cdot C<0$.
Thus the curve $C$ is contained in the base locus of the linear system
$|{-}K_{\widetilde{X}}-mE|$, where $m$ is given by the condition 
\eqref{eq:FanoI:-KEm}. Since $|{-}K_{\widetilde{X}}-mE|$
has no fixed components, the base locus $\Bs|{-}K_{\widetilde{X}}-mE|$
does not contain divisors.
\end{proof}

\begin{cor}
Any line $l$ on $X$ meets at most a finite number of other lines.
\end{cor}

\begin{proof}
Indeed, proper transforms of such lines are contracted by
the morphism $\phi$. If they form a family, then $\phi$
contracts a divisor.
\end{proof}

\subsection{Types of Sarkisov links}
\label{subs:FanoI:ex-diag} 

According to Lemma~\ref{lemma:l-s:d-p}\ref{double-projection:bir}
and Corollary~\ref{cor:F1:E} we can apply the construction of the
Sarkisov link
from Section~\ref{section:sl} with $C=l$.
If the divisor $-K_{\overline{X}}$ is ample, then there exists the second 
extremal contraction $\varphi\colon
\widetilde{X}\to Y$ on $\widetilde{X}$
(different from $\sigma$) and we obtain the diagram~\eqref{diagram-v}.
On the other hand, if the divisor $-K_{\overline{X}}$ is not ample, then by 
Theorem 
\ref{flop:theorem}
there exists a flop
$\chi\colon \widetilde{X} \dashrightarrow \overline{X}$ and we can include
our maps
in the diagram~\eqref{diagram}.

Let us find the possibilities for of the Mori contraction $\varphi\colon 
\overline{X}\to 
Y$ of the extremal ray $\rR_\varphi$.
Recall the notation.
By $\overline{E}$ we denote the proper transform of the divisor $E$ on
$\overline{X}$, by 
$M$ we denote the ample generator of the group $\Pic(Y)\simeq \ZZ$, and put
$\overline{M}:= \varphi^*M$.
If $\varphi$ is a birational contraction, then by $\overline{F}$ we denote
its exceptional divisor.
For uniformity, we also put $\overline{F}=\overline{M}$, if the 
contraction $\varphi$ is fiber type.

From Lemma~\ref{lemma-effective-divisors} we obtain the following

\begin{cor}
\label{cor:FanoI:F=a-K-bE}
Write $\overline{F}\sim a (-K_{\overline{X}})- b \overline{E}$. Then
$a$, $b\in\ZZ$ and $b\ge ma>0$, where $m$ is given by the condition 
\eqref{eq:FanoI:-KEm}.
Moreover, if the contraction $\varphi$ is birational, then
$b> ma$.
\end{cor}

\begin{proof}
Since the divisor $\sigma_*\comp \chi^{-1}_*(\overline{F})=a(-K_X)$ is 
effective, 
$a\ge 0$. Since $\chi^{-1}_*(\overline{F})\neq E$ (see
Lemma~\ref{lemma:barE-varphi}), we have $a>0$.
It follows from Lemma~\ref{lemma-effective-divisors} and
Remark~\ref{line:linear-system} that the class of the divisor
$-K_{\overline{X}}-m\overline{E}$
is a convex linear combination $\overline{F}\sim a (-K_{\overline{X}})- b
\overline{E}$
and~$\overline{E}$. Therefore, $b\ge ma$. If $b=ma$, then $\overline{F}\sim
a(-K_{\overline{X}}-m\overline{E})$.
Since $\dim|{-}K_{\widetilde{X}}-E|\ge 1$, the contraction $\varphi$ cannot
be birational.
\end{proof}

Denote by $\upmu=\upmu(\rR_\varphi)$ the length of the extremal ray
$\rR_\varphi$ (see~\eqref{eq:def:length}).

\begin{lem}
\label{mu=1}
If $\upmu=1$, then $\overline{M}\sim -K_{\overline{X}}-\overline{E}$.
\end{lem}

\begin{proof}
The linear system $|{-}K_{\overline{X}}-m\overline{E}|$ is of positive 
dimension
and a general divisor $\overline{D}\in|{-}K_{\overline{X}}-m\overline{E}|$ 
is irreducible (see
Remark~\ref{line:linear-system}).
Thus $\overline{D}$ is non-negative on the ray $\rR_\varphi$.
Put $\widetilde{D}:= \chi^{-1}_*(\overline{D})$.
Then $\widetilde{D}\in|{-}K_{\widetilde{X}}-mE|=|H^*-(m+1)E|$.

If $\Exc(\phi)=\varnothing$, then $\overline{X}=\widetilde{X}$ and
$$
\overline{D}\cdot \rR_\sigma=(H^*-(m+1)E)\cdot \rR_\sigma>0.
$$
And if $\Exc(\phi)\neq \varnothing$, then
$$
\overline{D}\cdot \rR_{\bar\theta}=-\widetilde{D}\cdot \rR_{\theta}=(m+1)E\cdot
\rR_{\theta}>0,
$$
where $\rR_{\theta}$ is the extremal ray, generated by the curves in the fibers 
of $\theta$.

Hence, the divisor $\overline{D}$ is nef
(see Fig.~\ref{figure:Mori-cone}, page~\pageref{figure:Mori-cone}).
For a minimal rational curve $\ell$ of the ray $\rR_{\varphi}$ we have
$$
0\le (-K_{\overline{X}}-m\overline{E})\cdot \ell=\upmu-m\overline{E}\cdot 
\ell=1-m\overline{E}\cdot \ell,
$$
where $\overline{E}\cdot \ell> 0$.
Therefore, $m=\overline{E}\cdot \ell=1$ and 
$(-K_{\overline{X}}-\overline{E})\cdot 
\ell=0$, i.e.
$-K_{\overline{X}}-\overline{E}$ is a supporting divisor of the ray 
$\rR_{\varphi}$.
\end{proof}

From Lemma~\ref{lemma:i-form} taking~\eqref{eq:intersection-numbers} into 
account we obtain
\begin{equation}
\label{eq:intersection-numbers1}
-K_{\overline{X}}^3 = 2g-6,\qquad
(-K_{\overline{X}})^2\cdot \overline{E} = 3,\qquad
-K_{\overline{X}}\cdot \overline{E}^2 = -2.
\end{equation}

\begin{lem}
\label{lemma:FanoI:B2-5}
The ray $\rR_{\varphi}$ cannot be of type \type{B_2}--\type{B_5}.
\end{lem}

\begin{proof}
Assume that $\varphi$ contracts a divisor $\overline{F}$ to a point.
As in Corollary~\ref{cor:FanoI:F=a-K-bE} we write $\overline{F}\sim
a(-K_{\overline{X}})-b \overline{E}$. It follows from~\eqref{equations-E25} 
that
$$
\overline{F}\cdot(-K_{\overline{X}})^2+3b=(2g-6)a=3b+k, \qquad k=1,\, 2, \, 4.
$$
On the other hand, since the linear system $|{-}K_{\overline{X}}-m 
\overline{E}|$ 
has no fixed components,
and $|{-}K_{\overline{X}}|$ has no base points, we have
$$
0\le (-K_{\overline{X}}-m \overline{E})\cdot \overline{F}\cdot 
(-K_{\overline{X}})=
(2g-6)a-3b -3am-2bm.
$$
Hence,
$$
3b+4\ge 3b+k=(2g-6)a\ge 3b+3am+2bm.
$$
Since $a,\, b,\, m \ge 1$ (Corollary~\ref{cor:FanoI:F=a-K-bE}), the last
inequality gives us a contradiction.
\end{proof}

\begin{lem}
\label{lemma:FanoI:D}
If the ray $\rR_{\varphi}$ of type \type{D}, then the case~\ref{g=7} of
Theorem~\ref{th:double-projection} occurs and the contraction $\varphi\colon 
\overline{X}\to Y=\PP^1$
is given by the linear system $|{-}K_{\overline{X}}- \overline{E}|$.
\end{lem}

\begin{proof}
Let $d$ be the degree of the general fiber of the del Pezzo surface fibration
$\varphi:\overline{X}\to Y\simeq\PP^1$.
According to Lemma~\ref{lemma-F},
we have $\overline{M}\sim a(-K_{\overline{X}})-\upmu \overline{E}$, where $a>0$.
Recall (see Theorem~\ref{class:ext-rays}) that $1\le \upmu\le 3$. Moreover,
$1\le d\le 6$ for $\upmu=1$,
$d=8$ for $\upmu=2$ and $d=9$ for $\upmu=3$.

The relations~\eqref{equations-D} can be written in the form
\begin{align}
\label{eqnarray:sect8-1}
d&=(-K_{\overline{X}})^2\cdot \overline{M}=(2g-6)a-3\upmu,
\\[3pt]
\label{eqnarray:sect8-2}
0&=(-K_{\overline{X}})\cdot\overline{M}^2=(2g-6)a^2-6a\upmu-2\upmu^2.
\end{align}
Recall that $g\ge 7$ by our assumption. It follows
from the equation~\eqref{eqnarray:sect8-2} that $\upmu^2 \equiv 0\mod 
a$.
If $\upmu=1$, then by Lemma~\ref{mu=1} we have $a=1$, and so $g=7$ and $d=5$,
i.e. this is exactly the case~\ref{th:double-projection}\ref{g=7}.
If $\upmu=2$, then $d=8$ and from the equation~\eqref{eqnarray:sect8-1} we 
obtain
$(g-3)a=7$, $a=1$, $g=10$.
This contradicts~\eqref{eqnarray:sect8-2}.
Finally, if $\upmu=3$, then $d=9$ and from the equation~\eqref{eqnarray:sect8-1}
we obtain $(g-3)a=9$, $a=1$, $g=12$.
This also contradicts~\eqref{eqnarray:sect8-2}.
\end{proof}

\begin{lem}
\label{lemma:FanoI:C}
If the ray $\rR_{\varphi}$ is of type \type{C}, then the case~\ref{g=8} of
Theorem~\ref{th:double-projection} occurs and the contraction $\varphi\colon 
\overline{X}\to Y=\PP^2$
is given by the linear system $|{-}K_{\overline{X}}- \overline{E}|$.
\end{lem}

\begin{proof}
Let $\Delta$ be the discriminant curve of the conic bundle
$\varphi\colon \overline{X}\to Y\simeq\PP^2$ and let $d:= \deg \Delta$.
As above, according to Lemma~\ref{lemma-F}
we have $\overline{M}\sim a(-K_{\overline{X}})-\upmu \overline{E}$.
Recall (see Theorem~\ref{class:ext-rays}) that $1\le \upmu\le 2$. Moreover,
$d\ge 3$ for $\upmu=1$ and
$d=0$ (i.e. $\Delta=\varnothing$) for $\upmu=2$.
The relations~\eqref{equations-C} can be written in the form
$$
\begin{alignedat}{2}
12-d&=(-K_{\overline{X}})^2\cdot \overline{F}&&=(2g-6)a-3\upmu,
\\[3pt]
2&=(-K_{\overline{X}})\cdot\overline{F}^2&&=(2g-6)a^2-6a\upmu-2\upmu^2.
\end{alignedat}
$$
If $\upmu=1$, then by Lemma~\ref{mu=1} we have $a=1$, and so $g=8$ and $d=5$.
We obtain the case~\ref{th:double-projection}\ref{g=8}. On the other hand, if 
$\upmu=2$, then
$d=0$
and the first equation gives us $(g-3)a=9$, $a=1$, and $g=12$.
This contradicts the second equation.
\end{proof}

\begin{lem}
\label{lemma:FanoI:B1}
If the ray $\rR_{\varphi}$ of type \type{B_1}, then one of the
cases~\ref{g=9},~\ref{g=10},~\ref{g=12} of
Theorem~\ref{th:double-projection} occurs and the following relations
\begin{equation}
\label{eq:FanoI:FMlast}
\overline{M}\sim -K_{\overline{X}}- \overline{E}, \qquad
\overline{F}\sim (\iota-1) (-K_{\overline{X}})-\iota \overline{E}
\end{equation}
hold, 
where $\iota=4$, $3$,~$2$ in the cases~\ref{g=9},~\ref{g=10},~\ref{g=12},
respectively.
\end{lem}

\begin{proof}
Assume that $\varphi\colon \overline{X}\to Y$ is a birational contraction of 
a divisor $\overline{F}$
to a curve $Z$. Then $Y$ is a (nonsingular) Fano threefold
of index $\iota=\iota(Y)$.
From Lemmas~\ref{lemma-F} and~\ref{mu=1} we obtain the relations and,
\eqref{eq:FanoI:FMlast},
moreover, $\iota>1$.
Furthermore, the relations~\eqref{equations-E1} in our case have form
$$
\begin{aligned}
(-K_{\overline{X}}- \overline{E})^2\cdot (-K_{\overline{X}})&=\iota\cdot\dd(Y),
\\[3pt]
(-K_{\overline{X}}-\overline{E})\cdot ((\iota-1) (-K_{\overline{X}})-\iota 
\overline{E})\cdot(-K_{\overline{X}})&=\deg Z,
\\[3pt]
((\iota-1) (-K_{\overline{X}})-\iota \overline{E})^2\cdot 
(-K_{\overline{X}})&=2\g(Z)-2.
\end{aligned}
$$
Taking~\eqref{eq:intersection-numbers} into account we obtain
\begin{align}
\label{eq:FanoI-le1}
(2g-6) -8&=\iota\cdot\dd(Y),
\\[3pt]
\label{eq:FanoI-le2}
(2g-6) (\iota-1) -8 \iota+3&=\deg Z,
\\[3pt]
\label{eq:FanoI-le3}
(2g-6) (\iota-1)^2 -6(\iota-1)\iota -2\iota^2&=2\g(Z)-2.
\end{align}
If $\iota=2$, then it follows from~\eqref{eq:FanoI-le1} 
and~\eqref{eq:FanoI-le3} 
that
$g-7=\dd(Y)$
and $g-12=\g(Z)\ge 0$. Therefore, $g=12$ and $\g(Z)=0$ because $\dd(Y)\le 5$
(Theorem~\ref{th:d5}).
We obtain the only possibility~\ref{th:double-projection}\ref{g=12}.

Let $\iota=3$. Then $Y$ is a quadric in $\PP^4$
(Theorem~\ref{thm:large-index}). In particular, $\dd(Y)=2$. Then from
\eqref{eq:FanoI-le1} we obtain the only possibility $g=10$ and from
\eqref{eq:FanoI-le2} and~\eqref{eq:FanoI-le3} we can compute the invariants of 
the curve~$Z$:
$\g(Z)=2$ and $\deg Z=7$. This the case~\ref{th:double-projection}\ref{g=10}.

Finally, let $\iota=4$.
Then $Y\simeq\PP^3$ (Theorem~\ref{thm:large-index}). In particular, $\dd(Y)=1$.
Then from~\eqref{eq:FanoI-le1} we obtain the only possibility $g=9$ and from
\eqref{eq:FanoI-le2} and~\eqref{eq:FanoI-le3} we compute invariants of the curve
$Z$:
$\g(Z)=3$ and $\deg Z=7$. This the case~\ref{th:double-projection}\ref{g=9}.
It remains to show that in this case the curve $Z$ is not hyperelliptic.
Since the linear system
$$
\varphi^*|4M-Z|=|\varphi^*(-K_Y)-\overline{F}|=
|{-}K_{\overline{X}}|=\bar{\theta}^*|{-}K_{\overline{X_0}}|
$$
is base point free, the curve
$Z\subset\PP^3$ is an intersection of quartics. In particular, it has no
$5$-secant lines. Then the assertion follows from the lemma below.
\end{proof}

\begin{lem}
\label{lemma:hyp}
Let $Z\subset\PP^3$ be a nonsingular curve of genus $3$ and degree $7$.
Then $Z$ has a $5$-secant line if and only if it
is hyperelliptic.
\end{lem}

\begin{proof}
If $Z$ has a $5$-secant line $L$, then the projection from it defines a 
hyperelliptic
linear series $\mathfrak g^1_2$ on $Z$.

Conversely, assume that on $Z$ there exists a hyperelliptic
linear series $\mathfrak g^1_2$. Let $D$ be a hyperplane section of the curve 
$Z$. By
the Riemann--Roch Theorem
$\dim|D|=4$. Therefore, $Z\subset\PP^3$ is a projection of $Z\subset\PP^4$
from a point $O\in\PP^4\setminus Z$.
Put $D':= D-\mathfrak g^1_2$. Then $\deg D'=5$ and $\dim|D-D'|=1$.
Therefore, the linear span of each divisor from the linear system $|D'|$
is a $5$-secant plane $\Pi\subset\PP^4$ to the curve $Z\subset\PP^4$.
Again by the Riemann--Roch Theorem $\dim|D'|=2$. Hence, such $5$-secant planes
cover the whole space $\PP^4$ and the point $O$ lies in one of
such planes $\Pi\subset\PP^4$. Then the image of $\Pi$ under the projection 
from $O$ is a $5$-secant line to our curve $Z\subset\PP^3$.
\end{proof}

\begin{proof}[Proof of Theorem~\ref{th:double-projection}]
The existence of the link~\eqref{diagram} is proved in~\ref{subs:FanoI:ex-diag}.
By Lemma~\ref{lemma:FanoI:B2-5} the contraction $\varphi$ cannot be of type
\type{B_2}--\type{B_5}.
If the contraction $\varphi$ is of type \type{D}, then by Lemma 
\ref{lemma:FanoI:D}
we obtain the case~\ref{th:double-projection}\ref{g=7}.
If the contraction $\varphi$ is of type \type{C}, then by Lemma 
\ref{lemma:FanoI:C}
we obtain the case~\ref{th:double-projection}\ref{g=8}.
Finally, if the contraction~$\varphi$ is of type \type{B_1}, then by Lemma 
\ref{lemma:FanoI:B1} we obtain one of the
cases~\ref{g=9},~\ref{g=10},~\ref{g=12}. Let $\overline{H}$ be the proper
transform of the divisor $H$ on $\overline{X}$. Then 
$-K_{\overline{X}}=\overline{H}-\overline{E}$ and from
\eqref{eq:FanoI:FMlast} it is not difficult to show that
$$
\overline{H}\sim (2\iota-1)\overline{M}-2\overline{F},\qquad
\overline{E}\sim (\iota-1)\overline{M}-\overline{F}.
$$
This implies that the map $Y \dashrightarrow X$ is given by the linear
system $|(2\iota-1)M-2Z|$. This proves the last assertions
in~\ref{g=9},~\ref{g=10} and~\ref{g=12}.
\end{proof}

\begin{proof}[Proof of Corollaries~\ref{corollary:line-flop} and~\ref{cor:line-line}]
In the assumptions of Corollary~\ref{cor:line-line} the exceptional locus
$\Exc(\phi)$ of the contractions~$\phi$ consists exactly of proper 
transforms of lines meeting $l$ (see
Lemma~\ref{lemma:l-s:d-p}\ref{double-projection:bir}).
Thus for the proof of assertions it is sufficient to establish the non-emptiness
of the set $\Exc(\phi)$.
A simplest computations show that the defect $\deff(\Psi):= 
E^3-\overline{E}^3$
of the link~\eqref{diagram} is strictly positive (see Exercise 
\ref{zad:FanoI:def} below).
Therefore, $\chi$ is not an isomorphism and the set $\Exc(\phi)$ is non-empty 
(see
Proposition~\ref{prop:defect}).
\end{proof}

\begin{zadachi}
\eitem
\label{zad:P2-2}
Let $S\subset\PP^N$ be a (not necessarily nonsingular) irreducible surface,
containing a two-dimensional family of lines. Prove that $S$ is a plane.

\eitem
Prove that on an anticanonically embedded
Fano threefold $X=X_{2g-2}\subset\PP^{g+1}$
with $\Pic(X)\simeq \ZZ\cdot K_X$ a \textit{general} line
meets at most a finite number of lines.

\eitem
Let $X=X_{2\cdot 2\cdot 2}\subset\PP^6$ be an intersection of three quadrics.
Let $l\subset X$ be a line. Assume that the line $l$ meets at most
a finite number of other lines in $X$. Prove that there exists a Sarkisov link
\eqref{diagram} with center $l$ and in this situation $Y\simeq \PP^2$ and 
$\varphi$ is a conic bundle.
Compute the degree of the discriminant curve.

\eitem
Let $X=X_{10}\subset\PP^7$ be an anticanonically embedded Fano threefold with 
$\iota(X)=1$, $\uprho(X)=1$ and $\g(X)=6$.
Let $l\subset X$ be a line. Assume that the line $l$ meets at most
a finite number of other lines in $X$. Prove that there exists a Sarkisov link
\eqref{diagram} with center $l$ and in this situation $Y$ is again a Fano 
threefold with
$\iota(Y)=1$, $\uprho(Y)=1$, $\g(Y)=6$, and $\varphi$ is the blowup of a line.

\eitem
\label{zad:FanoI:def}
Compute the defect $\deff(\Psi):= E^3-\overline{E}^3$ of the 
link~\eqref{diagram} in 
Theorem~\ref{th:double-projection}
and prove that it is strictly positive.
\hint{Compute $\overline{M}^3=(-K_{\overline{X}}-\overline{E})^3$ in two 
ways.}

\eitem
\label{zad:line-lineF}
Prove the assertion of Corollary~\ref{cor:line-line} for the case of genus 
$\g(X)\ge
4$.
\end{zadachi}

\newpage\section{Fano threefolds with Picard 
number~$1$.~II}
\label{sec:FanoII}

\subsection{Hilbert scheme of conics}

Let $X=X_{2g-2}\subset\PP^{g+1}$ be an anticanonically embedded 
Fano threefold of index $\iota=1$ with $\uprho(X)=1$ and $g\ge 5$, and let $H$ 
be 
the class of a hyperplane section. Assume that there exists at least one
conic\footnote{Later in Corollary~\ref{cor:exist:line+conic} we show,
that a conic does there exists.} on $X$. Let $\Conics=\Conics (X)$ be the 
Hilbert scheme conics on $X$, i.e. the scheme parametrizing closed
subschemes in $X$ 
with Hilbert polynomial $\upchi(t)=2t+1$.

\begin{prp}
\label{prop:index1-conics}
Let $X=X_{2g-2}\subset\PP^{g+1}$ be an anticanonically embedded
Fano threefold of index $\iota=1$
with $\uprho(X)=1$ and $g\ge 5$.
Assume that there exists a nondegenerate conic on $X$.
The following assertions hold.
\begin{enumerate}
\item
\label{prop:index1-conics-hs}
For any nondegenerate conic $C\subset X$ there is a nonsingular hyperplane 
section passing through~$C$.
\item
\label{prop:index1-conics-nb}
For the normal bundle $\NNN_{C/X}$ of a nondegenerate conic $C\subset X$ there 
are only three possibilities:
\begin{equation}
\label{eq:conic-N}
\NNN_{C/X}\simeq
\begin{cases}
\OOO_{\PP^1} \oplus\OOO_{\PP^1} & \text{type $(0,0)$,}
\\
\OOO_{\PP^1}(1) \oplus\OOO_{\PP^1}(-1) & \text{type $(1,-1)$,}
\\
\OOO_{\PP^1}(2) \oplus\OOO_{\PP^1}(-2) & \text{type $(2,-2)$.}
\end{cases}
\end{equation}
\textup(We identify $C$ with $\PP^1$.\textup)

\item
\label{prop:index1conics-sh}
Let $[C]\in\Conics(X)$ be the point corresponding to a nondegenerate conic
$C\subset X$.
If the normal bundle $\NNN_{C/X}$ is of type $(0,0)$ or $(1,-1)$, then
the scheme $\Conics(X)$ is nonsingular at the point~$[C]$ and 
$\dim_{[C]}\Conics(X)=2$.

\item
\label{prop:index1conics-sh00}
If a nondegenerate conic $C$ corresponds to a sufficiently general point of some 
component $\mathrm{F}_2'\subset\Conics(X)$, then it is of
type $(0,0)$.

\item
\label{prop:index1conics-dim2}
$\dim\Conics(X)=2$ and any component of the scheme $\Conics(X)$ containing a
nondegenerate conic is two-dimensional.

\item
\label{prop:index1-conics-pt}
There is only a finite number of conics passing through a general point $P\in X$
and there is at least one nondegenerate conic passing through $P$.

\item
\label{prop:index1-conics-conic-line}
A general \textup(nondegenerate\textup) conic $C\subset X$ meets at most
a finite number of lines.

\item
\label{prop:index1-conics-every-pt}
For $g\ge 10$
there are at most
a finite number of conics passing through an \textsl{arbitrary} point $P\in X$.

\item
\label{prop:index1-conics-cl}
For $g\ge 9$ \textsl{every} conic $C\subset X$ meets at most a finite 
number of lines.
\end{enumerate}
\end{prp}

\begin{rem}
In Proposition~\ref{prop:index1-conics} we do not claim that any component
of the scheme $\Conics(X)$ is two-dimensional (see 
\ref{prop:index1conics-dim2}). In principle, 
there may exist components
of dimension $\le 1$, consisting of reducible or non-reduced conics.
Actually this is impossible~\cite[Lemma~2.3.4]{KPS:Hilb} but 
this is not needed for our 
purposes.
\end{rem}

\begin{proof}
Let us prove~\ref{prop:index1-conics-hs}.
Let $\sigma\colon \widetilde{X}\to X$ be the blowup of the conic~$C$
and let $E$ be the exceptional divisor.
Since $X$ is an intersection of quadrics, the conic~$C$ is cut out on $X$ by the 
linear system $|H-C|$ of
hyperplane sections passing through~$C$.
Hence, the linear system $|\sigma^*H-E|$ is base point free.
By Bertini's theorem a general member $\widetilde{H}\in|\sigma^*H-E|$ is 
nonsingular.
The restriction of the linear system $|\sigma^*H-E|$ to $E$ also has no base
points
and is not composed of fibers of the projection $\sigma_E\colon E\to C$.
Again by Bertini's theorem the intersection $\widetilde{H}\cap E$ is a 
nonsingular
irreducible curve, a section of $\sigma_E\colon E\to C$.
This means that the restriction $\sigma_H\colon \widetilde{H}\to 
\sigma(\widetilde{H})$ is an isomorphism. Therefore, $\sigma(\widetilde{H})$ is 
a nonsingular 
hyperplane
section passing through $C$.

The proof of
\ref{prop:index1-conics-nb} is
similar to the proof of
Proposition~\ref{proposition-index=2-lines}~\ref{proposition-index=2-lines-nb}:
we have a decomposition $\NNN_{C/X}=\OOO_{\PP^1}(a)+\OOO_{\PP^1}(b)$.
Let $H$ be a nonsingular hyperplane section passing through $C$.
The vector bundle
$\NNN_{C/X}$ can be presented in the form of
an extension
$$
0\longrightarrow\NNN_{C/H}\longrightarrow\NNN_{C/X}
\longrightarrow\NNN_{H/X}|_C\longrightarrow 0,
$$
where $\NNN_{H/X}|_C=\OOO_{\PP^1}(2)$.
Since $H$ is a \K3 surface and $C\subset H$ is a
nonsingular curve of genus zero, we have $\NNN_{C/H}=\OOO_{\PP^1}(-2)$.
Therefore, $a+b=0$. Since $\hr^0(\NNN_{C/X})\le \hr^0(\NNN_{H/X}|_C)=2$, we 
have 
$a,\, b\le 2$.

Let us prove~\ref{prop:index1conics-sh}.
According to the deformation theory (see~\cite[Theorems~2.8 and 
2.15]{Kollar-1996-RC}), 
at any point $[C]\in\Conics$ we have
$$
\dim_{[C]}\Conics \ge \hr^0 (\NNN_{C/X})-\hr^1 (\NNN_{C/X})=2.
$$
Moreover, if the sheaf
$\NNN_{C/X}$ is of type $(0,0)$ or $(1,-1)$, then
$$
\hr^0 (\NNN_{C/X})=2,\qquad \hr^1(\NNN_{C/X})=0.
$$
In this case
the scheme $\Conics$ is nonsingular at the corresponding point $[C]$ and has 
dimension~$2$.
This, in particular, proves~\ref{prop:index1conics-sh}.

If
$\NNN_{C/X}$ is of type $(2,-2)$, then
$$
\hr^0(\NNN_{C/X})=3,\qquad \hr^1(\NNN_{C/X})=1.
$$
In this case
$\dim T_{[C],\Conics}=3$ and there are two possibilities:
\begin{enumerate}
\item
\label{Hilbert:conics-1}
$\dim_{[C]} \Conics = 2$ and the scheme $\Conics$ is singular at the point 
$[C]$;
\item %
\label{Hilbert:conics-2}
$\dim_{[C]}\Conics = 3$ and the scheme $\Conics$ is nonsingular at the point 
$[C]$.
\end{enumerate}

Let $\mathrm{F}_2'\subset\Conics(X)$ be an irreducible component containing the
class $[C]$ of a (nondegenerate) conic~$C$.
Consider the universal family of conics $\Univ'\to \mathrm{F}_2'$ and the 
diagram
$$
\xymatrix@R=1.8em{
&\Univ'\ar[dl]_{p}\ar[dr]^{q}&
\\
\mathrm{F}_2'&&X
}
$$
We claim that the morphism $q$ is surjective.
Indeed, in the opposite case, $q(\Univ')$ is a surface containing an
(at least) two-dimensional family of conics.
It is known that such a surface is a projection of the Veronese surface
(see e.~g.~\cite[Lemma~A.1.2]{KPS:Hilb}).
In particular, $\deg q(\Univ')\le 4$.
By our condition the group $\Pic(X)$ is generated by the class of a
hyperplane section. Thus, in particular, $X$ does not contain surfaces
of degree less than $2g-2$.
Therefore, $q(\Univ')=X$, i.e. the morphism $q$ is surjective.

Let us prove~\ref{prop:index1conics-sh00}.
Take a nondegenerate conic $C$ corresponding to a sufficiently general
point $[C]\in\mathrm{F}_2'$ (we regard the component $\mathrm{F}_2'$ with 
reduced
structure). Let 
$$
\NNN_{C/X}\simeq \OOO_{\PP^1}(a)\oplus \OOO_{\PP^1}(-a),
\quad a=0,1,2.
$$
Then the variety $\mathrm{F}_2'$ is nonsingular at $[C]$ and, therefore, so
is the variety~$\Univ'$ along fiber over $[C]$ (which we also denote by $C$).
We have natural morphism of normal bundles
$$
\mathrm{d} q \colon \NNN_{C/\Univ'}\simeq \bigoplus_{i=1}^{\dim(\mathrm{F}_2')}
\OOO_{\PP^1}
\longrightarrow\NNN_{C/X}\simeq \OOO_{\PP^1}(a)\oplus \OOO_{\PP^1}(-a).
$$
Let $V:= \Univ'\setminus \Sing(\Univ')$. Then $C\subset V$.
The morphism $q_V\colon V\to X$ is smooth over open subset $U\subset X$
and $C\cap U\neq\varnothing$.
Hence, the differential $\mathrm{d}q$ is non-degenerate at general point of $C$.
Since
$$
\Hom (\OOO_{\PP^1}, \OOO_{\PP^1}(a))=0\quad \text{whenever}\quad a<0,
$$ 
we have
$a=0$. This proves~\ref{prop:index1conics-sh00}. Furthermore, this implies that 
the case~\ref{Hilbert:conics-1} occurs and 
proves~\ref{prop:index1conics-dim2}.
Then the morphism $q$ must be generically finite. This implies
the assertion~\ref{prop:index1-conics-pt}.

The assertion~\ref{prop:index1-conics-conic-line} follows from the fact that 
there are at most a finite number of lines passing through
any point of the variety $X$, all the
lines on $X$ cover a surface, and a general conic does not lie in this surface
(see Proposition~\ref{prop:index1-lines} and the 
assertion~\ref{prop:index1-conics-pt}).

To prove~\ref{prop:index1-conics-every-pt}
assume that there exists a one-dimensional family of conics passing through 
$P\in
X$. Let $F$ be the surface covered by these conics and let $\MMM:= | H
-3P |$ be the linear system of hyperplane sections $X$ having a singularity
of multiplicity $\ge 3$ at the point $P$. Then
$$
\dim \MMM \ge \dim|H|-10=g-9> 0.
$$
Every conic passing through $P$ is contained in every divisor from $\MMM$.
Therefore, $F$ is a fixed component of $\MMM$. This contradicts the fact
that $\Pic (X)=\ZZ \cdot H$. The arguments for the proof of 
\ref{prop:index1-conics-cl} are similar.
The proposition is proved.
\end{proof}

\subsection{Sarkisov links with center a conic: main result}

\begin{teo}[\cite{Takeuchi-1989}]
\label{thm:proj-conic}
Let $X=X_{2g-2}\subset\PP^{g+1}$ be an anticanonically embedded 
Fano threefold
of index~$1$ and genus $g\ge 7$ with $\Pic(X)=\ZZ\cdot H$, where $-K_X=H$. 
Assume that there exists a nondegenerate conic
$C$ on $X$. For $g\le 8$ we additionally require that the conic $C$ meets at 
most a finite 
number of lines \textup(i.e. it is sufficiently general, see
Proposition~\ref{prop:index1-conics}~\ref{prop:index1-conics-conic-line}
\textup).
Then the blowup $\sigma\colon \widetilde{X}\to X$ of the conic $C$ suits to the
Sarkisov link~\eqref{diagram}.

Furthermore, let $E:= \sigma^{-1}(C)$ be the exceptional divisor, let
$\overline{E}\subset\overline{X}$ be the proper transform of $E$,
let $M$ be the ample generator of $\Pic(Y)\simeq \ZZ$, and let
$\overline{M}=\varphi^*M$. Then
for the extremal contraction $\varphi$
there is one of the following possibilities:
\begin{enumerate}

\item
\label{thm:proj-conic7}
$g=7$,
$Y=Q\subset\PP^4$ is a nonsingular quadric and the morphism $\varphi$ 
contracts a divisor
$\overline{F}\sim 5(-K_{\overline{X}})-3\overline{E}$
and is the blowup of a nonsingular curve $Z\subset Q\subset\PP^4$
of degree $10$ and genus $7$, $\overline{M}\sim 
2(-K_{\overline{X}})-\overline{E}$;

\item
$g=8$,
$Y=Y_{14}\subset\PP^9$ is also a Fano threefold
of index~$1$ and genus $8$
with $\Pic(Y)=-K_Y\cdot\ZZ$ and the morphism $\varphi$
contracts a divisor $\overline{F}\sim -K_{\overline{X}}-\overline{E}$
and is the blowup of a
conic $Z\subset Y$,
$\overline{M}\sim 2(-K_{\overline{X}})-\overline{E}$;

\item
$g=9$, $Y=\PP^1$,
$\varphi\colon \overline{X}\to\PP^1$ is a del Pezzo fibration of 
degree $6$
and 
$\overline{M}\sim -K_{\overline{X}}-\overline{E}$;

\item
$g=10$,
$Y=\PP^2$ and $\varphi\colon \overline{X}\to\PP^2$ is a 
conic bundle whose discriminant curve $\Delta\subset\PP^2$
has degree $4$ and $\overline{M}\sim -K_{\overline{X}}-\overline{E}$;

\item
$g=12$,
$Y=Q\subset\PP^4$ is a nonsingular quadric and the morphism $\varphi$
contracts a divisor
$\overline{F}\sim 2(-K_{\overline{X}})-3\overline{E}$
and is the blowup of a
nonsingular rational curve $Z\subset Q\subset\PP^4$ of degree~$6$,
$\overline{M}\sim -K_{\overline{X}}-\overline{E}$.
\end{enumerate}
In particular,
$g\le 12$ and $g\ne 11$.
\end{teo}

A large part of this section will be dedicated to the proof of the theorem.
It follows almost the same scheme as
the proof of~\ref{th:double-projection}.

\begin{notation}
\label{notation:X2g-2:conic}
Below, throughout the whole section we use the notation of
Theorem~\ref{thm:proj-conic}. Put also $H^*:= \sigma^*H$.
\end{notation}

From relations~\eqref{eq:blowup-curve-intersection} and the fact that
$-K_{\widetilde{X}}=H^*-E$ it follows immediately the following.

\begin{lem}
In the notation of Theorem~\ref{thm:proj-conic} we have $E^3=0$,
\begin{equation}
\label{e:conic-intersection-numbers}
E^3=0,
\qquad
-K_{\widetilde{X}}^3=2g-8,
\qquad
(-K_{\widetilde{X}})^2\cdot E=4,
\qquad
-K_{\widetilde{X}}\cdot E^2=-2.
\end{equation}
\end{lem}

According to~\eqref{eq:conic-N}, for the surface $E$ there are three 
possibilities:
$E\simeq \PP^1\times \PP^1$, $\FF_2$ and~$\FF_4$.

\begin{lem}
\label{lemma:conic:dim}
In the notation of Theorem~\ref{thm:proj-conic} we have the following 
assertions.
\begin{enumerate}
\item
\label{lemma:conic:dim1}
$\hr^0(\widetilde{X}, \OOO_{\widetilde{X}}(-K_{\widetilde{X}}))=g-1$.
\item
\label{lemma:conic:dim2}
The linear system $|{-}K_{\widetilde{X}}|$ is base point free
and defines a morphism
$$
\phi=\Phi_{|{-}K_{\widetilde{X}}|}\colon \widetilde{X}\longrightarrow\PP^{g-2}.
$$
The image $X_{\bullet}=\phi(X)$ of this
morphism is three-dimensional.
\item
\label{lemma:conic:ED}
$$
-K_{\widetilde{X}}|_E=
\begin{cases}
\Sigma+2\Upsilon,&\text{if $E\simeq \PP^1\times \PP^1$,}
\\
\Sigma+3\Upsilon,&\text{if $E\simeq \FF_2$,}
\\
\Sigma+4\Upsilon,&\text{if $E\simeq \FF_4$.}
\end{cases}
$$
where $\Sigma$, $\Upsilon$ are classes of the exceptional
section and fiber of the rational ruled surface
$E$, respectively.

\item
\label{lemma:conic:dim3}
$\hr^0(E, \OOO_E(-K_{\widetilde{X}}))=6$.
\end{enumerate}
\end{lem}

\begin{proof}
Since our variety $X=X_{2g-2}\subset\PP^{g+1}$ is an intersection of
quadrics and does not contain
planes, 
the intersection $X\cap\langle C \rangle$ coincides with $C$
(as a scheme). Therefore, $\Bs|H-C|=C$ and so the proper transform of
$|H^*-E|=|{-}K_{\widetilde{X}}|$ on $\widetilde{X}$ of 
the linear system of hyperplane sections passing through~$C$ has no base
points. By the Kawamata--Viehweg Vanishing Theorem 
$H^1(\widetilde{X}, \OOO_{\widetilde{X}}(-K_{\widetilde{X}})=0$. Therefore, 
there is the following
exact sequence
$$
0\longrightarrow H^0(\OOO_{\widetilde{X}}(-K_{\widetilde{X}}))\longrightarrow 
H^0(\OOO_{\widetilde{X}}(H^*))\longrightarrow
H^0(\OOO_E(H^*))\longrightarrow 0.
$$
This implies that $\dim|{-}K_{\widetilde{X}}|=g-2$. Since
$-K_{\widetilde{X}}^3>0$, the image of the morphism defined by the linear 
system
$|{-}K_{\widetilde{X}}|$ is three-dimensional. This proves 
\ref{lemma:conic:dim1}
and~\ref{lemma:conic:dim2}.

The assertion~\ref{lemma:conic:ED} follows from the fact that 
$(-K_{\widetilde{X}})^2\cdot E=4$
and $-K_{\widetilde{X}}\cdot \Upsilon=1$.
The equality~\ref{lemma:conic:dim3} follows from~\ref{lemma:conic:ED}.
\end{proof}

We study the anticanonical map $\phi=\Phi_{|{-}K_{\widetilde{X}}|}$ of
the variety $\widetilde{X}$. It is clear that $\phi$ suits to the commutative 
diagram
\begin{equation}
\label{eq:diag:FanoIIp}
\vcenter{
\xymatrix@R=2em{
&\widetilde{X}\ar[rd]^{\phi}\ar[ld]_{\sigma}&
\\
X\ar@{-->}[rr]^{\Phi_{|H-C|}}&&\PP^{g-2}
} }
\end{equation}
where
$\Phi_{|H-C|}\colon X \dashrightarrow\PP^{g-2}$ is the projection
from the linear span $\langle C \rangle=\PP^2$ of the conic $C$.

\begin{lem}
\label{conic:dim}
We have
\begin{equation}
\label{eqconic:dim}
\hr^0\bigl(\widetilde{X},\,\OOO_{\widetilde{X}}(-K_{\widetilde{X}}-E)\bigr)
\ge g- 1- \hr^0\bigl(E, \OOO_E(-K_{\widetilde{X}})\bigr)=g-7.
\end{equation}
If in the above settings the equality holds
$$
\hr^0\bigl(\widetilde{X},\,\OOO_{\widetilde{X}}(-K_{\widetilde{X}}-E)\bigr)=g-7,
$$
then the restriction $|{-}K_{\widetilde{X}}|\bigr|_E$ is a complete linear 
system
of dimension $5$.
\end{lem}

\begin{proof}
It is obtained from the exact sequence 
$$
0\longrightarrow H^0(\OOO_{\widetilde{X}}(-K_{\widetilde{X}}-E))\longrightarrow
H^0(\OOO_{\widetilde{X}}(-K_{\widetilde{X}}))
\longrightarrow H^0(\OOO_E(-K_{\widetilde{X}}))
$$
and the fact that $\hr^0(\OOO_E(-K_{\widetilde{X}}))=6$ (see
Lemma~\ref{lemma:conic:dim}\ref{lemma:conic:dim3}).
\end{proof}

\begin{lem}
\label{lemma:F2:E}
The morphism $\phi$ does not contract any divisors.
\end{lem}

\begin{proof}
Assume that the morphism $\phi$ contracts a prime divisor $\widetilde{D}$, 
i.e.
$\dim \phi(\widetilde{D})\le 1$.
This divisor is covered by a one-dimensional family of curves
$\widetilde{L}$ such that $-K_{\widetilde{X}}\cdot \widetilde{L}=(H^*-E)\cdot 
\widetilde{L}=0$
and $E\cdot \widetilde{L}>0$.
By the construction~\eqref{eq:diag:FanoIIp} the image $D:= 
\sigma(\widetilde{D})$ is 
contained
in the intersection of $X$ with the linear span $\langle 
C,\,\phi(\widetilde{D}) 
\rangle$.
Since $\Pic(X)=\ZZ\cdot H$, it follows that
$\dim \langle C,\,\phi(\widetilde{D}) \rangle\ge g$. Therefore,
$\dim \langle \phi(\widetilde{D}) \rangle\ge g-3$.
In particular, this means that $\phi(\widetilde{D})$ is a curve of degree $\ge 
4$.

Since $(-K_{\widetilde{X}}-E)\cdot \widetilde{L}<0$, the divisor
$\widetilde{D}$ must be a fixed component of the linear system
$|{-}K_{\widetilde{X}}-E|$ (if it is non-empty). Since $\Pic(X)=\ZZ\cdot H$ and
$\widetilde{D}\neq E$, we have $\dim|{-}K_{\widetilde{X}}-E|\le 0$.
The inequality~\eqref{eqconic:dim} and
Lemma~\ref{lemma:conic:dim}~\ref{lemma:conic:dim3} show that
$$
\dim \langle \phi(E) \rangle=\hr^0\bigl(E, \OOO_E(-K_{\widetilde{X}})\bigr)-1\ge
g-3.
$$
Thus the linear span of the surface $\phi(E)$ has dimension
at least~$4$. In particular, the map $\phi_E\colon E \to \phi(E)$
is birational
(otherwise $\phi(E)$ is a surface of degree $\le 2$).

By our assumption the conic $C$ meets at most a finite number of lines.
Therefore, the surface $\widetilde{D}$ is covered by curves~$\widetilde{L}$ 
such 
that 
$K_{\widetilde{X}}\cdot \widetilde{L}=0$ and $E\cdot
\widetilde{L}>1$. Thus the image of the exceptional divisor $\phi(E)$ must be is 
singular along the 
curve
$\phi(\widetilde{D})$ of degree $\ge 4$. However the singular locus of a surface 
of 
degree
$4$ in $\PP^4$
cannot contain a curve of degree $\ge 2$ (see Exercise~\ref{zad:SdPd} at 
the 
end of this section), a contradiction.
\end{proof}

\subsection{The link}
\label{conic:new-cases} 

Let us proceed to the proof of Theorem~\ref{thm:proj-conic}.
According to Lemmas~\ref{lemma:conic:dim}\ref{lemma:conic:dim2} and 
\ref{lemma:F2:E},
we can apply the construction of
Sarkisov links
from Section~\ref{section:sl}.
We obtain the diagram~\eqref{diagram}
(or its degenerate variant~\eqref{diagram-v}).
As in the proof of~\ref{th:double-projection}, we consider the possibilities for 
the
contraction $\varphi$.
Here we can use the classification of extremal rays
(Theorem~\ref{class:ext-rays}) and
the estimate~\eqref{eqconic:dim} of the dimension of the linear system
$|{-}K_{\widetilde{X}}-E|$.
Similarly to the proof of~\ref{th:double-projection}
we obtain.

\begin{lem}
\label{lemma:FanoII:ggge9}
For $g\ge 9$ the equality $\overline{M}\sim 
-K_{\overline{X}}-\overline{E}$ holds.
\end{lem}

Furthermore, besides the solutions listed in Theorem~\ref{thm:proj-conic}
there are
the following three possibilities. All they correspond to birational 
contractions:
\begin{enumerate}
\item
\label{conic:new-cases11}
$\g(X)=11$, $Y=\PP^3$, $\varphi$ is the blowup of a rational curve of degree 
$6$,
$\overline{F}\sim 3(-K_{\overline{X}})-4\overline{E}$, and $\overline{E}^3=-5$;
\item
\label{conic:new-cases8}
$\g(X)=8$, $Y=Y_{16}\subset\PP^{10}$ is a Fano threefold of index~$1$ and genus
$9$, $\varphi$ is the blowup of a point, $\overline{F}\sim 
-K_{\overline{X}}-\overline{E}$, and
$\overline{E}^3=-11$;
\item
\label{conic:new-cases7}
$\g(X)=7$, $Y=Y_3\subset\PP^4$ is a nonsingular cubic, $\varphi$ is the blowup 
of a
rational curve of degree $4$, $\overline{F}\sim 
3(-K_{\overline{X}})-2\overline{E}$, and 
$\overline{E}^3=-15$.
\end{enumerate}
Below we show that these cases do not occur.

\begin{cor}
\label{corollary:conic-flop}
In the conditions of Theorem~\ref{thm:proj-conic} the map $\chi$ from
the diagram~\eqref{diagram} is not an isomorphism.
\end{cor}

\begin{proof}
As in the proof of~\ref{corollary:line-flop}, it follows from the computations 
that
the defect $\deff(\Psi):= E^3-\overline{E}^3$ of the link~\eqref{diagram} 
is strictly positive
(see Exercise~\ref{zad:FanoII:def} at the end of this section).
\end{proof}

\begin{cor}
\label{cor:ex-line}
Let $X=X_{2g-2}\subset\PP^{g+1}$ be an anticanonically embedded 
Fano threefold
of index~$1$ and genus $g\ge 7$. If there exists a conic on $X$, then also
there exists
a line on $X$.
\end{cor}

\begin{proof}
It is clear that we may assume that the conic $C$ is non-degenerate.
We also additionally assume that the conic~$C$ is chosen to be general, i.e. its
normal bundle has the form $\NNN_{C/X}\simeq \OOO_C\oplus\OOO_C$.
Then $E\simeq \PP^1\times \PP^1$ and the morphism $\phi$ cannot contract 
curves on
$E$.
According to~\ref{corollary:conic-flop}, the map $\chi$ is not 
an isomorphism and the set $\Exc(\phi)$ is non-empty (see
Proposition~\ref{prop:defect}). Let us prove that it consists of proper
transforms of lines meeting $C$.

By Lemma~\ref{lemma:FanoII:ggge9} we have $\overline{M}\sim 
-K_{\overline{X}}-\overline{E}$.
Thus $\dim|{-}K_{\overline{X}}-\overline{E}|=g-8$ in these cases.
For $g=8$ the exceptional divisor $\overline{F}$ is the unique element of the 
linear
system
$|{-}K_{\overline{X}}-\overline{E}|$.
For $g=7$ the exceptional divisor $\overline{F}$ is an element of the linear
system
$\mo|-5K_{\overline{X}}-3\overline{E}|$ or 
$\mo|-3K_{\overline{X}}-2\overline{E}|$ (depending on whether the 
possibility~\ref{thm:proj-conic}\ref{thm:proj-conic7} or possibility from the 
table in~\ref{conic:new-cases} is realized).
It is clear that then $|{-}K_{\overline{X}}-\overline{E}|=\varnothing$.
Therefore, in all cases
$$
\dim|{-}K_{\overline{X}}-\overline{E}|=g-8.
$$
Thus, by Lemma~\ref{conic:dim}, the image 
$\phi(E)$ is a \textup(nondegenerate\textup)
surface $S_4\subset\PP^5$.
In particular, it is nonsingular. This means
that for any one-dimensional fiber
$\widetilde{L}$ of the morphism $\phi\colon \widetilde{X}\to X_{\bullet}$
we have $E\cdot \widetilde{L}=H^*\cdot \widetilde{L}=1$.
Thus $\sigma(\widetilde{L})$ is a line meeting the conic~$C$.
\end{proof}

\begin{proof}[Proof of Theorem~\ref{thm:proj-conic}]
It remains to exclude only
the cases~\ref{conic:new-cases11}--\ref{conic:new-cases7}
from subsection~\ref{conic:new-cases}.
But now we can use the existence of a line on $X$.
Then by Theorem~\ref{th:double-projection} the case $g=11$ is impossible.
In the case $g=8$ from Table~\ref{conic:new-cases}
we obtain a contradiction, comparing the topological Euler numbers
$X$ and $Y$ (by using Corollary~\ref{Euler-numbers}).
In the case $g=7$ the variety $X$ is rational according to
Corollary~\ref{corollary:rationality}.
Hence, it cannot be are birationally equivalent to nonsingular cubic
$Y_3\subset\PP^4$, which is not rational (see~\cite{Clemens-Griffiths}).
\end{proof}

\begin{cor}
\label{Euler-numbers:7}
Let $X$ be a Fano threefold with $\iota(X)=1$, $\uprho(X)=1$, and $\g(X)\ge 7$.
Assume that there exists nondegenerate conic on $X$.
Then the topological Euler number $\chit(X)$ takes the following values:
$$
\begin{array}{rrrrrr}
\g(X): & 7 & 8 & 9 & 10 & 12 \\[5pt]
\chit(X): & -10 & -6 & -2 & 0 & 4
\end{array}
$$
\end{cor}

\begin{proof}
According to Corollary~\ref{cor:ex-line}, on $X$ there exists a line.
Then for $\g(X)\ge 8$ we can use Corollary~\ref{Euler-numbers}.
For $\g(X)=7$ as in the proof of~\ref{Euler-numbers}, the value of $\chit(X)$
is computed by using
invariants of the curve $Z$ in birational
transformation~\eqref{diagram} (see~\ref{thm:proj-conic}\ref{thm:proj-conic7}).
\end{proof}

\subsection{Genus $7$ curves in $\PP^4$} 

Similar to Remark~\ref{remark:jacobian} it should be noticed that in the 
conditions of 
\ref{thm:proj-conic7} in Theorem~\ref{thm:proj-conic} the curve of genus $7$ up 
to 
abstract isomorphism is uniquely defined the variety
$X$. As in~\ref{th:double-projection}~\ref{g=9} we can 
impose certain conditions on this curve, i.e. it must be sufficiently general:

\begin{prp}
\label{proposition:curve:g7}
Let $Z\subset Q\subset\PP^4$ be a curve of degree $10$ of genus $7$ from
~\ref{thm:proj-conic7} of Theorem~\ref{thm:proj-conic}. Then 
the following assertions hold.
\begin{enumerate}
\item
\label{proposition:curve:g7:hyp}
The curve $Z$ is not hyperelliptic and the embedding $Z\subset\PP^4$ is the 
projection
$\pi\colon Z_{12} \dashrightarrow \PP^4\supset Z$ of a canonical curve 
$Z_{12}\subset
\PP^6$ from a line meeting $Z_{12}$ at two points $P_1,P_2\in Z_{12}$
\textup(possibly coinciding\textup).
\item
\label{proposition:curve:g7:trig}
The curve $Z$ is not trigonal nor tetragonal.
\end{enumerate}
\end{prp}

A complete description of canonical models curves of genus~$7$ is given in
\cite{mukai-1995-1}.

\begin{proof}
\Ref{proposition:curve:g7:hyp} 
First of all note that the curve $Z\subset\PP^4$ does not lie in a hyperplane.
Indeed, it follows from~\ref{thm:proj-conic7}\ref{thm:proj-conic} that 
$\overline{E}+2\overline{F}\sim 5\overline{M}$.
Thus the surface $E_Q:= \varphi(\overline{E})$ is cut out on the quadric $Q$
by a hypersurface $R$ of degree $5$. If the curve $Z\subset\PP^4$ lies in a 
hyperplane, then it would be a complete intersection $Q\cap R\cap\PP^3$
(because $\deg Z=10$). But then $\g(Z)=16$, a contradiction.

Let $D$ be a hyperplane section of the curve $Z$. Then $\dim|D|\ge 4$. By
the Riemann--Roch Theorem
$$
\dim|D|-\dim|K_Z-D|=\deg D-\g(Z)+1=4.
$$
Therefore, $|K_Z-D|\neq \varnothing$, i.e. there exist two points $P_1$, $P_2$ 
on 
$Z$ 
such that $K_Z=D+P_1+P_2$.
Since the divisor $D$ is very ample, the curve~$Z$ non-hyperelliptic and is a
projection of a canonical curve $Z_{12}\subset\PP^6$ from the points
$P_1$, $P_2\in Z_{12}$.

\ref{proposition:curve:g7:trig}
Assume that $Z$ is trigonal. Then its canonical model $Z_{12}\subset
\PP^6$ has one-dimensional family of trisecant lines that cover a surface
of minimal degree (see Theorem~\ref{th:NEP:}). For the projection $\pi\colon 
\PP^6\subset Z_{12} \dashrightarrow Z\subset\PP^4$ we obtain a one-dimensional
family of trisecant lines that must lie in our quadric $Q\subset\PP^4$.
The proper transforms of all these~trisecant lines on $\overline{X}$ are 
contracted by 
the morphism
$\Phi_{|{-}K_{\overline{X}}|}$. This contradicts the fact that this morphism is 
small.

Assume now that $Z$ is not hyperelliptic nor trigonal, 
but it has a linear series of type $\mathfrak g^1_4$. For the analysis of this 
case we need two lemmas.

\begin{lem}
\label{lemma:II:tetr}
Let $Z$ be a smooth non-hyperelliptic, non-trigonal curve of genus $7$.
Assume that $Z$ has a linear series $|L|$ of type $\mathfrak g^2_d$, where $d\le
6$. Then $d=6$,
the curve $Z$ tetragonal, and its canonical model $Z_{12}\subset\PP^6$
is contained in a normal surface $S_6\subset\PP^6$ of degree $6$. At the same 
time
$Z_{12}$ is cut out on $S_6$ by a quadric. In particular, $Z_{12}$ is contained 
in
nonsingular parts of $S_6$.
\end{lem}

\begin{proof}
Subtracting the base points, we may assume that $|L|$ is base point free and 
defines a morphism $\gamma\colon Z \to \PP^2$
(then $d$ is decreasing).
Denote by $Z^\circ$ the image of $\gamma$.

First we consider the case where the morphism $\gamma$ is birational onto
$Z^\circ$.
Then $Z^\circ$ is a plane curve of degree $d$ and arithmetic genus
$\p(Z^\circ)=\frac12 (d-1)(d-2)\ge \g(X)=7$. Therefore, $d=6$ and $ 
\p(Z^\circ)=10$.
Since $Z$ is not trigonal, the curve $Z^\circ$ has no triple singular points.
Hence, the singularities of $Z^\circ$ are resolved by the blowup of three 
double points
(possibly infinitely near)~\cite[Ch.~5, \S3]{Hartshorn-1977-ag}:
$$
\eta\colon \widehat{S} \longrightarrow\PP^2,
$$
and the projection from a singular point defines a linear series $\mathfrak
g_4^1$ on $Z$.
Let $\widehat{Z}\subset\widehat{S}$ be the proper transform of $Z^\circ$.
We have $2K_{\PP^2}+Z^\circ\sim 0$ and this relation is preserved under blowups 
of
double points:
$$
2K_{\widehat{S}}+\widehat{Z}=\eta^*(2K_{\PP^2}+Z^\circ)\sim 0.
$$
Therefore, $-K_{\widehat{S}}\cdot \widehat{Z}=\frac12 (-K_{\widehat{S}})^2=3$. 
Thus 
the anticanonical divisor $-K_{\widehat{S}}$ is nef and big, i.e.
$\widehat{S}$ is a weak del Pezzo surface of degree 6.
The anticanonical map $\Phi_{|{-}K_{\widehat{S}}|}\colon \widehat{S}\to \PP^6$ 
is 
a morphism and its image $S_6\subset\PP^6$ is a del Pezzo surface with 
Du Val singularities (or smooth)~\cite{Hidaka-Watanabe-1981}.
From the adjunction formula it is clear that the restriction of the morphism 
$\Phi_{|{-}K_{\widehat{S}}|}$
on $\widehat{Z}$ is the canonical map and its
image $\Phi_{|{-}K_{\widehat{S}}|}(\widehat{Z})$ is a canonical curve.
Since the curve $Z$ is not trigonal and the curve
$Z_{12}=\Phi_{|{-}K_{\widehat{S}}|}(\widehat{Z})$
is an intersection of quadrics.
Thus $Z_{12}\subset\PP^6$ is contained in a del Pezzo surface
$S_6\subset\PP^6$ of degree 6 with at worst Du Val singularities and is cut out 
on $S_6$ by a quadric.

Now we consider the case where the morphism $\gamma\colon Z\to 
Z^\circ\subset\PP^2$ is not birational.
It is clear that $(\deg Z^\circ)\cdot(\deg \gamma)=d\le 6$, where $\deg 
\gamma\ge 2$ and
$\deg Z^\circ\ge 2$. Since the curve $Z$ is not trigonal, the case $\deg 
Z^\circ=2$
is impossible. It remains to consider the only possibility: $d=6$,
$\deg Z^\circ=3$ and $\deg \gamma=2$. Since the curve $Z$ non-hyperelliptic, 
the curve $Z^\circ$ is nonsingular.
The projection from any point $p\in Z^\circ$ defines on $Z$ a linear
series~$\mathfrak g_4^1$.
By the Hurwitz formula $K_Z=\gamma^*R$, where $R$ is divisor of degree 6 
on~$Z^\circ$.
Take a divisor $R'$ of degree~3 on~$Z^\circ$ such that $R=2R'$. Then
$K_Z=2\pi^*R'$ and $|L'|:= |\pi^*R'|$ is a linear series of type $\mathfrak 
g^2_6$
(possibly different from $|L|$). Thus there exists subsystem $\pi^*|2R'|$ in the 
complete linear system
$|K_Z|=|2L'|$ on $Z$ that is base point free and defines a map of $Z$ to an 
elliptic curve $Z^\circ\subset\PP^5\subset\PP^6$
of degree 6. By the construction this curve is a projection of a canonical 
model
$Z_{12}\subset\PP^6$ from some point that does not lie on $Z_{12}$. This means,
that $Z_{12}$ lie on the cone $S_6\subset\PP^6$ over $Z^\circ$. A simple
computations (see Corollary~\ref{cor:quadrics-can-curve}) show that 
there are exactly~$10$ linearly independent quadrics passing 
through
$Z_{12}$ in $\PP^6$ but at the same time there are only~$9$ quadrics passing 
$S_6$. Therefore, $Z_{12}$ is cut out on $S_6$ by a quadric.
\end{proof}

\begin{cor}
\label{cor:FanoII:g26}
In our assumptions~\ref{proposition:curve:g7} the curve
$Z$ has no linear series $\mathfrak g^2_6$.
\end{cor}

\begin{proof}
Since the centers $P_1$, $P_2$ the projection $\pi\colon \PP^6 \dashrightarrow 
\PP^4$ lie
on $S_6$, the image $\pi(S_6)$ is a surface of degree $4$, containing the curve
$Z$.
If $\pi(S_6)$ does not lie in our quadric $Q$, then the curve $Z$ is a
component of the intersections $\pi(S_6)\cap Q$ and then $10=\deg Z\le 
\deg\pi(S_6)\cap
Q=8$. The contradiction shows that $\pi(S_6)\subset Q$.
But then for the proper transform $\overline{S}\subset\overline{X}$ of the
surface
$\pi(S)$ we have $(-K_{\overline{X}})^2\cdot \overline{S}<0$. This contradicts 
the nef property of $-K_{\overline{X}}$.
\end{proof}

\begin{lem}
Let $Z$ be a smooth non-hyperelliptic tetragonal curve of genus $7$.
Assume that $Z$ has no linear series $\mathfrak g^2_6$. Then the
canonical model $Z_{12}\subset\PP^6$ of this curve is contained in nonsingular
three-dimensional variety $W=W_4\subset\PP^6$,~which is a rational 
scroll of degree $4$ \textup(see Proposition~\ref{prop:scrolls}).
\end{lem}

\begin{proof}
Let $|L|$ be a linear series of type $\mathfrak g^1_4$.
By geometric version of the Riemann--Roch Theorem the linear spans $\langle
q_1+\dots+q_4\rangle$ of divisors $q_1+\dots+q_4\in|L|$ on the curve 
$Z_{12}\subset
\PP^6$ are two-dimensional. These planes cover a three-dimensional variety 
$W\subset
\PP^6$.

Consider this variety in detail. For this we consider the linear
system $|K_Z-L|$ of degree~$8$. By the Riemann--Roch Theorem $\dim|K_Z-L|=3$. 
Since
$Z$ has no linear series $\mathfrak g^2_6$, we have $\dim|K_Z-L-P'-P''|=1$ for
any points $P', P''\in Z$. Therefore, $|K_Z-L|$ is a very ample
linear system of type $\mathfrak g^3_8$
(see~\cite[Ch.~IV, \S3]{Hartshorn-1977-ag}). It defines an embedding $Z
\hookrightarrow \PP^3$ and both linear systems $|L|$ and $|K_Z-L|$ define 
embeddings $Z\hookrightarrow
\PP^1\times \PP^3$ so that $\pr_1^* H_1+\pr_3^* H_3=K_Z$, where $\pr_i$ is 
the projection and $H_i\subset\PP^i$ are hyperplane section. Thus, if now
we consider the composition
$$
Z\hookrightarrow \PP^1\times \PP^3\hookrightarrow \PP^7
$$
with the Segre embedding, then the canonical class $K_Z$ will be cut out by
hyperplanes in $\PP^7$. Since $\dim|K_Z|=6$, the image of $Z$ is contained in 
hyperplane section $(\PP^1\times \PP^3)\cap\PP^6$, which has a structure of
$\PP^2$-bundle over $\PP^1$. It is not difficult to see that the fibers of this 
fibration are 
exactly the linear spans $\langle q_1+\dots+q_4\rangle$ of divisors from
$|L|$. Thus $W=(\PP^1\times \PP^3)\cap\PP^6$ and $\deg W=4$. Since
the variety $W$ is not a cone, it is nonsingular.
\end{proof}

Now to finish the proof of the proposition we recall that the curve $Z$
is a projection of a canonical curve $Z_{12}\subset\PP^6$ from two points
$P_1$, $P_2\in Z_{12}$. The image of the variety $W$ under this projection is
a quadric $W'\subset\PP^4$ containing a family of planes. Therefore,
$W'$ is distinct from our nonsingular quadric $Q$. The intersection $S:= W'\cap 
Q$ is 
a surface of degree 4 and contains our curve $Z\subset\PP^4$. But then, as 
in the proof of Corollary~\ref{cor:FanoII:g26}, for the proper transform
$\overline{S}\subset\overline{X}$ of the surface $S$ we have 
$(-K_{\overline{X}})^2\cdot
\overline{S}<0$. This contradicts the nef property of $-K_{\overline{X}}$.
Proposition~\ref{proposition:curve:g7} is proved.
\end{proof}

\begin{zadachi}
\eitem
Let $X=X_{10}\subset\PP^7$ be an anticanonically embedded Fano threefold with 
$\iota(X)=1$, $\uprho(X)=1$, and $\g(X)=6$.
Let $C\subset X$ be a sufficiently general conic.
Prove that there exists a Sarkisov link
\eqref{diagram} with center $C$, where
$Y=Y_{10}\subset\PP^7$ is also Fano threefold
of index~$1$ and genus~$6$
with $\uprho(X)=1$ and 
$\varphi$ is the blowup of a conic.

\eitem
\label{zad:SdPd}
Let $S_d\subset\PP^d$ be an irreducible surface of degree $d$ that does not lie 
in 
a hyperplane.
Prove that its singular locus cannot contain curve of degree $>1$.

\eitem
\label{zad:FanoII:def}
Compute the defect $\deff(\Psi):= E^3-\overline{E}^3$ of the 
link~\eqref{diagram} in 
Theorem~\ref{thm:proj-conic}
and prove that it is strictly positive.
\hint{Compute $\overline{M}^3=(-K_{\overline{X}}-\overline{E})^3$ in two 
ways.}

\eitem
Let $S=S_4\subset\PP^5$ be a ruled surface of degree $4$,
the image of $\PP^1\times \PP^1$
under the map by the complete linear system of bidegree $(1,2)$.
Describe images of projections to $\PP^4$.
\hint{Come up with a construction, similar to the proof of Lemma 
\ref{projection-Veronese}.}
\end{zadachi}

\newpage\section{Fano threefolds of Picard 
number~$1$.~III}
\label{sec:FanoIII}

\subsection{Sarkisov links with center a point: main result} 

The following theorem was first formulated and proved by
K.~Takeuchi~\cite{Takeuchi-1989}. However some its
parts were know before (see e.~g.~\cite{Tregub1985a}).

\begin{teo}
\label{th:main:point}
Let $X=X_{2g-2}\subset\PP^{g+1}$ be an anticanonically embedded 
Fano threefold
of index~$1$ and genus $g\ge 7$ with $\Pic(X)=\ZZ\cdot H$, where $-K_X=H$ and 
let 
$P\in X$
be a point that does no lie on a line.
For $g\le 9$ we additionally require there are at most
a finite number of conics passing through $P$ \textup(i.e. $P$ is sufficiently 
general,
see~\ref{prop:index1-conics}~\ref{prop:index1-conics-pt}) and for $g=7$
we additionally require that the number of
irreducible rational curves
$C\subset X$ of degree $4$ having singularity at $P$ is at most finite,
see~\ref{lemma4.5.2}.
Then the blowup $\sigma\colon \widetilde{X}\to X$ of the point $P$ suits 
to 
a Sarkisov 
link~\eqref{diagram}.

Furthermore, let $E:= \sigma^{-1}(P)$ be the exceptional divisor, let
$\overline{E}\subset\overline{X}$ be the proper transform of $E$,
let $M$ be the ample generator of $\Pic(Y)\simeq \ZZ$, and let
$\overline{M}=\varphi^*M$. Then
for the extremal contraction $\varphi$
there is one of the following possibilities:
\begin{enumerate}
\item
\label{th:main:point7}
$g=7$, $Y={Y_5}\subset\PP^6$ is a nonsingular del Pezzo threefold of degree $5$, 
the morphism $\varphi$ 
contracts a
divisor $\overline{F}\sim 5(-K_{\overline{X}})-2\overline{E}$,
and $\varphi$ is the blowup of a
nonsingular curve $Z\subset Y_5\subset\PP^6$
of genus $7$, degree $12$, and
$\overline{M}\sim 3(-K_{\overline{X}})-\overline{E}$;

\item
\label{th:main:point:g=8}
$g=8$, $Y=Y_3\subset\PP^4$ is a nonsingular cubic, the morphism $\varphi$ 
contracts a divisor $\overline{F}\sim 3(-K_{\overline{X}})-2\overline{E}$,
and $\varphi$ is the blowup of a
nonsingular rational curve $Z\subset Y_3\subset\PP^4$ of degree $4$, and
$\overline{M}\sim 2(-K_{\overline{X}})-\overline{E}$;

\item
$g=9$, $Y=Y_{16}\subset\PP^{10}$
is a nonsingular Fano threefold of index~$1$ of genus $9$, the morphism 
$\varphi$
contracts a
divisor $\overline{F}\sim -K_{\overline{X}}-\overline{E}$ and $\varphi$ is the 
blowup of
a point, and
$\overline{M}\sim 3(-K_{\overline{X}})-2\overline{E}$;

\item
$g=10$, $Y=\PP^1$ and 
$\varphi\colon \overline{X}\to\PP^1$ a del Pezzo fibration 
of degree $6$;

\item
$g=12$, $Y=\PP^3$, the morphism
$\varphi$ contracts a
divisor $\overline{F}\sim 3(-K_{\overline{X}})-4\overline{E}$
and $\varphi$ is the blowup of a
nonsingular rational curve
$Z\subset\PP^3$ of degree $6$, and
$\overline{M}\sim -K_{\overline{X}}-\overline{E}$.
\end{enumerate}
In particular, $g\le 12$ and $g\ne 11$.
\end{teo}

\begin{notation}
\label{notation:X2g-2:point}
Throughout this section we use the notation of
Theorem~\ref{th:main:point}. We also put $H^*:= \sigma^*H$.
Let 
$$
\psi=\Phi_{|H-2P|}\colon X \dashrightarrow X_{\bullet}\subset
\PP^{g-3}
$$
be the projection
from the tangent space to $X$ at the point $P$, where 
$X_{\bullet}=\psi(X)$.
\end{notation}

It is clear that $-K_{\widetilde{X}}=H^*-2E$.
Also, the following relations hold
\begin{equation}
\label{e:point-intersection-numbers}
E^3=1,
\qquad
-K_{\widetilde{X}}^3=2g-10,
\qquad
(-K_{\widetilde{X}})^2\cdot E=4,
\qquad
-K_{\widetilde{X}}\cdot E^2=-2.
\end{equation}
Note that $E\simeq \PP^2$.

\begin{lem}
\label{lemma4.5.2}
There exists a non-empty open subset $U\subset X$ such that
for any point $P\in U$ the number
irreducible rational curves
$C\subset X$ of degree $4$ having singularity at $P$ is at most finite.
\end{lem} 

\begin{proof}
Let $C\subset X$ be a singular irreducible rational curve of degree $4$.
Note that $\p(C)=1$, $\dim \langle C\rangle=3$, and $C$ has a unique 
singular point. Moreover, $C$ is a complete intersection of two
quadrics in $\PP^3$
and $\langle C\rangle\cap X=C$ (as a scheme).

We claim that a general hyperplane section passing through $C$
is nonsingular. Indeed, let $\Lambda=|H-C|\subset |H|$ be the
linear system of hyperplane sections passing through $C$.
It is easy to see that for any point $P\in X$ we have
$\overline{T_{P,X}}\ne\langle C\rangle$
(otherwise $C\subset\overline{T_{P,X}}\cap X$ but $\overline{T_{P,X}}\cap X$ 
is a point or
an intersection of lines because
$X$ is a union of quadrics).
Therefore, a general divisor $H\in\Lambda$ is nonsingular at the point
$\Sing(C)$.
Since $\Bs\Lambda=C$, by Bertini's theorem $H$ is nonsingular outside $C$.
To prove that $H$ is nonsingular on $C\setminus \Sing(C)$
we need to consider the blowup $\hat{\sigma}\colon \widehat{X}\to X$ of
the curve $C$. The variety $\widehat{X}$ is nonsingular outside
$\hat{\sigma}^{-1}(\Sing(C))$ and since $\Bs\Lambda=C$,
the linear system
$\hat{\sigma}^*\Lambda$ is base point free.
Then the proof completely similar the corresponding part in the proof of
Proposition~\ref{prop:index1-conics}\ref{prop:index1-conics-hs}.
Thus there exists a nonsingular hyperplane section $H$
passing through $C$.

Let $\mathfrak{C}$ be a subscheme the Hilbert scheme parametrizing
irreducible reduced rational curves $C\subset X$ of degree $4$ and arithmetic 
genus $\p(C)=1$.
Consider the incidence variety
$$
\mathfrak{R}\subset |{-}K_X|\times\mathfrak{C},\qquad
\mathfrak{R}=\big\{(H,C)\ \big|\ H\in|{-}K_X|,\ C\in\mathfrak{C},\ H\supset
C\big\}
$$
and projections
$$
\xymatrix{
&\mathfrak{R}\ar[dr]^{\operatorname{pr}_2}\ar[dl]_{\operatorname{pr}_1}&
\\
|H|&&\mathfrak{C}
}
$$
The projection $\operatorname{pr}_2$ is obviously surjective and its fibers 
are projective spaces $\PP^{g-3}$. Therefore,
$\dim \mathfrak{R}=\dim\mathfrak{C}+g-3$.

On the other hand, a general hyperplane section passing through $C$
is a nonsingular $\K3$ surface 
and it cannot contain family rational curves, i.e.
general fiber of
$\operatorname{pr}_1\colon \mathfrak{R}\to \operatorname{pr}_1(\mathfrak{R})$
zero-dimensional. Thus
$\dim \operatorname{pr}_1(\mathfrak{R})=\dim\mathfrak{C}+g-3$.
According to the generalized Noether--Lefschetz hyperplane theorem 
\cite{Mishezon1967},
for a general surface $H\in|{-}K_X|$ the restriction map
$\Pic(X)\to \Pic(H)$ is an isomorphism. Therefore, a general member
$H\in|{-}K_X|$ does not contain curves of degree $4$ and so the image 
$\operatorname{pr}_1(\mathfrak{R})$ is not dense in $|H|$.
Thus
$$
g+1=\dim|H|>\dim \operatorname{pr}_1(\mathfrak{R})=\dim\mathfrak{C}+g-3.
$$
Therefore, $\dim\mathfrak{C}\le 3$ and so the map
$$
\mathfrak{C}\longrightarrow X,\qquad C\longmapsto \Sing(C)
$$
either has a two-dimensional image or
is generically finite.
\end{proof}

\begin{lem}
The map $\phi \colon \widetilde{X} \dashrightarrow \PP^{g-3}$ induced by the
projection $\psi=\Phi_{|H-2P|}\colon X \dashrightarrow \PP^{g-3}$
from the tangent space to $X$ at the point $P$ is a morphism.
\end{lem}

\begin{proof}
It is clear that
the map $\phi \colon \widetilde{X} \dashrightarrow \PP^{g-3}$ is given by the
linear system $|H^*-2E|$ and this linear system is base point free outside 
the exceptional divisor $E$.
Let us prove that $|H^*-2E|$ is also base point free on $E$ as well.
To show this we consider the affine chart
$\A^{g+1}\subset\PP^{g+1}$ with the origin at the point $P$.

Let $I\subset\CC[x_1,\dots,x_{g+1}]$
be the ideal defining the affine variety $X\cap\A^{g+1}$.
Let $I_2\subset I$ be the subspace consisting of all elements of degree $\le 
2$.
Thus every element $s\in I_2$ is represented in the form
$$
s=\ell_s(x_1,\dots,x_{g+1})+q_s(x_1,\dots,x_{g+1}),
$$
where $\ell_s$ and $q_s$ are homogeneous forms of degree~$1$ and~$2$,
respectively.
By
Theorem~\ref{theorem-1.1}\ref{theorem-1.1:int:q} the ideal $I$ is generated by 
the
space $I_2$.
Here, the equations
$\ell_s=0$, where $s\in I_2$ cut out $T_{P,X}$ in~$\A^{g+1}$.

There is a natural isomorphism
$$
H^0(\widetilde{X},\, \OOO_{\widetilde{X}}(H^*-2E))\simeq
H^0(X,\, \OOO(H)\otimes \mathfrak{m}_P^2).
$$
A non-zero section of $H^0(X,\, \OOO(H)\otimes \mathfrak{m}_P^2)$ corresponds
a linear function $\ell$ on $\A^{g+1}$ vanishing on $T_{P,X}$.
Hence, $\ell=\ell_s$ for some $s\in I_2$ and so
$\ell=\ell_s=-q_s$ on~$X$.
Thus the restriction of the non-zero section
$\ell\in H^0(\widetilde{X},\, \OOO_{\widetilde{X}}(H^*-2E))$ to
the exceptional divisor $E$ is a quadratic function of the form $q_s|_E$
for some $s\in I_2$. The inverse is also true: any
quadratic function of the form $q_s|_E$, $s\in I_2$ is 
the restriction of a global section of the line bundle
$\OOO_{\widetilde{X}}(H^*-2E)$.

Since there are no lines passing through the point $P$, we have
$T_{P,X}\cap X=\{P\}$ (set-theoretically).
Hence, the common zero-locus in $T_{P,X}$ of all forms $q_s$, $s\in I_2$ 
consists only of zero.
This means that the common zero-locus on $E=\PP(T_{P,X})$ of all
global sections of the line bundle $\OOO_{\widetilde{X}}(H^*-2E)$
is empty.
Therefore, $|H^*-2E|$ is base point free on $E$.
\end{proof}

\begin{cor}
The linear system $|{-}K_{\widetilde{X}}|$ is base point free,
$\dim|{-}K_{\widetilde{X}}|=g-3$, and the image $\overline{X}=\phi(X)$ of
the morphism
$$
\phi=\Phi_{|{-}K_{\widetilde{X}}|}\colon \widetilde{X}\longrightarrow\PP^{g-3},
$$
defined by this linear system is three-dimensional.
\end{cor}

\begin{notation}
\label{notation:X2g-2:point-1}
Consider the induced morphism
$$
\phi=\Phi_{|H^*-2E|} \colon \widetilde{X} \longrightarrow X_{\bullet}\subset\PP^{g-3}.
$$
The surface $\phi(E)$ is the image of the projection of the Veronese surface
$S_4\subset\PP^5$.
\end{notation}

\begin{lem}
\label{P2:dim}
One has
\begin{equation}
\label{eqP2:dim}
\hr^0\bigl(\widetilde{X},\,\OOO_{\widetilde{X}}(-K_{\widetilde{X}}-E)\bigr)\ge 
g-8,
\end{equation}
and
the linear span of the surface $\phi(E)$ has dimension
\begin{equation}
\label{eqP2:dim1}
3\le \dim \langle \phi(E)\rangle\le g-3.
\end{equation}
If furthermore the equality 
$$
\hr^0\bigl(\widetilde{X},\,\OOO_{\widetilde{X}}(-K_{\widetilde{X}}-E)\bigr)=g-8
$$
holds,
then the restriction $|{-}K_{\widetilde{X}}|\bigr|_E$ is a complete linear 
system that
defines a morphism of
$E$ to the Veronese surface $S_4\subset\PP^5$.
\end{lem}

\begin{proof}
Consider the exact sequence
$$
0\longrightarrow H^0(\OOO_{\widetilde{X}}(-K_{\widetilde{X}}-E))\longrightarrow
H^0(\OOO_{\widetilde{X}}(-K_{\widetilde{X}}))
\xarr{\nu} H^0(\OOO_E(-K_{\widetilde{X}})).
$$
Since $\hr^0(\OOO_E(-K_{\widetilde{X}}))=6$, from this we get
\eqref{eqP2:dim}.
The remaining assertions are obtained from the fact that 
$|{-}K_{\widetilde{X}}|\bigr|_E$
is a
subsystem in the linear system of conics on $E\simeq \PP^2$ and it is a
projectivization the image of the map $\nu$.
\end{proof}

\begin{lem}
\label{projection-Veronese}
Let $S=S_4\subset\PP^5$ be the Veronese surface.
\begin{enumerate}
\item
\label{projection-Veronesei}
Let $S^{\#}\subset\PP^4$ be its
projection from a point $P\in\PP^5\setminus S$.
Then for the surface $S^{\#}\subset\PP^4$
there is exactly two possibilities:
\begin{enumerate}
\item
\label{projection-Veroneseia}
the singular locus of $S^{\#}$ consists of a line, $S^{\#}$ is the an 
intersection of
two quadrics $Q_1$ and $Q_2$,
and any quadric from the pencil $\langle Q_1,\, Q_2\rangle$ is singular;
\item
\label{projection-Veroneseib}
$S^{\#}$ is nonsingular and is not contained in a quadric.
\end{enumerate}
\item
\label{projection-Veroneseii}
Let $S^{\#}\subset\PP^3$ be the
projection of the surface $S_4\subset\PP^5$ from a line $l\subset\PP^5\setminus 
S$.
Then for $S^{\#}\subset\PP^4$ there is two possibilities as well:
\begin{enumerate}
\item
$S^{\#}\subset\PP^3$ is a Steiner surface, i.e. surface of degree $4$ whose
singular locus consists of three \textup(possibly coinciding\textup) lines;
\item
$S^{\#}\subset\PP^3$ is singular quadric.
\end{enumerate}
\end{enumerate}
\end{lem}

\begin{proofsk}
Identify the space $\PP^5$ with projectivization of the space $M_2$ of
quadratic forms $q(x_0,x_1,x_2)$ in three variables. Then the surface
$S_4\subset\PP^5$ is identified with the set of
forms of rank~$1$, i.e. forms $\ell(x_0,x_1,x_2)^2$, where 
$\ell(x_0,x_1,x_2)$
is a non-zero linear form. A line in $\PP^5=\PP(M_2)$ intersecting $S_4$ at
two points $[\ell^2]$ and $[\ell'^2]$ has the form
$$
[\lambda\ell^2+\lambda'\ell'^2], \qquad (\lambda:\lambda')\in\PP^1.
$$
Therefore, the union of all bisecant lines (and tangent lines) to the surface
$S_4\subset\PP^5$ is a hypersurface
$$
W=\{[q] \mid \rk q \le 2\}=\{[q] \mid \det (q)=0\}
$$
of degree~$3$. If the point $P$ does not lie on this hypersurface, then the 
projection
$S_4\to S^{\#}$ from $P$ is an isomorphism. On the other hand, if the point $P$ 
lie in 
$W\setminus 
S_4$, then it has the form $[\ell\ell']$, where $\ell$ and $\ell'$ are 
non-proportional linear forms. Then the plane $\Pi$ generated by three points
$P=[\ell\ell']$, $[\ell^2]$, $[\ell'^2]$ is contained in $W$ and intersects 
$S_4$ by a conic. The projection from $P$ maps $S_4$ to a surface 
$S^{\#}\subset
\PP^4$ singular along a line that is the image of this conic. For the assertion 
on~quadrics
passing through $S^{\#}$, see Exercise~\ref{zad:Veronese} at the end of 
this section.
The assertion~\ref{projection-Veronesei} is proved.

To prove~\ref{projection-Veroneseii} we note that if the line $l$
is contained in $W$, then it is not contained in the plane $\Pi\subset W$ 
(because 
does not meet $S_4$).
Therefore, it meets a one-dimensional family of planes in $W$
and in this case the projection from $l$ will be a finite map of degree $2$ to a 
quadric
in $\PP^3$.

Finally, let $l\not\subset W$. Then $l$ intersects
$W$ by three points (counted with multiplicities). There are three planes 
$\Pi_i$ passing through these points and
contained in $W$. As above, they give us three lines of singularities on
$S^{\#}$.
\end{proofsk}

\begin{lem}
\label{lemma:FanoIII:n-bir}
The morphism $\phi\colon \widetilde{X} \to X_{\bullet }\subset\PP^{g-3}$ is 
birational and
$X_{\bullet }$ is a Fano threefold with canonical Gorenstein
singularities.
\end{lem}

\begin{proof}
Assume that the morphism $\phi$ is not birational. Since
$$
(\deg X_{\bullet })\cdot (\deg \phi)=(-K_{\widetilde{X}})^3=2g-10,
$$
according to~\eqref{equation:LDV-inequality},
the morphism $\phi$ must be of degree~$2$ and $X_{\bullet}$ must be a variety
of minimal degree $\deg X_{\bullet}=g-5$.

If its hyperplane section $H_\bullet$ admits a decomposition in a sum of
movable Weil divisors $H_{\bullet}'$ and $H_{\bullet}''$, then the
anticanonical divisor
$-K_{\widetilde{X}}=\phi^* (H_{\bullet}'+H_{\bullet}'')$ on $\widetilde{X}$ also
admits a decomposition in a sum of movable divisors. Therefore, the same
holds and for the
anticanonical divisor $-K_X=\sigma_*(-K_{\widetilde{X}})$ on $X$.
But this contradicts our assumption $\Pic(X)=\ZZ\cdot K_X$.

Then according to Corollary~\ref{corollary:varieties-minimal-degree} the variety
$X_{\bullet }\subset\PP^4$ is a nonsingular quadric and $g=7$.
In particular, $\Pic(X_{\bullet })$ is generated by the class of a hyperplane 
section
and so the surface $\phi(E)$ is cut out on~$X_{\bullet }$ by some hypersurface.
By Lemma~\ref{projection-Veronese} we have the only possibility:
$\phi(E)$ is singular quadric that is the tangent hyperplane section of the 
quadric
$X_{\bullet }\subset\PP^4$. Consider the Stein factorization 
$$
\phi \colon \widetilde{X} \xarr{\theta} X_0 \xarr{\gamma} 
X_{\bullet}\subset\PP^4.
$$
Then $\theta(E)=\gamma^* \phi(E)$ is a Cartier divisor on $X_0$ and the
singularities $X_0$ are canonical Gorenstein (see
Corollary~\ref{cor:SL:singX0c}).
Since the group $\Cl(X_0)$ is generated by the divisors $\theta(E)$ and 
$K_{X_0}$, the variety $X_0$ is locally factorial. Thus the morphism $\theta$ 
cannot be small
(by the Purity of Exceptional Locus Theorem
\cite[Ch.~II,~\S4,~Theorem~2]{Shafarevich:basic}). Since
$\theta(E)=\gamma^* \phi(E)\sim -K_{X_0}$, we have $\uprho(X_0)=1$. Thus 
$\theta$ cannot be an isomorphism.
Therefore, $\theta$ contracts some divisor $D$. From the fact that
$\theta(E)\sim -K_{X_0}$ it immediately follows that $D\sim 
-K_{\widetilde{X}}-E$.
Since $(-K_{\widetilde{X}})\cdot D^2\neq 0$, the image $\phi(D)$ is a curve
lying on
$\phi(E)$. The morphism $\phi_E\colon E\to \phi(E)$ is finite of degree~$2$. 
Thus 
for a general
fiber $\widetilde{C}$ of 
the morphism $\phi_D\colon D\to \phi(D)$ we have $E\cdot \widetilde{C}\le 2$.
Let $C:= \sigma(\widetilde{C})$. Then $\deg C=H^*\cdot \widetilde{C}=2E\cdot
\widetilde{C}$.
If $E\cdot \widetilde{C}=1$, then $C$ is a conic passing through $P$.
This contradicts our choice of the point $P\in X$ (the number such conics must
be
finite). Hence, $E\cdot \widetilde{C}=2$ and $C$ is a curve of degree $4$, 
having a
singularity at $P$. Again, according to our choice of the point $P\in X$
and Lemma~\ref{lemma4.5.2}, this is impossible.

Therefore, the morphism $\phi$ is birational. The last assertion is obtained 
from
Proposition~\ref{proposition:nef-big} and Corollary~\ref{cor:SL:singX0c}.
\end{proof}

\begin{lem}
\label{linear-system:point-fibers}
The morphism $\phi$ does not contract any divisors and the singularities of
$X_\bullet$ are terminal.
\end{lem}

\begin{proof}
Assume that the morphism $\phi$ contracts a prime divisor $\widetilde{D}$.
Since $\psi=\phi\comp \sigma^{-1}\colon X \dashrightarrow X_\bullet\subset
\PP^{g-3}$
is the projection from $\overline{T_{P,X}}$, the image $D:= 
\sigma(\widetilde{D})$ is cut out on $X$ by the
linear span $\langle \overline{T_{P,X}},\,\phi(\widetilde{D}) \rangle$.
Thus $\dim \langle \overline{T_{P,X}},\,\phi(\widetilde{D}) \rangle\ge g$
and $\dim \langle \phi(\widetilde{D}) \rangle\ge g-4$.
In particular, this means that $\phi(\widetilde{D})$ is a curve of degree $\ge 
3$.

The divisor $\widetilde{D}$ is covered by a one-dimensional family of curves
$\widetilde{C}$ such that $K_{\widetilde{X}}\cdot \widetilde{C}=0$.
Since $(-K_{\widetilde{X}}-E)\cdot \widetilde{C}<0$, the divisor
$\widetilde{D}$ must be a fixed component of the linear system
$|{-}K_{\widetilde{X}}-E|$ (if it is non-empty). Since $\Pic(X)=\ZZ\cdot K_X$ 
and
$\widetilde{D}\neq E$ (because $(-K_{\widetilde{X}})^2\cdot E=4>0$),
it follows that $\dim|{-}K_{\widetilde{X}}-E|\le 0$.
From the inequalities~\eqref{eqP2:dim} and~\eqref{eqP2:dim1} we obtain that 
$g\le 9$ and the
linear span of the image $\phi(E)$ has dimension
at least $3$.
By our assumption there are at most a finite number of
conics passing through the point $P$.
Therefore, the surface $\widetilde{D}$ is covered by curves
$\widetilde{C}$ such that $K_{\widetilde{X}}\cdot\widetilde{C}=0$ and $E\cdot
\widetilde{C}>1$. Thus
the image of the exceptional divisor $\phi(E)$ must be singular along an 
irreducible
curve
$\phi(\widetilde{D})$ of degree $\ge 3$.
This is impossible by Lemma~\ref{projection-Veronese}. Now the last
assertion is obtained from Lemma~\ref{lemma:FanoIII:n-bir} because the morphism
$\phi$ does not contract any divisors.
\end{proof}

\begin{cor}
\label{corollary:V}
If $g=7$, then $\phi(E)$ is not contained in a quadric. Therefore,
$|{-}2K_{\widetilde{X}}-E|=\varnothing$.
\end{cor}

\begin{proof}
By Lemma~\ref{lemma:FanoIII:n-bir} the morphism $\phi\colon \widetilde{X}\to 
X_\bullet
\subset\PP^4$ is birational and $X_\bullet$ is a hypersurface of degree 4.
Moreover, by Lemma~\ref{linear-system:point-fibers} the fibers of $\phi$ are
connected and the exceptional locus of $\phi$ consists of at most of a finite 
number of curves.
Therefore, the surface $\phi(E)$ has at most a finite number of singular 
points.
Then by Lemma~\ref{projection-Veronese} the surface $\phi(E)$ actually
is nonsingular and is not contained in a quadric. Finally, if
$\widetilde{D}\in|{-}2K_{\widetilde{X}}-E|$, then $\phi(\widetilde{D})$ is 
a quadric
passing through $\phi(E)$. Thus $|{-}2K_{\widetilde{X}}-E|=\varnothing$.
\end{proof}

\subsection{The link}
\label{point-coini}

Let us proceed to the proof of~\ref{th:main:point}.
According to Lemmas~\ref{lemma:FanoIII:n-bir}
and~\ref{linear-system:point-fibers},
there exists
a Sarkisov link with center $P$.
We obtain the diagram~\eqref{diagram}
(or~\eqref{diagram-v}).
Consider the possibilities for the 
contraction $\varphi$.
Recall (see~\eqref{eqP2:dim}) that
$$
\dim|{-}K_{\overline{X}}-\overline{E}|\ge g-9.
$$
As in the proof of Theorem~\ref{th:double-projection} we use
the classification of extremal rays (Theorem~\ref{class:ext-rays}).
Besides solutions listed in Theorem~\ref{th:main:point} we obtain the following 
cases:
\begin{enumerate}
\item
\label{poinblowup:ex:g=13}
$\g(X)=13$, $Y=Q\subset\PP^4$ is a nonsingular quadric, $\varphi$ is the blowup 
of a
rational curve of degree $6$, $\overline{F}\sim 
2(-K_{\overline{X}})-3\overline{E}$, and
$\overline{M}-K_{\overline{X}}-\overline{E}$;

\item
\label{poinblowup:ex:g=11}
$\g(X)=11$, $Y=\PP^2$, $\varphi$ is a conic bundle whose discriminant curve
has degree $4$, and $\overline{M}\sim-K_{\overline{X}}-\overline{E}$;

\item
\label{poinblowup:ex:g=9}
$\g(X)=9$, $Y=Y_{14}\subset\PP^9$ is a Fano threefold of index~$1$ and genus 
$8$,
$\varphi$ is the blowup of a conic, $\overline{F}\sim 
-K_{\overline{X}}-\overline{E}$, and
$\overline{M}\sim2(-K_{\overline{X}})-\overline{E}$;

\item
\label{poinblowup:ex:g=8}
$\g(X)=8$, $Y=Q\subset\PP^4$ is a nonsingular quadric, $\varphi$ is the blowup 
of a
curve of degree $10$ and genus $7$, $\overline{F}\sim 
5(-K_{\overline{X}})-3\overline{E}$, and
$\overline{M}\sim2(-K_{\overline{X}})-\overline{E}$;

\item
\label{poinblowup:ex:g=7}
$\g(X)=7$, $Y=Y_{10}\subset\PP^7$ is a Fano threefold of index~$1$ and genus 
$6$,
$\varphi$ is the blowup of a conic, $\overline{F}\sim 
2(-K_{\overline{X}})-\overline{E}$, and
$\overline{M}\sim3(-K_{\overline{X}})-\overline{E}$;

\item
\label{th:main:point7b}
$\g(X)=7$, $Y=Y_{12}\subset\PP^8$ is a Fano threefold of index~$1$ and genus 
$7$,
$\varphi$ is the blowup of a point, $\overline{F}\sim 
2(-K_{\overline{X}})-\overline{E}$, and
$\overline{M}\sim5(-K_{\overline{X}})-2\overline{E}$.
\end{enumerate}
Below we show that these cases do not occur.

\begin{cor}
\label{corollary:point-flop}
In the conditions of Theorem~\ref{th:main:point} the map $\chi$ in 
diagram~\eqref{diagram} is not an isomorphism.
\end{cor}

\begin{proof}
As in the proof of~\ref{corollary:line-flop} it follows from computations that
the defect $\deff(\Psi):= E^3-\overline{E}^3$ of the link~\eqref{diagram} is 
strictly positive
(see Exercise~\ref{zad:FanoIII:def} at the end of this section).
\end{proof}

\begin{cor}
\label{cor:exist:line+conic}
Let $X=X_{2g-2}\subset\PP^{g+1}$ be an anticanonically embedded 
Fano threefold
of index~$1$ and genus $g\ge 7$. Then there exists a nondegenerate conic
and
a line on $X$.
\end{cor}

\begin{proof}
Since the variety $X_\bullet$ is normal, the morphism $\phi$ is birational 
and
does not contract any divisors, the surface $\phi(E)$ either is nonsingular or 
has only isolated singularities (see Lemma~\ref{linear-system:point-fibers}).
The latter alternative is impossible by Lemma~\ref{projection-Veronese}.
Therefore, the surface $\phi(E)$ is nonsingular.

According to~\ref{corollary:point-flop},
the morphism $\phi$
contracts at least one (irreducible) curve $\widetilde{C}$.
Since $K_{\widetilde{X}}\cdot \widetilde{C}=0$, we have
$\deg \sigma(\widetilde{C})=2 E\cdot \widetilde{C}$.
Since the surface $\phi(E)$ is nonsingular, we have $E\cdot \widetilde{C}=1$.
This means that $\deg \sigma(\widetilde{C})=2$, i.e. $\sigma(\widetilde{C})$ is 
a nondegenerate conic.
\end{proof}

\begin{proof}[Proof of Theorem~\ref{th:main:point}]
It remains to exclude only
the cases~\ref{point-coini},~\ref{poinblowup:ex:g=13}--\ref{th:main:point7b}.
But now according to Corollary~\ref{cor:exist:line+conic} we know that on $X$
there exists a line.
Then by Theorem~\ref{th:double-projection}
the cases~\ref{poinblowup:ex:g=13} and~\ref{poinblowup:ex:g=11}
(i.e. $g=13$ and $11$) are impossible.
In the cases~\ref{poinblowup:ex:g=9} and~\ref{poinblowup:ex:g=8} from
Table~\ref{point-coini}
we obtain a contradiction comparing the topological Euler numbers
$X$ and $Y$ (and applying Corollary~\ref{Euler-numbers}).

Consider the cases~\ref{poinblowup:ex:g=7} and~\ref{th:main:point7b}.
Then $\overline{F}+\overline{E} \sim 2(-K_{\overline{X}})$.
Here the proper transform $\widetilde{F}\subset\widetilde{X}$ of the 
divisor $\overline{F}$
is the unique element of the linear system $|{-}2K_{\widetilde{X}}-E|$.
This contradicts Corollary~\ref{corollary:V}.
\end{proof}

\begin{zadachi}
\eitem
Show that the image of the morphism $\Phi_{|{-}K_{\widetilde{X}}|}$ is a normal
variety.
\hint{Use Theorem~\ref{theorem-A-Saint-Donat}.}

\eitem
Let $X=X_d\subset\PP^{d+1}$ be a del Pezzo threefold of degree $3\le d\le
5$.
Prove that there exists a Sarkisov link
with center a point $P\in X$ and describe the result.

\eitem
Let $X=X_{10}\subset\PP^7$ be an anticanonically embedded Fano threefold with 
$\iota(X)=1$, $\uprho(X)=1$, and $\g(X)=6$.
Let $P\in X$ be a sufficiently general point.
Prove that there exists a Sarkisov link
\eqref{diagram} with center~$P$, where
$Y=Y_{10}\subset\PP^7$ is also a Fano threefold
of index~$1$ and genus $6$ with $\uprho(X)=1$ and $\varphi$ is the blowup of a 
point.

\eitem
\label{zad:FanoIII:def}
Compute the defect $\deff(\Psi):= E^3-\overline{E}^3$ of the 
link~\eqref{diagram} in 
Theorem~\ref{th:main:point}
and prove that it is strictly positive.
\hint{Compute $\overline{M}^3=(-K_{\overline{X}}-\overline{E})^3$ in two 
ways.}

\eitem
\label{zad:Veronese}
Prove the non-proved parts of Lemma~\ref{projection-Veronese}: in the cases 
\ref{projection-Veroneseia} and~\ref{projection-Veroneseib} the surfaces are
unique up to projective equivalences and find defining
equations.
\hint{Use the fact that the Veronese embedding induces the action of the group 
$\GL_3(\CC)$
on $\PP^5$ so that there exists exactly three orbits.}
\end{zadachi}

\newpage\section{Fano threefolds of Picard 
number~$1$.~IV}
\label{sect:Fano-IV}

\subsection{Central curve of the link}

In this and the next section we prove the existence of Fano threefolds of 
the main
series and genus $g\ge 7$.
Moreover, we will give a way to construct \textit{all} such varieties.
First we discuss the cases $g=12$, $10$ and $9$. We will use the
construction that is
inverse to the double projection from a line.

\begin{notation}
\label{notation:IV}
The starting point for constructing a Sarkisov link is a pair $(Y,Z)$, where 
$Y$ is a
Fano threefold of index $\iota(Y)>1$ with $\uprho(Y)=1$ and $Z\subset Y$ is a
smooth curve of genus $\g(Z)$ and degree $\deg Z:= (-K_Y\cdot Z)/\iota(Y)$.
We will consider the cases, presented in Table~\ref{table11.1} (see
Theorem~\ref{th:double-projection}\ref{g=9-12}), where $d:= (\iota(Y)-1)(\deg 
Y)$
(this notation will be used below).
\begin{table}[ht]\small
\caption{Pairs $(Y,Z)$}
\label{table11.1}
\begin{center}\def\arraystretch{1.3}
\begin{tabular}{|c|c|c|c|c|c|c|}\hline
\textnumero&$g$ & \heading{$Y$} &$\iota(Y)$ &$\deg Z$ & $\g(Z)$ & $d$
\\\hline\setcounter{NN}{8}
\nr
\label{constr:9}
&$9$ & $\PP^3$& $4$ &$7$ & $3$ & $3$
\\
\nr
\label{constr:10}
&$10$ &a quadric $Q\subset\PP^4$& $3$ &$7$ &$2$ & $4$
\\\refstepcounter{NN}
\nr
\label{constr:12}
&$12$&a del Pezzo threefold $Y_5\subset\PP^6$ & $2$ & $5$ & $0$ & $5$ \\ \hline
\end{tabular}
\end{center}
\end{table}
We assume that in the case~\ref{constr:9} the curve $Z$ \textit{is not 
hyperelliptic} (see Theorem~\ref{th:double-projection}\ref{g=9}).
By $M$ we denote the ample generator of the group $\Pic(Y)\simeq \ZZ$.
\end{notation}

\begin{prp}
\label{lemma:surfaceG}
In the above notation the following assertions hold.
\begin{enumerate}
\item
\label{lemma:surfaceG:1}
The dimension of the linear span of the curve $Z$ equals $d$.
\item
\label{lemma:surfaceG:2}
The curve $Z$
is contained in an irreducible surface
$S\in |(\iota-1)M|$ of degree $d$, where $\iota:= \iota(Y)$.
\item
\label{lemma:surfaceG:3}
The surface $S=S_d\subset\PP^d$ is not a cone, it is rational
and belongs to one of the following types:
\begin{enumerate}
\item
\label{lemma:surfaceG:3a}
$S=S_d\subset\PP^d$ is a normal, anticanonically embedded del Pezzo surface with 
Du Val
singularities \textup(or smooth\textup);
\item
\label{lemma:surfaceG:3b}
$S=S_d\subset\PP^d$ is a non-normal ruled surface covered by $(6-d)$-secant 
lines to
the curve~$Z$. In this case the singular locus of $S=S_d\subset\PP^d$ is a
$(7-d)$-secant line to $Z$.
\end{enumerate}

\item
\label{lemma:surfaceG:unique}
The surface $S\in (\iota-1)M$ containing $Z$ is unique.

\item
\label{lemma:surfaceG:curve}
The curve $Z\subset\PP^d$ is an intersection of hypersurfaces of degree $7-d$.
\end{enumerate}
\end{prp}

\begin{proof}\ref{lemma:surfaceG:1}
Assume that $Z$ is contained in a subspace $\PP^{d-1}$.
In the case~\ref{constr:9} this contradicts the genus formula of the
plane curve
$Z_7\subset\PP^2$. In the case~\ref{constr:10} the curve $Z$ must be
contained in a two-dimensional quadric $Q\cap\PP^3$ (possibly singular).
It is easy to check that this is also impossible.
Finally, in the case~\ref{constr:12} by the degree reason of we obtain that
$Z=Y\cap\PP^4$. But then by the adjunction formula $K_Z=0$,
i.e. $\g(Z)=1$. Again we have a contradiction.

\ref{lemma:surfaceG:2}
Consider the exact sequence
\begin{equation}
\label{eq:exact-seq:ideal}
0 \longrightarrow\mathscr{J}_Z(\iota-1) \longrightarrow\OOO_Y(\iota-1) \longrightarrow\OOO_Z(\iota-1)\longrightarrow 0,
\end{equation}
where $\mathscr{J}_Z$ is the ideal sheaf of the curve $Z$ in $Y$.
Note that $\hr^0(Y,\,\OOO_Y(\iota-1))=20$, $14$ and $7$ in the cases 
\ref{constr:9},~\ref{constr:10} and~\ref{constr:12}, respectively and, at the 
same time,
$\hr^0(Z,\,\OOO_Z(\iota-1))=19$, $13$ and $6$ in these cases.
From the corresponding long exact sequence of cohomology it
follows that
$$
H^0(\mathscr{J}_Z(\iota-1))\neq 0,
$$
i.e. there exists a surface $S\in |(\iota-1)M|$ passing through $Z$.
Moreover, such a surface is irreducible. This is obvious in the cases 
\ref{constr:12} and~\ref{constr:10},
and in the case~\ref{constr:9} it
follows from the fact that $Z$ does not lie on a quadric (it is suggested to the
reader to fulfill these simple computations on his own).

\ref{lemma:surfaceG:3}
Assume that $S$ is a cone over some curve $\Gamma$ in $\PP^{d-1}$.
In the case~\ref{constr:12} there are only a finite number
of lines passing 
through any point of the variety 
$Y=Y_5\subset\PP^6$
(see
Proposition~\ref{proposition-index=2-lines}\ref{proposition-index=2-lines-pt})
and so $Y$ does not contain any cones. In the case~\ref{constr:10} any
cone on
the quadric $Y=Q\subset\PP^4$ is cut out by tangent hyperplane section, and
hence this case also is impossible.

First, consider the case~\ref{constr:9}.
Then $\Gamma\subset\PP^2$
is an irreducible cubic curve. Assume first that it is nonsingular.
Let $\sigma\colon \widehat{S}\to S$ be the blowup of the vertex. Then 
$\widehat{S}$ 
has
a structure of a
ruled surface over $\Gamma$ and the exceptional divisor $\Sigma\subset
\widehat{S}$ is 
its section. Let $A$ be a hyperplane section of $S$ and let $A^*:= \sigma^*A$.
We have standard relations:
$$
(A^*)^2=3,\quad A^*\cdot \Sigma=0,\quad (K_{\widehat{S}}+\Sigma)\cdot \Sigma=0,
\quad (K_{\widehat{S}}+A^*)\cdot A^*=0.
$$
Since classes of the divisors $\Sigma$ and $A^*$ generate the space 
$\N^1(\widehat{S})$, from this we obtain
$$
K_{\widehat{S}}+\Sigma+A^*\approxident 0.
$$
By Noether's formula
$$
0=K_{\widehat{S}}^2=(\Sigma+A^*)^2.
$$
Hence, $\Sigma^2=-3$.
Furthermore, let $\widehat{Z}$ be the proper transform of $Z$. Then 
$A^*\cdot \widehat{Z}=7$.
If $Z$ does not pass through the vertex of the cone, then
the numerical class of $\widehat{Z}$ is proportional to the class of $A^*$, but 
this is 
impossible
because $(A^*)^2=3$.
Therefore, $\widehat{Z}\cdot \Sigma>0$. Since the curve
$Z$ is nonsingular, we have $\widehat{Z}\cdot \Sigma=1$. Hence, 
$\widehat{Z}\approxident 
\frac 73
A^*-\frac13 \Sigma$. Then by the adjunction formula
$$
2\g(Z)-2=2\g(\widehat{Z})-2=(K_{\widehat{S}}+\widehat{Z})\cdot \widehat{Z}=8, 
\quad \g(Z)=5.
$$
The contradiction shows that the surface $S$ cannot be a cone over a nonsingular
elliptic curve.

Since the variety $Y$ is nonsingular and $-K_Y=\iota M$, by 
the adjunction formula the anticanonical divisor $-K_S$ is an ample Cartier 
divisor. 
If the surface~$S$ is normal and is not a cone, then it
is rational and has at worst Du Val singularities (this is a classical fact,
a modern exposition see for instance in~\cite{Hidaka-Watanabe-1981}). We obtain
the case~\ref{lemma:surfaceG:3a}.

Assume that the surface $S$ is not normal.
It is known~\cite[\S8.1]{Dolgachev-ClassicalAlgGeom} that $S$ is a the image 
of the projection of a surface
$S'\subset\PP^{d+1}$ of minimal degree (see
Proposition~\ref{varieties-minimal-degree}) from a point that does not lie on 
$S'$. 
Here the morphism $\nu\colon S'\to S$ coincides with the normalization. It is 
clear 
that $S'\subset
\PP^{d+1}$ cannot be the Veronese surface (because in this case $d=4$ and the 
degree of any
curve on $S'$ must be even). Hence, according to
Proposition~\ref{varieties-minimal-degree}, the surface $S'\subset\PP^{d+1}$
is the image of a rational ruled surface $\widehat{S}=\FF_e$ under the map 
given by the linear system
$\LLL:= |\Sigma+n\Upsilon|$, where $\Sigma$ and $\Upsilon$ are classes of the 
minimal
section and the fiber of $\FF_e$, respectively. Since $\deg S'=d$, we have the 
following 
relations:
$$
(\Sigma+n\Upsilon)^2=2n-e=d,\quad (\Sigma+n\Upsilon)\cdot \Sigma=n-e\ge 0.
$$
Therefore, $d\ge n\ge e$.
Moreover, as the surface $S'$ is not a cone for $d\neq 3$,
the strict inequality $n> e$ holds in this case. We obtain the following
possibilities for $(d,e,n)$:
$$
(3, 3, 3),\ (3, 1, 2),\ (4, 0, 2),\ (4, 2, 3),\ (5, 1, 3),\ (5, 3, 4).
$$
Furthermore, let $\widehat{Z}$ be the proper transform of $Z$ on 
$\widehat{S}=\FF_e$.
Write $\widehat{Z}\sim a\Sigma+b\Upsilon$. Then
$$
\begin{aligned}
\widehat{Z}\cdot \Sigma&= b-ea\ge 0,
\\[3pt]
\deg Z&=\widehat{Z}\cdot (\Sigma+n\Upsilon)=-ae+b+na,
\\[3pt]
2\g(Z)-2&=(K_{\FF_e}+\widehat{Z})\cdot \widehat{Z}=-ea (a-2)+a(b-e-2)+b(a-2).
\end{aligned}
$$
Since for $d=3$ the curve $Z$ is not hyperelliptic,
$a\ge 3$ in this case. Now, it is not difficult to obtain all the 
possibilities for 
$(d,e,n,a,b)$:
$$
(3, 1, 2, 3, 4),\quad
(4, 0, 2, 2, 3),\quad
(4, 2, 3, 2, 5),\quad
(5, 1, 3, 1, 3),\quad
(5, 3, 4, 1, 4).
$$
In particular, $(d,e,n)\neq (3, 3, 3)$, i.e. $S'=\widehat{S}=\FF_e$ and $S$ is 
not a cone.

Since the surface $S$ is a locally complete intersection and is not 
normal, its singular locus must have a one-dimensional component. 
By Bertini's theorem and the adjunction formula a general hyperplane section 
$S\cap\PP^{d-1}$ is an irreducible curve of arithmetic genus~$1$, which must be
singular at the points of $\Sing(S)\cap\PP^{d-1}$. There is exactly one such a 
point and so the
one-dimensional component of the set $\Sing(S)$ is a line. Denote it by
$\Lambda$.
In all cases the surface $S'\subset\PP^{d+1}$ is an intersection of quadrics
(see Proposition~\ref{cor:VMD:R}). Thus the inverse image $\Lambda'\subset S'$ 
of
the line $\Lambda$ under the projection $S' \to S$ is a conic (possibly
reducible) lying in plane generated by $\Lambda$ and center the projection.
Thus $\LLL\cdot \Lambda'=2$. Now it is not difficult to compute and the class 
of the curve
$\Lambda'$. The answer is recorded in Table~\ref{non-normal:G}.

\begin{table}[ht]\small
\caption{The class of $\widehat{Z}$ on the normalization of $S$}

\label{non-normal:G}
\begin{center}\def\arraystretch{1.3}
\begin{tabular}{|c|c|c|c|c|c|} \hline
&$d$& $S'$& $\LLL$& The class $\widehat{Z}$ & The class $\Lambda'$
\\\hline
\setcounter{NNr}{0}\nrr&$5$& $\FF_1$& $|\Sigma+3\Upsilon|$& $\Sigma+3\Upsilon$& 
$\Sigma$
\\
\nrr&$5$& $\FF_3$& $|\Sigma+4\Upsilon|$& $\Sigma+4\Upsilon$& $\Sigma+\Upsilon$
\\
\nrr&$4$& $\FF_0$& $|\Sigma+2\Upsilon|$& $2\Sigma+3\Upsilon$& $\Sigma$
\\
\nrr&$4$& $\FF_2$& $|\Sigma+3\Upsilon|$& $2\Sigma+5\Upsilon$& $\Sigma+\Upsilon$
\\
\nrr&$3$& $\FF_1$& $|\Sigma+2\Upsilon|$& $3\Sigma+4\Upsilon$& $\Sigma+\Upsilon$ 
\\ \hline
\end{tabular}
\end{center}
\end{table}

In all cases $\widehat{Z}\cdot \Upsilon=a=6-d$ and $\Sigma\cdot
(\Sigma+n\Upsilon)>0$. This means that the images of fibers of~$\Upsilon$ are
$(6-d)$-secant lines of~$Z$. Since $\widehat{Z}\cdot \Lambda'=7-d$, the line 
$\Lambda$ is 
$(7-d)$-secant of~$Z$. The assertion~\ref{lemma:surfaceG:3} is proved.

\ref{lemma:surfaceG:unique}
Let us prove the uniqueness of the surface $S$.
In the case~\ref{constr:12} this is obvious because otherwise
$\dim \langle Z\rangle=4$ and $Z=Y\cap\langle Z\rangle$ by the degree reason. 
But then $Z$ must be an elliptic curve by 
the adjunction formula, a contradiction.

Consider the case~\ref{constr:10}. Assume that there are two
surfaces $S_1,\,S_2\in|2M|$ passing through $Z$. Then $S_1\cap S_2=Z+\Lambda$, 
where
$\Lambda\subset Y\subset\PP^4$ is some line.
We have the exact sequence
$$
0\longrightarrow\OOO_{Z+\Lambda}\longrightarrow\OOO_Z\oplus\OOO_\Lambda
\longrightarrow\bigoplus_{P\in Z\cap\Lambda}\FFF_P\longrightarrow 0,
$$
where $\FFF_P$ is a sheaf with support in $P$. It is clear that
$$
\hr^1(\OOO_{Z+\Lambda})=\p(Z+\Lambda)=5\quad \text{and}\quad
\hr^1(\OOO_Z\oplus\OOO_\Lambda)=\p(Z)+\p(\Lambda)=2.
$$
Since the curve $Z+\Lambda$ is connected and reduced, we have
$$
\# (Z\cap\Lambda)=\sum_{P\in Z\cap
\Lambda}\hr^0(\FFF_P)=5-\hr^0(\OOO_{Z+\Lambda})=4,
$$
i.e. the curve $Z$ has $4$-secant line $\Lambda$. But
then the projection from $\Lambda$ defines a birational map of $Z$ to a plane
curve
of degree $3$, which is impossible.

Consider the case~\ref{constr:9}.
Assume that there are two cubic surfaces $S_1$ and $S_2$
passing through $Z$.
Write $S_1\cap S_2=Z+\Lambda$, where $\Lambda$ is the residual curve.
Since $\deg\Lambda=2$,
there are only the following possibilities:
$$
\Lambda=
\begin{cases}
\text{plane reduced conic;}
\\
\Lambda=\Lambda_1+\Lambda_2, \ \text{where $\Lambda_1$ and
$\Lambda_2$ are disjoint lines;}
\\
\Lambda=2\Lambda_0,\ \text{where $\Lambda_0$ is a line.}
\end{cases}
$$
By Bertini's theorem a general member $S$ of the pencil generated by the 
surfaces $S_1$ and $S_2$ is nonsingular outside $S_1\cap S_2$.
In all cases a general member $S$ cannot have singularities along a component of 
the intersection $S_1\cap S_2$ (otherwise both surfaces
$S_1$ and $S_2$ would be singular 
along this component and the multiplicity of intersections along it would be 
$\ge 
4$).
Thus $S$ is a normal cubic surface with Du Val
singularities (see~\ref{lemma:surfaceG:3a}) and $\Lambda:= S_1\cap S=S_2\cap S$
is a Cartier divisor 
on $S$ such that $\Lambda\sim -3K_S$.

If $\Lambda$ is a plane reduced conic, then $\# (Z\cap \Lambda)\le
\deg Z=7$. On the other hand, similar to the case~\ref{constr:10} it can be
shown that $\# (Z\cap \Lambda)=8$, a contradiction.
If $\Lambda=\Lambda_1+\Lambda_2$,
where $\Lambda_1$ and
$\Lambda_2$ are disjoint lines, then similar to~\ref{constr:10}
we obtain $\# (Z\cap\Lambda_1)+\# (Z\cap\Lambda_2)=9$.
But then projections from
$\Lambda_1$ and $\Lambda_2$ define on $Z$ linear systems $\mathfrak{g}_3^1$ and
$\mathfrak{g}_2^1$, which contradicts our assumption that the curve $Z$ is 
non-hyperelliptic.

Finally, let $\Lambda=2\Lambda_0$. We estimate the genus of the 
\textit{reduced} 
curve
$C:= Z+\Lambda_0$.
We have the exact sequence
$$
0\longrightarrow\OOO_S(-C)\longrightarrow\OOO_S\longrightarrow\OOO_C\longrightarrow 0.
$$
(Since the surface $S$ can be singular, the sheaf $\OOO_S(-C)$
not necessarily is invertible, but it is reflexive of rank~$1$). By the 
Kawamata--Viehweg Vanishing Theorem 
(the version for singular varieties~\cite{KMM}) we have
$H^1(S,\,\OOO_S)=H^2(S,\,\OOO_S)=0$.
Therefore,
$H^1(C,\OOO_C)\simeq H^2(S,\,\OOO_S(-C))$. By the Serre duality
$$
\hr^2(S,\,\OOO_S(-C))=\hr^0(S,\,\OOO_S(K_S+C)),
$$
where $K_S+C\sim -2K_S-\Lambda_0$.
Let $\mu\colon \widetilde{S}\to S$ be the
minimal resolution of singularities (put $\widetilde{S}=S$, if the surface
$S$ is nonsingular)
and let $\widetilde{\Lambda}_0$ be the proper transform of the line 
$\Lambda_0$.
Then $\widetilde{\Lambda}_0$ is a $(-1)$-curve.
Since singularities $S$ are Du Val, we have $\mu^*K_S=K_{\widetilde{S}}$.
Therefore,
$$
\hr^0(S,\,\OOO_S(-2K_S-\Lambda_0)) \ge \hr^0(\widetilde{S},\,
\OOO_S(-2K_{\widetilde{S}}-\widetilde{\Lambda}_0).
$$
By the Riemann--Roch Theorem
$$
\hr^0(\widetilde{S},\,
\OOO_S(-2K_{\widetilde{S}}-\widetilde{\Lambda}_0)\ge 7.
$$
Hence,
$$
\p(C)=\hr^1(C,\OOO_C)=\hr^0(S,\,\OOO_S(K_S+C))\ge 7.
$$
Then as above, we obtain that $\Lambda_0$ is a $5$-secant line of the curve $Z$.
Therefore, $Z$ is a hyperelliptic curve. The contradiction proves 
\ref{lemma:surfaceG:unique}.

\ref{lemma:surfaceG:curve}
We use the following result by Mumford--Castelnuovo
\cite[Lecture~14]{Mumford-Lectures-on-curves}.
We reproduce its in a slightly modified form adapted for
our purposes.

\begin{prp}
\label{prop:mumford}
Let $Z\subset\PP^d$, $d\ge 3$ be a curve and let $\JJJ_Z$ be its its ideal 
sheaf. Assume that for some $m>0$ the following two conditions are satisfied:
\begin{enumerate}
\item
\label{prop:mumford1}
$H^2(\PP^d,\,\JJJ_Z(m-2))=0$,
\item
\label{prop:mumford2}
$H^1(\PP^d,\,\JJJ_Z(m-1))=0$.
\end{enumerate}
Then for $k>m$ the space $H^0(\PP^d,\,\JJJ_Z(k))$ is generated by
$$
H^0(\PP^d,\,\JJJ_Z(k-1))\otimes H^0(\PP^d,\,\OOO_{\PP^d}(1)).
$$
In particular, the curve $Z\subset\PP^d$ is an intersection of
hypersurfaces of degree $m$.
\end{prp}

\begin{rem}
It is not difficult to see that the condition~\ref{prop:mumford1} is a 
consequence of the following
\begin{enumerate}
\item
\label{prop:mumford1a}
$(m-2)(\deg Z)> 2\g(Z)-2$.
\end{enumerate}
\end{rem}

We show that the conditions of Proposition~\ref{prop:mumford} are fulfilled in 
our 
case with $m=7-d$.
Indeed, consider, fro instance, the case~\ref{constr:9}.
The inequality
\ref{prop:mumford1a} can be checked directly and for
the proof of~\ref{prop:mumford2} we
consider the exact sequence
$$
0 \longrightarrow\mathscr{J}_Z(3) \longrightarrow\OOO_{\PP^3}(3) \longrightarrow\OOO_Z(3)\longrightarrow 0.
$$
According to the assertion~\ref{lemma:surfaceG:unique} we have
$\hr^0(\PP^3,\,\mathscr{J}_Z(3))=1$.
It is not difficult to compute that $\hr^0(\PP^3,\,\OOO_{\PP^3}(3)=20$ and
$\hr^0(X,\,\OOO_Z(3)=19$.
Therefore, the map
$$
H^0(\PP^3,\,\OOO_{\PP^3}(3))\longrightarrow H^0(Z,\,\OOO_Z(3))
$$
is surjective.
Since $H^1(\PP^3,\,\OOO_{\PP^3}(3))=0$, we have $H^1(\PP^3,\,\JJJ_Z(3))=0$.
The cases~\ref{constr:10} and~\ref{constr:12} are considered similarly.

Proposition~\ref{lemma:surfaceG} is proved.
\end{proof}

\subsection{Construction of Fano threefolds of genus $9$, $10$, and $12$}

\begin{teo}
\label{theorem:9-10-12}
\begin{enumerate}
\item
\label{theorem:9-10-12:12}
Let $Z\subset Y_5$ be a smooth rational curve of degree $5$
lying on the variety $Y=Y_5\subset\PP^6$. Then the linear system
$\left|3M-2Z\right|$ defines a birational map of
$Y_5$ to an anticanonically embedded Fano threefold
$X_{22}\subset\PP^{13}$ of genus $12$.
\item
\label{theorem:9-10-12:10}
Let $Z\subset\PP^4$ be a smooth irreducible curve of degree $7$
and genus~$2$ lying on a smooth quadric $Y=Y_2\subset\PP^4$. Then
the linear system $\left|5M-2Z\right|$ defines
a birational map of $Y_2$ to an anticanonically embedded
Fano threefold $X_{18}\subset\PP^{11}$ of genus $10$.
\item
\label{theorem:9-10-12:9}
Let $Z\subset\PP^3$ be a smooth irreducible curve of degree $7$
and genus $3$ that is not hyperelliptic. Then the linear
system $\left|7M-2Z\right|$ defines a birational
map of $\PP^3$ to an anticanonically embedded Fano threefold
$X_{16}$ of genus $9$ in $\PP^{10}$.
\item
In all cases the birational
map contracts the proper transform of the surface $S$ to a line $l\subset
X$.

\item
\label{theorem:9-10-12:k}
Each Fano threefold $X=X_{2g-2}\subset\PP^{g+1}$ of genus $g=9$,
$10$ or $12$ with $\uprho(X)=\iota(X)=1$ can be obtained by one of the ways
described above.
\end{enumerate}
\end{teo}

\begin{direct}{Explanation}
To prove the theorem we construct a Sarkisov link~\eqref{diagram} that is 
inverse to the link from Theorem~\ref{th:double-projection}:
\begin{equation}
\label{CD:IV-constr}
\vcenter{
\xymatrix{
&\overline{X}\ar[dl]_{\varphi}\ar@{-->}[r]^{\chi}& 
\widetilde{X}\ar[dr]^{\sigma}&
\\
Y\ar@{-->}[rrr]^{}&&&X
}}
\end{equation}
Here $\varphi\colon \overline{X}\to Y$ is the blowup of $Z$, $\chi$ is a flop,
and 
$\sigma$ is 
the contraction of the proper transform of the surface $S$ to a line $l\subset
X=X_{2g-2}\subset\PP^{g+1}$.
\end{direct}

\begin{proof}
Let 
$\overline{F}:= \varphi^{-1}(Z)$ be the exceptional divisor, let
$\overline{M}:= \varphi^*M$ be the pull-back of a hyperplane section of $Y$, and 
let
$\overline{E}=\varphi^{-1}_*(S)$ be the proper transform of
the surface $S$. Thus
\begin{equation}
\label{lemma:intersection-numbers:7}
\begin{gathered}
\Pic(\overline{X})\simeq \overline{M}\cdot\ZZ\oplus\overline{F}\cdot\ZZ\simeq 
(-K_{\overline{X}})\cdot\ZZ\oplus\overline{E}\cdot\ZZ,
\\[3pt]
-K_{\overline{X}}\sim \iota(Y)\overline{M}-\overline{F},\qquad
\overline{E}\sim (\iota(Y)-1)\overline{M}-\overline{F}.
\end{gathered}
\end{equation}

\begin{lem}
\label{lemma:intersection-numbers}
We have the following relations on $\overline{X}$:
\begin{equation}
\label{lemma:intersection-numbers:1}
\overline{M}^3=M^3=\dd(Y),\qquad
\overline{M}^2\cdot \overline{F}=0,\qquad
\overline{M}\cdot\overline{F}^2=-\deg Z,
\end{equation}
\begin{equation}
\label{lemma:intersection-numbers:2}
\overline{F}^3=
\begin{cases}
-8 & \text{in the case~\ref{constr:12}}
\\
-23&\text{in the case~\ref{constr:10}}
\\
-32&\text{in the case~\ref{constr:9}}
\end{cases}
\end{equation}
or, in another basis:
\begin{equation}
\label{lemma:intersection-numbers:3}
\overline{E}^3=-\iota(Y),\qquad
\left(-K_{\overline{X}}\right)^3=2g-6,\qquad
\left(-K_{\overline{X}}\right)^2\cdot\overline{E}=3,\qquad
\left(-K_{\overline{X}}\right)\cdot \overline{E}^2=-2,
\end{equation}
where $g=12-\g(Z)$.
\end{lem}

\begin{proof}
From Lemma~\ref{lemma-blowup-curve-intersection} we immediately obtain
\eqref{lemma:intersection-numbers:1}--\eqref{lemma:intersection-numbers:2}.
Then relations~\eqref{lemma:intersection-numbers:3} are obtained from
\eqref{lemma:intersection-numbers:7}.
\end{proof}

\begin{lem}
\label{lemma:4.4}
\begin{enumerate}
\item
\label{lemma:4.4-1}
The linear system $|{-}K_{\overline{X}}|$ is base point free.
In particular, the divisor $-K_{\overline{X}}$ is nef.
\item
\label{lemma:4.4-2}
For any effective divisor $D$ we have
$$
\left(-K_{\overline{X}}\right)^2\cdot D>0.
$$
\end{enumerate}
\end{lem}

\begin{proof}\ref{lemma:4.4-1}
According to~\ref{lemma:surfaceG}\ref{lemma:surfaceG:curve},
the curve $Z$ is cut out on $Y$ by hypersurfaces of degree $\iota=\iota(Y)=d-7$.
Therefore, the linear system $|{-}K_{\overline{X}}|=|\iota 
\overline{M}-\overline{F}|$ is base point free and so it is nef.

\ref{lemma:4.4-2}
Assume that
$\left(-K_{\overline{X}}\right)^2\cdot D=0$ for some prime divisor $D$.
From~\eqref{lemma:intersection-numbers:7} we obtain $-K_{\overline{X}}\sim
\overline{M}+\overline{E}$.
Therefore,
$$
0=(-K_{\overline{X}})^2\cdot D=(-K_{\overline{X}})\cdot \overline{M}\cdot D+
(-K_{\overline{X}})\cdot \overline{E}\cdot D.
$$
Since divisors $-K_{\overline{X}}$ and $\overline{M}$ are nef,
and the divisors $\overline{E}$ and $D$
are effective and have no common components,
both terms in the right hand side are non-negative. Hence, they are equal 
to zero. Furthermore,
$$
0=(-K_{\overline{X}})\cdot \overline{M}\cdot D=\overline{M}^2\cdot 
D+\overline{E}\cdot
\overline{M}\cdot D.
$$
As above both terms in the right hand side are equal to zero: 
$\overline{M}^2\cdot 
D=\overline{E}\cdot
\overline{M}\cdot D=0$. In particular, this means that the
divisor $D$ is contracted by the morphism $\varphi$. But then $D=\overline{F}$.
On the other hand, $(-K_{\overline{X}})^2\cdot \overline{F}>0$. The 
contradiction 
proves
the lemma.
\end{proof}

If the divisor $-K_{\overline{X}}$ is not ample\footnote{A posteriori,
applying Corollary~\ref{cor:line-line}, we obtain that the divisor 
$-K_{\overline{X}}$
is never ample.}, then by Theorem~\ref{flop:theorem} and
Lemma~\ref{lemma:4.4} there exists a flop
\begin{equation}
\label{eq:IV:flop}
\vcenter{
\xymatrix{
\overline{E}\subset\overline{X}\ar@{-->}[rr]^{\chi}\ar[dr]^{\bar\theta} &&
\widetilde{X}\supset E\ar[dl]_{\widetilde\theta}
\\
&E_0\subset X_0&
} }
\end{equation}
where $E_0=\bar\theta(\overline{E})=\widetilde\theta(E)$. Thus we can 
apply
the construction~\eqref{diagram}. In both
cases~\ref{construction:sl-cases}\ref{case:nef-big} and
~\ref{construction:sl}\ref{case:ample} there exists
the diagram~\eqref{CD:IV-constr}, in which the map $\chi$ is either a flop or an 
isomorphism.

\begin{prp}
\label{4.9}
The morphism $\sigma\colon \widetilde{X}\to X$ contracts the proper 
transform $E$ of the surface
$\overline{E}$ to a curve
$l\subset X$, the variety $X$ is a smooth Fano threefold
of index~$1$ of genus $g=12$ if $Y=Y_5$, $g=10$ if $Y=Y_2$, and $g=9$
if $Y=\PP^3$. Furthermore, the curve $l$ is a line on $X$.
\end{prp}

\begin{proof}
We claim that divisor $E$ is negative on the ray $\rR_\sigma$.

First, consider the case where $\chi$ is an isomorphism. Then 
$E=\overline{E}$.
If $\overline{E}\cdot \rR_\varphi \ge 0$, then $E$ is nef.
This contradicts the fact that $\left(-K_{\overline{X}}\right)\cdot 
\overline{E}^2<0$
(see~\eqref{lemma:intersection-numbers:3}).

Now consider the case where $\chi$ is not an isomorphism.
The Mori cone $\NE(\overline{X})$ is generated by two rays: the ray 
$\rR_\varphi$,
corresponding to the contraction $\varphi$, and $K_{\overline{X}}$-trivial 
ray
$\rR_{K_{\overline{X}}}$. In view of~\eqref{lemma:intersection-numbers:7} the
divisor $\overline{E}=\varphi_*^{-1}(S)$ is positive on $\rR_\varphi$.
On the other hand, it is not nef
(because $\left(-K_{\overline{X}}\right)\cdot \overline{E}^2<0$).
Hence, the divisor $\overline{E}$ is negative on $\rR_{K_{\overline{X}}}$.
The Mori cone $\NE(\widetilde{X})$ is also generated by two rays: the ray 
$\rR_\sigma$
corresponding to the contraction $\sigma$ and the $K_{\widetilde{X}}$-trivial 
ray
$\rR_{K_{\widetilde{X}}}$. By the property of flops $E\cdot
\rR_{K_{\widetilde{X}}}>0$.
If $\overline{E}\cdot \rR_\varphi \ge 0$, then the
divisor $E$ is nef. But this again contradicts the inequality
$$
\left(-K_{\widetilde{X}}\right)\cdot {E}^2=\left(-K_{\overline{X}}\right)\cdot
\overline{E}^2<0.
$$
Therefore, $E\cdot \rR_\sigma<0$. This means that
the morphism $\varphi$ is birational and contracts the divisor $E$
(see Theorem~\ref{class:ext-rays}).
Since $E\cdot (-K_{\widetilde{X}})^2=\overline{E}\cdot 
(-K_{\overline{X}})^2=3$,
according to~\eqref{equations-E25},
the morphism $\sigma$ cannot contract the divisor $E$ to a point.
Then by the classification of extremal rays
$l:= \sigma(E)$ is a nonsingular curve and $X$ is nonsingular Fano threefold
with $\uprho(X)=1$ (see Theorem~\ref{class:ext-rays}).
The group $\Pic(\widetilde{X})$ is generated by the classes of the divisors
$-K_{\widetilde{X}}$ and $E$.
Thus $\Pic({X})$ is generated by the class of
$-K_{{X}}$, i.e. $\iota(X)=1$.
Taking~\eqref{lemma:intersection-numbers:3} and~\eqref{equations-E1} into 
account we obtain
$$
\begin{gathered}
(-K_{\widetilde{X}}+E)^2\cdot (-K_{\widetilde{X}})=(-K_X)^3=2g-2,\qquad \g(X)=g,
\\[3pt]
(-K_{\widetilde{X}}+E)\cdot E\cdot (-K_{\widetilde{X}})=\deg l=1,
\\[3pt]
E^2\cdot (-K_{\widetilde{X}})=-2=2\g(l)-2, \qquad \g(l)=0,
\end{gathered}
$$
i.\,e.\ $X=X_{2g-2}\subset\PP^{g+1}$ and $l\subset X$ is a line. The proposition
is proved.
\end{proof}
Now, let us proceed to the proof of Theorem~\ref{theorem:9-10-12}. We have the 
commutative
diagram~\eqref{CD:IV-constr},
where $\sigma{\scriptstyle{\circ}} \chi{\scriptstyle{\circ}}
\varphi^{-1}=\psi$ is a birational map. It is easy to see that
$$
\sigma^*(-K_X)=-K_{\widetilde{X}}+E=\chi_*(-K_{\overline{X}}+\overline{E})=\chi_
*(
\varphi^*(-K_Y)+\overline{E}-\overline{F}),
$$
Taking~\eqref{lemma:intersection-numbers:7} into account we obtain
$$
\sigma^*(-K_X)=\chi_*\big((\iota \overline{M}-\overline{F})+((\iota-1)
\overline{M}-\overline{F}). \big)=\chi_*\big(((2\iota-1) 
\overline{M}-2\overline{F}) \big)
$$
This implies the assertions~\ref{theorem:9-10-12:12}--\ref{theorem:9-10-12:9} of 
the
theorem. To prove~\ref{theorem:9-10-12:k} we compare
diagrams~\eqref{CD:IV-constr} and~\eqref{diagram}. The map $\psi$ is 
inverse to the double projection $\Psi$ from~\eqref{eq:double-projection}. For 
$g=9$
the curve~$Z$ cannot be hyperelliptic
according to~\ref{th:double-projection}\ref{g=9}. Theorem~\ref{theorem:9-10-12}
is proved completely.
\end{proof}

In the terms 
of the diagram~\eqref{CD:IV-constr}, it can be given a criterion to 
distinguish the types of lines on the variety $X_{2g-2}\subset\PP^{g+1}$ for 
$g\ge 9$.

\begin{prp}
\label{proposition:9-10-12:cr}
In the conditions of Theorem~\ref{theorem:9-10-12} the surface $S$ is normal if 
and
only if the normal bundle of the line~$l$ has the form
\begin{equation}
\label{eq:N:11}
\NNN_{l/\overline{X}}\simeq \OOO_{\PP^1}\oplus \OOO_{\PP^1}(-1). 
\end{equation}
\end{prp}

\begin{proof}
Consider the diagram~\eqref{CD:IV-constr}.
First, let the normal bundle of the line $l$ have the form~\eqref{eq:N:11}.
Then $E\simeq\FF_1$.
Assume that the surface $S$ is not normal (and its normalization $S'$
is as one in Table~\ref{non-normal:G}. In this case a general 
fiber
$\Upsilon_E\subset E$ does not meet the flopping locus (see
Lemma~\ref{lemma:l-s:d-p}\ref{double-projection:bir}).
Therefore, its image $\varphi\comp \chi^{-1}(\Upsilon_E)$ on $Y$ is of
degree $({-}K_{\widetilde{X}}-E)\cdot \Upsilon_E=2$, i.\,e.\ $\varphi\comp
\chi^{-1}(\Upsilon_E)$ is an irreducible conic lying in the nonsingular part 
of $S$.
Using Table~\ref{non-normal:G}, it is not difficult to check that for $g=9$ 
every 
conic
on $S$ meets the singular locus of $\Lambda=\Sing(S)$, for $g=12$;
the surface $S$ cannot contain a family conics, and for $g=10$ the surface
$S$ can contain a family of conics only in one case: $d=4$, 
$S'\simeq
\PP^1\times \PP^1$, $\LLL=|\Sigma+2\Upsilon|$, $\Lambda'\sim \Sigma$. In this 
case any curve on $S$ that is distinct from conics in our one-dimensional family
meets $\Lambda$. However, as above, the image of any sufficiently 
general very ample divisor on $E$ must lie in the smooth part of $S$. The 
contradiction
proves the sufficiency.

Let the line $l$ has normal bundle
$$
\NNN_{l/X}=\OOO_{\PP^1}(1)\oplus\OOO_{\PP^1}(-2).
$$
Then
$E\simeq\FF_3$ and the self-intersection number of the exceptional section 
$\Sigma\subset E\simeq\FF_3$ is equal to $-3$.
Assume that $S$ is a normal surface with Du Val singularities
(in particular, it can be nonsingular). Then the pair $(Y,S)$ has purely
log terminal singularities. Since 
$K_{\overline{X}}+\overline{E}=\varphi^*(K_Y+S)$,
the pair $(\overline{X},\overline{E})$ has plt 
singularities as well. By the adjunction formula
$$
K_{\overline{E}}=(K_{\overline{X}}+\overline{E})|_{\overline{E}}=\varphi^*(K_Y+S
)|_{\overline{E}}=
\varphi_{\overline{E}}^* K_S.
$$
Therefore, $\varphi_{\overline{E}}\colon \overline{E}\to S$ is a crepant 
morphism, 
the surface
$\overline{E}$ has only Du Val singularities and the anticanonical class
$-K_{\overline{E}}$ is nef and big (i.e. $\overline{E}$ is a weak
del Pezzo surface). In this case any contraction on $\overline{E}$ keeps 
the type of singularities.

Consider the flop diagram~\eqref{eq:IV:flop}.
Since $\widetilde\theta$ contracts the exceptional section $\Sigma\subset
E\simeq\FF_3$,
the surface $E_0$ is isomorphic to a cone over a rational twisted cubic curve. 
But
this contradicts the fact that the singularities of $E_0$ are (at worst) Du Val.
The proposition is proved.
\end{proof}

\begin{zadachi}
\eitem
In the notation~\ref{notation:IV} and~\ref{lemma:surfaceG} above, assume that
the surface $S$ is nonsingular.
Prove that, under suitable choice of the standard basis
$\mathbf h,\mathbf e_1,\dots,\mathbf e_n$ in Picard group of the del Pezzo 
surface~$S$, the class of the curve $Z$ can be expressed in the following form

\smallskip
$d=5$: \ $Z\sim 2\mathbf h-\mathbf e_1$;

\smallskip
$d=4$: \ $Z\sim 4\mathbf h-2\mathbf e_1-\mathbf e_2-\mathbf e_3-\mathbf e_4$;

\smallskip
$d=3$: \ $Z\sim 4\mathbf h-\sum_{i=1}^5\mathbf e_i$.

\eitem
In the conditions of the previous exercise prove that the curve $Z$ has exactly 
$8-d$
distinct $(7-d)$-secant lines and these lines do not meet each other.

\eitem
Let in the conditions of Theorem~\ref{theorem:9-10-12} the surface $S$ is 
nonsingular.
Prove that
for any flopping curve $C\subset\overline{X}$ the normal bundle has the form
$\NNN_{C/\overline{X}}\simeq \OOO_{\PP^1}(-1)\oplus \OOO_{\PP^1}(-1)$ and
the curve~$C$ is the inverse image of a $(7-d)$-secant line of the curve $Z$.

\eitem
Using the construction from
Theorem~\ref{theorem:9-10-12} compute the dimension of moduli spaces of Fano 
threefolds $X=X_{2g-2}\subset
\PP^{g+1}$ of genus $g\ge 9$.

\eitem
Prove that on a general Fano threefold $X=X_{2g-2}\subset\PP^{g+1}$ of genus
$g=9$, $10$, $12$ the Hilbert scheme of lines $\Lines (X)$ is nonsingular. 
What is 
the codimension of the subset corresponding to those Fano threefolds whose 
Hilbert scheme 
$\Lines(X)$ is singular?
\hint{Use Proposition~\ref{proposition:9-10-12:cr}.}

\eitem
What properties distinguish the curve $Z=Z_5\subset Y_5\subset\PP^6$ from
the construction~\eqref{CD:IV-constr} for the case where 
$X=X_{22}\subset\PP^{13}$
is the Mukai-Umemura threefold? (see Example~\ref{remark:MU:22}.) Prove that
on the Mukai-Umemura threefold every line has normal bundle of the form
$\OOO_{\PP^1}(1)\oplus\OOO_{\PP^1}(-2)$ and so the Hilbert scheme of lines on
this variety is non-reduced at every point. Actually this property
is another characterization Mukai-Umemura threefold (see 
\cite{Prokhorov-1990b}).
\end{zadachi}

\newpage\section{Fano threefolds of Picard 
number~$1$.~V}
\label{sect:Fano-V}

Now we discuss constructions that allow to obtain 
all Fano threefolds of Picard number~$1$ and genus~$7$ or~$8$.
We will use Sarkisov links that are inverse to ones considered in Sections 
\ref{sec:FanoII} and~\ref{sec:FanoIII}.

\subsection{Fano threefolds of genus~$8$} 

Similar to Theorem~\ref{theorem:9-10-12}
there are birational constructions that allow to obtain all Fano threefolds
of genus $8$. Below we provide one of them that is inverse to the Sarkisov link
with center a point (see Theorem~\ref{th:main:point}\ref{th:main:point:g=8}).
Another construction due to Fano is given in~\cite[Ch.~3,~\S1]{Iskovskikh1980}.

\begin{teo}[\cite{Tregub1985a}]
\label{th:g8:constr}
Let $Z\subset\PP^4$ be a rational normal curve of degree $4$ and let
$Y=Y_3\subset\PP^4$ be a nonsingular cubic hypersurface passing through
$Z$. Then the linear system $|8M-5Z|$, where $M$ is the class of 
hyperplane section of~$Y$, defines a birational map $Y$ to
an anticanonically embedded Fano threefold $X_{14}\subset\PP^9$ of genus $8$.
The map contracts a divisor that is linearly equivalent to $3M-2Z$ and
is inverse to the birational map constructed in 
Theorem~\ref{th:main:point}\ref{th:main:point:g=8}.

Each Fano threefold $X=X_{14}\subset\PP^9$ of genus $8$ with $\uprho(X)=1$
can be obtained in the way described above.
\end{teo}

\begin{proofsk}
Let $\varphi\colon \overline{X}\to Y$ be the blowup of $Z$, let $\overline{F}$ 
be the
exceptional divisor, and let $\overline{M}:= \varphi^* M$. Since the curve $Z$
is an intersection of quadrics, the linear system
$$
|{-}K_{\overline{X}}|=|2\overline{M}-\overline{F}|
$$
is base point free. It is not difficult to compute also that
$(2\overline{M}-\overline{F})^3=6$. Therefore, the divisor $-K_{\overline{X}}$ 
is nef and big.

Let us prove that there exists an effective divisor $\overline{E}\sim 
3\overline{M}-2\overline{F}$.
By the Riemann--Roch Theorem (see the proof of Theorem~\ref{theorem-sections})
we have
$$
\hr^0(\overline{X},\,\OOO_{\overline{X}}(3\overline{M}))=\hr^0(
X,\,\OOO_X(3M))=34.
$$
Since $\overline{F}\simeq \PP_Z\big(\NNN_{Z/Y}^\vee\big)$ is a rational
ruled surface, we have
$$
\hr^0(\overline{F},\,\OOO_{\overline{F}}(3\overline{M}))=3M\cdot Z+1=13.
$$
It follows from the exact sequence
$$
0 \longrightarrow\OOO_{\overline{X}}(3\overline{M}-\overline{F}) 
\longrightarrow\OOO_{\overline{X}}(3\overline{M}) \longrightarrow
\OOO_{\overline{F}}(3\overline{M}) \longrightarrow 0
$$
that
$$
\hr^0(\overline{X},\,\OOO_{\overline{X}}(3\overline{M}-\overline{F}))\ge 21.
$$
For brevity, we put $\LLL:= \OOO_{\overline{F}}(3\overline{M}-\overline{F})$.
Let $\pi\colon \overline{F}\to Z\simeq \PP^1$ be the projection (that coincides with the 
restriction of $\varphi$ to $\overline{F}$).
Since $\overline{F}\simeq \PP_Z\big(\NNN_{Z/Y}^\vee\big)$, 
we have
$$
\pi_*\OOO_{\overline{F}}(-\overline{F})= 
\pi_*\OOO_{\PP_Z(\NNN_{Z/Y}^\vee)}(1)=\NNN_{Z/Y}^\vee.
$$
According to the above, the sheaf
$\NNN_{Z/Y}^\vee\otimes 
\OOO_Z(2)=\pi_*\OOO_{\overline{F}}(2\overline{M}-\overline{F})$
is generated by global sections.
Therefore, the same holds for the sheaf
$\pi_*\LLL=\NNN_{Z/Y}^\vee\otimes \OOO_Z(3)$.
Thus
$$
H^1(\overline{F},\,\LLL)=
H^1(Z,\,\pi_*\LLL)=0
$$
(see~\cite[Ch.~V, Lemma~2.4]{Hartshorn-1977-ag}).
Similarly, by the Riemann--Roch Theorem on~$Z$ we have
$$
\hr^0(\overline{F},\LLL)=\hr^0(Z,\,\NNN_{Z/Y}^\vee\otimes \OOO_Z(3))=
2+\deg \NNN_{Z/Y}^\vee+3\deg Z=20.
$$
Finally, as above,
from the exact sequence 
$$
0 \longrightarrow\OOO_{\overline{X}}(3\overline{M}-2\overline{F}) 
\longrightarrow\OOO_{\overline{X}}(3\overline{M}-\overline{F})
\longrightarrow\OOO_{\overline{F}}(3\overline{M}-\overline{F}) \longrightarrow 0
$$
it follows that
$$
H^0(\overline{X},\,\OOO_{\overline{X}}(3\overline{M}-2\overline{F}))\neq 0,
$$
i.e. there exists an effective divisor $\overline{E}\sim 
3\overline{M}-2\overline{F}$.

It is clear that this divisor is irreducible. As in the proof of
Lemma~\ref{lemma:4.4}\ref{lemma:4.4-2} it can be shown that the linear system
$|{-}K_{\overline{X}}|$ defines a small contraction. Then the proof is 
completely
similar to the one of Theorem~\ref{theorem:9-10-12}.
\end{proofsk}

\subsection{Fano threefolds of genus~$7$} 

Now we consider the case Fano threefolds of genus~$7$.
Our construction is inverse to the Sarkisov link with center a conic (see
Theorem~\ref{thm:proj-conic}\ref{thm:proj-conic7}).

\begin{teo}
\label{prop:g=7:constr}
Let $Z\subset\PP^4$ be a nonsingular curve of degree $10$ and genus $7$
satisfying the following two conditions

\begin{enumerate}
\item
\label{proposition:g=7:constr1}
$Z$ is not hyperelliptic, nor trigonal, nor tetragonal
\textup(see Proposition~\ref{proposition:curve:g7});
\item
\label{proposition:g=7:constr2}
$Z$ is contained in a nonsingular quadric $Q\subset\PP^4$.
\end{enumerate}

Then the linear system $\left|8M-3Z\right|$, where $M$ is the class of 
a hyperplane
section $Q$, defines a birational map of $Q$ to an anticanonically embedded
Fano threefold $X_{12}\subset\PP^8$ of genus $7$. The map contracts a
divisor that is linearly equivalent to $5M-2Z$. This transformation is inverse to the
birational map constructed in Theorem 
\ref{thm:proj-conic}\ref{thm:proj-conic7}.

Each Fano threefold $X=X_{12}\subset\PP^8$ of genus $7$ with $\uprho(X)=1$
can be obtained in the way described above.
\end{teo}

As in the proof of Proposition~\ref{lemma:surfaceG}\ref{lemma:surfaceG:2}
it is easy to show that there exists a quadric passing through $Z$.
In~\ref{prop:g=7:constr}\ref{proposition:g=7:constr2} we require that
this quadric is nonsingular.

\begin{proof}
Let $\varphi\colon \overline{X}\to Q$ be the blowup of $Z$, let $\overline{F}$ 
be the
exceptional divisor and let $\overline{M}:= \varphi^* M$.
Then
$$
-K_{\overline{X}}=3\overline{M}-\overline{F}.
$$
Using the relations~\eqref{eq:blowup-curve-intersection} we obtain
$$
(\overline{M})^3=2,\qquad (\overline{M})^2\cdot \overline{F}=0,\qquad 
\overline{M}\cdot
\overline{F}^2=-10, \qquad \overline{F}^3=-42.
$$
From this it is not difficult to compute that
\begin{equation}
\label{equation:g=7:int}
(-K_{\overline{X}})^3=6,\quad (-K_{\overline{X}})^2\cdot \overline{M}=8,\quad
(-K_{\overline{X}})^2\cdot \overline{F}=18.
\end{equation}

\begin{lem}
The following inequality holds:
\begin{equation}
\label{equation:g=7:dim-K}
\hr^0 (\overline{X}, \,\OOO_{\overline{X}}(-K_{\overline{X}}))\ge 6.
\end{equation}
\end{lem}

\begin{proof}
By the Riemann--Roch Theorem (see~\eqref{eq:theorem:3dim-RRi=2}) we have
$$
\hr^0(\overline{X},\,\OOO_{\overline{X}}(3\overline{M}))=\hr^0(Q,\OOO_Q(3M))=30.
$$
Since $\overline{F}$ is a ruled surface of genus~$7$, we have
$$
\hr^0(\overline{F},\,\OOO_{\overline{F}}(3\overline{M}))=
\hr^0(Z,\,\OOO_Z(3))=24.
$$
Then the inequality~\eqref{equation:g=7:dim-K} follows from
the exact sequence
$$
0 \longrightarrow\OOO_{\overline{X}}(3\overline{M}-\overline{F}) 
\longrightarrow\OOO_{\overline{X}}(3\overline{M}) \longrightarrow
\OOO_{\overline{F}}(3\overline{M}) \longrightarrow 0.\qedhere
$$
\end{proof}

\begin{lem}
\label{slemma:g=7:nef}
The divisor $-K_{\overline{X}}=3\overline{M}-\overline{F}$ is nef.
\end{lem}

\begin{proof}
Assume, the converse, i.e. there exists an irreducible curve 
$\overline{C}\subset
\overline{X}$ such that $K_{\overline{X}}\cdot\overline{C}>0$.
It follows from~\eqref{equation:g=7:dim-K}
that $\dim |{-}K_{\overline{X}}|\ge 5$.
Take two general elements $\overline{D}, 
\overline{D}'\in|{-}K_{\overline{X}}|$. We can
write
$$
\overline{D}\cap\overline{D}'=m \overline{C}+\overline{R},
$$
where $m\ge 1$ and $\overline{R}=\sum m_i\overline{R}_i$ is an effective 
$1$-cycle. It follows
from the restriction exact sequence that
\begin{equation}
\label{eq:g=7:H0SR}
\hr^0(\overline{D},\,\OOO_{\overline{D}}(-K_{\overline{X}}))\ge 5.
\end{equation}
Thus the cycle $\overline{R}$ varies in a four-dimensional family 
on~$\overline{D}$.
In particular, $\overline{R}\neq 0$.

Put $D:= \varphi(\overline{D})$, $D':= \varphi(\overline{D}')$, $C:= 
\varphi_*\overline{C}$, and
$R:= \varphi_*\overline{R}$. Thus $D$ and $D'$ are elements the linear
system $|{-}K_Q|$ and they are general members passing through~$Z$.
Then
\begin{equation}
\label{equation:g=7:degC}
m \deg C+\deg R=\overline{D}\cdot\overline{D}'\cdot 
\overline{M}=(-K_{\overline{X}})^2\cdot
\overline{M}=8.
\end{equation}
Since $\deg Z=10$, this implies that the base locus $\Bs
|{-}K_{\overline{X}}|$ does not dominate~$Z$, that is,
$\varphi\left(\Bs
|{-}K_{\overline{X}}|\right) \not \supset Z$. In particular,
$\overline{C}\not \subset\overline{F}$, i.e. $C\neq Z$ and the intersection 
$D\cap D'$ contains $Z$ with multiplicity~1. Thus
$D\cap D'=Z+mC+R$.

We have
$$
3\deg C=3\varphi^* M\cdot \overline{C} < \overline{F}\cdot \overline{C}.
$$
Thus if $\deg C=1$, then the curve $C$ must be a $4$-secant line to $Z$. But
then the projection from $C$ defines a linear series $\mathfrak g^2_6$ on $Z$. This
contradicts our assumptions and Lemma~\ref{lemma:II:tetr}. Similarly,
if $\deg C=2$, then $C$ is a conic and the projection from the plane this conic 
defines a
linear series $\mathfrak g^1_3$ on $Z$. This again contradicts our
assumptions.
Thus $\deg C\ge 3$ and so $\deg R\le 5$.

Then we claim that the surface $D$ is normal. Indeed, otherwise it
must be singular along some curve $\Lambda$, which, by Bertini's theorem,
must lie in the base locus the linear system $|3M-Z|$. Moreover, both
surfaces $D$ and $D'$ must have multiplicity $\ge 2$ along $\Lambda$. 
Since
$m\le 2$, this implies that $\Lambda$ is a component of $R$ and
$R=4\Lambda+R_1$, where $R_1$ is a line. But then $D$ contains a
four-dimensional family of lines (see~\eqref{eq:g=7:H0SR}). It is clear 
that
this is impossible
(see~\cite[Corollary A.1.3]{KPS:Hilb}).

Hence, the surface $D$ is normal and $R$ is a Weil divisor on $D$ such that 
$\dim|R|\ge 4$.
Write $R$ in the form $R=R_{\mathrm{f}}+R_{\mathrm{m}}$, where 
$R_{\mathrm{m}}$ 
is movable and $R_{\mathrm{f}}$ is the fixed parts. Thus
$|R|=R_{\mathrm{f}}+|R_{\mathrm{m}}|$.
According to~\cite[Lemma A.1.2, Corollary A.1.3]{KPS:Hilb},
a surface of degree $\ge 5$ cannot contain a
two-dimensional family of conics or lines.
Therefore, $\deg R_{\mathrm{m}}\ge 3$.
Thus
$$
\deg R_{\mathrm{m}}\ge 3, \qquad 3\le \deg R\le 5, \qquad 3\le \deg C\le
5,\qquad m=1.
$$

If a general member $R_{\mathrm{m}}\in|R_{\mathrm{m}}|$ is reducible, then
by Bertini's theorem
the linear system $|R_{\mathrm{m}}|$ is composed of a pencil:
$|R_{\mathrm{m}}|=k |L|$, where $k \deg L\le 5$.
But then $\dim|R_{\mathrm{m}}|\le 2$. The contradiction shows that a general
element $R_{\mathrm{m}}\in|R_{\mathrm{m}}|$ is irreducible.

Now we claim that the arithmetic genus of the curve
$R_{\mathrm{m}}$ is at most~$1$.
Indeed, if the linear span $\langle R_{\mathrm{m}}\rangle$ is two-dimensional,
then $R_{\mathrm{m}}$ is a conic coinciding with $\langle R_{\mathrm{m}}\rangle\cap Q$.
If $\langle R_{\mathrm{m}}\rangle$ is three-dimensional, then $\p(R_{\mathrm{m}})=2$,
$\deg R_{\mathrm{m}}=5$ and a
cubic passing through $R_{\mathrm{m}}$ cut out on~$D$
the curve $R_{\mathrm{m}}+\Lambda$,
where $\Lambda$ is a line.
Since $D$ contains at most a one-dimensional family of lines,
this is impossible.
On the other hand, if $R_{\mathrm{m}}$ does not lie in a hyperplane in $\PP^4$, 
then the estimate
$\p(R_{\mathrm{m}})\le 1$ follows from the Castelnuovo inequality for
space curves (see e.~g.~\cite[Ch.~2,~\S3]{Griffiths1994}).

\begin{claim}
\label{claim:g=7:}
Let $D\subset\PP^4$ be a normal surface of degree $6$ that is
an intersection of a quadric and cubic, and containing the four-dimensional linear
system~$\mathcal C$ of
irreducible curves of degree $d$ and arithmetic genus $\p(\Gamma)\le 1$. 
Then $d\ge 5$.
\end{claim}

\begin{proof}
Let $A$ be a hyperplane section of the surface $D$ and let $\Gamma\in\mathcal 
C$ be a general curve.
Let $\mu\colon \widetilde{D}\to D$ be a resolution of singularities.
Let $\widetilde{\Gamma}\subset\widetilde{D}$ be the proper transform of 
$\Gamma$ and
let $A^*:= \mu^*A$. Then $A^*\cdot \widetilde{\Gamma}=d$. Put
$n:= \widetilde{\Gamma}^2$.
We can choose the resolution $\mu$ so that the linear system
$|\widetilde{\Gamma}|$ is base point free. Since $\dim
|\widetilde{\Gamma}|\ge 4$, we have $n\ge 3$.
It follows from the exact sequence
$$
0 \longrightarrow H^0(\widetilde{D}, \OOO_{\widetilde{D}}) \longrightarrow H^0(\widetilde{D},
\OOO_{\widetilde{D}}(\widetilde{\Gamma})) \longrightarrow H^0(\widetilde{\Gamma},
\OOO_{\widetilde{\Gamma}}(\widetilde{\Gamma}))
$$
that
$$
4\le \hr^0(\widetilde{D}, \OOO_{\widetilde{D}}(\widetilde{\Gamma}))-1
\le \hr^0(\widetilde{\Gamma}, \OOO_{\widetilde{\Gamma}}(\widetilde{\Gamma})).
$$
It is clear that $g:= \g(\widetilde{\Gamma})\le \p(\Gamma)\le 1$. By 
the Riemann--Roch Theorem
$$
\hr^0(\widetilde{\Gamma}, \OOO_{\widetilde{\Gamma}}(\widetilde{\Gamma}))=
n-g+1.
$$
Therefore, $n\ge 3+g$. By the Hodge Index Theorem
$$
0\ge (A^*)^2\,\widetilde{\Gamma}^2- (A^* \cdot\widetilde{\Gamma})^2=6n-d^2
$$
(see~\cite[Ch.~V, Exercise 1.9]{Hartshorn-1977-ag}). Hence,
$$
d^2\ge 6(3+g).\qedhere
$$
\end{proof}
From this claim we immediately obtain that
$$
\deg R_{\mathrm{m}}=5,\qquad R=R_{\mathrm{m}},\qquad \deg C=3.
$$
We use the notation of the proof of Claim~\ref{claim:g=7:}.
Consider the morphism $\Phi_{|\widetilde{R}|}\colon \widetilde{D}\to 
\check{D}\subset\PP^N$,
$N\ge 4$ given by the linear system $|\widetilde{\Gamma}|$.
Then $\check{D}$ is a surface of degree $\widetilde{R}^2=n\le 4$ (and this 
morphism is birational).
Here the image of a general divisor $A^*$ is a curve 
$\check{A}\subset\PP^N$
of degree $\widetilde{R}\cdot A^*=5$ and genus $4$.
According to the genus bound~\cite[Ch.~IV, \S6, Theorem 6.4]{Hartshorn-1977-ag},
the curve $\check{A}$ must be contained in a plane.
But then $\check{A}$ is a hyperplane section of the surface $\check{D}\subset
\PP^N$. We obtained a contradiction with our assumption.
Therefore, the divisor $-K_{\overline{X}}$ is nef.
Lemma~\ref{slemma:g=7:nef} is proved.
\end{proof}

Hence, by the Base Point Free Theorem~\ref{th:mmp:bpf},
the linear
system $\mo|-nK_{\overline{X}}|$ for $n\gg 0$ defines a morphism and so 
there exists
a birational contraction 
$\bar\theta \colon \overline{X}\to X_0$ such that
$-K_{\overline{X}}=\bar\theta^* A_0$ for
some ample Cartier divisor $A_0$ on $X_0$.

\begin{lem}
The morphism $\bar\theta$ does not contract any divisors.
\end{lem}

\begin{proof}
Assume the converse, i.e. there exists a prime divisor $\overline{D}$ such that
$(-K_{\overline{X}})^2\cdot \overline{D}=0$.
Using~\eqref{equation:g=7:int}, it is not difficult to show that
$D\sim m(9\overline{M}-4\overline{F})$ for some $m>0$.
Since
$$
(-K_{\overline{X}})\cdot \overline{D}^2=-42m^2\neq 0,
$$
the image $\Lambda:= \bar{\theta}(\overline{D})$ is a curve on $X_0$.

For $n\gg 0$, we consider a general member $\overline{G}\in\mo|-n 
K_{\overline{X}}|$, its image 
$G_0:= \bar{\theta}(\overline{G})$, and the morphism-restriction 
$$
\mu=\bar{\theta}|_{\overline{G}}\colon \overline{G}\longrightarrow G_0.
$$
Then $\overline{G}$ is a nonsingular surface and the surface $G_0$ is
singular 
exactly at the intersection points $\{P_1,\dots,P_k\}=\Lambda\cap G_0$.
Put $l_i:= \mu^{-1}(P_i)_{\mathrm{red}}$. Then $\sum l_i=\overline{G}\cap D$
(as a scheme) and $K_{\overline{X}}\cdot l_i=K_{\overline{G}}\cdot l_i=0$.
This means that the morphism $\mu\colon \overline{G}\to G_0$ is crepant and the 
singularities of the surface $G_0$ are Du Val.
For any component of $l_i'\subset l_i$ we have $l_i\cdot l_i'=D\cdot l_i'<0$.
This implies that every curve $l_i$ either is irreducible or is a
union of two irreducible components meeting each other transversally at a 
single point.
In both cases $l_i^2=-2$. Thus
$$
-2k=\sum l_i^2=\overline{G}\cdot \overline{D}^2=(-nK_{\overline{X}})\cdot
\overline{D}^2=-42m^2n, \quad 21m^2n=k.
$$
On the other hand,
$$
k\deg \varphi(l_i)=\overline{M}\cdot \sum l_i=\overline{M}\cdot 
\overline{G}\cdot D=14 n m
$$
Therefore, $3m\deg \varphi(l_i)=2$, a contradiction.
\end{proof}

Thus there exists a flop and we can include our maps to the commutative 
diagram~\eqref{CD:IV-constr}, where $Y=Q$. Furthermore, similar to
the proof of Theorem~\ref{theorem:9-10-12} we obtain that $X$ is a nonsingular
Fano threefold of index~$1$ and genus~$7$, and $\sigma$ is the blowup of a conic on 
$X$.
The theorem is proved.
\end{proof}

Finally, we note that our birational technique does not allow obtain a description of
Fano threefolds of genus 6. It turns out that any Sarkisov link results again to a variety of the same type.
As we already noticed, these Fano threefolds can be described by using the vector bundle
method~\cite{Gushelcprime1982}, see
also~\cite{Debarre-Kuznetsov:GM} and~\ref{Gushel-Mukai}.
However our birational techniques works well for Fano threefolds of 
genus~6 admitting singularities. For example, the following assertion holds.

\begin{prp}
\label{prop:g=6}
Let $V=V_{10}\subset\PP^7$ be an anticanonically embedded Fano threefold
of genus $6$ with $\uprho(V)=1$ having exactly one singular point, which is 
ordinary double.
\begin{enumerate}
\item
\textup(see~\cite[5.6.2]{Beauville:Prym} and 
\cite{Debarre-Iliev-Manivel-2011,P:factorial-Fano:e}\textup)
If the variety $V$ is $\QQ$-factorial, then it suits to a link
of the form~\eqref{diagram}:
$$
\vcenter{
\xymatrix@R=1em{
&\widetilde{X}\ar[dr]_{\theta}\ar@/_4.0pt/[ddl]_{\sigma}\ar@{-->}[rr]^\chi &&
\overline{X}\ar[dl]^{\bar{\theta}}\ar@/^4.0pt/[ddr]^{\varphi}
\\
&&X_0&&
\\
V\ar@{-->}[rrrr]^{}&&&&\PP^2
}}
$$
where $\sigma$ is a Mori contraction of type \type{B_3} \textup(the blowup 
of 
singular 
point\textup) and $\varphi$ is a conic bundle whose discriminant curve
has degree $6$.

\item
\textup(see~\cite[Example~1.11]{Przhiyalkovskij-Cheltsov-Shramov-2005en}\textup)
If the variety $V$ is not $\QQ$-factorial, then it is the midpoint of the link~\eqref{diagram}:
$$
\vcenter{
\xymatrix@R=1em{
&\widetilde{X}\ar[dr]_{\theta}\ar@/_4.0pt/[ddl]_{\sigma}\ar@{-->}[rr]^\chi &&
\overline{X}\ar[dl]^{\bar{\theta}}\ar@/^4.0pt/[ddr]^{\varphi}
\\
&&V&&
\\
X\ar@{-->}[rrrr]^{}&&&&\PP^1
}}
$$
where $\sigma\colon \widetilde{X}\to X$ is the blowup of a line on a nonsingular 
del Pezzo variety of degree~$2$ and $\varphi\colon \overline{X}\to \PP^1$ 
is a del Pezzo fibration of degree $3$.
\end{enumerate}
\end{prp}

\begin{rem}[(Added in translation)]
A complete classification of one-nodal Fano threefolds of Picard number one was obtained in 
the recent work \cite{KP:1node}.
It uses the Sarkisov links techniques similar to one in 
Proposition~\ref{prop:g=6}. 
\end{rem}

\subsection{Summary of the results} 

Let us summarize the results.
The results of previous lectures allow us to compile the following table of Fano threefolds with $\Pic(X)\simeq \ZZ\cdot K_X$.
Note, however that we did not provide a proof of the fact that all the varieties of genus $6$
are described by the row~\ref{tabl:g=6} of the table.
A sketch of the proof and references will be given in Section~\ref{Gushel-Mukai}.

\begin{table}[ht]\small
\caption{Fano threefolds of index~$1$ with 
$\uprho(X)=1$}
\label{table-main}
\begin{center}\def\arraystretch{1.3}
\begin{tabular}{|l|r|r|p{120mm}|}
\hline
& \heading{$\g$} & \heading{$\hr^{1,2}$} & \heading{$X$}
\\\hline
\setcounter{NNN}{0}\rownumber
\label{tabl:g=2} & $2$ & 52 &a hypersurface of
degree $6$ in $\PP(1^4,3)$
\\\hline
\multirow2{*}{\rownumber
\label{tabl:g=3} } & \multirow2{*}{$3$} &
\multirow2{*}{$30$} &$X=X_4\subset\PP^4$, a hypersurface of degree $4$
\\\cline{4-4}
& & &a double cover of a nonsingular quadric $Q\subset\PP^4$ branched over 
a surface $B=B_8\subset Q$ of degree~$8$
\\\hline
\rownumber & $4$ & 20 &$X=X_6\subset\PP^5$, an intersection of a quadric and 
a cubic
\\\hline
\rownumber & $5$ & 14 &$X=X_8\subset\PP^6$, an intersection of three quadrics
\\\hline
\multirow2{*}{\rownumber
\label{tabl:g=6} } & \multirow2{*}{$6$} &
\multirow2{*}{$10$}
&$X=X_5\subset\PP^6$, a section of Grassmannian $\Gr(2,5)$
embedded by Pl\"ucker to $\PP(\wedge^2 \CC^5)=\PP^9$ by a 
subspace of codimension~$2$ and a quadric
\\\cline{4-4}
& & &a double cover of a nonsingular del Pezzo threefold $Y=Y_5\subset\PP^6$
of degree $5$ branched over a surface
$B=B_{10}\subset Y$ of degree~$10$
\\\hline
\rownumber & $7$ & 7&$X=X_{12}\subset\PP^8$
(see~\ref{prop:g=7:constr})
\\\hline
\rownumber & $8$ & 5&
$X=X_{14}\subset\PP^9$
(see~\ref{th:g8:constr})
\\\hline
\rownumber & $9$ & 3&$X=X_{16}\subset\PP^9$
(see~\ref{theorem:9-10-12}\ref{theorem:9-10-12:9})
\\\hline
\rownumber & $10$ & 2&$X=X_{18}\subset\PP^{10}$
(see~\ref{theorem:9-10-12}\ref{theorem:9-10-12:10})
\\\hline
\rownumber & $12$ & 0& $X=X_{22}\subset\PP^{13}$
(see~\ref{theorem:9-10-12}\ref{theorem:9-10-12:12})
\\\hline
\end{tabular}
\end{center}
\end{table}

\begin{rem}
Two types of varieties in the row~\ref{tabl:g=3} of Table~\ref{table-main} can be combined in one
irreducible family: all of them are intersections of hypersurfaces of 
degree
2 and 4 in the weighted projective space $\PP(1^5,2)$
(see Exercise~\ref{zad:g=3} in Section~\ref{sec5}).

Similarly, two types of varieties in the row~\ref{tabl:g=6} of Table~\ref{table-main} also can be combined
in one irreducible family: all of them are sections of the cone
$\widetilde{\Gr(2,5)}\subset\PP^{10}$
over the Grassmannian
$\Gr(2,5)\subset\PP^9$ by a linear
subspace of codimension $3$ and a quadric.
\end{rem}

Note however that the Fano threefolds of Picard number~$1$ of genus $7\le g\le10$
admit more invariant description as sections of some homogeneous
spaces.
For genus $g=12$ also there are alternative descriptions.
The methods are used in this approach are different from those used in these notes.
An overview of these techniques see in~\ref{Gushel-Mukai}.

From Table~\ref{table-main} and Table~\ref{table:del-Pezzo-3-folds} on
page~\pageref{table:del-Pezzo-3-folds} we immediately obtain the following 
interesting observation.

\begin{cor}
\label{cor:H3}
There exist exactly four types of Fano threefolds with $H^3(X,\ZZ)=0$ and 
$\uprho(X)=1$:
\begin{quote}
$\PP^3$,\quad $Q_2\subset\PP^4$,\quad $X_5\subset\PP^6$,\quad $X_{22}\subset
\PP^{13}$.
\end{quote}
\end{cor}

\begin{zadachi}
\eitem
Prove that any Fano threefold of genus $8$
can be also obtained from a nonsingular cubic in $\PP^4$
by using a Sarkisov link with center an elliptic curve of degree $5$
(see~\cite{Iskovskikh1980}).

\eitem
\label{zad:g7d10:hyp}
Let $Z\subset\PP^3$ be a curve of genus~$7$ and degree~$10$ lying on a quadric.
Prove that it is hyperelliptic.

\eitem
Using the construction from
Theorems~\ref{th:g8:constr} and~\ref{prop:g=7:constr} compute the dimension of
moduli spaces of Fano threefolds $X=X_{2g-2}\subset
\PP^{g+1}$ of genus $g=7$ and $8$.

\eitem
Prove Proposition~\ref{prop:g=6}.
\end{zadachi}

\newpage\section{Other methods}
\label{sec13}

\subsection{Classification of Fano threefolds with 
$\uprho>1$ by means of extremal rays}

Fano threefolds of Picard rank $\uprho>1$ were classified by Mori and Mukai
\cite{Mori-Mukai-1981-82} by means of the extremal rays techniques. Unlike
the case $\uprho=1$ the classification in this case is very extensive we
are not even trying to reproduce it here entirely. We provide only the main result.

\begin{teo}[see~\cite{Mori-Mukai-1981-82}]
\label{theorem51}
\begin{enumerate}
\item
There exist exactly $88$ deformation families of Fano threefolds 
of Picard number $\uprho>1$.
\item
Any Fano threefold $X$ with $\uprho(X)\ge 6$ is isomorphic to
the product $S\times \PP^1$, where $S$ is a del Pezzo surface of degree
$11-\uprho(X)$. In particular, $\uprho(X)\le 10$.
\item
If $X$ is a Fano threefold with $\uprho(X)\ge 4$, then it is the
blowup of another Fano threefold along a smooth curve and also $X$ has
a structure of a \textup(not necessarily relatively minimal\textup) conic bundle over a smooth del Pezzo surface.
\end{enumerate}
\end{teo}

The proof is based on detailed analysis of extremal rays on Fano threefolds.
Let us give a very brief sketch of the main idea of the classification and the proof of
Theorem~\ref{theorem51}.

\begin{dfn}
A Fano threefold $X$ is called \textit{imprimitive}, if it
can be obtained by blowing up of a curve on another Fano threefold~$Y$.
In the opposite case the Fano threefold $X$ is called \textit{primitive}.
\end{dfn}

It turns out that in the most cases, for a Fano threefold~$X$ of Picard 
number $\uprho(X)>1$, there exists an extremal ray of type~\type{B_1}.
One would expect that then $X$ is imprimitive. However, this is not always the case.

\begin{dfn}
Let $X$ be a nonsingular three-dimensional variety and let $f\colon X\to Y$ be 
a 
the contraction of an extremal ray of type \type{B_1} \textup(i.e. $f$ is the blowup of a
nonsingular
curve $C\subset Y$\textup). We say that the contraction~$f$ is of type
\type{B_1^0}, if $C\simeq \PP^1$ and
$$
\NNN_{C/Y}\simeq \OOO_{\PP^1}(-1)\oplus \OOO_{\PP^1}(-1).
$$
In the opposite case we say that $f$ is of type \type{B_1^1}.
\end{dfn}

\begin{rem}
\label{remark:B10}
Let $f\colon X\to Y$ be a the contraction of an extremal ray of type \type{B_1^0}, 
let
$E$ be the exceptional divisor and let $C:= f(E)$. Then $K_Y\cdot C=0$
(in particular, $Y$ is not a Fano threefold) and there exists another extremal
contraction $f'\colon X\to Y'$ of the divisor $E$ to a curve $C'\subset Y'$. 
The induced birational map $Y\dashrightarrow Y'$ is an Atiyah--Kulikov flop~\ref{flop:AK}.

Note also that since the divisor $E$ can be contracted to a point, it intersects
all the curves $\Gamma\subset E$ negatively.
\end{rem}

\begin{prp}
Let $f\colon X\to Y$ be the blowup of a smooth curve $C$ on a three-dimensional 
nonsingular
projective variety $Y$. Assume that $X$ is a Fano threefold.
Then $Y$ is a weak Fano threefold, i.e. its anticanonical divisor
$-K_Y$ is nef and big.
Moreover, $C$ is a unique curve on $Y$ such that $K_Y\cdot C=0$.
\end{prp}

\begin{proof}
Let $E$ is the exceptional divisor.
Assume that divisor
\begin{equation}
\label{eq:rho=2:KXKYE}
f^*(-K_Y)=-K_X+E
\end{equation}
is not nef. Then there exists an extremal ray $\rR$ on
$X$ such that $(-K_X+E)\cdot \rR<0$. Let $f'\colon X\to Y'$ be its contraction. 
Then $E\cdot \rR<K_X\cdot \rR<0$. Hence, for a minimal rational
curve $l$ generating the ray $\rR$, we have $E\cdot l\le -2$. This contradicts
Theorem~\ref{class:ext-rays}.

Thus the divisor $-K_Y$ is nef. It can be seen from the decomposition
\eqref{eq:rho=2:KXKYE} that $-K_Y$ is also big. By the Base Point 
Free Theorem~\ref{th:mmp:bpf} the linear system $\mo|-nK_Y|$, $n\gg 0$ defines
a (birational) contraction $\theta\colon Y\to Y_0$. Let $Z\subset Y$ be 
an irreducible curve such that $-K_Y\cdot Z=0$ and let $\widehat{Z}\subset X$ 
be any irreducible curve dominating $Z$. Then $E\cdot \widetilde{Z}=K_X
\cdot \widetilde{Z}<0$. Thus $\widetilde{Z}\subset E$ and $Z=C$.
\end{proof}

\begin{cor}
Let $f\colon X\to Y$ be the blowup of a smooth curve $C$ on a three-dimensional 
nonsingular
projective variety $Y$. Assume that $X$ is a Fano threefold
and $Y$ is not Fano. Then $f$ is a contraction of type \type{B_1^0}.
\end{cor}

\begin{proof}
We use the notation of the proof of the last proposition.
Then
$\theta$ is a flopping contraction that contracts only $C$. This implies 
that
$C\simeq \PP^1$ (see Corollary~\ref{cor:flops:rat}).
Let $Z\subset E$ be a very ample divisor. By the Cone Theorem its
class is a convex linear combination of the classes of extremal curves:
$Z\approxident \sum_{i=1}^r \alpha_i l_i$ and there is more than one term in this sum. 
Since the
divisor $-K_Y$ is nef and $f^*(-K_Y)\cdot Z=0$, we have $f^*(-K_Y)\cdot
l_i=0$ and $E\cdot l_i<0$ for all $i$ (see~\eqref{eq:rho=2:KXKYE}).
Since the contraction $f$ is extremal, it contracts only one curve from
$l_1,\dots,l_r$.
Hence, there exists an extremal curve
$l_i$ that is not contracted by the morphism $f$. Here the contraction 
$f'\colon X\to Y'$ of
the curve~$l_i$ is also birational and
$E$ is the exceptional divisor for both $f$ and $f'$,
i.e. $E$ is a ruled surface with two distinct rulings.
In this case, $E\simeq \PP^1\times \PP^1$.
Therefore, $f\colon X\to Y$ suits to the flop diagram of the Atiyah--Kulikov flop
(Fig.~\ref{ris:flop}, page~\pageref{ris:flop}).
\end{proof}

\begin{lem}
\label{lemma:cb0}
Let $X$ be a Fano threefold with $\uprho(X)>1$. Assume that there exists an extremal ray $\rR$ of type \type{D}
on $X$. Then $\uprho(X)=2$ and the second extremal ray on $X$ is of type \type{C} or \type{B_1^1}.

In other words, $X$ either has a structure of standard conic bundle over $\PP^2$ or is the blowup of a curve on a Fano threefold of Picard number~$1$.
\end{lem}

\begin{proof}
Let $f\colon X\to Y$ be the contraction of the ray $\rR$. Then $Y\simeq \PP^1$ 
and $\uprho(X)=2$
(see Corollaries~\ref{cor:Omega} and~\ref{cor:mmp:contraction}).
Therefore, on $X$ there exists one more extremal ray $\rR'\neq
\rR$.
Let $f'\colon X\to Y'$ be its contraction. Assume that $f'$ has a two-dimensional
fiber
$E'$. Then $E'$ meets the fibers of $f$, hence there exists a curve $l$ contracted by
both morphisms $f$ and $f'$. But this means that $\rR=\RR_+[l]=\rR'$, a 
contradiction. Therefore, the contraction~$f'$ has no two-dimensional fibers, and hence
it is of type \type{C} or \type{B_1}. In latter case $Y'$ is a Fano threefold (because $\uprho(Y')=1$). Therefore, $f'$ cannot be of type 
\type{B_1^0}.
\end{proof}

\begin{prp}
\label{proposition-cb}
Let $X$ be a Fano threefold with $\uprho(X)>1$. Assume that
$X$ is primitive. Then $X$ has a structure of standard conic bundle over 
$\PP^2$ or $\PP^1\times \PP^1$. In particular, $\uprho(X)\le 3$.
\end{prp}

\begin{proof}[(cf. proof Lemma~\ref{lemma-intersection-exceptional})]
Let $\rR_1,\dots,\rR_n$ be all the all extremal rays on $X$. For each
extremal ray $\rR_i$, let $f_i\colon X\to Y_i$ be the corresponding
extremal contraction.

Assume that $X$ has no conic bundle structures.
By Lemma~\ref{lemma:cb0} we also may assume that $X$ has no contractions of type
\type{D}, (i.e.
has no structures of del Pezzo fibration).
Then any extremal ray $\rR_i$ is of type \type{B_2}--\type{B_5} or
\type{B_1^0}.
Let 
$E_i$ be the corresponding exceptional divisor.

Assume that for some $i\neq j$ the divisors $E_i$ and $E_j$ meet each other.
Then the intersection $E_i\cap E_j$ contains a curve $D$ and since the surface
$E_i$ is isomorphic to either $\PP^2$, $\PP^1\times \PP^1$ or a quadratic cone in 
$\PP^3$, 
there exists a curve $D'\subset E_i$ that is algebraically equivalent to $D$ and
such that $D'\not \subset E_j$. Thus $E_j\cdot D=E_j\cdot D'\ge 0$. This
contradicts the fact that $E_j$ is an exceptional divisor contracted to a point
(see Remark~\ref{remark:B10}). Thus the divisors $E_1,\dots,E_n$
do not meet each other (pairwise).

Let now $L$ be the curve on $X$ obtained as an intersection of two general
hyperplane sections. Then $E_i\cdot L>0$ for all $i$. On the other hand,
the class of $L$ can be decomposed as a convex linear combination of some elements of extremal rays. Each ray intersects negatively exactly one
divisor from the set $E_1,\dots,E_n$ and intersects trivially other ones.
This implies that $E_i\cdot L\le 0$. The contradiction shows that $X$ has an
extremal contraction $f_k\colon X\to Y_k$ of type \type{C_1} or \type{C_2}, i.e.
$Y_k$ is a surface and $f_k$ is a standard conic bundle.

Assume that the surface $Y_k$ contains a $(-1)$-curve $l$ and let
$F:= f_k^{-1}(l)$. By the projection formula divisor $F$ is not nef. 
Therefore, there exists an extremal ray $\rR_j$ such that $F\cdot
\rR_j<0$. Since our Fano threefold is primitive, the ray $\rR_j$ is of type
\type{B_2}--\type{B_5} or \type{B_1^0}. But then $F$ has negative
intersection with any curve $\Gamma\subset F$
and cannot contain any fibers of $f_k$.
\end{proof}

\subsection{Fano threefolds of Picard number~$2$}
Let us enumerate all Fano threefolds $X$ with $\uprho(X)=2$.
In this case on variety $X$ there are exactly two extremal rays
$\rR_1,\,\rR_2\subset\NE(X)$ and we can apply the computations
from Lemmas~\ref{lemma:sl-comput} and~\ref{lemma-F}.
We start with the primitive case.

\subsubsection{}
\label{rho=2:primitive}

Let $X$ be a primitive Fano threefold with $\uprho(X)=2$. According to
Proposition~\ref{proposition-cb}, one of the extremal rays on $X$ defines
a structure of conic bundle $f\colon X\to \PP^2$. Let $\Delta\subset\PP^2$ be 
the discriminant curve and let
$$
d:= \deg \Delta.
$$
Let $f'\colon X\to Y'$ be the contraction of the second extremal ray. Since $X$ 
is primitive, $f'$ can be of only one of the types \type{C_1}--\type{C_2},
\type{D_1}--\type{D_3}, \type{B_2}--\type{B_5}.

Let us introduce the notation.
Let $l$ be a line on $\PP^2$ and let $M:= f^* l$.
If $f$ is a $\PP^1$-bundle (type \type{C_1}), then
$K_X$ and $M$ generate a sublattice of index~$2$ in $\Pic(X)$.
On the other hand, if the discriminant curve $\Delta$ is non-empty (type 
\type{C_2}), then 
$K_X$ and
$M$ generate the whole lattice $\Pic(X)$. The intersection theory on $X$ has 
the 
following form (see Lemma~\ref{lemma:sl-comput}):
$$
M^3=0, \quad M^2\cdot (-K_X)=2, \quad M\cdot (-K_X)^2=12-d, \quad
(-K_X)^3=2g-2,
$$
where $g:= \g(X)$. In particular, it follows that $d<12$.

In the cases where the second contraction $f'\colon X\to Y'$ is of type
\type{D_1}--\type{D_3}, by $D$ we denote fiber of $f'$.
In the cases \type{C_1}--\type{C_2} by $D$ we denote inverse image of general line
$l\subset Y'\simeq \PP^2$.
Finally, in the cases \type{B_2}--\type{B_5} by $D$ we denote the exceptional
divisor.

Similar to Lemma~\ref{lemma-F} we write
$$
D\sim a (-K_X)- b M,
$$
where $a$ and $b$ are positive numbers. Moreover these numbers are integral if 
$f$ is of type \type{C_1}
and are contained in $\frac12 \ZZ$ if $f$ is of type \type{C_2}.
Then we consider the cases according to Lemma~\ref{lemma:sl-comput}. The results
will be summarized in Table~\ref{tabl:rho2:prim}.
Note that our numbering is different from the one used 
in~\cite{Mori-Mukai-1981-82}.

\smallskip
\textit{The cases} \type{D_1}--\type{D_3}.
Then $Y'\simeq \PP^1$ and $f'$ is a del Pezzo fibration.
Let $d'$ is the degree general fiber. Then as in Lemma 
\ref{lemma:sl-comput}, we write
$$
\begin{aligned}
d'&= (-K_X)^2\cdot D=(2g-2) a-(12-d) b,
\\[3pt]
0&= (-K_X)\cdot D^2=(2g-2) a^2-2(12-d)ab+2b^2,
\\[3pt]
0&= D^3/a=(2g-2) a^2-3(12-d)ab+6b^2.
\end{aligned}
$$
Subtracting the third equality from the second one we obtain
$$
(12-d)a=4b,\qquad 8(2g-2)=3 (12-d)^2.
$$
From these relations by simple search we obtain for $(d,g)$ only the following 
possibilities:
$$
(d,g)=\ (0,28),\quad (4,13),\quad (8,4).
$$
If $(d,g)=(0,28)$, then $b=3a$ and $d'=18a$. Since $a\in\frac12\ZZ$ and $d'\le
9$, the only possibility is
$d'=9$, $a=1/2$, and $b=3/2$. We obtain the case~\ref{Fano2:P1P2} in Table 
\ref{tabl:rho2:prim} below.

If $(d,g)=(4,13)$, then $b=2a$ and $d'=8a$. In this case $\Delta\neq 0$, hence
$a\in\ZZ$. Thus
the only possibility is
$d'=8$, $a=1$, and $b=2$. We obtain the case~\ref{Fano2:D2C1} in Table 
\ref{tabl:rho2:prim}.

Finally, if $(d,g)=(8,4)$, then $b=a$ and $d'=2a$.
If the anticanonical class $-K_X$ is very ample, then $X=X_6\subset
\PP^5$ is an intersection of a quadric and cubic (see
Exercise~\ref{problem:Fano:g-ge5} in Section~\ref{sec6}).
But in this case $\uprho(X)=1$ by the Lefschetz hyperplane theorem.
Hence, either the linear system $|{-}K_X|$ has base points or
the variety~$X$ is hyperelliptic. According to
Theorems~\ref{theorem-Bs} and~\ref{th:hyp:any},
we obtain the case~\ref{Fano2:D1C1}.

\smallskip
\textit{The cases} \type{C_1}--\type{C_2}.
Then $Y'\simeq \PP^2$ and $f'$ is a conic bundle.
Let $d'$ be the degree of the discriminant curve. Then
$$
\begin{aligned}
12-d'&= (-K_X)^2\cdot D=(2g-2) a-(12-d) b,
\\[3pt]
2&= (-K_X)\cdot D^2=(2g-2) a^2-2(12-d)ab+2b^2,
\\[3pt]
0&= D^3/a=(2g-2) a^2-3(12-d)ab+6b^2.
\end{aligned}
$$
Again subtracting the third relation from the second one we obtain
$$
(12-d)ab=4b^2+2,\qquad
(2g-2)a^2=6b^2+6.
$$
If $f\colon X\to \PP^2$ is a contraction of type \type{C_1}, then $a$ and $b$ 
are integral
(positive) numbers.
Then there are only two possibilities:
\begin{itemize}
\item
$b=a=1$, $d=d'=6$, $2g-2=12$ (the case~\ref{Fano2:C1C1}),
\item
$b=2$, $a=1$, $d=3$, $d'=0$, $2g-2=30$ (the case~\ref{Fano2:C1C2}).
\end{itemize}
And if $f\colon X\to \PP^2$ is a contraction of type \type{C_2}, then by 
symmetry we may assume that
$f'\colon X\to \PP^2$ is also a contraction of type \type{C_2}. Then $d=d'=0$, 
$a=1/2$, $b=1$, and $2g-2=48$ (the case~\ref{Fano2:C2C2}).

\smallskip
\textit{The cases} \type{B_2}--\type{B_5}.
Then $Y'$ is a (possibly singular) Fano threefold and $f'$ is the blowup 
of the maximal ideal sheaf of a point. Then
$$
\begin{aligned}
k&= (-K_X)^2\cdot D=(2g-2) a-(12-d) b,
\\[3pt]
-2&= (-K_X)\cdot D^2=(2g-2) a^2-2(12-d)ab+2b^2,
\\[3pt]
4/k&= D^3=(2g-2) a^3-3(12-d)a^2b+6ab^2.
\end{aligned}
$$
where $k=1$,~$2$, $4$ in the cases \type{B_5}, \type{B_3}--\type{B_4} and 
\type{B_2}, respectively.
As above, we obtain the following solutions:
\begin{itemize}
\item
\type{B_5}:
$2g-2=62$, $d=0$, $a=1/2$, and $b=5/2$ (the case~\ref{Fano2:C2B5}),
\item
\type{B_3}--\type{B_4}:
$2g-2=14$, $d=6$, $a=1$, and $b=2$ (the case~\ref{Fano2:C1B3}),
\item
\type{B_2}:
$2g-2=56$, $d=0$, $a=1/2$, and $b=2$ (the case~\ref{Fano2:C2B2}).
\end{itemize}
\begin{table}[ht]\small
\caption{Primitive Fano threefolds with $\uprho(X)=2$}
\label{tabl:rho2:prim}
\begin{center}\def\arraystretch{1.3}
\begin{tabular}{|c|r|r|p{0.55\textwidth}|c|c|}
\hline
No.&\heading{$(-K_X)^3$}&\heading{$\hr^{1,2}$}&\heading{$X$}
&$\rR_1$&$\rR_2$
\\\hline
\setcounter{NN}{0}\nr
\label{Fano2:D1C1}&6&20& a double cover of $\PP^1\times
\PP^2$ 
whose branch divisor
has bidegree $(2,4)$&\type{C_1}&\type{D_1}
\\
\nr
\label{Fano2:C1C1}&12&9&a) a divisor on $\PP^2\times\PP^2$ of bidegree
$(2,2)$&\type{C_1}&\type{C_1}
\\
&&&b) a double cover of the variety \textnumero\ref{Fano2:C2C2} whose branch 
divisor is
an element of the anticanonical linear system&&
\\
\nr
\label{Fano2:C1B3}&14&9& a double cover of the variety
\textnumero\ref{Fano2:C2B2} whose branch divisor is an element of the 
anticanonical linear system
& \type{C_1}&\type{B_3}--\type{B_4}
\\
\nr
\label{Fano2:D2C1}&24&2& a double cover of $\PP^1\times\PP^2$ whose branch
divisor has bidegree
$(2,2)$& \type{C_1}&\type{D_2}
\\
\nr
\label{Fano2:C1C2}&30&$0$& a divisor on $\PP^2\times\PP^2$ of bidegree
$(1,2)$&\type{C_2}&\type{C_1}
\\
\nr
\label{Fano2:C2C2}&48&$0$& a divisor on $\PP^2\times\PP^2$ of bidegree
$(1,1)$&\type{C_2}&\type{C_2}
\\
\nr
\label{Fano2:P1P2}&54&$0$&$\PP^1\times\PP^2$&\type{C_2}&\type{D_3}
\\
\nr
\label{Fano2:C2B2}&56&$0$&$\PP(\OOO\oplus\OOO(1))$ over
$\PP^2$&\type{C_2}&\type{B_2}
\\
\nr
\label{Fano2:C2B5}&62&$0$&$\PP(\OOO\oplus\OOO(2))$ over
$\PP^2$&\type{C_2}&\type{B_5}
\\\hline
\end{tabular}
\end{center}
\end{table}

\subsubsection{}
\label{rho=2:imprimitive}

Now let $X$ be an imprimitive Fano threefold with $\uprho(X)=2$ and
$g:= \g(X)$.
Then there exists an extremal contraction $f\colon X\to Y$ that is the blowup of 
a smooth curve~$\Gamma$ on a nonsingular Fano threefold $Y$.
If either the linear system $|{-}K_X|$ has base points or $X$ is 
hyperelliptic, then $Y$ is a del Pezzo threefold of degree~$1$ or~$2$,
according to Theorems~\ref{theorem-Bs} and~\ref{th:hyp:any} and $f$ is the
blowup of a curve $\Gamma=D_1\cap D_2$, where $D_i\in\mo|-\frac12 K_X|$
(the case~\ref{Fano2:blVd} of Table~\ref{tabl:rho2:imp-b} with $d=1$ and~$2$). 
Thus from now on we assume that
the linear system
$|{-}K_X|$ defines an embedding $X=X_{2g-2}\subset\PP^{g+1}$.
In particular, $\g(X)\ge 4$. Moreover, if $\g(X)\le 5$, then by Theorem 
\ref{th:trig:any} the variety
$X$ must be a complete intersection in $\PP^{g+1}$ and then $\uprho(X)=1$. The 
contradiction shows that $\g(X)\ge 6$.
Furthermore, the curve $\Gamma$ coincides with the base locus of the linear system
$f_*|{-}K_X|\subset |{-}K_Y|$.
According to Theorem~\ref{theorem-Bs-rho=1}, the linear system $|{-}K_Y|$ is
base point free.
Therefore, $f_*|{-}K_X|\varsubsetneq |{-}K_Y|$ and so
$$
\g(Y)>\g(X)\ge 6.
$$

Assume that $\iota(Y)=1$. Then $\Gamma$ cannot be a line according to
Corollary~\ref{corollary:line-flop}.
Furthermore, according to Proposition 
\ref{prop:index1-lines}\ref{prop:index1-lines-sh}
all the lines on $Y$ cover a (non-normal) ruled surface, which
must intersect $\Gamma$ (because $\uprho(Y)=1$).
Thus there exists a line $l$ on $Y$ meeting $\Gamma$. But then the
intersection number of the proper transform of $l$ on $Y$ and the canonical class
is non-negative. The contradiction shows that $\iota(Y)>1$. Let $f'\colon X\to 
Y'$
be the second extremal contraction. Then the classification is carried out by
a detailed analysis of the cases.
We omit long computations that are similar to ones in the paragraph
\ref{rho=2:primitive}.
Here is the result.

\begin{teo}
An imprimitive Fano threefold $X$ with $\uprho(X)=2$ is isomorphic to the 
blowup of
Fano threefold $Y$ with $\uprho(Y)=1$ and $\iota=\iota(Y)\ge 2$ along a
nonsingular curve $\Gamma\subset Y$ that is the scheme-theoretic intersection of elements of
the linear system $\MMM\subset \mo|-\frac{\iota-1}{\iota}K_Y|$. Conversely, 
any 
blowup of a Fano threefold $Y$ with $\uprho(Y)=1$ and $\iota=\iota(Y)\ge 2$ along
such a curve is an imprimitive Fano threefold.
\end{teo}

Now it is not difficult to obtain a complete list.
The results summarized in Tables~\ref{tabl:rho2:imp-a} and 
\ref{tabl:rho2:imp-b}. 
In particular, from the classification one can deduce the following
fact.

\begin{teo}[{\cite[Theorem~5.1]{Mori-Mukai1983}}]
\label{th:rho=2}
Let $X$ be a Fano threefold with $\uprho(X)=2$.
Let $f\colon X\to Y$ and $f'\colon X\to Y'$ be 
distinct extremal contractions. Then
$$
\Pic(X)=f^*\Pic(Y)\oplus f'^*\Pic(Y').
$$
\end{teo} 

In~\cite{Mori-Mukai1983} this result is proved without using
a detailed classification.

\begin{table}[ht]\small
\caption{Imprimitive Fano threefolds with $\uprho=2$ (both contractions are of 
type \type{B_1})}
\label{tabl:rho2:imp-a}
\begin{center}\def\arraystretch{1.3}
\begin{tabular}{|c|c|c|c|c|c|c|c|}
\hline
No.&$(-K_X)^3$&$Y$&\heading{$\deg \Gamma$}&$\g(\Gamma)$&$Y'$&\heading{$\deg 
\Gamma'$}&$\g(\Gamma')$
\\\hline
\setcounter{NN}{9}\nr&20&$\PP^3$ & $6$& $3$&$\PP^3$&6 &3
\\
\nr&24&$\PP^3$&$5$& $1$& $Q$ &5 &1
\\
\nr&26&$\PP^3$&5 &2& $V_4$ & 1& 0
\\
\nr&28&$Q$& 4 &0&$Q$& 4&0
\\
\nr&30&$\PP^3$& 4& 0&$V_5$& 2 & 0
\\
\nr&34&$Q$&3 &0& $V_5$ & 1& 0
\\
\hline
\end{tabular}
\end{center}
\end{table}

\begin{table}[ht]\small
\caption{Imprimitive Fano threefolds with $\uprho=2$ (the second contraction 
is not of type \type{B_1})}
\label{tabl:rho2:imp-b}
\begin{center}\def\arraystretch{1.3}
\begin{tabular}{|c|c|c|c|c|c|p{90mm}|}
\hline
No.&$(-K_X)^3$&$Y$&\heading{$\deg \Gamma$}&$\g(\Gamma)$ 
&$Y'$&\heading{$f'\colon X\to Y'$} \\\hline
\setcounter{NN}{15}\nr
\label{Fano2:blVd} & $4d$ & $V_d$& $d$&1&$\PP^1$ &
a del Pezzo fibration of degree $d$, ${1\le d\le 5}$
\\
\nr&18& $V_3$ & 1& 0 &$\PP^2$&a conic bundle whose discriminant curve has
degree
$5$
\\
\nr&22& $V_4$& 2 & 0& $\PP^2$&a conic bundle whose discriminant curve has
degree
$4$
\\
\nr&10&$\PP^3$&9 &10&$\PP^1$&a del Pezzo fibration of degree $3$
\\
\nr&14&$Q$&8 &5&$\PP^1$& a del Pezzo fibration of degree $4$
\\
\nr&16&$\PP^3$ & $7$& $5$&$\PP^2$& a conic bundle whose discriminant curve
has degree $5$
\\
\nr&20&$Q$&
$6$& $2$&$\PP^2$&a conic bundle whose discriminant curve has degree $4$
\\
\nr&22&$\PP^3$ & 6 &4&$V_3^{\mathrm{s}}$&a contraction to a Gorenstein point 
on a 
cubic $Y'=V_3^{\mathrm{s}}\subset\PP^4$ (type \type{B_3} or \type{B_4})
\\
\nr&26&$V_5$& 3 &0&$\PP^2$&a conic bundle whose discriminant curve has
degree $3$
\\
\nr&30&$Q$& $4$& $1$&$V_4^{\mathrm{s}}$&a contraction to a Gorenstein point on
$Y'=V_4^{\mathrm{s}}\subset\PP^5$ (type \type{B_3} or \type{B_4})
\\
\nr&32&$\PP^3$ & 4 &1&$\PP^1$&a quadric bundle
\\
\nr&38&$\PP^3$ & 3 &0&$\PP^2$&a $\PP^1$-bundle
\\
\nr&40&$\PP^3$ & 3 &1& &a contraction to a non-Gorenstein point on hypersurface
$Y'=Y_3^{\mathrm{n}}\subset\PP(1^4,2)$ of degree $3$ (type \type{B_5})
\\
\nr&40&$Q$& 2& 0&$\PP^1$&a quadric bundle
\\
\nr&46&$\PP^3$ & 2& 0&$Q$&a contraction to a nonsingular point on a quadric
\\
\nr&46&$Q$& 1& 0&$\PP^2$&a $\PP^1$-bundle
\\
\nr&54&$\PP^3$ & 1& 0&$\PP^1$&a $\PP^2$-bundle
\\
\hline
\end{tabular}
\end{center}
\end{table}

\subsection{Vector bundle method}
\label{Gushel-Mukai}

N. Gushel~\cite{Gushelcprime1982,Gushel:g8,Gushel:g8:new} and later
S. Mukai~\cite{Mukai:CK3F,Mukai-1989,Mukai-1992} used
the vector bundle method 
for the classification of Fano varieties\footnote{Added in translation: see also \cite{BKM}}. The main idea of this method is to construct some special vector bundle
$\EEE$ on Fano variety $X$, which is generated by global sections, and
hence it defines a morphism
$X\to \Gr(r,N)$, where $r:= \rk \EEE$ and $N:= \hr^0(X,\EEE)$.

This method works very well for Fano varieties of coindex $3$
with Picard number~1.
In this case a general member of the linear system $\mo|-\frac1{n-2}K_X|$ is again
a nonsingular Fano variety of coindex~3 if $n\ge 4$ and a nonsingular \K3 
surface if $n=3$ (see~\cite{Mella-1999,Ambro-1999}). The classification of
such Fano varieties of genus $g\le 5$ is completely similar to the 
three-dimensional case
(see Theorem~\ref{theorem-1.1}). Moreover, for $g\ge 5$ the linear system
$\mo|-\frac1{n-2}K_X|$ is very ample and defines an embedding $X\hookrightarrow
\PP^{g+n-2}$ so that the image has degree $2g-2$ and is an intersection of
quadrics. It follows also from the classification in the three-dimensional case 
that $2\le g\le 12$ and $g\neq 11$.

\begin{teo}[\cite{Mukai:CK3F}]
\label{th:mukai}
Let $X$ be a Fano variety of dimension $n$, coindex~$3$, and genus $6\le g \le 10$ such that $\uprho(X)=1$.
Let $X=X_{2g-2}\subset\PP^{g+n-2}$ be the embedding defined by the linear system
$\mo|-\frac1{n-2}K_X|$. Then
$X=X_{2g-2}\subset\PP^{g+n-2}$ is a linear section of some
$N$-dimensional Fano variety
$\Omega_{2g-2}\subset\PP^{g+N-2}$ of coindex $3$, where $N=N(g)$.
Moreover, for $7\le \g(X)\le 10$
the variety $\Omega_{2g-2}$ is homogeneous with respect to a connected simple
algebraic group $G$. All the varieties $\Omega_{2g-2}$ listed in Table 
\ref{table:Mukai} below.
\end{teo}

\begin{table}[ht]\small
\caption{``Maximal'' Fano varieties of coindex $3$}
\label{table:Mukai}
\begin{center}\def\arraystretch{1.3}
\begin{tabular}{|c|c|p{120mm}|c|}\hline
$g$ &$N(g)$& \heading{$\Omega_{2g-2}\subset\PP^{g+N-2}$} & $G$
\\ \hline
$6$ & $6$& the intersection of the cone $\rule{0pt}{15pt}
\widetilde{\Gr(2,5)}\subset\PP^{10}$
over the Grassmannian $\Gr(2,5)\subset\PP^9$ with a quadric &
\\
$7$& $10$&the orthogonal Grassmannian $\operatorname{OGr}(4,9)\subset\PP^{15}$
& $\operatorname{SO}_{10}(\CC)$
\\
$8$ & $8$&the Grassmannian $\Gr(2,6)\subset\PP(\wedge^2\CC^6)$&
$\operatorname{SL}_6(\CC)$
\\
$9$& $6$&the symplectic Grassmannian $\operatorname{SGr}(3,6)\subset\PP^{13}$&
$\operatorname{Sp}_6(\CC)$
\\
$10$& $5$&$\operatorname{G}_2/P\subset\PP^{13}$ & $\operatorname{G}_2$ \\ \hline
\end{tabular}
\end{center}
\end{table}

Let us explain the notation of in the table.
In the case $g=7$ the variety
$\Omega_{12}=\operatorname{OGr}(4,9)=\operatorname{OGr}(5,10)^+$
parametrizes
four-dimensional vector subspaces in $\CC^9$ that are
isotropic with respect to a
nondegenerate symmetric bilinear form on $\CC^9$. In the case $g=9$
the variety $\Omega_{16}=\operatorname{SGr}(3,6)$ parametrize three-dimensional
vector subspaces in $\CC^6$ that are isotropic with respect to a nondegenerate
skew-symmetric bilinear form on~$\CC^6$. For more detailed explanations and 
sketches of proofs we refer to 
the works~\cite{Mukai:CK3F,Mukai-1989,Mukai-1992}
(see also~\cite{Debarre-Kuznetsov:GM} for the case $g=6$).

Consider the case $g=10$ in more details. Let $G$ be a simple 
algebraic
group of exceptional type $\mathrm{G}_2$. Recall that $G$ has dimension
$14$ and rank~$2$. Let $P\subset G$ be a its parabolic subgroup of 
dimension
$9$, the stabilizer of highest (long) root. Then
$\Omega_{18}=G / P\subset\PP(\mathfrak {g})=\PP^{13}$ (the adjoint
variety), where
$\mathfrak{g}$ is the corresponding Lie algebra (of type \type{G_2}).

The group $G$ can be realized as 
the automorphism group of the algebra of Cayley's octonions 
$\OO_\CC$ over $\CC$. Let $\OO_\CC^0\subset\OO_\CC$ be the hyperplane of
purely imaginary octonions. Then the adjoint variety $\Omega_{18}$ can be
realized as subvariety in Grassmannian $\Gr (2,\OO_\CC^0)$ of isotropic
two-dimensional subspaces $\Lambda\subset\OO_\CC^0$ with respect to 
multiplication in 
$\OO_\CC$. The latter means that the restriction of the multiplication in $\OO_\CC$ 
to the 
subspace $\Lambda$ is trivial.

It turns out that Fano varieties of coindex~$3$ and genus~$12$ occurs only in 
dimension~$3$:
they cannot be represented as
\textit{linear} sections of varieties of higher dimension. However 
there is an
invariant description of the threefolds $X_{22}\subset\PP^{13}$:
there exists an embedding $X_{22} \hookrightarrow \Gr(3,7)$ so that the image 
consists
of three-dimensional subspaces in $\CC^7$ that are isotropic with respect to three linearly
independent bilinear skew-symmetric forms $\sigma_1, \sigma_2, \sigma_3\in
\Lambda^2\CC^{7\vee}$ (see~\cite{Mukai-1989,Mukai-2002,KP-Mu}).
Conversely, a three-dimensional subspace $\Lambda\subset\Lambda^2\CC^{7\vee}$
defines a Fano threefold $X_{22}\subset\PP^{13}$ as a subvariety in $\Gr(3,7)$ 
under some condition that $\Lambda$ is \textit{non-degenerate}~\cite[\S5]{Mukai-2002}.
This description results also to the following remarkable construction.

\begin{construction}[\cite{Mukai-1989}]
Let 
\begin{equation}
\label{eq:V22:quartics}
C=\{f(x_0, x_1,x_2)=0\}\subset\PP^2
\end{equation}
be a (not necessarily smooth) plane curve of degree $d$. Consider the set of
collections of $n$ lines
$$
L_i=\{\ell_i(x_0, x_1,x_2)=0\}\subset\PP^2,\qquad i=1,\dots,n
$$
such that there exist constants $\lambda_1,\dots,\lambda_n$ satisfying
the following condition
$$
f=\sum_{i=1}^n \lambda_i \ell_i^d.
$$
The union $\cup L_i$ of such lines is called a \textit{polar $n$-side
of~$C$}.
The closure of the set of all polar $n$-sides of $C$ regarded as
elements of the symmetric product $\Sym^n (\PP^2{^\vee})$ we denote by
$\mathrm{VSP}(C,n)$ (the variety of sums of powers).

We are interested in polar six-sides of plane quartics $C\subset
\PP^2$. A six-side $\cup L_i$ is called \textit{complete}, if there exists
a quadruple of points $P_1,\dots,P_4\in\PP^2$ such that $\cup L_i$
consists of lines connecting all pairs of points $P_i$, $P_j$, 
$i\neq j$.
\end{construction}

\begin{teo}[\cite{Mukai-1989,Mukai-1992,Schreyer2001}]
\label{theorem:VPS}
Assume that a quartic $C\subset\PP^2$ has no polar five-sides.
Then the variety $\mathrm{VSP}(C,6)$ polar six-sides of $C$ is a
Gorenstein Fano threefold of genus $12$. If, additionally, among
polar six-sides of $C$ there are no complete ones,
then the variety $\mathrm{VSP}(C,6)$ is nonsingular and has index~$1$.

Conversely, any \textup(nonsingular\textup) Fano threefold $X$ of index~$1$ 
and genus
$12$ with $\uprho(X)=1$ is represented in the form $\mathrm{VSP}(C,6)$ for 
some quartic $C\subset\PP^2$.
\end{teo}
The condition that the quartic~\eqref{eq:V22:quartics} has polar five-sides
is equivalent to vanishing of the determinant
$$
\det \,\bigl\|\partial_i\partial_j f \bigr\|_{1\le i,j\le 6},
$$
where
$$
\partial_1:= \partial^2/\partial x_0^2,\ \partial_2:= \partial^2/\partial
x_0\partial x_1,\dots,\partial_6:= \partial^2/\partial x_2^2.
$$
Thus
a general quartic $C\subset\PP^2$ has no polar five-sides.
In particular, Theorem~\ref{theorem:VPS} shows that there exists birational
isomorphism
between moduli spaces of Fano threefolds of genus~$12$ and index~$1$, and genus~$3$ curves.

\begin{exa}
If $f=(x_0^2+x_1^2+x_2^2)^2$, i.e. $C$ is double conic, then the variety
$\mathrm{VSP}(C,6)$ is the Mukai-Umemura threefold (see
Example~\ref{remark:MU:22}).
\end{exa}

Note that the vector bundle method for each value of genus 
$\g(X)$ is applied separately. Thus, for application of this method to the classification, 
the genus bound $\g(X)\le 12$, $g\neq 11$ is needed and it is obtained by other, for instance, birational methods. 
On the other hand, the
constructions in Theorem~\ref{th:mukai} are invariant (unlike
birational transformations of the form~\eqref{diagram}) and very helpful, for example, in the study
of the automorphism groups~\cite{P:JAG:simple,KPS:Hilb}.

\subsection{Boundedness of the degree}

Classification Theorems~\ref{theorem51} and~\ref{theorem:main} (see 
the introduction) imply the 
boundedness of the degree of Fano threefolds:
$(-K_X)^3\le 64$. A more rough estimate can be obtained without using
the classification (cf.~\cite[Ch.~IV, \S~4]{Isk:anti-e}):

\begin{teo}
\label{thm:bound}
Let $X$ be a Fano threefold with $\uprho(X)=1$. Then
$$
(-K_X)^3\le 72.
$$
\end{teo}

The significance of the method is that it can be applied to the proof of
boundedness of the degree (and singularities) of Fano threefolds with 
terminal singularities~\cite{Kawamata:bF} that immediately implies 
finiteness of the number of algebraic families of such varieties 
\cite{Kollar1985}.

\begin{proof}
The proof uses the Bogomolov--Miyaoka inequality~\cite{Bogomolov:ineq-e}.
Note that as $\Pic(X)\simeq \ZZ$, the semistable property does not depend on the
choice of polarization.
Since $\upchi(X,\,\OOO_X)=1$, by the Riemann--Roch Theorem in the standard 
form
(see e.~g.~\cite[Ch.~A, \S5, Exercise~6.7]{Hartshorn-1977-ag}) we have
\begin{equation}
\label{eq:Kc2}
(-K_X)\cdot \mathrm{c}_2(X)=24.
\end{equation}

Take a 
general surface $S\in\mo|-nK_X|$, $n\gg 0$.
It is a nonsingular surface of general type.

First we assume that the sheaf $\Omega_X^1$ is semistable.
The restriction $\Omega_X^1|_S$ is semistable as well
\cite{Mehta-Ramanathan:1982}.
Applying the Bogomolov--Miyaoka inequality~\cite{Bogomolov:ineq-e}, we obtain
$$
n(-K_X)^3=\mathrm{c}_1 \left(\Omega_X^1|_S\right)^2 \le 3
\mathrm{c}_2\left(\Omega_X^1|_S\right)=-3nK_X\cdot
\mathrm{c}_2\left(\Omega_X\right)\le 72n,
$$
i.e. the desired inequality holds in this case.

Let now $\Omega_X^1$ be not semistable.
Then there exists a destabilizing reflexive subsheaf $\EEE\subset
\Omega_X^1$
of rank $r<3$ such that
$$\frac 1 r \mathrm{c}_1(\EEE)\cdot (-K_X)^2> \frac 13
\mathrm{c}_1(\Omega_X^1)\cdot (-K_X)^2=-\frac 13 (-K_X)^3.
$$
Put $\FFF:= (\wedge^r\EEE)^{\vee\vee}$. Then $\FFF$ is a reflexive
subsheaf in $\Omega_X^r$
of rank~$1$. Since the variety $X$ is nonsingular, the sheaf $\FFF$ is invertible.
Consider the divisor $D=\mathrm{c}_1(\FFF)=\mathrm{c}_1(\EEE)$ and write $D\sim
tH$, where $-K_X=\iota H$, $\iota=\iota(X)$.
The the previous inequality can be rewritten as follows
$$-\frac{t}{\iota}<\frac{r}{3}.
$$

If $t>0$, then $\varkappa(X,\FFF)=\varkappa(X,D)=3$. This contradicts a
result from~\cite{Bogomolov:ineq-e}:
$\varkappa(X,\FFF)\le 1$ for any $\FFF\subset\Omega^q_X$ of rank~$1$.
Since $H^0(X,\Omega_X^q)=0$ for $q>0$, we have $t\neq 0$.

Thus
$$0<-t<\frac13 r\iota.
$$
In particular, $\iota\ge 2$.
According to the computations in the proof of Theorem~\ref{thm:large-index}, we 
have
$\iota\le4$ and also $(-K_X)^3=64$ for $\iota=4$ and $(-K_X)^3=54$ for 
$\iota=3$.
Thus we may assume that $\iota=2$, $r=2$ and $t=-1$.

If the sheaf $\EEE$ is not semistable, then there exists a reflexive subsheaf
$\mathscr V\subset\EEE\subset\Omega_X^1$
of non-negative degree.
This again contradicts~\cite{Bogomolov:ineq-e} and the vanishing
$H^0(X,\Omega_X^q)=0$.

Therefore, the sheaf $\EEE$ is semistable. Since $\mathrm{c}_1 (\EEE)=D=\frac12 
K_X$,
we have, as above,
$$
\frac n4(-K_X)^3=D^2\cdot S=\mathrm{c}_1 \left(\EEE\right)^2\cdot S 
=\mathrm{c}_1 \left(\EEE|_S\right)^2
\le 4\mathrm{c}_2\left(\EEE|_S\right).
$$
Note that the codimension of the singular locus of $\EEE$ is at least~$3$.
Thus we may assume that $\EEE$ locally free in a neighborhood of~$S$.
From the exact sequence
$$
0\longrightarrow \EEE \longrightarrow \Omega_X^1 \longrightarrow \Omega_X^1/\EEE\longrightarrow 0
$$
we obtain
$$
\mathrm{c}_2\left(\Omega_X^1\right)=
\mathrm{c}_2(\EEE)
+\mathrm{c}_1(\EEE)\cdot \mathrm{c}_1\left(\Omega_X^1/\EEE\right)=
\mathrm{c}_2(\EEE)
+\mathrm{c}_1(\EEE)\cdot \mathrm{c}_1(\Omega_X^1)
-\mathrm{c}_1(\EEE)^2.
$$
Thus
$$
\begin{aligned}
4n(-K_X)\cdot \mathrm{c}_2(X)
&=4S \cdot \mathrm{c}_2\left(\Omega_X^1\right)
= 4S \cdot \mathrm{c}_2(\EEE)+ 4S \cdot D\cdot \mathrm{c}_1(\Omega_X^1)- 4S 
\cdot D^2
\\[3pt]
&\ge \frac n4(-K_X)^3+2n (-K_X)^3 - n(-K_X)^3= \frac{5n}{4}(-K_X)^3.
\end{aligned}
$$
Canceling by $n$ we obtain
$$(-K_X)^3\le \frac{16}5 (-K_X)\cdot \mathrm{c}_2(X)< 77.
$$
Since $\iota=2$, the degree $(-K_X)^3$ must be divisible by $8$.
Therefore, $(-K_X)^3\le 72$.
\end{proof}

Note that the (co)tangent bundle of nonsingular Fano threefolds
with Picard number~$1$ is stable~\cite{Peternell-Wisniewski:tF}, but this is not 
always true for Fano threefolds of larger Picard number 
\cite{Steffens}.

\smallskip

Another method to prove the boundedness of Fano varieties is based on
the study of deformations of rational curves.
It allows us to get significantly more general but less effective than
Theorem~\ref{thm:bound} results.

\begin{teo}[\cite{Kollar-Miyaoka-Mori-1992c}]
Let $X$ be a nonsingular Fano variety of dimension $n$. Then its degree
$(-K_X)^n$ is bounded by a constant $\const(n)$ that depends only on dimension.
\end{teo}

\begin{cor}
Nonsingular Fano varieties of given dimension lie in a finite number of
algebraic families.
\end{cor}

The idea of the proof of theorem is to use the deformation techniques of 
rational
curves to show that general two points $P_1,P_2\in X$ of a
Fano variety $X$ can be connected by an \textit{irreducible} rational curve
$C$ of degree $-K_X\cdot C\le \const_1(n)$. If the degree $(-K_X)^n$ would not be
bounded, then by the asymptotic Riemann--Roch Theorem~\eqref {equation-RR}
one can construct a divisor $D\in\mo|-mK_X|$ having at the point 
$P_1$ 
the multiplicity
$> m\,\const_1(n)$. But then the curve~$C$ would be contained in $D$, which contradicts
our generality of the choice of the point~$P_2$. 

Now the above theorem implies the boundedness of the degree
$(-K_X)^n$ and also the finiteness of the choice of two highest
terms in the
Hilbert polynomial (see~\eqref {equation-RR}).

To prove the corollary, first from theorem it is
deduced the finiteness of the number of Hilbert polynomials of Fano varieties (of given dimension). Then the finiteness of the number of 
algebraic families follows from general result~\cite{Matsusaka1972}.

A generalization of the above facts to the case of singular Fano varieties was obtained recently by C.~Birkar 
\cite{Birkar-BAB}.

\subsection{Vector fields and automorphisms}

Recall that the automorphism group $\Aut(X)$ of any Fano variety $X$
is a linear algebraic group
(Theorem~\ref{thm-1-prop}\ref{1-properties-aut}). Thus the group
$\Aut(X)$ is infinite if and only if
its connected component of unity is not trivial.

Let $X$ be a Fano variety of dimension $n$. For the tangent vector bundle
$\TTT_X$ we have
$$
\TTT_X\simeq (\Omega_X^1)^\vee \simeq \Omega^{n-1} \otimes \OOO_X(-K_X).
$$
Then by the Kodaira--Nakano Vanishing Theorem
\cite[Ch.~1, \S2]{Griffiths1994}
$$
H^q(X, \TTT_X)=H^q(X, \Omega^{n-1} \otimes \OOO_X(-K_X))=0\quad \text{for
$q>1$}.
$$
Note that the space $H^0(X, \TTT_X)$ is naturally identified
with the
tangent space to the connected component of unity $\Aut(X)^0$ of the automorphism group 
\cite[\S2.3]{Akhiezer1995} and $H^1(X, \TTT_X)$ is identified with the tangent space 
to the versal deformation space of $X$ (see~\cite{Kodaira:book-def}).
Since $H^2(X, \TTT_X)=0$, the deformations of $X$ are unobstructed
and so the versal deformation space is nonsingular at the corresponding point and
has dimension
$\hr^1(X, \TTT_X)$. Let now the variety $X$ be three-dimensional. Then by 
the Riemann--Roch Theorem
$$
\hr^0(X, \TTT_X)- \hr^1(X, \TTT_X)=\frac 12(-K_X)^3-\frac{19}{24}(-K_X)\cdot
\mathrm{c}_2(X)+\frac 12\mathrm{c}_3(X).
$$
Taking the relations
$$
(-K_X)\cdot \mathrm{c}_2(X)=24 \upchi(X, \OOO_X)=24
$$
(see~\eqref{eq:Kc2}) and
$$
\mathrm{c}_3(X)=\chit(X)=2+2\uprho(X)-2\hr^{1,2}(X)
$$
into account 
we obtain
$$
\hr^0(X, \TTT_X)- \hr^1(X, \TTT_X)=\frac 12(-K_X)^3+\uprho(X)-\hr^{1,2}(X) -18.
$$
We can rewrite this formula in the form
$$
\hr^0(X, \TTT_X)- \hr^1(X, \TTT_X)=\g(X)+\uprho(X)-\hr^{1,2}(X) -19.
$$
In particular, if the automorphism group of $X$ is finite, then the dimension
of the deformation space is equal to
$$
\hr^{1,2}(X)+19- \g(X)-\uprho(X).
$$

Consider one example.
In family varieties $X=X_{22}\subset\PP^{13}$ of index~1 of genus~12 with 
$\uprho(X)=1$ there exists a distinguished one that is quasihomogeneous with respect to
an action of the group $\PSL_2(\CC)$.
The construction is similar to the construction in Remark~\ref{remark:MU}:

\begin{exa}[\cite{Mukai-Umemura-1983}]
\label{remark:MU:22}
Consider the action of the group $\SL_2(\CC)$ on the space $M_{12}\oplus \CC$ and
its projectivization $\PP(M_{12}\oplus \CC)$.
The polynomial
$$
f(x_1,x_2)=x_1x_2(x_1^{10}+11x_1^5x_2^5+x_2^{10})\in M_{12}
$$
is a semi-invariant of the binary icosahedron group $\mathrm{Ico}\subset
\SL_2(\CC)$
\cite[Ch.~4]{Springer1977}.
Then the closure
$$
\overline{\PSL_2(\CC)\cdot [f+1]}\subset\PP(M_{12}\oplus \CC)
$$
of the orbit
$$
\PSL_2(\CC)\cdot [f+1]=\SL_2(\CC)/ \mathrm{Ico}
$$
is a Fano threefold
$X_{22}\subset\PP^{13}$. It is called \textit{the Mukai-Umemura threefold}.
\end{exa}

Note that there exists only four types of Fano threefolds with infinite automorphism group and
$\uprho(X)=1$:

\begin{teo}[see~\cite{P:90aut:en,KPS:Hilb,KP:V22}]
Let $X$ be a Fano threefolds with $\uprho(X)=1$. Assume that
its automorphism group is infinite.

If $\iota(X)>1$, then $X$ is isomorphic to one of the following three 
varieties:
\begin{quote}
$\PP^3$,\quad $Q_2\subset\PP^4$,\quad $X_5\subset\PP^6$.
\end{quote}

If $\iota(X)=1$, then $\g(X)=12$ and there are three possibilities:
\begin{enumerate}
\item
$\Aut(X)\simeq \PSL_2(\CC)$ and $X$ is the Mukai-Umemura threefold,
\item
$\Aut(X)\simeq \mathbb G_{\mathrm a}\rtimes \ZZ/4\ZZ$ and the corresponding
variety $X$ is unique up to isomorphism,
\item
$\Aut(X)\simeq \mathbb G_{\mathrm m}\rtimes \ZZ/2\ZZ$ and $X$ the corresponding
variety $X$ belongs to some one-dimensional family.
\end{enumerate}
\end{teo}
For all these varieties the middle cohomology group $H^3(X,\ZZ)$
is trivial (see Corollary~\ref{cor:H3}).

\appendix
\newpage\section{Some facts from the Mori theory }
\label{section:mori}

\subsection{Cone of curves}

Let $X$ be a complete variety. By
$\rZ_k (X)$ we denote the group of \textit{$k$-dimensional cycles} on $X$, i.e. 
free abelian
group generated by \textit{complete} reduced irreducible subvarieties
of dimension $k$ in $X$. Similarly, by
$\rZ^k (X)$ we denote the group of cycles \textit{of codimension $k$} on $X$.

We say that two cycles $Z,Z'\in\rZ_1(X)$ are \textit{numerically equivalent}
(we denote this by $Z\approxident Z'$), if $L\cdot Z=L\cdot Z'$ for
all $L\in\Pic(X)$; by duality the numerical
equivalence on $\Pic(X)$ is also defined. The bilinear form $\Pic(X)\times \rZ_1(X)\to
\ZZ$ induces a nondegenerate bilinear form
$$
\N^1(X)\times \N_1(X)\to \RR, 
$$
where $\N^1(X):= (\Pic(X)/\approxident)\otimes\RR$ and $\N_1(X):= 
(\rZ_1(X)/\approxident)\otimes \RR$.
For an~$1$-cycle $Z$ by $[Z]$ we will denote its class in $\N_1(X)$.

According to the N\'eron--Severi Theorem the
space $\N^1(X)$ is finite-dimensional. Its dimension (which coincides with 
the dimension of $\N_1(X)$) is called the
\textit{Picard number} of the variety $X$ and denoted by $\uprho(X)$.
By the construction, every Cartier divisor defines a linear function on 
$\N_1(X)$.

In the space
$\N_1(X)$ we consider the convex cone $\operatorname{NE}(X)$,
generated by all effective~$1$-cycles.
Denote by
$\NE(X)$ its closure. Thus $\NE(X)$ is a closed
convex cone in the finite-dimensional real space $\N_1(X)$.
It is called the \textit{Mori cone}.
Moreover, $\NE(X)$ generates $\N_1(X)$. Note however that the
elements $\NE(X)$ are not necessarily represented by effective~$1$-cycles
and not necessarily have rational coefficients.

\smallskip

The cone $\NE(X)$ contains a lot of important information on the~variety.
In particular, in terms of this cone it can be formulated of an 
ampleness criterion of divisors:

\begin{teo}[Kleiman's Ampleness Criterion
\cite{Kleiman1966,KM:book}]
\label{Kleiman}
A Cartier divisor $H$ is ample if and only if it
defines a strictly positive function on $\NE(X)\setminus
\{0\}$.
\end{teo}

A divisor $D$ on $X$ is called \textit{nef}, if $D\cdot L\ge 0$
for any curve $L$ on $X$. In other words, a divisor is nef, if
the corresponding linear function on $\N_1(X)$ is non-negative on the cone 
$\NE(X)$.
The classes of nef divisors form a closed convex cone
$\operatorname{Nef}(X)\subset\N^1(X)$ that is
dual to the cone $\NE(X)$. By Theorem~\ref{Kleiman} the classes of ample
$\RR$-divisors fill the interior of
$\operatorname{Nef}(X)$.

\begin{dfn}
Let $\rC\subset\RR^n$ be a closed convex cone with vertex $0$
generating the whole space $\RR^n$.
A ray $\rR\subset\rC$ with vertex $0$ is called \textit{extremal}, if for
any elements 
$z_1,\, z_2\in\rC$ the inclusion $z_1+z_2\in\rR$ implies that $z_1,\, z_2\in
\rR$.
\end{dfn}

It can be shown that a ray $\rR\subset\rC$ is extremal if and only
if there exists a linear function $f\colon \RR^n \to \RR$ such that
$f(z)\ge 0$ for all $z\in\rC$ and $f(z)=0$ only if $z\in\rR$.
Such a function is said to be \textit{supporting} of the ray $\rR$.

\subsection{Basic facts from the Minimal Model Program}

\begin{teo}[{(Cone Theorem, S. Mori~\cite{Mori:3-folds}, 
\cite[Theorem~4.7]{ClemensKollarMori1988})}]
\label{cone-th}
Let $X$ be a nonsingular projective variety and let $H$ be an ample
divisor on $X$. For any $\varepsilon >0$ there exists at most a finite number 
of extremal rays $\rR_i\subset\NE(X)$ such that $(K_X+\varepsilon H)\cdot
\rR_i<0$. Each ray $\rR_i$ is generated by a class of an irreducible 
rational
curve $C_i$ such that
\begin{equation}
\label{eq:ocenka:Cdeg}
0<-K_X\cdot C_i\le \dim(X)+1.
\end{equation}
The cone $\NE(X)$ is generated by the cone
$$
\NE(X)\cap\{z \mid (K_X+\varepsilon H)\cdot z\ge 0\}
$$
and rays
$\rR_i$ \textup(see Fig.~\ref{ris-NE}).
\end{teo}

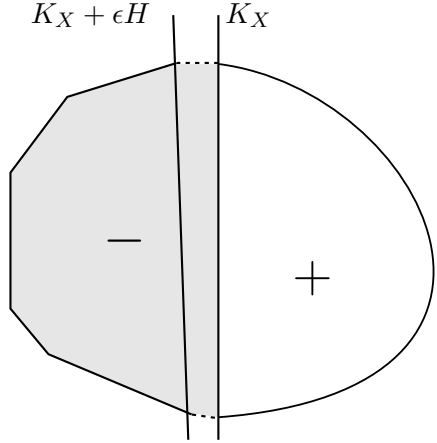
\begin{figure}[!t]\small
\begin{center}
\begin{tikzpicture}[xscale=.5, yscale=.5]
\draw[ultra thick,dotted] (4.5, 4.7)--(3.4,4.7);
\draw[ultra thick] (3.4,4.7)--(0.5,3.8)--(-1,1.8)--(-1,-1.8)--(0,
-3)--(3.8,-4.6);
\draw[ultra thick,dotted] (3.8,-4.6)--(4.5, -4.7);
\path[fill=gray!20] (4.5, 4.7)--(3.4,4.7)--(0.5,3.8)--(-1,1.8)--(-1,-1.8)--(0,
-3)--(3.8,-4.6)--(4.5, -4.7);
\draw[thick] (4.5,-5.3)--(4.5,6);
\draw[thick] (3.7,-5.3)--(3.3,6);
\draw[line width=0.7pt] (4.5, -4.7).. controls(14,-4) and (10,4).. (4.5, 4.7);
\node[font=\Huge] at (7,-1) {{$+$}};
\node[font=\Huge] at (2,0) {{$-$}};
\node[font=\large] at (5.3,6) {{$K_X$}};
\node[font=\large] at (1.1,6) {{$K_X+\epsilon H$}};
\end{tikzpicture}
\end{center}
\caption{The Mori cone $\NE(X)$ (a transversal section)}
\label{ris-NE}
\end{figure}

\begin{teo}[Base Point Free Theorem {\cite[Theorem~9.3]{ClemensKollarMori1988}}]
\label{th:mmp:bpf}
Let $X$ be a nonsingular projective variety and let $M$ be a nef Cartier divisor 
on~$X$ such that the divisor 
$aM-K_X$ is nef and big. Then there exists a number $b_0$ such that the linear
system $|bM|$ is base point free for all $b\ge b_0$.
\end{teo}

\begin{cor}
\label{cor:mmp:bpf}
In the notation of Theorem~\ref{th:mmp:bpf} there exists a contraction 
$\varphi\colon X\to Z$
to a normal projective variety $Z$
and an ample divisor $A$ on $Z$ such that $M=\varphi^*A$.
\end{cor}

\begin{teo}[Contraction Theorem
{\cite[Theorem~5.2]{ClemensKollarMori1988}}]
\label{th:mmp:contraction}
Let $X$ be a nonsingular projective variety and let $\rR \subset\NE(X)$ be 
an extremal ray that is negative with respect to $K_X$ \textup(i.e. such that
$K_X\cdot \rR<0$\textup).
Then there exists a contraction $\varphi\colon X\to Z$
to a normal projective variety $Z$ such that
the image $\varphi(\ell)$ of a the curve $\ell\subset X$ is a point if and only
if $[\ell]\in\rR$.
The contraction $\varphi\colon X\to Z$ is uniquely defined by the ray $\rR$.
\end{teo}

The contraction $\varphi\colon X\to Z$ considered above is called \textit{the 
contraction of the ray} $\rR$.

It should be noticed that we formulate Theorems~\ref{cone-th},~\ref{th:mmp:bpf}
and~\ref{th:mmp:contraction} for the case of nonsingular varieties. However the
corresponding facts are valid
(except for the bound~\eqref{eq:ocenka:Cdeg} in Theorem~\ref{cone-th}) in much 
more greater generality: instead of nonsingularity of the variety $X$ we may assume
that it has only log terminal singularities, see for instance~\cite{ClemensKollarMori1988}.

\begin{cor}
\label{cor:mmp:contraction}
In the notation of Theorem~\ref{th:mmp:contraction} let $D$ be a Cartier divisor 
on~$X$ such that $D\cdot \rR=0$. Then $D=\varphi^*M$ for some Cartier divisor $M$ on $Z$. Therefore, there is the following exact sequence
\begin{equation}
\label{exact-sequence-Pic}
0 \longrightarrow\Pic(Z) \xarr{\varphi^*} \Pic(X)
\xarr{\cdot \ell} \ZZ
\end{equation}
where $[\ell]\in\rR$. In particular, $\Pic(X)\simeq \Pic(Z)\oplus \ZZ$
and $\uprho(Z)=\uprho(X)-1$.
\end{cor}

Define the \textit{length} of an extremal ray as follows
\begin{equation}
\label{eq:def:length}
\upmu(\rR):= \min \{-K_X\cdot \ell \mid \text{$\ell$ is a rational curve such that
$[\ell]\in\rR$}\}.
\end{equation}
A curve, for which this minimum is achieved, we call a \textit{minimal
rational curve} of the ray $\rR$. It is clear that the order of the cokernel in 
the sequence~\eqref{exact-sequence-Pic} divides the length $\upmu(\rR)$ of the extremal ray.

For nonsingular varieties of dimension~$2$ and~$3$ extremal rays completely
classified. First we discuss the case of dimension~$2$.

\begin{teo}[see~\cite{Mori:3-folds}]
\label{class:ext-rays:surf}
Let $X$ be a nonsingular projective surface, let $\rR$ be a $K_X$-negative extremal
ray on $X$ and let $\ell$ be a minimal rational curve on $X$
generating $\rR$.
Then one of the following cases holds:
\begin{enumerate}
\item
$\upmu(\rR)=1$ and $\ell$ is a $(-1)$-curve;
\item
$\upmu(\rR)=2$, $X$ has a structure of $\PP^1$-bundle over a curve and $\ell$ 
is a fiber of this fibration;
\item
$\upmu(\rR)=3$, $X\simeq \PP^2$ and $\ell$ is a line.
\end{enumerate}
\end{teo}

The following important theorem was proved by S. Mori in 1982.

\begin{teo}[\cite{Mori:3-folds}]
\label{class:ext-rays}
Let $X$ be a nonsingular three-dimensional projective variety and let $\rR$ be 
a $K_X$-negative extremal ray on $X$.
Let $\ell$ be a minimal rational curve on $X$ generating $\rR$ and
let $f \colon X\to Z$ be the contraction of the ray $\rR$.

If $\rR$ is nef \textup(i.e. $D\cdot \rR\ge 0$
for any effective divisor $D$\textup), then $Z$ is nonsingular and
one of the cases presented in Table~\ref{tableA1} occurs.

If $\rR$ is not nef, then $f$ is a birational contraction of an irreducible exceptional divisor $E$, moreover, $f$ is the
blowup of the subvariety $f(E)$ \textup(with reduced structure\textup)
and one of the cases presented in Table~\ref{tableA2} occurs.
\end{teo}

\begin{table}[ht]\small
\caption{Extremal contractions of fiber type}
\label{tableA1} 
\begin{center}\def\arraystretch{1.3}
\begin{tabular}{|c|c|l|c|} \hline
Type & $\dim(Z)$ & \heading{$f \colon X\to Z$} & $\upmu(\rR)$
\\\hline
\type{C_1}& \multirow2{*}{$2$}& a conic bundle with non-trivial discriminant 
curve&$1$
\\\cline{1-1}\cline{3-4}
\type{C_2}& &a $\PP^1$-bundle& $2$
\\\hline
\type{D_1}&\multirow3{*}{$1$}&a del Pezzo fibration of degree $\le 
6$& $1$
\\\cline{1-1}\cline{3-4}
\type{D_2}&&a quadric bundle&$2$
\\\cline{1-1}\cline{3-4}
\type{D_3}&&a $\PP^2$-bundle&$3$
\\\hline
\type{F}&$0$&$X$ is a Fano threefold with $\Pic(X)\simeq \ZZ$ & $\iota(X)$ \\ 
\hline
\end{tabular}
\end{center}
\end{table}

\begin{table}[ht]\small
\caption{Birational extremal contractions}
\label{tableA2} 
\begin{center}\def\arraystretch{1.3}
\begin{tabular}{|c|p{0.8\textwidth}|c|} \hline
Type &\heading{$E$ and $f(E)$}& $\upmu(\rR)$
\\\hline
\type{B_1} &
$f(E)\subset Z$ is a smooth curve, $\ell$ is a fiber of ruled surface $E$,
$Z$ is a nonsingular variety & $1$
\\\hline
\type{B_2}&
$f(E)\in Z$ is a nonsingular point, $E\simeq \PP^2$, 
$\OOO_E(E)=\OOO_{\PP^2}(-2)$,
$\ell\subset E$ is a line &$2$
\\\hline
\type{B_3}&
$f(E)\in Z$ is an ordinary double point, $E\simeq \PP^1\times\PP^1$,
$\OOO_E(E)=\OOO_{\PP^1\times \PP^1}(-1,-1)$,
$\ell$ is a generator of
$\PP^1\times \PP^1$ \textup(here generators of $\PP^1\times \PP^1$ are numerically 
equivalent on $X$\textup)& $1$
\\\hline
\type{B_4}&
$f(E)\in Z$ is a double point that is analytically equivalent to the hypersurface
singularity
$x_1x_2+x_3^2+x_4^3=0$, $E$ is a quadratic cone in $\PP^3$,
$\OOO_E(E)=\OOO_E(-1)$, $\ell$ is a generator of the cone& $1$
\\\hline
\type{B_5}&
$f(E)\in Z$ is a non-Gorenstein point of multiplicity $4$ that is analytically equivalent to
the quotient singularity
$\CC^3/\{\pm \mathrm{id}\}$, $E\simeq \PP^2$,
$\OOO_E(E)=\OOO_{\PP^2}(-2)$,
$\ell$ is a line& $1$ \\ \hline
\end{tabular}
\end{center}
\end{table}

Unfortunately, in higher dimensions the situation much more complicated and 
a complete
classification very far to be completed. However there exists very a lot of 
partial
results in this direction (see for instance the
papers~\cite{Andreatta-Wisniewski:contr-1,Andreatta-Wisniewski:contr-2,Kachi1997}
and surveys~\cite{Andreatta-Wisniewski-view1997,P:G-MMP}).

\begin{rem}
\label{rem:new-contr}
In the case \type{C_2} the contraction $f$ is a smooth morphism and any 
its fiber is a nonsingular
rational curve. If, additionally, the surface $Z$ is rational, then
its Brauer group is trivial: $\operatorname{Br}(Z)=0$ \textup(see 
e.~g.~\cite{P:rat-cb:e}).
Thus the morphism $f$ has a rational section and it is a
$\PP^1$-bundle that locally
trivial in Zariski topology.
\end{rem}

\begin{rem}
If $\dim(Z)\neq 2$ or $\dim(Z)=2$ and the morphism is not smooth, then a curve
$\ell$
generating the ray $\rR$ can be chosen so that
the sequence~\eqref{exact-sequence-Pic} is also exact from the right:
\begin{equation}
\label{equation-exact-sequence}
0 \longrightarrow\Pic(Z) \xarr{f^*} \Pic(X)\xarr{\cdot \ell} \ZZ\longrightarrow 0.
\end{equation}
This holds also and in the case $\dim(Z)=2$ for smooth morphism $f$, if
the surface $Z$ is rational (see Remark~\ref{rem:new-contr}).
\end{rem}

\begin{rem}
\label{rem:Mdiscr}
Let the contraction $f$ be birational.
As in~\ref{subsection:discr}, we can write the standard formula
$$
K_X=f^* K_Z+\alpha E,
$$
where $\alpha \in\frac12\ZZ$. Intersecting both parts of this equality with a contracted
curve
$\ell$, we obtain that $\alpha$ takes the following values:
$$
\begin{array}{rccccc}
\text{Type}: & \Type{B_1}&\Type{B_2}&\Type{B_3}&\Type{B_4}&\Type{B_5}\\[5pt]
\alpha: & 1 & 2 & 1 & 1 & 1/2
\end{array}
$$
\end{rem}

\subsection{Flops}

A birational contraction $\theta\colon V\to V_0$ is said to be \textit{small}, 
if 
its exceptional
locus is of codimension $\ge 2$. If the contraction $\theta$ is not an 
isomorphism, then the variety
$V_0$ must be singular~\cite[Ch.~II,~\S4,~Theorem~2]{Shafarevich:basic}.

\begin{dfn}
Let $V$ be a nonsingular projective variety and let
$\theta \colon V\to V_0$ be a small birational contraction such that
the canonical class $K_V$ trivial on the fibers and $\uprho(V/V_0)=1$.
The \textit{flop} for $\theta$ is the following commutative diagram:
\begin{equation}
\label{eq:flop}
\vcenter{
\xymatrix@R=2em{
V\ar[dr]_{\theta}\ar@{-->}[rr]^{\chi}&&V'\ar[dl]^{\theta'}
\\
&V_0&
} }
\end{equation}
where the map $\chi$ is not an isomorphism, $\theta'$ is also small
birational contraction such that the canonical class $K_{V'}$ trivial on the fibers and
$\uprho(V'/V_0)=1$.
\end{dfn}

Note that the curves in the fibers of the morphism $\theta$ (respectively, 
$\theta'$) generate an extremal ray
$\rR\subset\NE(V)$ (respectively, $\rR'\subset\NE(V')$).
However these extremal rays intersect the canonical class trivially, hence
generally speaking, we cannot apply to them Theorem~\ref{th:mmp:contraction}
and Corollary~\ref{cor:mmp:contraction}.

Since the variety $V$ is nonsingular by our assumption, 
we have
$$
\Exc(\theta)=\theta^{-1}(\Sing(V_0))\quad \text{and}\quad
\Exc(\theta')=\theta'^{-1}(\Sing(V_0))
$$
by the Purity of Exceptional Locus Theorem 
\cite[Ch.~II,~\S4,~Theorem~2]{Shafarevich:basic}.
Thus the map
$\chi$ is an isomorphism on open sets $V\setminus \Exc(\theta)$ and
$V'\setminus\Exc(\theta')$. In particular,
there are canonical isomorphisms
$$
\chi_*\colon \rZ^1(V)\xarr{\simeq} \rZ^1(V'),
\qquad \chi_*\colon \Cl(V)\xarr{\simeq} \Cl(V').
$$
According to Theorem~\ref{th:mmp:bpf} there exists a divisor $D$ on $V$ such 
that
$$
\Cl(V)=\Pic(V)=\theta^*\Pic(V_0)\oplus D\cdot \ZZ.
$$
Let $D':= \chi_*D$ be its proper transform on $V'$.
Then
$$
\Cl(V')=\theta'^*\Pic(V_0)\oplus D'\cdot \ZZ.
$$
Since $\Pic(V')\neq \theta'^*\Pic(V_0)$, some multiple of $D'$ must be
a Cartier divisor. It is clear that $D'\cdot \rR'\neq 0$. Let $A$ be an 
ample
divisor on $V_0$. By Kleiman's Ampleness Criterion~\ref{Kleiman} the divisor
$aD'+b\theta'^*A$ is ample on $V'$ for some $a,b\in\ZZ$, $b\gg 0$.
Thus the variety $V'$ is isomorphic to the projective spectrum
$$
\operatorname{Proj} \R(V',\, D'+n\theta'^*A).
$$
Since $\chi$ is an isomorphism in codimension~$1$, the algebras $\R(V,\,
aD+b\theta^*A)$ and $\R(V',\, aD'+b\theta'^*A)$ are naturally
isomorphic.
Therefore,
$$
V'\simeq \operatorname{Proj} \R(V,\, aD+b\theta^*A).
$$
This means that the flop is unique (if it exists), i.e. given
contraction $\theta$ the right hand side of the diagram~\eqref{eq:flop} 
is
uniquely restored. The existence of the flop is equivalent to the finite generation
of the algebra
$\R(V,\, aD+b\theta^*A)$.

We give a simple example.

\begin{exa}[(Atiyah--Kulikov flop~{\cite{Atiyah1958}})]
\label{flop:AK}
Assume that in the preceding notation $V$ is a nonsingular three-dimensional
variety and
the exceptional locus of the morphism $\theta$ consists of a single nonsingular
rational curve $C\subset V$ such that
$$
\NNN_{C/V}=\OOO_{\PP^1}(-1)\oplus \OOO_{\PP^1}(-1).
$$
Let $\sigma\colon \widetilde{V}\to V$
be the blowup of $C$ and let $E$ is the exceptional divisor.
Then
\begin{equation}
\label{eq:Eflop}
E\simeq \PP^1\times \PP^1, \qquad \OOO_E(E)=\OOO_{\PP^1\times \PP^1}(-1,-1).
\end{equation}
According to~\cite{Moishezon:CastelnuovoEnriques}, the divisor $E$ can be contracted in
(in~the analytic
category) in another direction
$\sigma'\colon \widetilde{V}\to V'$ over $V_0$ (see Fig.~\ref{ris:flop}).
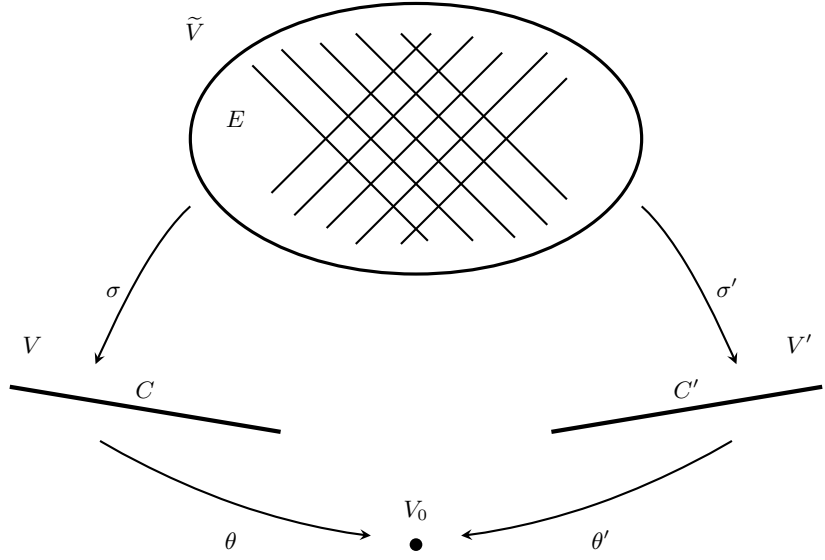
\begin{figure}[!t]\small
\begin{center}
\begin{tikzpicture}[scale=0.6]
\draw[very thick] (0,0) ellipse (5 and 3);

\draw[thick] (0+2,0)--++(-45:1.9);
\draw[thick] (0+2,0)--++(180-45:3.3);
\draw[thick] (0+2,0)--++(45:1.9);
\draw[thick] (0+2,0)--++(180+45:3.3);.

\draw[thick] (0+1,0)--++(-45:2.7);
\draw[thick] (0+1,0)--++(180-45:3.3);
\draw[thick] (0+1,0)--++(45:2.7);
\draw[thick] (0+1,0)--++(180+45:3.3);

\draw[thick] (0,0)--++(-45:3.1);
\draw[thick] (0,0)--++(180-45:3);
\draw[thick] (0,0)--++(45:2.7);
\draw[thick] (0,0)--++(180+45:2.8);

\draw[thick] (0-2,0)--++(-45:3.2);
\draw[thick] (0-2,0)--++(180-45:2.3);
\draw[thick] (0-2,0)--++(45:3.3);
\draw[thick] (0-2,0)--++(180+45:1.7);.

\draw[thick] (0 -1,0)--++(-45:3.2);
\draw[thick] (0-1,0)--++(180-45:2.8);
\draw[thick] (0-1,0)--++(45:3.2);
\draw[thick] (0-1,0)--++(180+45:2.4);

\draw[thick,->,>=stealth] plot [smooth,tension=1] coordinates{(-5,-1.5)
(-6,-2.8) (-7.1,-5.)};
\draw[thick,->,>=stealth] plot [smooth,tension=1] coordinates{(5,-1.5) (6,-2.8)
(7.1,-5.)};
\draw[thick,->,>=stealth] plot [smooth,tension=1] coordinates{(-7,-6.7)
(-4,-8.1) (-1,-8.8)};
\draw[thick,->,>=stealth] plot [smooth,tension=1] coordinates{(7,-6.7) (4,-8.1)
(1,-8.8)};

\draw[ultra thick] plot [smooth,tension=1] coordinates{(-9,-5.5) (-6,-6)
(-3,-6.5)};
\draw[ultra thick] plot [smooth,tension=1] coordinates{(9,-5.5) (6,-6)
(3,-6.5)};
\fill (0,-9) circle(4pt);
\node[above,yshift=1] () at (0,-8.7) {$V_0$};
\node[above,yshift=1] () at (-4,0) {$E$};
\node[above,yshift=1] () at (-8.5,-5) {$V$};
\node[above,yshift=1] () at (8.5,-5) {$V'$};
\node[above,xshift=0] () at (-4.9,2) {$\widetilde{V}$};
\node[above,yshift=1] () at (-6,-6) {$C$};
\node[above,yshift=1] () at (6,-6) {$C'$};
\node[above,xshift=0] () at (-6.7,-3.7) {$\sigma$};
\node[above,xshift=0] () at (6.9,-3.7) {$\sigma'$};
\node[above,xshift=0] () at (-4.1,-9.3) {$\theta$};
\node[above,xshift=0] () at (4.1,-9.3) {$\theta'$};
\end{tikzpicture}
\end{center}
\caption{Atiyah--Kulikov flop}

\label{ris:flop}
\end{figure}
Here the variety $V'$ is nonsingular and $\sigma'$ is the blowup of a
nonsingular
rational curve $C':= \sigma'(E)$. We also have $K_{\widetilde{V}}-E=\sigma'^*
K_{V'}$.
Therefore, $K_{V'}\cdot C'=0$. The divisor $E$ on~$\widetilde{V}$ can be 
contracted 
to a point
$(V_0\ni o)$, which is ordinary double. This contraction induces
contractions $\theta\colon V\to V_0$ and $\theta'\colon V'\to V_0$. We obtain
the diagram~\eqref{eq:flop}, where $\chi:= \theta'\comp \theta^{-1}$.

Conversely, assume that we are given an ordinary double point $(V_0\ni o)$. Up 
to analytic isomorphism we may assume that $V_0$ is defined in $\CC^4$ by the 
equation
\begin{equation}
\label{eq:Eflop:qu}
x_1x_2-x_3x_4=0.
\end{equation}
The blowup of the origin gives us a nonsingular variety 
$\widetilde{V}$ with 
exceptional divisor $E$ that is a quadric in $\PP^3$ defined by the 
equation~\eqref{eq:Eflop:qu} (cf.~\eqref{eq:Eflop}). It can be contracted in 
two directions and we obtain the situation shown on Fig.~\ref{ris:flop}. At the 
same 
time, the
contractions $\theta$ and $\theta'$ are blowups of planes $\{x_1=x_3=0\}$
and $\{x_1=x_4=0\}$ on~$V_0$.
\end{exa}
There is a similar but more complicated construction in the case where the contracted curve $C 
\subset V$ 
is nonsingular rational and has normal bundle
$$
\NNN_{C/V}\simeq \OOO_{\PP^1}\oplus \OOO_{\PP^1}(-2).
$$
This construction is called \textit{M. Reid's pagoda}~\cite{Reid:MM}.
In general case the existence of tree-dimensional flops follows from
the following simple observation by J. Koll\'ar.

\begin{teo}[\cite{Kollar:flops}]
\label{flop:theorem}
Let $V$ be a nonsingular\footnote{The theorem holds in much more general 
situation.
The arguments given in the proof below work in the case where
singularities $V$ are terminal (and so are the singularities of the
variety~$V'$).} three-dimensional variety and let
$\theta \colon V\to V_0$ be a small birational contraction such that
the canonical class $K_V$ trivial on the fibers and $\uprho(V/V_0)=1$.
Then the flop $\chi\colon V\dashrightarrow V'$ exists and the
variety~$V'$ is nonsingular.
\end{teo}

\begin{proofsk}
Apply the Base Point Free Theorem~\ref{th:mmp:bpf} to 
the divisor
$M=K_V+\theta^*A$, where $A$ is a very ample divisor on $V_0$. We obtain that
the linear system $|b(K_V+\theta^*A)|$ is base point free \textit{for all}
$b\gg 0$. It follows also from Theorem~\ref{th:mmp:bpf} that for any $b\gg 0$
there exists a Cartier divisor $N_b$ on~$V_0$ such that 
$b(K_V+\theta^*A)=\theta^*
N_b$. Therefore, there exists a Cartier divisor $N$ on~$V_0$ such that
$K_V=\theta^*N$. It is clear that $N=\theta_*K_V=K_{V_0}$. Hence, the canonical 
class
$K_{V_0}$ is a Cartier divisor and $K_V=\theta^*K_{V_0}$.

Since $\theta$ does not contract any divisors, the variety $V_0$ has only
terminal (Gorenstein) singularities.
Take a divisor $D$ on $V$ such that $-D$ is $\theta$-ample.
Then $D$
cannot be the pull-back of a Cartier divisor on $V_0$.
Therefore, $D_0:= \theta_*D$ is a Weil divisor, which is not Cartier. 
In particular, the variety $V_0$ is singular. Further, consider a connected
component $C$ of the exceptional locus $\Exc(\theta)$ and let $o:= \theta(C)$.
Consider $V$ and $V_0$ as small analytic neighborhoods of the curve $C$ and the 
point
$o$, respectively. According to the classification of terminal Gorenstein 
singularities (see~\cite[\S1]{Reid:MM} or~\cite[\S5.3]{KM:book}), the singularity
$(V_0\ni o)$ is a hypersurface of multiplicity~2, i.e. locally it can be
given in $\CC^4$ by an equation with non-zero quadratic part. According to
the Weierstrass preparation theorem 
\cite[Ch.~0, \S1]{Griffiths1994}, this equation after an
analytic coordinate change, can be written in the form
$$
y^2=\phi(x_1,x_2,x_3)
$$
(see e.~g.~\cite[Corollary~13.5]{P:book:sing-re}). Consider the involution
$$
\tau\colon (V_0\ni o)\longrightarrow (V_0\ni o),\qquad (x_1,x_2,x_3, y) \longmapsto
(x_1,x_2,x_3, -y)
$$
and the quotient 
$$
\pi\colon (V_0\ni o) \longrightarrow\left(\CC^3_{x_1,x_2,x_3}\ni 0\right).
$$
Put locally $(V' \supset C')=(V \supset C)$, but the structure morphism
$\theta'\colon V'\to V_0$ we define as the composition with the involution: $\theta'=
\theta\comp \tau$. Thus we put $\chi=\theta\comp
\tau\comp\theta^{-1}$ and $D'=\theta^{-1} (\tau(D_0))$. It is clear that the map
$\chi$ is an isomorphism on the open sets $V\setminus C$ and $V'\setminus C'$.
It is clear also that $D_0+\tau(D_0)$ is an invariant divisor on $(V \supset 
C)$ and
$D_0+\tau(D_0)=\pi^*B$ for some (Cartier) divisor $B$ on~$\CC^3$.
Since now our situation is local, we may assume that
$\tau(D_0)\sim -D_0$ and $D'\sim -D$ (under our identification of $V$ and~$V'$).
Thus the divisor $D'$ is positive on all components of the curve $C=C'$, i.e.
is $\theta'$-ample. Then the construction can be globalized by 
``regluing'' the neighborhood $(V \supset C)$. The projectivity of the 
variety $V'$ follows from the projectivity of the morphism $\theta'$.
For details we refer to the work~\cite{Kollar:flops} and also
\cite[Ch.~6]{KM:book}.
\end{proofsk}

In some cases the construction with involution $\tau$ and double cover $\pi$ can 
be globalized, 
i.e. $\tau$ and $\pi$ can be defined on the whole
projective variety (see Example~\ref{example:sl-cubic}). But this occurs
quite rarely.

\begin{cor}
\label{cor:flop:DC}
In the conditions of Theorem~\ref{flop:theorem} there is an isomorphism of 
the exceptional
curves $\Exc(\theta)\simeq \Exc(\theta')$ \textup(as abstract
varieties\textup). Furthermore, for any component $C_i\subset\Exc(\theta)$ 
and
the corresponding component $C_i'\subset\Exc(\theta')$ we have
$$
D\cdot C_i=-D'\cdot C'_i.
$$
\end{cor}

\begin{proof}
Follows from the construction and the fact that $\tau$ induces the multiplication 
by $-1$ in the local Weil divisor class group
$$
\Cl(V_0\ni o)\simeq \Pic(V\supset C)\simeq \Pic(V'\supset C').\qedhere
$$
\end{proof}

\begin{cor}
In the conditions of Theorem~\ref{flop:theorem}, let $H_0\subset V_0$ be the a 
general hyperplane section passing through the point~$o$. Put $H:= \theta^{-1}(H_0)$
and $H':= \theta'^{-1}(H_0)$. Then $H_0$, $H$ and $H'$ are surfaces with only
Du Val singularities and the restrictions $\theta_H\colon H\to H_0$ and 
$\theta'_H\colon H'\to H_0$ of the corresponding maps $\theta\colon V\to V_0$ 
and $\theta'\colon V'\to
V_0$ are crepant, i.e. $\theta_H^*K_{H_0}=K_H$ and $\theta_H'^*K_{H_0}=K_{H'}$. 
In particular, the minimal resolution $\widetilde{H}_0\to (H_0\ni 0)$ 
passes
through $H$ and $H'$, i.e. surfaces $H$ and $H'$ are obtained by contracting 
certain configurations of $(-2)$-curves on $\widetilde{H}_0$.
\end{cor}

\begin{proof}
According to the classification (see~\cite[\S1]{Reid:MM} or 
\cite[\S5.3]{KM:book}),
the singularity $(V_0\ni o)$ is an isolated hypersurface of type~cDV, i.e.
a general hyperplane section $H_0$ passing through~$o$ has only Du Val
singularity at $o$. According to the Inversion of Adjunction
\cite[Theorem~5.50]{KM:book}, the pair $(V_0,H_0)$ is plt. Since the morphism 
$\theta$ is log crepant with respect to 
$K_V+H$, the pair $(V,H)$ is plt as well 
(see~\cite[Lemma~3.10]{Kollar95:pairs}). Again by the Inversion of Adjunction 
the 
surface
has only Du Val singularities. The arguments for~$H'$ are completely similar.
The rest follows from the adjunction formula.
\end{proof}

\begin{cor}
\label{cor:flops:rat}
In the conditions of Theorem~\ref{flop:theorem}, the exceptional curves
$\Exc(\theta)$ and
$\Exc(\theta')$ are unions of smooth rational curves that form a combinatorial tree.
\end{cor}

\begin{proof}
We use the notation of the previous corollary. Consider the fiber
$C=\theta^{-1}(o)\subset H\subset V$ (the arguments for the right hand side part 
of the flopping
diagram are similar). Let $E=\sum E_i\subset\widetilde{H}_0$ be the exceptional
divisor over~$o$ in the minimal resolution of singularities of $(H_0\ni o)$ and let
$E'\subset E$ be the sum of the components that are contracted on $H$. Then $C$ is a
the image of the divisor $E-E'$. In particular, all components of $C$ are 
rational 
and their
configuration forms a combinatorial tree. For the smoothness of a 
component~$C_i\subset C$
it is necessary and sufficient that the intersection number of its proper transform $E_i\subset E$ 
and the fundamental cycle $Z_j$ of each singularity $P_j\in H$ equals~$1$
\cite[\S10]{P:book:sing-re}. But this immediately follows from the classification
of Du Val singularities~\cite[\S10, Exercise~7]{P:book:sing-re} (consider all the possibilities!).
\end{proof}
It is possible give another, more universal the proof of this fact,
that uses the vanishing of $R^1\theta_* \OOO_V$ (by the Kawamata--Viehweg 
Vanishing Theorem). The interested reader is invited to figure this out on
their own.

\begin{zadachi}
\eitem
Let $W$ be a nonsingular three-dimensional variety and let $C_1$,\, $C_2\subset 
W$
be smooth curves meeting each other transversally.
Let us perform the following actions: let $V\to W$ be the successive blowup of first
$C_1$ and then the proper transform of $C_2$; let
$V'\to W$ be the successive blowup of first $C_2$ and then the proper
transform of $C_1$. Show that $V$ and $V'$ can be included to the flopping diagram over 
$W$, where $V_0\to W$ is the blowup of the singular curve $C_1\cup C_2$.
\end{zadachi}

\newpage\section{Varieties of minimal degree}
\label{sect:VMD}


\subsection{Varieties of minimal degree and rational scrolls}
In this section we will provide some necessary information on the varieties of 
minimal degree.
These are classical, well-known results.
They repeatedly used in the course.

\begin{prp}
\label{proposition:minimal-degree}
Let $X\subset\PP^N$ be a variety \textup(i.e. an irreducible reduced
scheme\textup),
not lying in any hyperplane $\PP^{N-1}\subset\PP^N$. Then
\begin{equation}
\label{equation:LDV-inequality}
\deg X \ge \codim X+1.
\end{equation}
\end{prp}

\begin{proof}
Both parts of the inequality~\eqref{equation:LDV-inequality} remain unaltered on passing to 
hyperplane section of the variety~$X$, and if hyperplane section is sufficiently
general and $\dim(X)>1$, then 
the conditions the lemma will continue behold. Therefore, we may assume that $\dim(X)=1$. Intersecting the 
curve $X$ with a sufficiently general hyperplane
we obtain a finite number of (equals $\deg X$) points generating $\PP^{N-1}$.
The inequality~\eqref{equation:LDV-inequality} obviously holds for them.
The proposition is proved.
\end{proof}

A variety $X\subset\PP^N$ not lying in a hyperplane, for which 
the equality in~\eqref{equation:LDV-inequality} holds, i.e.
such that
\begin{equation}
\label{equation-LDV-equality}
\deg X=\codim X+1,
\end{equation}
is called a \textit{variety of minimal degree}.
One-dimensional varieties of minimal degree are rational normal
curves (see e.~g.~\cite[Ch.~1, \S4]{Griffiths1994} or
\cite[Ch.~4, \S3]{Hartshorn-1977-ag}).

\begin{prp}
\label{cor:VMD:R}
Let $X\subset\PP^N$ be a variety of minimal degree.
Then the graded algebra $\R(X,\,\OOO_X(1))$ is generated by its
component of degree~$1$, i.e. the natural homomorphism of graded algebras
$$
\alpha \colon \Sym^*H^0(X, \OOO_X(1))\longrightarrow\R(X, \OOO_X(1))
$$
is an epimorphism.
The kernel of $\alpha$ is generated by elements of degree~$2$,
i.e. the variety $X$ is an intersection of quadrics containing $X$.
\end{prp}

\begin{proof}
Taking general hyperplane sections, by induction on the dimension (see
Theorem~\ref{th-hyp-sect}) the assertion is reduced to case where
$X$ is a rational normal curve of degree $N$ in $\PP^N$.
Then
$X\simeq \PP^1$ and the algebra $\R(X,\,\OOO_X(1))$
is the Veronese subalgebra in $\R(\PP^1,\OOO_{\PP^1}(1))$:
$$
\R(X,\,\OOO_X(1))\simeq \R(\PP^1,\OOO_{\PP^1}(N))\subset
\R(\PP^1,\OOO_{\PP^1}(1))\simeq \CC[x_0,x_1].
$$
This case is left to the reader for self-analysis (cf.~\cite[Ch.~4,
\S3, Exercise~3.4]{Hartshorn-1977-ag}.
\end{proof}

Typical examples of varieties of minimal degree are
projectivization vector bundles on $\PP^1$.
Recall that according to Grothendieck's theorem any vector bundle on
$\PP^1$ decomposable (see e.~g.~\cite[Ch.~V, Exercise~2.6]{Hartshorn-1977-ag}).

\begin{prp}[({\cite[\S1]{Saint-Donat-1974}}, cf. {\cite[Ch.~5, 
\S2]{Hartshorn-1977-ag}})]
\label{prop:scrolls}
Consider the vector bundle
$$
\EEE:= \OOO_{\PP^1} (d_1)\oplus\cdots\oplus \OOO_{\PP^1} (d_m)
$$
of rank $m$ on $\PP^1$, where $d_i\ge 0$ for all $i$.
Put $X=\PP_{\PP^1}(\EEE)$ and let $\pi \colon X\to \PP^1$ be the natural
projection. Let 
$\MMM:= \OOO_{X/\PP^1}(1)$ be the tautological invertible Grothendieck's 
sheaf on $X$
\textup(i.e. an invertible sheaf
such that $\pi_*\OOO_{X/\PP^1}(1)=\EEE$\textup).
The following assertions hold.
\begin{enumerate}
\item
$H^q(X,\MMM^{\otimes p})=0$ for $q\ge 1$, $p\ge 0$ and
$$
\hr^0(X,\MMM)=\hr^0(\PP^1,\pi_*\MMM)=\sum (d_i+1)=m+\sum d_i+1.
$$

\item
The sheaf $\MMM$ is generated by global sections and defines a
birational morphism that is onto its image:
\begin{equation}
\label{eq:min-deg:image}
\Phi_{|\MMM|}\colon X\longrightarrow\PP^{\sum d_i+m-1}.
\end{equation}
The sheaf $\MMM$ is very ample if
and only if all the $d_i$ are strictly positive.

\item
The degree of the image $\Phi_{|\MMM|}(X)$ is computed by the formula
$$
\deg \Phi_{|\MMM|}(X)=\sum d_i.
$$
In particular, $\Phi_{|\MMM|}(X)$ is a variety of minimal degree, i.e. for
it the equality in~\eqref{equation:LDV-inequality} holds.

\item
Put
$$
\begin{alignedat}{2}
\EEE' &:= \bigoplus_{d_i>0}\OOO_{\PP^1}(d_i),\qquad&
m' &:= \rk \EEE'=\#\{ i \mid d_i>0 \},
\\[3pt]
\EEE'' &:= \bigoplus_{d_i=0}\OOO_{\PP^1}(d_i),\qquad&
m'' &:= \rk \EEE''=\#\{ i \mid d_i=0 \}.
\end{alignedat}
$$
If $\Phi_{|\MMM|}$ is not an isomorphism, then
the image $\Phi_{|\MMM|}(X)$ is a cone with vertex $\PP^{m''-1}\subset\PP^{\sum
d_i+m-1}$ over the variety
$X'=\PP_{\PP^1}(\EEE')\subset\PP^{\sum d_i+m'-1} $ and the exceptional locus of
$\Phi_{|\MMM|}$ is the variety $\PP_{\PP^1}(\EEE'')\simeq \PP^1\times \PP^{m''-1}$.

\item
\label{scroll:canonical-divisor}
There is an isomorphism of sheaves
$$
\OOO_X(-K_X)\simeq \MMM^{\otimes m}\otimes \pi^*\OOO_{\PP^1}\left(2-\sum
d_i\right).
$$
\end{enumerate}
\end{prp}

If the sheaf $\MMM$ is very ample (i.e. $d_i>0$ $\forall i$), then
the image of the embedding~\eqref{eq:min-deg:image} is called
\textit{rational scroll}.
Usually rational scroll is considered with the polarization 
$\MMM$ (i.e. it is fixed the tautological bundle $\MMM=\OOO_{\PP(\EEE)}(1)$).

\subsection{Classification}
\begin{prp}
\label{varieties-minimal-degree}
Let $X\subset\PP^N$ be a variety of minimal degree.
Then $X$ is one of the following varieties:
\begin{enumerate}
\item
\label{varieties-minimal-degree-1}
$\PP^N$;
\item
\label{varieties-minimal-degree-2}
a quadric $Q\subset\PP^N$;
\item
\label{varieties-minimal-degree-3}
a rational normal curve $C_N\subset\PP^N$;
\item
\label{varieties-minimal-degree-4}
a rational scroll \textup(embedded by the tautological linear system\textup);
\item
\label{varieties-minimal-degree-5}
the Veronese surface $S=S_4\subset\PP^5$;
\item
\label{varieties-minimal-degree-cone}
the cone over a variety of type~\ref{varieties-minimal-degree-3},~\ref{varieties-minimal-degree-4} or~\ref{varieties-minimal-degree-5}.
\end{enumerate}
In particular, the variety $X$ is normal.
\end{prp}

\begin{proof}
The proof is by induction on the dimension of the variety $X$.
Denote by $H$ a general hyperplane section of $X$.
Put also $d:= \deg X$ and $n=\dim(X)$.

The case $\dim(X)=1$ is well-known: an irreducible curve of degree $N$ in 
$\PP^N$ is a rational normal curve.

Let $\dim(X)=2$. Then by the inductive hypothesis $H$ is a rational normal curve 
$C_{N-1}\subset\PP^{N-1}$.
Therefore, singularities $X$ are isolated (and $H$ does not pass through them). 
Consider
the minimal resolution of singularities 
$f\colon \widetilde{X}\to X$. Then
$$
K_{\widetilde{X}}=f^*K_X-\Delta,
$$
where $\Delta$ is an effective $\QQ$-divisor
on $\widetilde{X}$ (see e.~g.~\cite[9.2]{P:book:sing-re}). Put
$\widetilde{H}:= f^{-1}_*H$. Then $\widetilde{H}\simeq C_{N-1}$.
By the adjunction formula $(K_{\widetilde{X}}+\widetilde{H})\cdot 
\widetilde{H}=-2$. This implies
that
the divisor $K_{\widetilde{X}}+\widetilde{H}$ is not 
nef, but it is nef over~$X$.
On the other hand, the divisor $\widetilde{H}=f^*H$
is nef and numerically trivial only on the exceptional divisors of $f$.
Thus $K_{\widetilde{X}}+t\widetilde{H}$
is nef for $t\gg 1$. Take $t$ minimal satisfying
this condition. According to the above, $t>1$.
There exists an extremal ray $\rR\subset\NE(\widetilde{X})$ such that
$$
(K_{\widetilde{X}}+t\widetilde{H})\cdot \rR=0\quad \text{and}\quad 
\widetilde{H}\cdot \rR>0.
$$
Let $\varphi\colon \widetilde{X}\to Y$ be the contraction of $\rR$
and let $\ell$ be the corresponding minimal rational curve (see 
\eqref{eq:def:length}).
Then
\begin{equation}
\label{eq:new:ell:surf}
\upmu(\rR)=-K_{\widetilde{X}} \cdot \ell=t\widetilde{H} \cdot \ell 
>\widetilde{H} \cdot
\ell>1.
\end{equation}
By the classification of extremal rays on surfaces
(Theorem~\ref{class:ext-rays:surf}) the morphism $\varphi$ cannot be
birational.
Moreover, if $Y$ is a point, then $\widetilde{X}\simeq \PP^2\simeq X$.
Then $\upmu(\rR)=3$ and $H \cdot \ell\le 2$ according to 
\eqref{eq:new:ell:surf}
we obtain the cases~\ref{varieties-minimal-degree-1}
and~\ref{varieties-minimal-degree-5}. And if
$Y$ is a curve, then it must be rational (because $\widetilde{H}\simeq
\PP^1$ dominates $Y$)
and $\widetilde{X}=\PP(\EEE)$ is a rational ruled
surface, where $\EEE=\OOO_{\PP^1}\oplus \OOO_{\PP^1}(-e)$ for some
$e\ge 0$. In this case $\upmu(\rR)=2$ and $\widetilde{H} \cdot \ell=1$ 
according to
\eqref{eq:new:ell:surf}, i.e.
$\widetilde{H}$ is a section. Thus the linear system
$|\widetilde{H}|$ maps the fibers of $\widetilde{X}=\PP(\EEE)$ to lines in 
$\PP^N$.
Thus $X$ is either a cone over a rational normal curve or a rational ruled surface
(the cases~\ref{varieties-minimal-degree-cone} and~\ref{varieties-minimal-degree-4}, respectively).

Let $n=\dim(X)>2$. We assume that the assertion of the proposition holds
for varieties of dimension $<n$.
First, consider the case where $X$ is nonsingular.
By the adjunction formula $K_H=(K_X+H)|_H$ and the divisor $K_X+(n-1)H$ is not 
nef because the divisor
$$
(K_X+(n-1)H)|_H=K_H+(n-2)H|_H
$$
is not 
nef on $H$ by the inductive hypothesis.
If $\uprho(X)=1$, then $X$ is a Fano variety of index $\ge \dim(X)$.
According to Theorem~\ref{thm:large-index}, we have
the case~\ref{varieties-minimal-degree-1} or~\ref{varieties-minimal-degree-2}.

Let $\uprho(X)>1$.
By the Lefschetz hyperplane theorem $\uprho(H)\ge \uprho(X)>1$.
Then by the inductive hypothesis, $H$ is a rational scroll
(the case~\ref{varieties-minimal-degree-4}) and
$$
\uprho(H)=\uprho(X)=2.
$$
Take $t$ to be minimal satisfying the condition that $K_X+tH$
is nef. According to the above, $t>n-1$.
There exists an extremal ray $\rR\subset\NE(X)$ such that
$$
(K_X+tH)\cdot \rR=0\quad \text{and}\quad H\cdot \rR>0.
$$
Take a minimal rational
curve $\ell$ for the ray $\rR$. Then $\rR=\RR_+[\ell ]$ and $-K_X\cdot \ell \le
n+1$
(see the Cone Theorem~\ref{cone-th}).
Let $\varphi\colon X\to Y$ be the contraction of $\rR$. By our conditions 
$H\cdot \ell=1$, i.e. $\ell $ is 
a line on $X\subset\PP^N$. As in the proof of
Proposition~\ref{proposition-index=2-lines}\ref{proposition-index=2-lines-hs}
there exists a nonsingular hyperplane section $H$ passing
through $\ell$. The divisor
$(K_X+tH)_H=K_H+(t-1)H_H$ 
is nef on $H$ and trivial on the curve $\ell$. Therefore, $\ell$ generates
an extremal ray $\rR_H$ on
$H$ and $K_H+(t-1)H_H$ is a supporting divisor of this extremal ray,
i.e. the equality $(K_H+(t-1)H_H)\cdot C=0$ holds for some curve $C$ on
$H$ if and only if $[C]\in\rR_H$.
In particular, $\varphi$ generates a contraction $\varphi_H\colon H\to \PP^1$ of an
extremal ray on $H$. Since $H$ is a rational scroll,
the fibers of $\varphi_H$ are projective spaces $\PP^{n-2}$ and they are
hyperplane sections of the fibers of 
$\varphi$. Therefore, the fibers of $\varphi$ are projective spaces
$\PP^{n-1}$ and $Y=\PP^1$
Thus $\varphi$ is a $\PP^{n-1}$-bundle over $\PP^1$.

Now let $X$ be singular. Then we take $H$ to be a general hyperplane section passing
through a singular point $P\in X$. By Bertini's theorem $H$ is irreducible. Then $H$ is 
a variety of minimal degree that is singular at the point $P$. By the inductive hypothesis
$H$
is a cone with vertex $P$. Thus and $X$ is a cone with vertex $P$ as well.
\end{proof}

\begin{cor}
\label{corollary:varieties-minimal-degree}
Let $X\subset\PP^N$ be a variety of minimal degree.
Assume that $X\neq \PP^N$ and $X$ is not a nonsingular quadric, defined by
a quadratic form of rank $\ge 5$.
Then a hyperplane section $H\subset X$ admits a decomposition $H=H_1+H_2$, where
$|H_i|$ are movable linear systems of Weil divisors, i.e.
the linear systems $|H_i|$ are non-empty and have no fixed components.
\end{cor}

\begin{proof}
Follows from classification Theorem~\ref{varieties-minimal-degree}.
It is left to the reader for self-analysis.
\end{proof}

\begin{zadachi}
\eitem
Let $X=X_d\subset\PP^{d+n-2}$ be a nonsingular non-degenerate $n$-dimensional
variety
of degree $d$. Prove that $X$ is either a projection of a variety of minimal
degree $X=X_d\subset\PP^d$ or del Pezzo variety (see Definition~\ref{def:coindex}).
\hint{Let $H$ be the general hyperplane section. In the case where $X$ is linearly 
normal by induction show that $K_X+(n+1)H\sim 0$. Use also that in this case
a general section $X\cap\PP^{d-1}$ is an elliptic curve.}

\eitem
\label{zad:Fano-K3}
Let $X\subset\PP^N$ be a nonsingular projective three-dimensional variety and 
let
$H=X\cap\PP^{N-1}$ be its hyperplane section.
Assume that $H$ is a minimal surface of Kodaira dimension 0.
Prove that $X$ is a Fano threefold and $H$ is a \K3 surface.

\eitem
\label{ex:vect-b}
Under what conditions the variety $X=\PP(\EEE)$ from Proposition 
\ref{prop:scrolls}
is a Fano threefold? Under what conditions 
the anticanonical class of this variety is nef and big?
Try to generalize these questions on the case of decomposable vector 
bundles on
projective spaces.
\end{zadachi}

\newpage\section{Projective models of K3 surfaces}

\label{ch:K3}

In this section we provide without proofs some standard information
on~canonical curves and \K3 surfaces.
We start with classical results on~canonical curves.
\subsection{Canonical curves}
\begin{teo}[Max Noether; see e.~g. {\cite[Ch.~2,~\S3]{Griffiths1994}}]
\label{th:M-Noether}
Let $C$ be a non-hyperelliptic curve. Then the canonical algebra
$\R(C,K_C)$
is generated by its component of degree~$1$.
\end{teo}

\begin{cor}
\label{cor:quadrics-can-curve}
There are exactly
$(g-2)(g-3)/2$ linearly independent quadrics passing
through the canonical model of a non-hyperelliptic curve of genus $g>2$.
\end{cor}

\begin{teo}[Noether-Enriques-Petri Theorem; see e.~g.
\cite{Shokurov-1971}]
\label{th:NEP:}
Let $C\subset\PP^{g-1}$ be a 
smooth canonical curve of genus $g\ge 4$. Then:

\begin{enumerate}
\item
\label{th:NEP:3}
$C$ is an intersection of quadrics and cubics passing through $C$ in $\PP^{g-1}$;
\item
\label{th:NEP:qua}
$C$ is not an intersection of quadrics only in the following cases:
\begin{enumerate}
\item
\label{Enrique-Petri-trigonal}
$C$ is a trigonal curve \textup(i.e. it has a one-dimensional linear series 
$\mathfrak
g^1_3$\textup);
\item
\label{Enrique-Petri:plane}
$C$ is a curve of genus $6$ that is isomorphic to a plane curve of degree $5$;
\end{enumerate}
\item
\label{th:NEP:nqua}
if $C$ is not an intersection of quadrics, then the intersection of all 
quadrics in $\PP^{g-1}$ passing through $C$ is a surface $R$
that is one of the following
\begin{itemize}
\item
\textup(for $g=4$\textup)
a quadric in $\PP^3$ \textup(possibly singular\textup),
\item
\textup(in the case~\ref{Enrique-Petri-trigonal}, $g\ge 5$\textup)
a nonsingular rational ruled surface of degree $g-2$
in $\PP^{g-1}$, moreover, the linear series $\mathfrak g^1_3$ on $C$ is cut out 
by the linear system
of rulings of $R$,
\item
\textup(in the case~\ref{Enrique-Petri:plane})
the Veronese surface $V_4\subset\PP^5$.
\end{itemize}
\end{enumerate}
\end{teo}

\subsection{\K3 surfaces}
Below we provide some facts on~\K3 surfaces.
Recall the definition.

\begin{dfn}
A smooth projective surface $S$ 
is called a \K3 surface, if $K_S=0$
and $H^1(S,\,\OOO_S)=0$.
\end{dfn}

The following fact is immediately deduced from the Riemann--Roch Theorem,
Serre duality, genus formula and Kawamata--Viehweg Vanishing Theorem
\ref{vanishing:KV}.

\begin{prp}
\label{prop:K3-simple}
Let $D$ be an effective non-zero divisor on a \K3 surface $S$. Then
$$
\dim|D|\ge \frac12 D^2+1.
$$
If $D$ is an irreducible
curve, then $H^1(S,\,\OOO_S(D))=0$ and the equality holds:
$$
\dim|D|=\frac12 D^2+1=\p(D).
$$
In particular, $D^2\ge -2$ for any irreducible
curve $D$ and $D^2=-2$ if and only if $D$ is a nonsingular rational
curve.
\end{prp}

Note that any nef divisor on a \K3 surface $S$ 
is effective (by the Riemann--Roch Theorem).

\begin{teo}[{\cite[\S\S2, 3]{Saint-Donat-1974}}, 
{\cite[Ch.~2]{Huybrechts:K3}}]
\label{theorem-AB-Saint-Donat}
Let $D$ be a nef divisor on a \K3 surface $S$.
Then $D$ belongs to one of the following classes.
\begin{enumerate}
\item
$D^2=0$;
in this case $D\sim mE$, where $m \ge 1$, $E$ is an elliptic curve, $\dim
|E|=1$ and $\Bs|E|=\varnothing$ \textup(i.e. $|E|$ is a base point free elliptic pencil\textup).

\item
\label{prop:K3:irr}
$D^2>0$ and
a general member of $|D|$ is an irreducible curve;
in this case the linear system $|D|$ is base point free, a general curve
$C\in |D|$ is nonsingular and
$$
\g(C)=\dim|D|>1;
$$
\item
$D^2>0$ and
a general member of $|D|$ is reducible; in this case
$$
|D|=m|E|+Z,
$$
where
$|E|$ is a base point free elliptic pencil, $Z$ is a nonsingular
rational curve that is a fixed component of $|D|$, moreover,
$E\cdot Z=1$ and $m\ge 2$.
\end{enumerate}
\end{teo}

\begin{dfn}
\label{def:DuVal:K3}
Let $S$ be a surface with Du Val singularities (see 
e.~g.~\cite[\S9]{P:book:sing-re}).
We say that $S$ is a \textit{singular \K3 surface}, if $K_S=0$ and
$H^1(S,\,\OOO_S)=0$.
\end{dfn}

Let $\mu\colon \widetilde{S}\to S$ is the minimal resolution of singularities 
of a singular \K3 surface. Then $K_{\widetilde{S}}=\mu^*K_S=0$.
Du Val singularities are rational. Thus
$H^1({\widetilde{S}},\,\OOO_{\widetilde{S}})=H^1(S,\,\OOO_S)=0$. Therefore,
$\widetilde{S}$ is a (nonsingular) \K3 surface.

\begin{cor}
\label{DuVal:K3}
Let $S$ be a singular \K3 surface and let $A$
be an ample Cartier divisor on $S$ such that the linear system $|A|$
has a fixed component $C$. Then the surface $S$ is nonsingular along $C$.
\end{cor}

\begin{proof}
Follows from Theorem~\ref{theorem-AB-Saint-Donat} applied to the minimal
resolution of
singularities.
\end{proof}


When studying the projective models of \K3 surfaces it is natural to consider a pair $(S,A)$ consisting of the surface $S$ itself and a nef
big divisor $A$ on $S$. Here the integral number
$$
g=\g(S,A):= \frac 12 A^2+1=\dim|A|
$$
is called the \textit{genus} of $(S,A)$. The class of the divisor $A$ is called
\textit{polarization} (respectively \textit{quasi-polarization}), if it is ample
(respectively, is not ample).

\begin{teo}[{\cite[\S\S5, 6]{Saint-Donat-1974}}]
\label{theorem-A-Saint-Donat}
Let $A$ be a nef and big divisor on a \K3 surface $S$ 
such that
the linear system $|A|$ has no fixed components. Put $g:= \g(S,A)$.
Then for the morphism
$$
\Phi_{|A|} \colon S\longrightarrow\PP^g
$$
one of the following possibilities occurs.
\begin{enumerate}
\item
\label{th:K3:surjective}
The morphism $\Phi_{|A|}$ is a birational onto its image $\Phi_{|A|} (S)$; this image
is a surface of degree $2g-2$ in $\PP^g$ that is a singular \K3 surface
\textup(in particular, it can be is nonsingular\textup).
The algebra $\R(S,A)$ is generated by its component of degree~$1$. If furthermore
the divisor $A$ is ample, then
$\Phi_{|A|} \colon S\to \Phi_{|A|} (S)$ is an isomorphism.

\item
The morphism $\Phi_{|A|} \colon S\to \Phi_{|A|} (S)$ is generically finite and
has degree~$2$;
its image $\Phi_{|A|} (S)$ is a surface of minimal degree $g-1$ in $\PP^g$.
In this case the surface $S$ \textup(with quasi-polarization $A$\textup)
is called \textit{hyperelliptic}. Such a surface are characterized the
property that any nonsingular curve $C\in|A|$ is hyperelliptic
\cite[Corollary~5.8]{Saint-Donat-1974}.
\end{enumerate}
\end{teo}

The following fact is contained in Theorems~6.1, 7.2
and Proposition~7.15 of~\cite{Saint-Donat-1974}.

\begin{teo}
\label{th:K3:}
Let $A$ be a very ample divisor on a \K3 surface $S$ such that
$g=\g(S,A)\ge 4$.
Let $I_S$ be the kernel of the natural homomorphism
$$
\alpha\colon \Sym^*H^0(S,\,\OOO_S (A)) \longrightarrow
\bigoplus_{m\ge 0} H^0 (S, \OOO_S (mA))=\R(S,A)
$$
Then the following assertions hold.
\begin{enumerate}
\item
\label{th:K3:23}
The ideal $I_S$ is generated by its elements of degree~$2$ and $3$.

\item
\label{th:K3:2}
Assume that the ideal
$I_S$ cannot be generated by elements of degree~$2$.
Then the homogeneous ideal $I_S'\subset I_S$ generated by all elements of degree
$2$ in $I_S$,
defines in $\PP^g$ a three-dimensional variety of minimal degree
$V=V_{g-2}\subset\PP^g$.
\item
\label{th:K3:3}
If $g\ge 5$, then $V$ \textup(see~\ref{th:K3:2}\textup) is a nonsingular 
rational scroll.
\end{enumerate}
\end{teo}

In the case~\ref{th:K3:2}
the surface $S$ \textup(with the polarization $A$\textup)
is called \textit{trigonal}. Such a surfaces are characterized the
property that
any nonsingular curve $C\in|A|$ is trigonal.

\begin{zadachi}
\eitem
Prove Proposition~\ref{prop:K3-simple}.

\eitem
Let $(S,A)$ be a hyperelliptic \K3 surface of genus $g\ge 3$, where
$A$ is a nef and big divisor and let $\Phi_{|A|}\colon S\to
\PP^g$ be the corresponding hyperelliptic map. Assume that
the image $W:= \Phi_{|A|}(S)$ is nonsingular. Prove that $W$ is either the
Veronese surface or isomorphic to rational geometrically ruled surface
$\FF_e$, where $e\in\{0,\, 1,\, 2,\, 3,\, 4\}$ and moreover all the possibilities
do occur. Compute the branch divisor $B\subset W$. 
What additional conditions are imposed on $W$ and $S$, if $A$ is ample?
\hint{Use the fact that the divisor $B$ is reducible and the pair $(W,\frac12 
B)$ has plt
singularities~\cite[Proposition 15.13]{P:book:sing-re}.}

\eitem
In the notation of the previous exercise assume that the image $W:= \Phi_{|A|}(S)$ is 
a singular surface. Prove that $W$ is a cone $W_g\subset\PP^g$ over a
rational normal curve of degree $e := g-1$, where $e\in\{2,\, 3,\, 4\}$,
and all the possibilities do occur. Compute the branch divisor $B\subset W$.
Prove that in this case the divisor $A$ cannot be ample.
\hint{Use the fact that the cone $W_g\subset\PP^g$ is isomorphic to weighted
projective space $\PP(1,1,e)$ and
construct a birational (onto its image) map $S\to \PP(1,1,e, e+2)$.}

\eitem
Use previous exercises to prove that there exists an elliptic pencil on any hyperelliptic \K3 surface of genus $g\neq 2,\, 5$.

\eitem
Prove that there exist hyperelliptic \K3 surfaces of arbitrary genus
$g\ge 2$.

\eitem
Prove that there exist trigonal \K3 surfaces of arbitrary genus
$g\ge 3$.

\eitem
Describe trigonal \K3 surfaces in $\PP^5$.
\end{zadachi}

\newpage\section{Bertini's Theorems}
\label{ch:Ber}

\label{app:Bertini}
\subsection{Linear systems}
Let $D$ be a Weil divisor on normal variety $X$ and let $|D|$ be the
corresponding linear system.
Assume that $|D|\neq \varnothing$.

\textit{The fixed part} of a linear system $|D|$ is called the effective 
divisor
$$
F:= \operatorname{Max} \left \{ F'\mid F'\le D'\quad\text{for all}\ D'\in|D|
\right\},
$$
where $\operatorname{Max}$ for a set divisors is regarded in the sense of
usual partial order~$\le$.
Components of $F$ are called \textit{fixed components} of $|D|$. There is
a decomposition
$$
|D|=F+|M|,
$$
where $|M|$ is a complete linear system without fixed components.
It is called the \textit{movable part} of $|D|$. It is clear that
$$
\dim|D|=\dim|M|.
$$
Note that an effective divisor $D$ on a projective normal surface such that
the linear system $|D|$ has no fixed components, 
is nef.

The \textit{base locus} of a linear system $|D|$ is the intersection of all
its elements:
$$
\Bs|D|:= \bigcap_{D'\in|D|} D'.
$$
Usually it regarded in the scheme sense.
Thus
$$
\Bs|D|\supset \Supp(F)
$$
and $\codim (\Bs|D|)>1$ if and only if $F=0$.

Recall that $\Phi_{|D|}$ denotes the rational map
$$
\Phi_{|D|}\colon X \dashrightarrow \PP^N, \qquad N:= \dim|D|=\dim|D|,
$$
given by the linear system $|D|$. By definition $\Phi_{|D|}=\Phi_{|M|}$.
\subsection{Bertini's Theorems}

\begin{teo}[First Bertini's Theorem, see {\cite[Ch.~1,\S3]{Shafarevich-at-al-Alg-Surf-er}}, {\cite[\S4]{Ueno1975}}]
\label{Bertini-1}
In the above notation, let $\dim|D|>0$.
\begin{enumerate}
\item
\label{Bertini-1:a}
If $\dim \Phi_{|D|}(X) \ge 2$, then a general member $M_{\mathrm{gen}}\in|M|$
is irreducible.
\item
\label{Bertini-1:b}
If $\dim \Phi_{|D|}(X)=1$
and $\dim|D|\ge 2$, then a general member $M_{\mathrm{gen}}\in|M|$ is reducible.
\end{enumerate}
\end{teo}

Consider the case~\ref{Bertini-1:b} of Theorem~\ref{Bertini-1}. Then the image of 
the map $\Phi_{|D|}(X)$ is a curve $C\subset\PP^N$
and the elements of $|M|$ are inverse images of hyperplane sections the curve $C$. There 
is the Stein factorization 
$$
\Phi_{|D|}\colon X \xdashrightarrow{\phi_1} \widehat{C}\xarr{\phi_2} C\subset\PP^N
$$
where $\widehat{C}$ is a nonsingular curve and $\phi_2$ is a finite morphism. 
Then a general
element $M'\in|M|$
has the form $M'=\sum M_i'$, where every summand $M_i'$ is a prime divisor and 
$P_i:= \phi_1(M_i')$ is a point on $\widehat{C}$. Moreover 
$\sum P_i=\phi_2^*(H)$ for a suitable (general) hyperplane section $H=C\cap
\PP^{N-1}$.

The map $\phi_1$ is regular on the open
subset
$$
X\setminus \Bs|M|.
$$
Moreover $\codim \Bs|M|\ge 2$.

\begin{rem}
\label{remark:Bertini}
Let the irregularity of the variety $X$ be equal to~$0$. Then
$\widehat{C}$ is a nonsingular rational curve.
Thus then we can write
$$
|D|=F+m |L|,
$$
where $|L|$ is a one-dimensional linear system (\textit{a pencil})
whose general member is irreducible.
It is clear that $\dim|D|=m$.
In this case we say that $|D|$ \textit{is composed of a rational pencil}.
\end{rem}

\begin{teo}[Second Bertini's Theorem, see {\cite[Ch.~1,~\S3]{Shafarevich-at-al-Alg-Surf-er}}, {\cite[\S4]{Ueno1975}}]
\label{Bertini-2}
The singular locus of a general element $D'\in|D|$ is contained in the union
$\Sing(X)\cup \Bs|D|$ of the singular locus of~$X$
and the base locus of the linear system $|D|$.
\end{teo}

%
\par\medskip\noindent
\textbf{Acknowledgements.}
This work has been partially funded within the framework of the HSE University Basic Research Program.


\newcommand{\cprime}{$^{\prime}$}


\newpage
\begin{thebibliography}{cc}

\bibitem{Shafarevich-at-al-Alg-Surf-er}
Algebraic surfaces~/ I.~R.~\v{S}afarevi\v{c}, B.~G.~Averbuh,
Ju.~R.~Va\u{\i}nberg et~al.~// \emph{Trudy Mat. Inst. Steklov.} --
\newblock 1965. --
\newblock Vol.~75. --
\newblock P.~1--215.

\bibitem{Akhiezer1995}
\emph{Akhiezer~Dmitri~N.}
\href{http://dx.doi.org/10.1007/978-3-322-80267-5}{Lie group actions in
complex analysis}. {Aspects of Mathematics, E27}. --
\newblock Braunschweig~: Friedr. Vieweg \& Sohn, 1995. --
\newblock P.~viii+201. --
\newblock
ISBN:~\href{http://isbndb.com/search-all.html?kw=3-528-06420-X}{3-528-06420-X}.
--
\newblock Access mode: \url{http://dx.doi.org/10.1007/978-3-322-80267-5}.

\bibitem{Ambro-1999}

\emph{Ambro~F.} {Ladders on {F}ano varieties}~// \emph{J. Math. Sci. (New
York)}. --
\newblock 1999. --
\newblock Vol.~94, no.~1. --
\newblock P.~1126--1135. --
\newblock Algebraic geometry, 9.

\bibitem{Andreatta-Wisniewski-view1997}

\emph{Andreatta~Marco, Wi{\'s}niewski~Jaros{\l}aw~A.} A view on contractions
of higher dimensional varieties~// Algebraic geometry. Proceedings of the
Summer Research Institute, Santa Cruz, CA, USA, July 9--29, 1995. --
\newblock Providence, RI: American Mathematical Society, 1997. --
\newblock P.~153--183.

\bibitem{Andreatta-Wisniewski:contr-2}

\emph{Andreatta~Marco, Wi{\'s}niewski~Jaros{\l}aw~A.} Contractions of smooth
varieties. {II}. {C}omputations and applications~// \emph{Boll. Unione
Mat. Ital. Sez. B Artic. Ric. Mat. (8)}. --
\newblock 1998. --
\newblock Vol.~1, no.~2. --
\newblock P.~343--360.

\bibitem{Andreatta-Wisniewski:contr-1}

\emph{Andreatta~Marco, Wi{\'s}niewski~Jaros{\l}aw~A.} On contractions of
smooth varieties~// \emph{J. Algebraic Geom.} --
\newblock 1998. --
\newblock Vol.~7, no.~2. --
\newblock P.~253--312.

\bibitem{Tan3}
Asai~Masaya, Tanaka~Hiromu. Fano threefolds in positive characteristic III~//
 \emph{Arxiv e-print}. --
\newblock 2025. --
\newblock \href{https://arxiv.org/abs/2308.08124}{2308.08124}.

\bibitem{Atiyah1958}
\emph{Atiyah~Michael~F.} On analytic surfaces with double points~//
\emph{Proc. R. Soc. Lond., Ser. A}. --
\newblock 1958. --
\newblock Vol. 247. --
\newblock P.~237--244.

\bibitem{BKM}
Bayer~Arend, Kuznetsov~Alexander, Macr{\`{\i}}~Emanuele. Mukai models of {Fano}
 varieties~// \emph{Arxiv e-print}. --
\newblock 2025. --
\newblock Access mode: \url{https://arxiv.org/abs/2501.16157}.

\bibitem{Beauville:Prym}
\emph{Beauville~Arnaud}. Vari{\'e}t{\'e}s de {P}rym et jacobiennes
interm{\'e}diaires~// \emph{Ann. Sci. {\'E}cole Norm. Sup. (4)}. --
\newblock 1977. --
\newblock Vol.~10, no.~3. --
\newblock P.~309--391.

\bibitem{Birkar-BAB}
\emph{Birkar~Caucher}. Singularities of linear systems and boundedness of
{F}ano varieties~//
\href{http://dx.doi.org/10.4007/annals.2021.193.2.1}{\emph{Ann. of Math.
(2)}}. --
\newblock 2021. --
\newblock Vol. 193, no.~2. --
\newblock P.~347--405. --
\newblock online; accessed:
\url{https://doi.org/10.4007/annals.2021.193.2.1}.

\bibitem{Bogomolov:ineq-e}
\emph{Bogomolov~F.~A.} Holomorphic tensors and vector bundles on projective
varieties~// \emph{Math. USSR, Izv.} --
\newblock 1979. --
\newblock Vol.~13. --
\newblock P.~499--555.

\bibitem{Bourbaki:comm-alg-e}
\emph{Bourbaki~Nicolas}. Elements of mathematics. {Commutative} algebra.
{Translated} from the {French}. --
\newblock 1972.




\bibitem{Campana1991a}
\emph{Campana~F.} On twistor spaces of the class {${\mathcal C}$}~//
\emph{J. Differ. Geom.} --
\newblock 1991. --
\newblock Vol.~33, no.~2. --
\newblock P.~541--549.

\bibitem{Campana1992}

\emph{Campana~F.} Connexit{\'e} rationnelle des vari{\'e}t{\'e}s de
{F}ano~// \emph{Ann. Sci. {\'E}cole Norm. Sup. (4)}. --
\newblock 1992. --
\newblock Vol.~25, no.~5. --
\newblock P.~539--545.

\bibitem{Cheltsov-Shramov:2008-e}
\emph{Chel'tsov~I.~A., Shramov~K.~A.} Log canonical thresholds of smooth
{Fano} threefolds~//
\href{http://dx.doi.org/10.1070/RM2008v063n05ABEH004561}{\emph{Russ. Math.
Surv.}} --
\newblock 2008. --
\newblock Vol.~63, no.~5. --
\newblock P.~859--958.

\bibitem{Clemens-Griffiths}
\emph{Clemens~C.~Herbert, Griffiths~Phillip~A.} The intermediate {J}acobian
of the cubic threefold~// \emph{Ann. of Math. (2)}. --
\newblock 1972. --
\newblock Vol.~95. --
\newblock P.~281--356.


\bibitem{ClemensKollarMori1988}
\emph{Clemens~Herbert, Koll{\'a}r~J{\'a}nos, Mori~Shigefumi}.
{Higher-dimensional complex geometry}~// \emph{Ast{\'e}risque}. --
\newblock 1988. --
\newblock no. 166. --
\newblock P.~144 pp. (1989).

\bibitem{Cutkosky1989}
\emph{Cutkosky~Steven~Dale}. On {F}ano {$3$}-folds~//
\href{http://dx.doi.org/10.1007/BF01160118}{\emph{Manuscripta Math.}}
--
\newblock 1989. --
\newblock Vol.~64, no.~2. --
\newblock P.~189--204.

\bibitem{Debarre-Iliev-Manivel-2011}
\emph{Debarre~Olivier, Iliev~Atanas, Manivel~Laurent}. {On nodal prime
{F}ano threefolds of degree 10}~//
\href{http://dx.doi.org/10.1007/s11425-011-4182-0}{\emph{Sci. China
Math.}} --
\newblock 2011. --
\newblock Vol.~54, no.~8. --
\newblock P.~1591--1609. --
\newblock Access mode: \url{http://dx.doi.org/10.1007/s11425-011-4182-0}.

\bibitem{Debarre-Kuznetsov:GM}

\emph{Debarre~Olivier, Kuznetsov~Alexander}. Gushel-{M}ukai varieties:
classification and birationalities~//
\href{http://dx.doi.org/10.14231/AG-2018-002}{\emph{Algebr. Geom.}}
--
\newblock 2018. --
\newblock Vol.~5, no.~1. --
\newblock P.~15--76. --
\newblock Access mode:
\url{http://content.algebraicgeometry.nl/2018-1/2018-1-002.pdf}.

\bibitem{Dolgachev-1982}

\emph{Dolgachev~Igor}. \href{http://dx.doi.org/10.1007/BFb0101508}{Weighted
projective varieties}~// {Group actions and vector fields ({V}ancouver,
{B}.{C}., 1981)}. --
\newblock Berlin~: Springer, 1982. --
\newblock Vol.~956 of \emph{{Lecture Notes in Math.}} --
\newblock P.~34--71. --
\newblock Access mode: \url{http://dx.doi.org/10.1007/BFb0101508}.

\bibitem{Dolgachev-ClassicalAlgGeom}
\emph{Dolgachev~Igor~V.}
\href{http://dx.doi.org/10.1017/CBO9781139084437}{Classical algebraic
geometry}. --
\newblock Cambridge~: Cambridge University Press, 2012. --
\newblock P.~xii+639. --
\newblock
ISBN:~\href{http://isbndb.com/search-all.html?kw=978-1-107-01765-8}{978-1-107-01765-8}.
--
\newblock Access mode: \url{http://dx.doi.org/10.1017/CBO9781139084437}.

\bibitem{Eisenbud1995}
\emph{Eisenbud~David}. Commutative algebra. {W}ith a view toward algebraic
geometry. --
\newblock Berlin: Springer-Verlag, 1995. --
\newblock Vol.~150. --
\newblock P.~xvi + 785. --
\newblock ISBN:~\href{http://isbndb.com/search-all.html?kw=3-540-94269-6/pbk;
3-540-94268-8/hbk}{3-540-94269-6/pbk; 3-540-94268-8/hbk}.

\bibitem{Fano1931}
\emph{{Fano}~G.} Sulle variet{\`a} algebriche a tre dimensioni aventi tutti
i generi nulli~// Atti Congresso Bologna. --
\newblock 1931. --
\newblock Vol.~4. --
\newblock P.~115--121.

\bibitem{Fano1942}
\emph{Fano~Gino}. Su alcune variet{\`a} algebriche a tre dimensioni
razionali, e aventi curve-sezioni canoniche~//
\href{http://dx.doi.org/10.1007/BF02565618}{\emph{Comment. Math. Helv.}}
--
\newblock 1942. --
\newblock Vol.~14. --
\newblock P.~202--211.

\bibitem{Utah}
Flips and abundance for algebraic threefolds~/ Ed.\ by\ J{\'a}nos~Koll{\'a}r.
--
\newblock Paris~: Soci{\'e}t{\'e} Math{\'e}matique de France, 1992. --
\newblock P.~1--258. --
\newblock Papers from the Second Summer Seminar on Algebraic Geometry held at
the University of Utah, Salt Lake City, Utah, August 1991, Ast{\'e}risque No.
211 (1992).

\bibitem{Fujita:book}
\emph{Fujita~Takao}. Classification theories of polarized varieties.
--
\newblock Cambridge~: Cambridge University Press, 1990. --
\newblock Vol.~155 of \emph{{London Mathematical Society Lecture Note
Series}}. --
\newblock P.~xiv+205. --
\newblock
ISBN:~\href{http://isbndb.com/search-all.html?kw=0-521-39202-0}{0-521-39202-0}.



\bibitem{Griffiths1994}
\emph{Griffiths~Phillip, Harris~Joseph}. Principles of algebraic geometry.
{Wiley Classics Library}. --
\newblock New York~: John Wiley \& Sons Inc., 1994. --
\newblock P.~xiv+813. --
\newblock
ISBN:~\href{http://isbndb.com/search-all.html?kw=0-471-05059-8}{0-471-05059-8}.
--
\newblock Reprint of the 1978 original.

\bibitem{Gushelcprime1982}
\emph{Gushel~N.~P.} {Fano varieties of genus {$6$}}~// \emph{Izv. Akad.
Nauk SSSR Ser. Mat.} --
\newblock 1982. --
\newblock Vol.~46, no.~6. --
\newblock P.~1159--1174, 1343.

\bibitem{Gushel:g8}
\emph{Gushel~N.~P.} {Fano varieties of genus {$8$}}~// \emph{Uspekhi Mat.
Nauk}. --
\newblock 1983. --
\newblock Vol.~38, no. 1(229). --
\newblock P.~163--164.

\bibitem{Gushel:g8:new}
\emph{Gushel~N.~P.} Fano {$3$}-folds of genus {$8$}~// \emph{Algebra i
Analiz}. --
\newblock 1992. --
\newblock Vol.~4, no.~1. --
\newblock P.~120--134.

\bibitem{Hartshorn-1977-ag}
\emph{Hartshorne~Robin}. Algebraic geometry. --
\newblock New York~: Springer-Verlag, 1977. --
\newblock P.~xvi+496. --
\newblock
ISBN:~\href{http://isbndb.com/search-all.html?kw=0-387-90244-9}{0-387-90244-9}.
--
\newblock Graduate Texts in Mathematics, No. 52.


\bibitem{Hidaka-Watanabe-1981}
\emph{Hidaka~Fumio, Watanabe~Keiichi}. Normal {G}orenstein surfaces with
ample anti-canonical divisor~// \emph{Tokyo J. Math.} --
\newblock 1981. --
\newblock Vol.~4, no.~2. --
\newblock P.~319--330.

\bibitem{Huybrechts:K3}
\emph{Huybrechts~Daniel}.
\href{http://dx.doi.org/10.1017/CBO9781316594193}{Lectures on {$K$3}
surfaces}. --
\newblock Cambridge: Cambridge University Press, 2016. --
\newblock P.~xi + 485. --
\newblock
ISBN:~\href{http://isbndb.com/search-all.html?kw=978-1-107-15304-2/hbk;
978-1-316-59419-3/ebook}{978-1-107-15304-2/hbk; 978-1-316-59419-3/ebook}.
--
\newblock online; accessed:
\url{http://www.math.uni-bonn.de/people/huybrech/K3Global.pdf}.



\bibitem{Isk:Fano1e}
\emph{Iskovskih~V.~A.} Fano threefolds. {I}~// \emph{Izv. Akad. Nauk SSSR
Ser. Mat.} --
\newblock 1977. --
\newblock Vol.~41, no.~3. --
\newblock P.~516--562, 717.

\bibitem{Isk:Fano2e}
\emph{Iskovskih~V.~A.} Fano threefolds. {II}~// \emph{Izv. Akad. Nauk
SSSR Ser. Mat.} --
\newblock 1978. --
\newblock Vol.~42, no.~3. --
\newblock P.~506--549.

\bibitem{Iskovskikh1988}
\emph{Iskovskikh~V.A.} {Lectures on three-dimensional algebraic varieties.
{F}ano varieties. ({L}ektsii po trekhmernym algebraicheskim mnogoobraziyam.
{M}nogoobraziya {F}ano)}. --
\newblock Moskva: Izdatel'stvo Moskovskogo Universiteta, 1988. --
\newblock P.~164. --
\newblock
ISBN:~\href{http://isbndb.com/search-all.html?kw=5-211-00588-0}{5-211-00588-0}.

\bibitem{Iskovskikh1990}
\emph{Iskovskikh~V.A.} {A double projection from a line on {F}ano threefolds
of the first kind}~// \emph{Math. USSR, Sb.} --
\newblock 1990. --
\newblock Vol.~66. --
\newblock P.~265--284.

\bibitem{Isk:anti-e}
\emph{Iskovskikh~V.~A.} Anticanonical models of three-dimensional algebraic
varieties~// \emph{J. Sov. Math.} --
\newblock 1980. --
\newblock Vol.~13. --
\newblock P.~745--814.

\bibitem{Iskovskikh1980}
\emph{Iskovskikh~.~A.} Birational automorphisms of three-dimensional
algebraic varieties~// \emph{J. Sov. Math.} --
\newblock 1980. --
\newblock Vol.~13. --
\newblock P.~815--868.

\bibitem{IP99}
\emph{Iskovskikh~V.~A., Prokhorov~Yu.} Fano varieties. {A}lgebraic geometry
{V}. --
\newblock Berlin~: Springer, 1999. --
\newblock Vol.~47 of \emph{{Encyclopaedia Math. Sci.}}

\bibitem{Jahnke-Radloff-2006}
\emph{Jahnke~Priska, Radloff~Ivo}. Gorenstein {F}ano threefolds with base
points in the anticanonical system~// \emph{Compos. Math.} --
\newblock 2006. --
\newblock Vol. 142, no.~2. --
\newblock P.~422--432.

\bibitem{Kachi1997}
\emph{Kachi~Yasuyuki}. {Extremal contractions from {$4$}-dimensional
manifolds to {$3$}-folds}~// 
\emph{Ann. Scuola Norm. Sup. Pisa Cl. Sci. (4)}. --
\newblock 1997. --
\newblock Vol.~24, no.~1. --
\newblock P.~63--131. --
\newblock Access mode:
\url{http://www.numdam.org/item?id=ASNSP_1997_4_24_1_63_0}.

\bibitem{Kawamata:bF}
\emph{Kawamata~Yujiro}. Boundedness of {$\mathbf{Q}$}-{F}ano threefolds~//
Proceedings of the {I}nternational {C}onference on {A}lgebra, {P}art 3
({N}ovosibirsk, 1989). --
\newblock Vol.~131 of \emph{Contemp. Math.} --
\newblock Amer. Math. Soc., Providence, RI, 1992. --
\newblock P.~439--445.

\bibitem{Kawamata1997}
\emph{Kawamata~Yujiro}. On {F}ujita's freeness conjecture for {$3$}-folds
and {$4$}-folds~// \emph{Math. Ann.} --
\newblock 1997. --
\newblock Vol. 308, no.~3. --
\newblock P.~491--505.

\bibitem{Kawamata:Adj-1}
\emph{Kawamata~Yujiro}. Subadjunction of log canonical divisors for a
subvariety of codimension {$2$}~// {Birational algebraic geometry (Baltimore,
MD, 1996)}. --
\newblock Providence, RI~: Amer. Math. Soc., 1997. --
\newblock Vol.~207 of \emph{{Contemp. Math.}} --
\newblock P.~79--88.

\bibitem{Kawamata:Adj-2}
\emph{Kawamata~Yujiro}. Subadjunction of log canonical divisors. {II}~//
\emph{Amer. J. Math.} --
\newblock 1998. --
\newblock Vol. 120, no.~5. --
\newblock P.~893--899.

\bibitem{KMM}
\emph{Kawamata~Yujiro, Matsuda~Katsumi, Matsuki~Kenji}.
\href{http://dx.doi.org/10.2969/aspm/01010283}{Introduction to the minimal
model problem}~// Algebraic geometry, {S}endai, 1985. --
\newblock North-Holland, Amsterdam, 1987. --
\newblock Vol.~10 of \emph{Adv. Stud. Pure Math.} --
\newblock P.~283--360. --
\newblock Access mode: \url{https://doi.org/10.2969/aspm/01010283}.

\bibitem{Kleiman1966}
\emph{Kleiman~Steven~L.} {Toward a numerical theory of ampleness}~//
\emph{Ann. of Math. (2)}. --
\newblock 1966. --
\newblock Vol.~84. --
\newblock P.~293--344.

\bibitem{Kodaira:book-def}
\emph{Kodaira~Kunihiko}.
\href{http://dx.doi.org/10.1007/978-1-4613-8590-5}{Complex manifolds and
deformation of complex structures}. --
\newblock Springer-Verlag, New York, 1986. --
\newblock Vol.~283 of \emph{Grundlehren der mathematischen Wissenschaften
[Fundamental Principles of Mathematical Sciences]}. --
\newblock P.~x+465. --
\newblock
ISBN:~\href{http://isbndb.com/search-all.html?kw=0-387-96188-7}{0-387-96188-7}.
--
\newblock Translated from the Japanese by Kazuo Akao, With an appendix by
Daisuke Fujiwara. Access mode:
\url{https://doi.org/10.1007/978-1-4613-8590-5}.

\bibitem{Kollar1985}
\emph{Koll{\'a}r~J{\'a}nos}. {Toward moduli of singular varieties}~//
\emph{Compositio Math.} --
\newblock 1985. --
\newblock Vol.~56, no.~3. --
\newblock P.~369--398.

\bibitem{Kollar:flops}
\emph{Koll{\'a}r~J{\'a}nos}. Flops~//
\href{http://dx.doi.org/10.1017/S0027763000001240}{\emph{Nagoya Math. J.}}
--
\newblock 1989. --
\newblock Vol. 113. --
\newblock P.~15--36. --
\newblock Access mode: \url{https://doi.org/10.1017/S0027763000001240}.

\bibitem{Kollar-1996-RC}
\emph{Koll{\'a}r~J{\'a}nos}. Rational curves on algebraic varieties.
--
\newblock Berlin~: Springer-Verlag, 1996. --
\newblock Vol.~32 of \emph{{Ergebnisse der Mathematik und ihrer
Grenzgebiete. 3. Folge. A Series of Modern Surveys in Mathematics [Results in
Mathematics and Related Areas. 3rd Series. A Series of Modern Surveys in
Mathematics]}}. --
\newblock P.~viii+320. --
\newblock
ISBN:~\href{http://isbndb.com/search-all.html?kw=3-540-60168-6}{3-540-60168-6}.

\bibitem{Kollar95:pairs}
\emph{Koll{\'a}r~J{\'a}nos}. Singularities of pairs~// Algebraic
geometry--Santa Cruz 1995. --
\newblock Providence, RI~: Amer. Math. Soc., 1997. --
\newblock Vol.~62 of \emph{{Proc. Sympos. Pure Math.}} --
\newblock P.~221--287.

\bibitem{Kollar-Miyaoka-Mori-1992c}
\emph{Koll{\'a}r~J{\'a}nos, Miyaoka~Yoichi, Mori~Shigefumi}. {Rational
connectedness and boundedness of {F}ano manifolds}~// \emph{J.
Differential Geom.} --
\newblock 1992. --
\newblock Vol.~36, no.~3. --
\newblock P.~765--779.

\bibitem{KM:book}
\emph{Koll\'{a}r~J\'{a}nos, Mori~Shigefumi}.
\href{http://dx.doi.org/10.1017/CBO9780511662560}{Birational geometry of
algebraic varieties}. --
\newblock Cambridge University Press, Cambridge, 1998. --
\newblock Vol.~134 of \emph{Cambridge Tracts in Mathematics}. --
\newblock P.~viii+254. --
\newblock
ISBN:~\href{http://isbndb.com/search-all.html?kw=0-521-63277-3}{0-521-63277-3}.
--
\newblock With the collaboration of C. H. Clemens and A. Corti, Translated from
the 1998 Japanese original. Access mode:
\url{https://doi.org/10.1017/CBO9780511662560}.

\bibitem{KP:V22}
\emph{Kuznetsov~Alexander, Prokhorov~Yuri}. Prime {F}ano threefolds of genus
{$12$} with a {$\mathbf G_m$}-action~// \emph{\'Epijournal de
G\'eom\'etrie Alg\'ebrique}. --
\newblock 2018. --
\newblock Vol.~2, no. epiga:4560. --
\newblock Access mode: \url{https://epiga.episciences.org/4560}.

\bibitem{KP-rF}
Kuznetsov~Alexander, Prokhorov~Yuri. Rationality of {F}ano threefolds over
 non-closed fields~// \emph{American Journal of Mathematics}. --
\newblock 2023. --
\newblock Vol. 145, no.~2. --
\newblock P.~335--411. --
\newblock Access mode: \url{muse.jhu.edu/article/885814}.

\bibitem{KP-Mu}
\emph{Kuznetsov~Alexander, Prokhorov~Yuri}.
\href{http://dx.doi.org/10.1007/978-3-030-75421-1_10}{Rationality of {M}ukai
varieties over non-closed fields}~// Rationality of Varieties~/ Ed.\ by\
Farkas~G., ~van~der~Geer~G., Shen~M., Taelman~L. --
\newblock Birkh\"auser, Cham, 2021. --
\newblock Vol.~342 of \emph{Progress in Mathematics}. --
\newblock P.~249--290. --
\newblock Access mode: \url{https://doi.org/10.1007/978-3-030-75421-1_10}.

\bibitem{KP:dP}
Kuznetsov~Alexander, Prokhorov~Yuri. On higher-dimensional del {P}ezzo
 varieties~// \href{https://doi.org/10.4213/im9385}{\emph{Izvestiya:
 Math.}} --
\newblock 2023. --
\newblock Vol.~87, no.~3. --
\newblock P.~75--148.

\bibitem{KP:1node}
Kuznetsov~Alexander, Prokhorov~Yuri. $1$-nodal {F}ano threefolds with {P}icard
 number $1$~// \href{https://doi.org/10.4213/im9585e}{\emph{Izvestiya:
 Mathematics}}. --
\newblock 2025. --
\newblock Vol.~89, no.~3. --
\newblock P.~495--594. --
\newblock Access mode: \url{https://doi.org/10.4213/im9585e}.

\bibitem{KPS:Hilb}
\emph{Kuznetsov~Alexander, Prokhorov~Yuri, Shramov~Constantin}. Hilbert
schemes of lines and conics and automorphism groups of {F}ano threefolds~//
\href{http://dx.doi.org/10.1007/s11537-017-1714-6}{\emph{Japanese J.
Math.}} --
\newblock 2018. --
\newblock Vol.~13, no.~1. --
\newblock P.~109--185.

\bibitem{Lazarsfeld2004}
\emph{Lazarsfeld~Robert}.
\href{http://dx.doi.org/10.1007/978-3-642-18808-4}{Positivity in algebraic
geometry. {II}}. --
\newblock Springer-Verlag, Berlin, 2004. --
\newblock Vol.~49 of \emph{{Ergebnisse der Mathematik und ihrer
Grenzgebiete. 3. Folge. A Series of Modern Surveys in Mathematics [Results in
Mathematics and Related Areas. 3rd Series. A Series of Modern Surveys in
Mathematics]}}. --
\newblock P.~xviii+385. --
\newblock
ISBN:~\href{http://isbndb.com/search-all.html?kw=3-540-22534-X}{3-540-22534-X}.
--
\newblock Positivity for vector bundles, and multiplier ideals. Access mode:
\url{http://dx.doi.org/10.1007/978-3-642-18808-4}.

\bibitem{Lazarsfeld2010}
\emph{Lazarsfeld~Robert}. {A short course on multiplier ideals}~// {Analytic
and algebraic geometry. Common problems, different methods. Lecture notes
from the Park City Mathematics Institute (PCMI) graduate summer school on
analytic and algebraic geometry, Park City, UT, USA, Summer 2008}. --
\newblock Providence, RI: American Mathematical Society (AMS), 2010. --
\newblock P.~451--494. --
\newblock online; accessed: \url{http://arxiv.org/abs/0901.0651}.

\bibitem{Manin:book:74}
\emph{Manin~Yu.~I.} Cubic forms: algebra, geometry, arithmetic. --
\newblock Amsterdam~: North-Holland Publishing Co., 1974. --
\newblock P.~vii+292. --
\newblock
ISBN:~\href{http://isbndb.com/search-all.html?kw=0-7204-2456-9}{0-7204-2456-9}.
--
\newblock Translated from the Russian by M. Hazewinkel, North-Holland
Mathematical Library, Vol. 4.

\bibitem{Matsumura1986}
\emph{Matsumura~Hideyuki}. {Commutative ring theory}. --
\newblock Cambridge~: Cambridge University Press, 1986. --
\newblock Vol.~8 of \emph{{Cambridge Studies in Advanced Mathematics}}.
--
\newblock P.~xiv+320. --
\newblock
ISBN:~\href{http://isbndb.com/search-all.html?kw=0-521-25916-9}{0-521-25916-9}.
--
\newblock Translated from the Japanese by M. Reid.

\bibitem{Matsusaka1972}
\emph{Matsusaka~T.} {Polarized varieties with a given {H}ilbert
polynomial}~// \emph{Amer. J. Math.} --
\newblock 1972. --
\newblock Vol.~94. --
\newblock P.~1027--1077.

\bibitem{Mehta-Ramanathan:1982}
\emph{{Mehta}~V.B., Ramanathan~A.} Semistable sheaves on projective
varieties and their restriction to curves~//
\href{http://dx.doi.org/10.1007/BF01450677}{\emph{{Math. Ann.}}} --
\newblock 1982. --
\newblock Vol. 258. --
\newblock P.~213--224.

\bibitem{Mella-1999}
\emph{Mella~Massimiliano}. Existence of good divisors on {M}ukai
varieties~// \emph{J. Algebr. Geom.} --
\newblock 1999. --
\newblock Vol.~8, no.~2. --
\newblock P.~197--206.

\bibitem{Mishezon1967}
\emph{{Moishezon}~B.G.} {{\"U}ber algebraische Homologieklassen auf
algebraischen Mannigfaltigkeiten.}~// \emph{{Izv. Akad. Nauk SSSR, Ser.
Mat.}} --
\newblock 1967. --
\newblock Vol.~31. --
\newblock P.~225--268.



\bibitem{Moishezon:CastelnuovoEnriques}
\emph{Moishezon~B.~G.} The {C}astelnuovo-{E}nriques contraction theorem for
arbitrary dimension~//
\href{http://dx.doi.org/10.1070/IM1969v003n05ABEH000810}{\emph{Math.
USSR-Izvestiya}}. --
\newblock 1969. --
\newblock Vol.~3, no.~5. --
\newblock P.~917–--966.

\bibitem{Mumford-Lectures-on-curves}
\emph{Mumford~David}. {Lectures on curves on an algebraic surface}. {With a
section by G. M. Bergman. Annals of Mathematics Studies, No. 59}. --
\newblock Princeton, N.J.~: Princeton University Press, 1966. --
\newblock P.~xi+200.

\bibitem{Mumford:red}
\emph{{Mumford}~David}. {The red book of varieties and schemes. Includes the
Michigan lectures (1974) on ``Curves and their Jacobians''. 2nd, expanded ed.
with contributions by Enrico Arbarello.} --
\newblock 2nd, expanded ed. with contributions by Enrico Arbarello edition.
--
\newblock Berlin: Springer, 1999. --
\newblock Vol.~1358. --
\newblock P.~x + 306. --
\newblock
ISBN:~\href{http://isbndb.com/search-all.html?kw=3-540-63293-X/pbk}{3-540-63293-X/pbk}.



\bibitem{Mori-1975}
\emph{Mori~Shigefumi}. {On a generalization of complete intersections}~//
\emph{J. Math. Kyoto Univ.} --
\newblock 1975. --
\newblock Vol.~15, no.~3. --
\newblock P.~619--646.

\bibitem{Mori:3-folds}
\emph{Mori~Shigefumi}. Threefolds whose canonical bundles are not
numerically effective~//
\href{http://dx.doi.org/10.2307/2007050}{\emph{Ann. Math. (2)}}. --
\newblock 1982. --
\newblock Vol. 116. --
\newblock P.~133--176. --
\newblock online; accessed: \url{https://doi.org/10.2307/2007050}.

\bibitem{Mori-Mukai-1981-82}
\emph{Mori~Shigefumi, Mukai~Shigeru}. Classification of {F}ano {$3$}-folds
with {$B\sb{2}\geq 2$}~// \emph{Manuscripta Math.} --
\newblock 1981/82. --
\newblock Vol.~36, no.~2. --
\newblock P.~147--162. --
\newblock Erratum: {M}anuscripta {M}ath. 110 (2003), 407.

\bibitem{Mori-Mukai1983}
\emph{Mori~Shigefumi, Mukai~Shigeru}. On {F}ano {$3$}-folds with
{$B\sb{2}\geq 2$}~// {Algebraic varieties and analytic varieties (Tokyo,
1981)}. --
\newblock Amsterdam~: North-Holland, 1983. --
\newblock Vol.~1 of \emph{{Adv. Stud. Pure Math.}} --
\newblock P.~101--129.

\bibitem{Mukai:CK3F}
\emph{Mukai~Shigeru}. Curves, {$K3$} surfaces and {F}ano {$3$}-folds of
genus {$\leq 10$}~// Algebraic geometry and commutative algebra, Vol.\ I.
--
\newblock Tokyo~: Kinokuniya, 1988. --
\newblock P.~357--377.

\bibitem{Mukai-1989}
\emph{Mukai~Shigeru}. Biregular classification of {F}ano {$3$}-folds and
{F}ano manifolds of coindex {$3$}~//
\href{http://dx.doi.org/10.1073/pnas.86.9.3000}{\emph{Proc. Nat. Acad.
Sci. U.S.A.}} --
\newblock 1989. --
\newblock Vol.~86, no.~9. --
\newblock P.~3000--3002. --
\newblock Access mode: \url{https://doi.org/10.1073/pnas.86.9.3000}.

\bibitem{Mukai-1992}
\emph{Mukai~Shigeru}.
\href{http://dx.doi.org/10.1017/CBO9780511662652.018}{Fano {$3$}-folds}~//
Complex projective geometry (Trieste, 1989/Bergen, 1989). --
\newblock Cambridge~: Cambridge Univ. Press, 1992. --
\newblock Vol.~179 of \emph{{London Math. Soc. Lecture Note Ser.}} --
\newblock P.~255--263. --
\newblock Access mode: \url{https://doi.org/10.1017/CBO9780511662652.018}.

\bibitem{mukai-1995-1}
\emph{Mukai~Shigeru}. Curves and symmetric spaces. {I}~//
\href{http://dx.doi.org/10.2307/2375032}{\emph{Amer. J. Math.}} --
\newblock 1995. --
\newblock Vol. 117, no.~6. --
\newblock P.~1627--1644.

\bibitem{Mukai-2002}
\emph{Mukai~Shigeru}. {New developments in the theory of {F}ano threefolds:
vector bundle method and moduli problems [translation of {S}\=ugaku {\bf 47}
(1995), no.\ 2, 125--144]}~// \emph{Sugaku Expositions}. --
\newblock 2002. --
\newblock Vol.~15, no.~2. --
\newblock P.~125--150.

\bibitem{Mukai-Umemura-1983}
\emph{Mukai~Shigeru, Umemura~Hiroshi}. Minimal rational threefolds~//
{Algebraic geometry (Tokyo/Kyoto, 1982)}. --
\newblock Berlin~: Springer, 1983. --
\newblock Vol.~1016 of \emph{{Lecture Notes in Math.}} --
\newblock P.~490--518.

\bibitem{Peternell-Wisniewski:tF}
\emph{Peternell~Thomas, Wi{\'s}niewski~Jaros{\l}aw~A.} On stability of
tangent bundles of {F}ano manifolds with {$b\sb 2=1$}~// \emph{J.
Algebraic Geom.} --
\newblock 1995. --
\newblock Vol.~4, no.~2. --
\newblock P.~363--384.


\bibitem{P:90aut:en}
\emph{Prokhorov~Yu.} Automorphism groups of {F}ano 3-folds~//
\emph{Russian Math. Surveys}. --
\newblock 1990. --
\newblock Vol.~45, no.~3. --
\newblock P.~222--223.

\bibitem{Prokhorov-1990b}
\emph{Prokhorov~Yu.} {Exotic {F}ano varieties}~// \emph{Moscow Univ.
Math. Bull.} --
\newblock 1990. --
\newblock Vol.~45, no.~3. --
\newblock P.~36--38.

\bibitem{Prokhorov1995a}
\emph{Prokhorov~Yu.} {On the existence of good divisors on {F}ano varieties
of coindex {$3$}}~// \emph{Trudy Mat. Inst. Steklov.} --
\newblock 1995. --
\newblock Vol. 208, no. Teor. Chisel, Algebra i Algebr. Geom. --
\newblock P.~266--277. --
\newblock Dedicated to Academician Igor\cprime\ Rostislavovich Shafarevich on
the occasion of his seventieth birthday (Russian).

\bibitem{P:book:sing-re}
\emph{Prokhorov~Yu.} Singularities in algebraic geometry. --
\newblock Moscow~: MCCME, 2009. --
\newblock P.~128. --
\newblock in Russian.

\bibitem{Prokhorov-re-rat-surf}
\emph{Prokhorov~Yu.} \href{http://dx.doi.org/10.4213/book1590}{Rational
surfaces}. --
\newblock Moscow~: Steklov Math. Inst., RAS, 2015. --
\newblock Vol.~24 of \emph{{Lekts. Kursy NOC}}. --
\newblock P.~78. --
\newblock
ISBN:~\href{http://isbndb.com/search-all.html?kw=978-5-98419-063-3}{978-5-98419-063-3}.
--
\newblock in Russian. online; accessed:
\url{http://mi.mathnet.ru/eng/book1590}.

\bibitem{P:factorial-Fano:e}
\emph{Prokhorov~Yuri.} On the number of singular points of terminal
factorial {F}ano threefolds~//
\href{http://dx.doi.org/10.1134/S0001434617050364}{\emph{Math. Notes}}.
--
\newblock 2017. --
\newblock Vol. 101, no. 5-6. --
\newblock P.~1068--1073.

\bibitem{P:G-MMP}
\emph{Prokhorov~Yuri}. Equivariant minimal model program~//
\href{http://dx.doi.org/10.1070/rm9990}{\emph{Russian Math. Surv.}}
--
\newblock 2021. --
\newblock Vol.~76, no.~3. --
\newblock P.~461--542. --
\newblock Access mode: \url{https://doi.org/10.1070/rm9990}.


\bibitem{P:JAG:simple}
\emph{Prokhorov~Yu.} Simple finite subgroups of the {C}remona group of rank
$3$~// \href{http://dx.doi.org/10.1090/S1056-3911-2011-00586-9}{\emph{J.
Algebraic Geom.}} --
\newblock 2012. --
\newblock Vol.~21, no.~3. --
\newblock P.~563--600.

\bibitem{P:rat-cb:e}
\emph{Prokhorov~Yuri}. The rationality problem for conic bundles~//
\href{http://dx.doi.org/10.1070/RM9811}{\emph{Russian Math. Surv.}}
--
\newblock 2018. --
\newblock Vol.~73, no.~3. --
\newblock P.~375--456. --
\newblock online; accessed: \url{http://dx.doi.org/10.1070/RM9811}.

\bibitem{Przhiyalkovskij-Cheltsov-Shramov-2005en}
\emph{Przhiyalkovskij~V.V., Chel'tsov~I.A., Shramov~K.A.} Hyperelliptic and
trigonal {F}ano threefolds~// \emph{Izv. Math.} --
\newblock 2005. --
\newblock Vol.~69, no.~2. --
\newblock P.~365--421.

\bibitem{Reid:MM}
\emph{Reid~Miles}. \href{http://dx.doi.org/10.2969/aspm/00110131}{Minimal
models of canonical {$3$}-folds}~// Algebraic varieties and analytic
varieties ({T}okyo, 1981). --
\newblock North-Holland, Amsterdam, 1983. --
\newblock Vol.~1 of \emph{Adv. Stud. Pure Math.} --
\newblock P.~131--180. --
\newblock online; accessed: \url{https://doi.org/10.2969/aspm/00110131}.

\bibitem{Reid:Kaw}

\emph{Reid~M.} Projective morphisms according to {K}awamata. --
\newblock Preprint. {U}niv. Warwick. --
\newblock 1983. --
\newblock Access mode: \url{http://www.maths.warwick.ac.uk/~miles/3folds}.

\bibitem{Saint-Donat-1974}

\emph{Saint-Donat~B.} Projective models of {$K3$} surfaces~// \emph{Amer.
J. Math.} --
\newblock 1974. --
\newblock Vol.~96. --
\newblock P.~602--639.

\bibitem{Schreyer2001}

\emph{Schreyer~Frank-Olaf}. Geometry and algebra of prime {F}ano $3$-folds
of genus $12$~// \emph{Compositio Math.} --
\newblock 2001. --
\newblock Vol. 127, no.~3. --
\newblock P.~297--319.




\bibitem{Shafarevich:basic}
\emph{Shafarevich~I.~R.} Basic algebraic geometry. {Translated} from the
{Russian} by {K}. {A}. {Hirsch}. --
\newblock Springer, Cham, 1974. --
\newblock Vol.~213 of \emph{Grundlehren Math. Wiss.}


\bibitem{Shepherd-Barron1997}
\emph{Shepherd-Barron~N.~I.} {Fano threefolds in positive characteristic}~//
\emph{Compositio Math.} --
\newblock 1997. --
\newblock Vol. 105, no.~3. --
\newblock P.~237--265.

\bibitem{Shokurov-1971}
\emph{Shokurov~V.~V.} The {Noether}-{Enriques} theorem on canonical
curves~//
\href{http://dx.doi.org/10.1070/SM1971v015n03ABEH001552}{\emph{Math. USSR,
Sb.}} --
\newblock 1972. --
\newblock Vol.~15. --
\newblock P.~361--403.

\bibitem{Shokurov:eleph}
\emph{Shokurov~V.~V.} Smoothness of a general anticanonical divisor on a
{F}ano variety~//
\href{http://dx.doi.org/10.1070/IM1980v014n02ABEH001123}{\emph{{Math.
USSR, Izv.}}} --
\newblock 1980. --
\newblock Vol.~14. --
\newblock P.~395--405.

\bibitem{Shokurov1980a}
\emph{Shokurov~V.~V.} {The existence of a straight line on {F}ano
{$3$}-folds}~//
\href{http://dx.doi.org/10.1070/IM1980v015n01ABEH001195}{\emph{Math.
USSR-Izvestiya}}. --
\newblock 1980. --
\newblock Vol.~15, no.~1. --
\newblock P.~173--209.

\bibitem{Shokurov:flips}
\emph{Shokurov~V.~V.} $3$-fold log flips~//
\href{http://dx.doi.org/10.1070/im1993v040n01abeh001862}{\emph{Izvestiya
Math.}} --
\newblock 1993. -- feb. --
\newblock Vol.~40, no.~1. --
\newblock P.~95--202. --
\newblock online; accessed:
\url{https://doi.org/10.1070/im1993v040n01abeh001862}.

\bibitem{Springer1977}
\emph{Springer~T.~A.} {Invariant theory}. {Lecture Notes in Mathematics,
Vol. 585}. --
\newblock Berlin~: Springer-Verlag, 1977. --
\newblock P.~iv+112.

\bibitem{Steffens}
\emph{Steffens~A.} On the stability of the tangent bundle of {F}ano
manifolds~// \href{http://dx.doi.org/10.1007/BF01446311}{\emph{Math.
Ann.}} --
\newblock 1996. --
\newblock Vol. 304, no.~4. --
\newblock P.~635--643. --
\newblock Access mode: \url{https://doi.org/10.1007/BF01446311}.

\bibitem{Takeuchi-1989}
\emph{Takeuchi~Kiyohiko}. Some birational maps of {F}ano {$3$}-folds~//
\emph{Compositio Math.} --
\newblock 1989. --
\newblock Vol.~71, no.~3. --
\newblock P.~265--283.



\bibitem{Tan4}
Tanaka~Hiromu. Fano threefolds in positive characteristic IV~// \emph{Arxiv
 e-print}. --
\newblock 2023. --
\newblock \href{https://arxiv.org/abs/2308.08127}{2308.08127}.

\bibitem{Tan1}
Tanaka~Hiromu. Fano threefolds in positive characteristic I~// \emph{Arxiv
 e-print}. --
\newblock 2024. --
\newblock \href{https://arxiv.org/abs/2308.08121}{2308.08121}.

\bibitem{Tan2}
Tanaka~Hiromu. Fano threefolds in positive characteristic II~// \emph{Arxiv
 e-print}. --
\newblock 2024. --
\newblock \href{https://arxiv.org/abs/2308.08122}{2308.08122}.

\bibitem{Tregub1985a}
\emph{Tregub~S.L.} {Construction of a birational isomorphism of a cubic
threefold and {F}ano variety of the first kind with $g=8$, associated with a
normal rational curve of degree 4}~// \emph{Mosc. Univ. Math. Bull.}
--
\newblock 1985. --
\newblock Vol.~40, no.~6. --
\newblock P.~78--80.

\bibitem{Ueno1975}
\emph{Ueno~Kenji}. Classification theory of algebraic varieties and compact
complex spaces. Lecture Notes in Mathematics, Vol. 439. --
\newblock Springer-Verlag, Berlin-New York, 1975. --
\newblock P.~xix+278. --
\newblock Notes written in collaboration with P. Cherenack.

\bibitem{Wilson1987}
\emph{Wilson~P. M.~H.} {Fano fourfolds of index greater than one}~//
\emph{J. Reine Angew. Math.} --
\newblock 1987. --
\newblock Vol. 379. --
\newblock P.~172--181.














\end{thebibliography}
\end{document}